\pdfoutput=1

\documentclass[11pt]{article}

\usepackage{amsmath,amsthm,verbatim,amssymb,amsfonts,amscd}
\usepackage{empheq} 
\usepackage{mathabx}
\usepackage{mathrsfs} 
\usepackage{authblk} 
\usepackage{mathtools} 
\usepackage[shortlabels]{enumitem}

\usepackage{hyperref}

\usepackage[utf8]{inputenc}
\usepackage[T1]{fontenc}

\usepackage{tikz}
\usepackage{pgfplots}
\usepgfplotslibrary{fillbetween}
\usetikzlibrary{plotmarks}
\usetikzlibrary{patterns}

\theoremstyle{plain}
\newtheorem{theorem}{Theorem}[section]

\newtheorem{lemma}[theorem]{Lemma}
\newtheorem{proposition}[theorem]{Proposition}

\theoremstyle{definition}

\numberwithin{equation}{section}

\newcommand{\R}{\mathbb{R}}
\newcommand{\C}{\mathbb{C}}
\newcommand{\N}{\mathbb{N}}
\newcommand{\Z}{\mathbb{Z}}
\newcommand{\Q}{\mathbb{Q}}

\newcommand{\mx}{\mathrm{d}x}
\newcommand{\my}{\mathrm{d}y}
\newcommand{\mz}{\mathrm{d}z}

\newcommand{\mxi}{\mathrm{d}\xi}

\DeclareMathOperator{\Const}{C}           
\DeclareMathOperator{\supp}{supp}         
\DeclareMathOperator{\sgn}{sgn}           

\renewcommand{\Re}{\operatorname{Re}}

\DeclareMathOperator{\FT}{\mathcal{F}}           

\begin{document}

\title{Mixed norm Strichartz-type estimates for hypersurfaces in three dimensions}

\author{Ljudevit Palle
\thanks{The author has been supported by Deutsche Forschungsgemeinschaft,
project number 237750060, \emph{Fragen der Harmonischen Analysis im Zusammenhang mit Hyperfl\"{a}chen.}
}
}

\affil{Christian-Albrechts-Universit\"{a}t zu Kiel}

\date{\today}


\maketitle

\begin{abstract}
In their work \cite{IM16} I.A. Ikromov and D. M\"{u}ller proved the full range $L^p-L^2$
Fourier restriction estimates for a very general class of hypersurfaces in $\R^3$
which includes the class of real analytic hypersurfaces.
In this article we partly extend their results to the mixed norm case where the coordinates are split in two
directions, one tangential and the other normal to the surface at a fixed given point.
In particular, we resolve completely the adapted case and partly the non-adapted case.
In the non-adapted case the case when the linear height $h_\text{lin}(\phi)$ is below two is settled completely.
\end{abstract}

\tableofcontents




\section{Introduction}

For a given smooth hypersurface $S$ in $\R^n$, its surface measure $\mathrm{d}\sigma$,
and a smooth compactly supported function $\rho \geq 0$, $\rho \in C_0^\infty(S)$,
the associated Fourier restriction problem asks for which $p,q \in [1,\infty]$ the estimate
\begin{equation}
\label{Introduction_FRP_general}
\Bigg( \int |\widehat{f}|^q \rho \mathrm{d}\sigma \Bigg)^{1/q} \leq \Const_{p,q} \Vert f \Vert_{L^p(\R^n)}, \qquad f \in \mathcal{S}(\R^n),
\end{equation}
holds true.
This problem was first considered by E.M. Stein in the late 1960s. Soon thereafter
the problem was essentially solved for curves in two dimensions, see \cite{Fef70}, \cite{CS72}, \cite{Zyg74}.
The higher dimensional case in its most general form is still wide open.
The three dimensional case, as of yet, is far from being completely understood even when $S$ is the sphere,
and there has been a lot of deep work in the direction of understanding $L^p-L^q$ estimates
for surfaces with both vanishing and non-vanishing Gaussian curvature.
A small sample of such works are \cite{Bou91}, \cite{W95}, \cite{MVV96}, \cite{W01}, \cite{T03}, \cite{LV10}, \cite{BG11}, \cite{Gu16}.

The case when $q = 2$ has proven to be more tractable since one can use the ``$R^* R$ technique''.
This was exploited by P.A. Tomas and E.M. Stein (see \cite{To75}) to obtain the full range of $L^p-L^2$
estimates when the hypersurface in question is the unit sphere, and later further developed by A. Greenleaf
in \cite{Grl81} where the full range of $L^p-L^2$ estimates was obtained for surfaces with non-vanishing
Gaussian curvature. In fact, Greenleaf proved that if one has a decay estimate on the Fourier transform of $\rho \mathrm{d}\sigma$
(which can be interpreted as a uniform estimate for an oscillatory integral), i.e.,
\begin{equation*}
|\widehat{\rho \mathrm{d}\sigma} (\xi)| \lesssim (1 + |\xi|)^{-1/h}, \qquad \xi \in \R^n,
\end{equation*}
then the associated restriction estimate holds true for $p' \geq 2(h+1)$ and $q = 2$. However,
in general, this range is not optimal. Recently I.A. Ikromov and D. M\"{u}ller in their series of works
(see \cite{IM11a}, \cite{IM11b}, \cite{IM16}, and also their work with M. Kempe \cite{IKM10})
have developed techniques for proving the full range of $L^p-L^2$ estimates for a very general class of surfaces.
Their work builds upon the work of V.I. Arnold and his school (in particular, the work by Varchenko \cite{Var76}) which highlighted
the importance of the Newton polyhedron within problems involving oscillatory integrals,
and upon the work of D.H. Phong and E.M. Stein \cite{PS97} and D.H. Phong, E.M. Stein, and J.A. Sturm \cite{PSS97}
in the real analytic case where the authors in addition to the Newton polyhedron used the Puiseux series expansions
of roots to obtain results on oscillatory integral operators.
For further and more detailed references we refer the reader to \cite{IM16}.

In \cite{IM16} Ikromov and M\"{u}ller proved the following theorem.
\begin{theorem}
\label{thm_main_result_IM16}
Let $S$ be a smooth hypersurface in $\R^3$ and $\mathrm{d}\sigma$ its surface measure.
After localisation and a change of coordinates assume that $S$ is given as the graph of a smooth function $\phi : \Omega \to \R$ of finite
type with $\phi(0) = 0$ and $\nabla\phi(0) = 0$, where $\Omega \subseteq \R^2$ is an open
neighbourhood of $0$. Furthermore, assume that $\phi$ is linearly adapted in its
original coordinates.
Let $\rho \geq 0$, $\rho \in C_c^\infty(S)$, be a smooth compactly supported function.
Then the estimate \eqref{Introduction_FRP_general} holds true for all $\rho$ with support contained in a sufficiently small neighbourhood of $0$
when $q=2$ and when either
\begin{enumerate}[(a),itemsep=0mm]
\item $\phi$ is adapted in its original coordinates and $p \geq 2(h(\phi)+1)$, or
\item $\phi$ is not adapted in its original coordinates, satisfies the Condition (R), and $p \geq 2(h^{\text{res}}(\phi)+1)$.
\end{enumerate}
\end{theorem}
Since linear transformations respect the Fourier transform, one can always assume linear adaptedness.
The quantities $h(\phi)$ and $h^\text{res}(\phi)$ are respectively the height and the restriction height of the function $\phi$
(the precise definitions can be found in Subsections \ref{Notation_and_assumptions} and \ref{Knapp_ratio_fixed} below respectively;
also note that we use $h^{\text{res}}(\phi)$ to denote the restriction height of the function $\phi$ instead of $h^{\text{r}}(\phi)$ as in \cite{IM16}).
Condition (R) is a factorisation condition which is true for
real analytic functions, but not for general smooth functions, and it remains open whether
this condition can be removed in the above theorem.

In this paper we shall be interested in the mixed norm case with $L^p(\R^3)$ denoting from now on the space
$L^{p_3}_{x_3} (L^{p_2}_{x_2} (L^{p_1}_{x_1}))$
and $q = 2$ in $\eqref{Introduction_FRP_general}$.
We shall be interested in the particular case when $p_1=p_2$, i.e., we only differentiate between the tangential and the normal
direction to the surface $S$ at the point $0 \in S$.
This means we take $\Vert f \Vert_{L^p(\R^3)}$ to mean
\begin{equation*}
\Bigg( \int \Bigg( \int \int |f|^{p_1}(x_1,x_2,x_3) \mx_1 \mx_2 \Bigg)^{p_3/p_1} \mx_3 \Bigg)^{1/p_3}.
\end{equation*}
Henceforth we shall denote by $p$ the pair $(p_1, p_3)$.
Our task is to determine for which $(p_1,p_3)$ the inequality
\begin{equation}
\label{Introduction_FRP_mixed}
\Bigg( \int |\widehat{f}|^2 \rho \mathrm{d}\sigma \Bigg)^{1/2} \leq
\Const_{p} \Vert f \Vert_{L^p(\R^3)} = \Const_{p} \Vert f \Vert_{L^{p_3}_{x_3} (L^{p_1}_{(x_1,x_2)})}, \qquad f \in \mathcal{S}(\R^3),
\end{equation}
holds true for $\rho \geq 0$ supported in a sufficiently small neighbourhood of $0$.

This question is of great interest in the theory of PDEs, as was noticed by Strichartz in \cite{Str77}.
Namely, one can obtain mixed norm Strichartz estimates for a wide collection of symbols $\phi$ determining the surface $S$
since the estimate \eqref{Introduction_FRP_mixed} can be reinterpreted as an a priori estimate
\begin{align*}
\Vert u \Vert_{L^p_{(x,t)}(\R^3)} \leq C \Vert g\Vert_{L^2(\R^2)}
\end{align*}
for the Cauchy problem
\begin{align*}
\begin{cases}
\partial_t u(x,t) &= i \phi(D) u(x,t), \quad (x,t) \in \R^2 \times \R,\\
\,\,\,\, u(x,0) &= g(x), \qquad \qquad \qquad x \in \R^2,
\end{cases}
\end{align*}
where $g$ has its Fourier transform supported in a small neighbourhood of the origin
and $\phi(D)$ is the operator with symbol $\phi(\xi)$.

It turns out that we can use the same basic techniques and phase space decompositions
as in \cite{IM16} in proving the estimate \eqref{Introduction_FRP_mixed} in the cases we consider
(namely, the adapted case and the non-adapted case with $h_\text{lin}(\phi) < 2$).
The main additional ingredients we shall use are some basic ideas from \cite{GV92} (see also \cite{KT98}) for handling mixed norms.
In our case additional complications appear which were absent in the corresponding cases in \cite{IM16}
and some of which resemble problems appearing in some of the final chapters of \cite{IM16}.
For example, after making a phase space decomposition of the kernel of the convolution operator obtained by the ``$R^* R$ technique'',
a recurring theme will be that we will not be able to sum absolutely the operators associated to the kernel decomposition pieces
whose operators were absolutely summable \cite{IM16}.
A further interesting feature of the mixed norm case is that estimates for the mixed norm endpoint for operators of certain kernel pieces
become invariant under scalings considered in \cite{IM16}.

The structure of this article is as follows.
In the following Subsection \ref{Notation_and_assumptions} we review some fundamental concepts such as
the Newton polyhedron and adapted coordinates.
In Subsection \ref{Statement} we state the main results of this paper,
namely Theorem \ref{Introduction_thm_Knapp} which states the necessary conditions, and
Theorem \ref{Introduction_theorem_FRP} which gives us the mixed norm Fourier restriction estimates in the adapted case
and the case $h_{\text{lin}}(\phi) < 2$.
In Section \ref{Knapp} we derive the necessary conditions (by means of Knapp-type examples) for the exponents in \eqref{Introduction_FRP_mixed}.
See Proposition \ref{Knapp_main_proposition}.
In Subsection \ref{Knapp_h_lin_less_than_two} we also determine explicitly the Newton polyhedra of $\phi$ in its original and adapted coordinates
in the case when the linear height of $\phi$ is strictly less than $2$.
Section \ref{Technical_results} contains auxiliary results that we shall often refer to.
In Subsection \ref{Oscillatory_auxiliary} we list results related to oscillatory integrals,
such as the van der Corput lemma, and also some results on oscillatory sums from \cite{IM16} that are useful in conjunction with complex interpolation.
In Subsection \ref{Mixed_norm_auxiliary} we state results which we need for handling mixed norms.
In Section \ref{Adapted}, Proposition \ref{Adapted_proposition_adapted}, we deal with the adapted case, i.e., we prove that if $\phi$ is adapted in its original coordinates,
then the estimate \eqref{Introduction_FRP_mixed} holds for all $p$'s determined by the necessary conditions, except occasionally for a certain endpoint.
In the same section (see Proposition \ref{Adapted_proposition_reduction}) we also reduce the general non-adapted case to considering the part near the principal root jet of $\phi$.
In Sections \ref{Section_h_lin_less_than_2} and \ref{Section_Airy} we handle the case when the linear height of $\phi$ is strictly less than $2$
for a class of functions $\phi$ which includes all analytic functions (see Theorem \ref{Section_h_lin_less_than_2_main_theorem} for a precise formulation).

For reasons of consistency we use the same notational conventions as in \cite{IM16}.
We use the ``variable constant'' notation meaning that constants appearing in calculations
and in the course of our arguments may have different values on different lines.
Furthermore we use the symbols $\sim, \lesssim, \gtrsim, \ll, \gg$ in order to avoid writing down constants.
If we have two nonnegative quantities $A$ and $B$, then
by $A \ll B$ we mean that there is a sufficiently small positive constant $c$ such that $A \leq cB$,
by $A \lesssim B$ we mean that there is a (possibly large) positive constant $C$ such that $A \leq CB$,
and by $A \sim B$ we mean that there are positive constants $C_1 \leq C_2$ such that $C_1 A \leq B \leq C_2 A$.
One defines analogously $A \gg B$ and $A \gtrsim B$.
Often the constants $c$ and $C$ shall depend on certain parameters $p$
in which case we occasionally write $A \ll_p B$, $A \lesssim_p B$, etc., in order to emphasize this dependence.

A further notational convention adopted from \cite{IM16} is the use of symbols $\chi_0$ and $\chi_1$
in denoting certain nonnegative smooth compactly supported functions on $\R$.
Namely, we require $\chi_0$ to be supported in a neighbourhood of the origin and identically $1$ near the origin,
and $\chi_1$ to be supported away from the origin and identically $1$ on some open neighbourhood of $1 \in \R$.
These cutoff functions $\chi_0$ and $\chi_1$ may vary from line to line,
and sometimes, when several $\chi_0$ and $\chi_1$ appear within the same formula, they may even designate different functions.
\bigskip

{\bf Acknowledgement.}
I would like to thank my supervisor Prof.\,\,Dr.\,\,Detlef M\"{u}ller for numerous useful discussions we had
and for his valuable comments on how to improve this paper.


\subsection{Fundamental concepts and basic assumptions}
\label{Notation_and_assumptions}

Let the surface $S$ be given as the graph
$S = S_\phi \coloneqq \{ (x_1,x_2,\phi(x_1,x_2)) : x = (x_1,x_2) \in \Omega \subset \R^2 \}$
of a smooth and real-valued function $\phi$ defined on an open neighbourhood $\Omega$ of the origin.
We can assume without loss of generality that $\phi(0) = 0$ and we take $\Omega$ to be a sufficiently small neighbourhood of the origin in $\R^2$.
In the mixed norm case we cannot use the rotational invariance of the Fourier transform in order to reduce to the case $\nabla \phi(0) = 0$.
Instead we use a different linear transformation (for details see Subsection \ref{Basic_assumptions_phi}), and so we may and shall assume $\nabla \phi(0) = 0$.

Next, we impose on $\phi$ to be a function of \emph{finite type} at $0$.
This means that there exists a multi-index $\alpha \in \N_0^2$ such that $\partial^\alpha \phi(0) \neq 0$.
By continuity, $\phi$ is of finite type on a neighbourhood of $0$.
We may therefore assume that $\phi$ is of finite type in each point of $\Omega$.
We define the \emph{Taylor support} of $\phi$ as the set
\begin{align*}
\mathcal{T}(\phi) \coloneqq \Big\{ \alpha \in \N_0^2 : \partial^\alpha \phi(0) \neq 0 \Big\}.
\end{align*}
The \emph{Newton polyhedron} $\mathcal{N}(\phi)$ of $\phi$ is the convex hull of the set
\begin{align*}
\bigcup_\alpha \Big\{(t_1,t_2) \in \R^2 : t_1 \geq \alpha_1, t_2 \geq \alpha_2 \Big\},
\end{align*}
where the union is over all $\alpha$ such that $\partial^\alpha \phi(0) \neq 0$ (and so $|\alpha| \geq 2$).
See Figure \ref{Knapp_figure_Newton}.
Both edges and vertices are called \emph{faces} of $\mathcal{N}(\phi)$.
We define the \emph{Newton diagram} $\mathcal{N}_d(\phi)$ of $\phi$ to be the union of all \emph{compact} faces of $\mathcal{N}(\phi)$.

If we are given a face $e_0$ of $\mathcal{N}(\phi)$, we can define its associated (formal) series
\begin{align}
\label{Introduction_hom_poly}
\phi_{e_0}(x_1, x_2) \coloneqq \sum_{\alpha \in e_0 \cap \mathcal{T}(\phi)} \frac{\partial^\alpha \phi(0)}{\alpha!} x^\alpha.
\end{align}
If $e_0$ is a compact face, then $\phi_{e_0}(x_1, x_2)$ is a mixed homogeneous polynomial.
This means that there exists a weight $\kappa^{e_0} = (\kappa^{e_0}_1,\kappa^{e_0}_2) \in [0,\infty)^2$ such that for any $r > 0$ we have
\begin{align*}
\phi_{e_0}(r^{\kappa^{e_0}_1} x_1, r^{\kappa^{e_0}_2} x_2) = r \phi_{e_0}(x_1, x_2),
\end{align*}
and we call $\phi_{e_0}$ a $\kappa^{e_0}$-homogeneous polynomial.
$\kappa^{e_0}$ is uniquely determined if and only if $e_0$ is not a vertex.
In fact, in the case when $e_0$ is an edge, we define $L_{\kappa^{e_0}}$ to be the unique line containing $e_0$:
\begin{align}
\label{Introduction_hom_poly_line}
e_0 \subseteq L_{\kappa^{e_0}}.
\end{align}
Then the weight $\kappa^{e_0}$ is uniquely determined by the relation
\begin{align}
\label{Introduction_def_line}
L_{\kappa^{e_0}} = \Big\{(t_1,t_2) \in \R^2 : \kappa^{e_0}_1 t_1 + \kappa^{e_0}_2 t_2 = 1 \Big\}.
\end{align}
When $e_0$ is an unbounded face, $\phi_{e_0}(x_1, x_2)$ is to be taken only as a formal power series.
Note that then $e_0$ is either a vertical or horizontal edge of $\mathcal{N}(\phi)$, and we can
also find unique $\kappa^{e_0}_1$ and $\kappa^{e_0}_2$ (one of them being $0$ in this case) such that \eqref{Introduction_hom_poly_line} holds.

Of particular interest is the \emph{principal face} $\pi(\phi)$ defined as the face of minimal dimension
of $\mathcal{N}(\phi)$ which intersects the bisectrix $\{(t_1,t_2) \in \R^2 : t_1 = t_2\}$.
Its associated series (or homogeneous polynomial) we call the \emph{principal part} of $\phi$ and denote by $\phi_{\text{pr}} \coloneqq \phi_{\pi(\phi)}$.
Let $\kappa = (\kappa_1,\kappa_2)$ determine the line $L_{\kappa}$ as in \eqref{Introduction_def_line} containing the principal face if it is an edge,
or when it is a vertex, let $\kappa$ determine the edge of $\mathcal{N}(\phi)$ having the principal face as its left endpoint.
Interchanging the $x_1$ and $x_2$ coordinates, if necessary, we may always assume that
\begin{align*}
\kappa_2 \geq \kappa_1.
\end{align*}
We shall denote the ratio $\kappa_2/\kappa_1$ by $m$, and so $m \geq 1$.

The \emph{Newton distance} $d(\phi)$ of $\phi$ is defined to be the coordinate $d$ of the point $(d,d)$
which is the intersection of the bisectrix and the principal face of $\mathcal{N}(\phi)$.
One can easily see that if $\kappa = (\kappa_1,\kappa_2)$ determines the line containing the principal face
(or any of the supporting lines to $\mathcal{N}(\phi)$ in case $\pi(\phi) = \{(d,d)\}$), then we have
\begin{align*}
d(\phi) = \frac{1}{\kappa_1+\kappa_2}.
\end{align*}

The \emph{Newton height} $h(\phi)$ of $\phi$ is defined as
\begin{align*}
h(\phi) = \sup \{d(\phi \circ \varphi) : \varphi \,\, \text{a smooth local coordinate change} \}.
\end{align*}
By a smooth local coordinate change we mean a function $\varphi$ which is smooth and invertible in a neighbourhood of the origin,
and $\varphi(0) = 0$. We also define the \emph{linear height} as
\begin{align*}
h_{\text{lin}}(\phi) = \sup \{d(\phi \circ \varphi) : \varphi \,\, \text{a linear coordinate change} \}.
\end{align*}
For a coordinate change $\varphi$ we shall denote the new cooridnates by $y = \varphi(x)$.
In this case we also denote $d_y = d(\phi \circ \varphi)$.
We say that $\phi$ is \emph{adapted} in the $y$ coordinates if $d_y = h(\phi)$.
Analogously, we say that $\phi$ is \emph{linearly adapted} in coordinates $y$ if $d_y = h_{\text{lin}}(\phi)$.
When $\phi$ is adapted in its original coordinates $x$ we say that $\phi$ is adapted,
and if $\phi$ is not adapted in its original coordinates, then we say that $\phi$ is non-adapted.
Analogous expressions we shall use for linear adaptedness.
We obviously always have
\begin{align*}
d_x = d(\phi) \leq h_{\text{lin}}(\phi) \leq h(\phi).
\end{align*}
The existence of an adapted coordinate system for real analytic functions on $\R^2$ was first proven by Varchenko in \cite{Var76}.
He gave an explicit algorithm on how to construct an adapted coordinate system.
His result was generalised in \cite{IM11a} where it was shown that an adapted coordinate system exists for general smooth functions.
It turns out that in the smooth case one can also essentially use Varchenko's algorithm.
In this article when we refer to Varchenko's algorithm we shall always mean the variant used in \cite{IM11a}.
In this variant one constructs an adapted coordinate system in the form of a non-linear shear transformation
\begin{align*}
y_1 = x_1, \quad y_2 = x_2 - \psi(x_1).
\end{align*}
The smooth real-valued function $\psi$ can be taken in the real-analytic case to be
the \emph{principal root jet} of $\phi$ as defined in \cite{IM16}.
We denote the function $\phi$ in the new (adapted) coordinates by $\phi^a$. Then we have
\begin{align*}
\phi^a(y) = \phi(y_1, y_2 + \psi(y_1)).
\end{align*}
We remark that when $\phi$ is not adapted, then $m = \kappa_2/\kappa_1$ is a positive integer and
$\psi(x_1) - b_1 x_1^m = \mathcal{O}(x_1^{m+1})$ for some nonzero real constant $b_1$.

We introduce next \emph{Varchenko's exponent} $\nu(\phi) \in \{0,1\}$.
If $h(\phi) \geq 2$ and there exists an adapted coordinate system $y$ such that in these coordinates the principal face of $\phi^a(y)$
is a vertex, we define $\nu(\phi) \coloneqq 1$.
In all other cases we take $\nu(\phi) \coloneqq 0$. In particular $\nu(\phi) = 0$ whenever $h(\phi) < 2$.
A concrete characterisation for determining when an adapted coordinate system having the principal face as a vertex exists
can be found in \cite[Lemma 1.5]{IM11b}.

Let us discuss next linear adaptedness.
We assume that $h_{\text{lin}}(\phi) < h(\phi)$, i.e., that we cannot achieve adapted coordinates with a linear coordinate change.
In \cite[Section 1.3]{IM16} it was shown that in this case we can always find a linearly adapted coordinate system,
and \cite[Proposition 1.7]{IM16} gives an explicit characterisation of when a coordinate system is linearly adapted.
It was shown in particular that if the coordinate system $x$ is not already linearly adapted,
then one just needs to apply the first step of Varchenko's algorithm in order to obtain it.

Since in our mixed norm case we consider only $p_1=p_2$, we can freely use linear coordinate changes
in ``tangential'' variables $(x_1,x_2)$ in the expression \eqref{Introduction_FRP_mixed}.
Thus we may assume without loss of generality that either the original coordinate system $x$ is already adapted,
or that it is at least linearly adapted. In particular, we may assume $d(\phi) = h_{\text{lin}}(\phi)$.

The final important concept we introduce is the \emph{augmented Newton polyhedron}
$\mathcal{N}^{\text{res}}(\phi^a)$ of a non-adapted $\phi$
(note the slight change in notation compared to \cite{IM16}, where $\mathcal{N}^{\text{r}}(\phi^a)$ is used instead).
$\mathcal{N}^{\text{res}}(\phi^a)$ is defined as the convex hull of the set
\begin{align*}
\mathcal{N}(\phi^a) \cup L^+,
\end{align*}
where $L^+$ is defined as follows.
Let $L_{\kappa}$ be the line containing the principal face $\pi(\phi)$ of $\mathcal{N}(\phi)$ and let $P = (t^P_1, t^P_2)$
be the point on $L_{\kappa} \cap \mathcal{N}(\phi^a)$ with the smallest $t_2$ coordinate.
Such a point always exists.
Then $L^+$ is the ray
\begin{align*}
\Big\{ (t_1,t_2) \in L_{\kappa} : t_2 \geq t^P_2 \Big\}.
\end{align*}
(See Figure \ref{Knapp_figure_Newton}).



\begin{figure}[!h]
\centering

\begin{tikzpicture}
\begin{axis}[
xlabel=$t_1$, ylabel=$t_2$,
xmin= -1, xmax=15,
ymin= 0, ymax=15,
axis lines=middle,
xtick=\empty, ytick=\empty,
scale = 2.2,
extra x ticks={7.02},
extra x tick labels={
    $1/\kappa_1$
},
extra y ticks={11.02},
extra y tick labels={
    $1/\kappa_2$
}
]

\addplot+ [
mark size=1pt,
solid,
thick,
gray,
mark options={color = gray, fill = gray},
mark = *,
] coordinates {
(1,16)
(1,12)
(1.5,10)
};

\addplot+ [
no marks,
dashed,
gray,
mark options={color = gray, fill = gray},
mark = *,
] coordinates {
(1.5,10)
(1.65,9.4)
};

\addplot+ [
no marks,
dashed,
gray,
mark options={color = gray, fill = gray},
mark = *,
] coordinates {
(2,8.2)
(2.25,7.5)
};

\addplot+ [
mark size=1pt,
solid,
thick,
gray,
mark options={color = gray, fill = gray},
mark = *,
] coordinates {
(2.25,7.5)
(5,5)
(7,3.7)
};

\addplot+ [
no marks,
dashed,
gray,
mark options={color = gray, fill = gray},
mark = *,
] coordinates {
(7,3.7)
(7.5,3.45)
};

\addplot+ [
no marks,
dashed,
gray,
mark options={color = gray, fill = gray},
mark = *,
] coordinates {
(10,2.85)
(11,2.7)
};

\addplot+ [
mark size=1pt,
solid,
thick,
gray,
mark options={color = gray, fill = gray},
mark = *,
] coordinates {
(11,2.7)
(13,2.5)
(16,2.5)
};

\addplot+ [
no marks,
solid,
gray,
mark options={color = gray, fill = gray},
mark = *,
] coordinates {
(0,11.02)
(7.02,0)
};

\node at (axis cs:  5.3,  2) {$L_{\kappa}$};


\node at (axis cs:  1+0.8,12) {$(A_{0},B_{0})$};
\node at (axis cs:  1.5+0.8,10) {$(A_{1},B_{1})$};
\node at (axis cs:  2.25+1.2,7.5+0.3) {$(A_{l_0-1},B_{l_0-1})$};
\node at (axis cs:  5+0.7,5+0.3) {$(A_{l_0},B_{l_0})$};
\node at (axis cs:  7+1.2,3.7+0.3) {$(A_{l_0+1},B_{l_0+1})$};
\node at (axis cs:  11,2.7+0.5) {$(A_{n-1},B_{n-1})$};
\node at (axis cs:  13+0.3,2.5+0.4) {$(A_{n},B_{n})$};

\node at (axis cs:  10,  10) {$\mathcal{N}(\phi^a)$};


\addplot+[
draw=none,
mark=none,
gray,
solid,
pattern=north east lines,
pattern color=gray]
coordinates {
(-1,12.6)
(0,11.02)
(2.25,7.5)
(2,8.2)
(1.65,9.4)
(1.5,10)
(1,12)
(1,15)
(-1,15)
}
\closedcycle;

\addplot+[
draw=none,
mark=none,
gray,
solid,
fill=gray, 
fill opacity=0.1]
coordinates {
(1,15)
(1,12)
(1.5,10)
(1.65,9.3)
(2,8.1)
(2.25,7.5)
(5,5)
(7,3.7)
(7.5,3.45)
(10,2.85)
(11,2.7)
(13,2.5)
(15,2.5)
(15,15)
};


\end{axis}
\end{tikzpicture}

\caption{The (augmented) Newton polyhedron associated to $\phi^a$.}
\label{Knapp_figure_Newton} 

\end{figure}
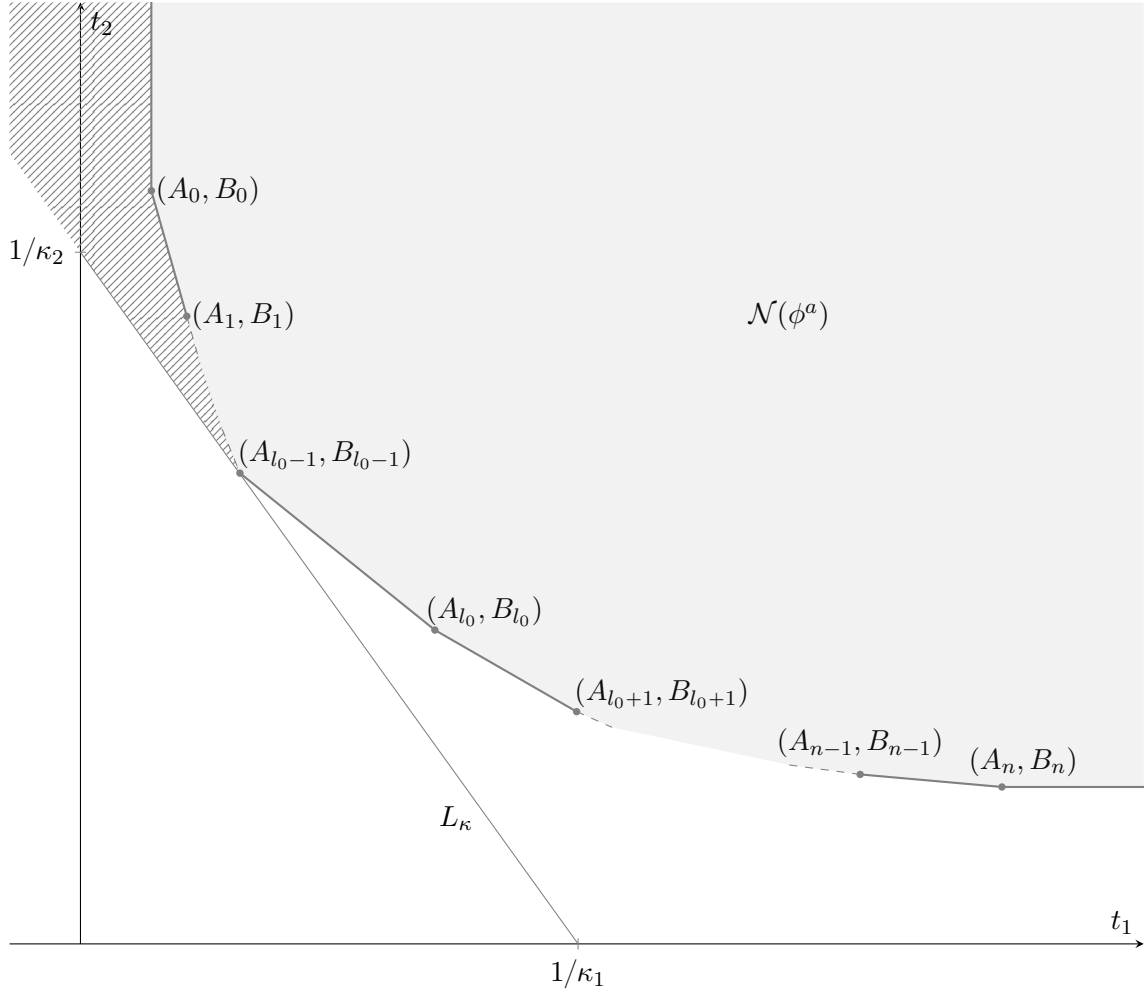



\subsection{The main results}
\label{Statement}

Let us briefly review all the conditions on the function $\phi$
which we may assume without loss of generality when considering the mixed norm restriction problem:
\begin{itemize}
\item
$\phi(0)=0$ and $\nabla \phi(0) = 0$,
\item
$\phi$ is of finite type on $\Omega$,
\item
the weight $\kappa$ determined by the principal face of $\mathcal{N}(\phi)$ (or by the edge containing the principal face as its left endpoint)
satisfies $m = \kappa_2 / \kappa_1 \geq 1$, and
\item
the original coordinate system $x$ is either adapted, or linearly adapted but not adapted.
In both cases we have $d(\phi) = h_{\text{lin}}(\phi)$.
\end{itemize}
Recall that $S$ denotes the surface given as the graph of $\phi$ and $\mathrm{d}\sigma$ its surface measure.
We are considering the mixed norm Fourier restriction problem \eqref{Introduction_FRP_mixed}
when $\rho$ is supported in a sufficiently small neighbourhood of the origin.

We begin by stating necessary conditions which will be obtained by means of Knapp-type examples.
When $\phi$ is not adapted we denote by
\begin{align*}
K : [o, \kappa_1] \to [0,+\infty]
\end{align*}
the function defined in the following way.
Consider all lines of the form
\begin{align}
\label{Statement_supporting_line}
L_{\tilde{\kappa}} = \Big\{ (t_1,t_2) \in \R^2 : \tilde{\kappa}_1 t_1 + \tilde{\kappa}_2 t_2 = 1 \Big\},
\end{align}
where $\tilde{\kappa} \in [0,\infty)^2$ is a weight.
For each $0 \leq \tilde{\kappa}_1 \leq \kappa_1$ there is a unique $\tilde{\kappa}_2$ so that \eqref{Statement_supporting_line}
determines a supporting line $L_{\tilde{\kappa}}$ to $\mathcal{N}^{\text{res}}(\phi^a)$.
We then define $K(\tilde{\kappa}_1)$ to be $\tilde{\kappa}_2$ for $\tilde{\kappa}_1 \in [0,\kappa_1]$
(see Figure \ref{Knapp_figure_Legendre}).
Note that then the weight $(0, K(0))$ determines line containing the horizontal edge of the augmented Newton polyhedron,
i.e., the right most edge of $\mathcal{N}^{\text{res}}(\phi^a)$.
The weight $(\kappa_1, K(\kappa_1)) = \kappa$ determines the line containing the edge associated to the principal face of $\mathcal{N}(\phi)$
which is the left most edge of $\mathcal{N}^{\text{res}}(\phi^a)$.



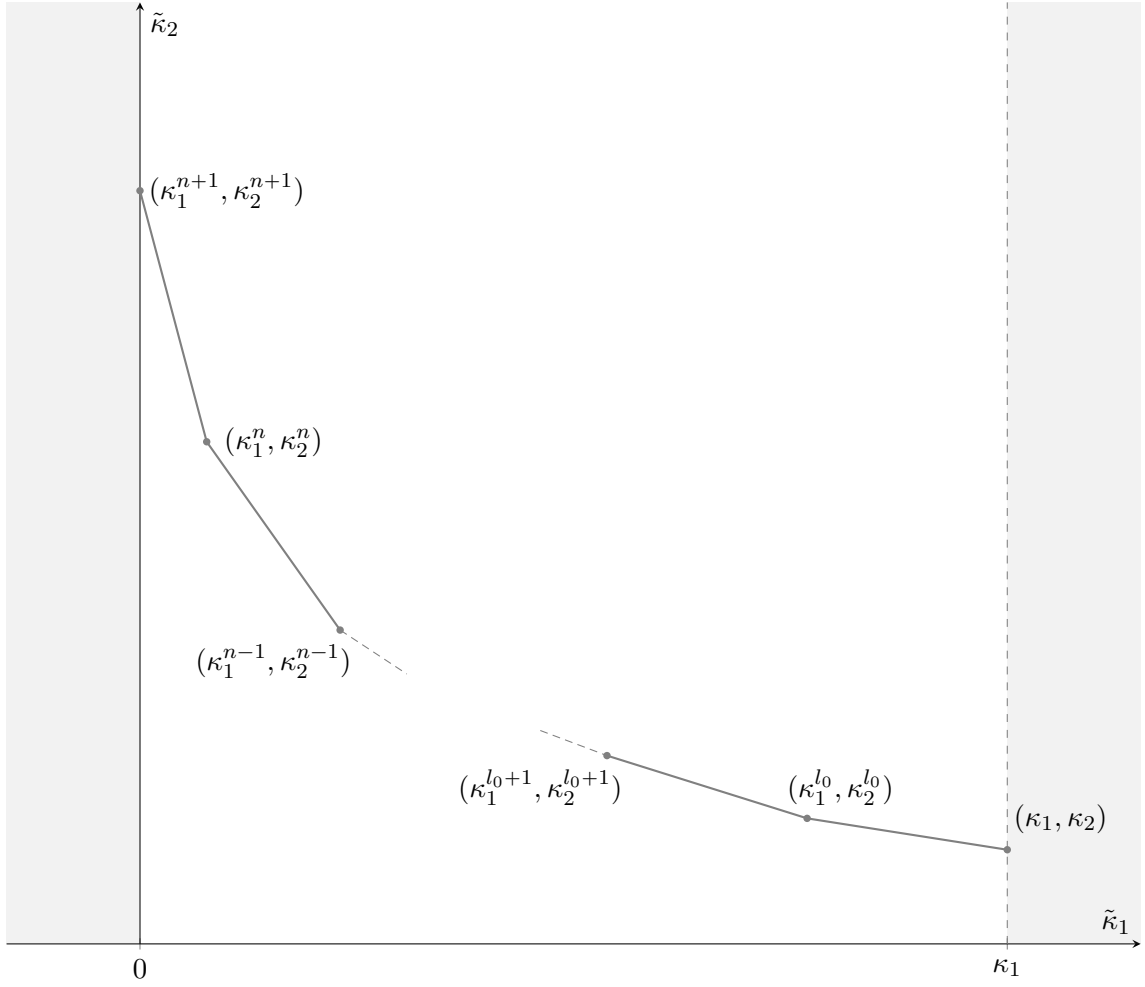
\begin{figure}[!h]
\centering

\begin{tikzpicture}
\begin{axis}[
xlabel=$\tilde{\kappa}_1$, ylabel=$\tilde{\kappa}_2$,
xmin=-2, xmax=15,
ymin=-0, ymax=15,
axis lines=middle,
xtick=\empty, ytick=\empty,
scale = 2.2,
extra x ticks={0, 13},
extra x tick labels={$0$, $\kappa_1$}
]

\addplot+ [
mark size=1pt,
solid,
thick,
gray,
mark options={color = gray, fill = gray},
mark = *,
] coordinates {
(0,12)
(1,8)
(3,5)
};

\addplot+ [
no marks,
densely dashed,
gray,
] coordinates {
(3,5)
(4,4.3)
};

\addplot+ [
no marks,
densely dashed,
gray,
] coordinates {
(6,3.4)
(7,3)
};

\addplot+ [
mark size=1pt,
solid,
thick,
gray,
mark options={color = gray, fill = gray},
mark = *,
] coordinates {
(7,3)
(10,2)
(13,1.5)
};

\path[name path=left_up_fill_axis] (axis cs:-2,0) -- (axis cs:0,0);
\path[name path=left_down_fill_axis] (axis cs:-2,15) -- (axis cs:0,15);

\addplot [
thick,
color=gray,
fill=gray, 
fill opacity=0.1
]
fill between[
of=left_up_fill_axis and left_down_fill_axis,
soft clip={domain=-2:0},
];


\path[name path=right_up_fill_axis] (axis cs:13,0) -- (axis cs:15,0);
\path[name path=right_down_fill_axis] (axis cs:13,15) -- (axis cs:15,15);

\addplot [
thick,
color=gray,
fill=gray, 
fill opacity=0.1
]
fill between[
of=right_up_fill_axis and right_down_fill_axis,
soft clip={domain=13:15},
];


\addplot+ [
no markers,
dashed,
gray,
] coordinates {
(13,0)
(13,1)
(13,15)
};

\node at (axis cs:  0+1.3, 12) { $(\kappa_1^{n+1},\kappa_2^{n+1})$};
\node at (axis cs:  1+1, 8) { $(\kappa_1^{n},\kappa_2^{n})$};
\node at (axis cs:  3-1, 5-0.5) { $(\kappa_1^{n-1},\kappa_2^{n-1})$};
\node at (axis cs:  7-1, 3-0.5) { $(\kappa_1^{l_0+1},\kappa_2^{l_0+1})$};
\node at (axis cs:  10+0.5, 2+0.5) { $(\kappa_1^{l_0},\kappa_2^{l_0})$};
\node at (axis cs:  13+0.8, 1.5+0.5) { $(\kappa_1,\kappa_2)$};


\end{axis}
\end{tikzpicture}

\caption{The typical form of the graph of the function $K : \tilde{\kappa}_1 \mapsto \tilde{\kappa}_2$.}
\label{Knapp_figure_Legendre} 

\end{figure}


Denote by $\mathcal{L}$ the \emph{Legendre transformation} for a real-valued convex function $K$:
\begin{align*}
\mathcal{L}(K) [w] \coloneqq \sup_{u \in [0, \kappa_1]} (wu-K(u)).
\end{align*}
Then we may state the necessary conditions in the following way:
\begin{theorem}
\label{Introduction_thm_Knapp}
Let $\phi$ be as above and let us assume that the estimate \eqref{Introduction_FRP_mixed} holds true with $\rho(0) \neq 0$.
If $\phi$ is adapted, then we have the necessary condition
\begin{align*}
    \frac{1}{d(\phi) p_1'} + \frac{1}{p_3'} \leq \frac{1}{2 d(\phi)}.
\end{align*}
If $K$ is as above and $\phi$ is linearly adapted, but not adapted, then we necessarily have
\begin{align*}
\frac{1}{p_3'} \leq -\frac{1}{2} \mathcal{L}(K) \Big[\frac{2+2m}{p_1'} - 1\Big].
\end{align*}
\end{theorem}
Recall that $d(\phi) = h(\phi)$ when $\phi$ is adapted.
The above theorem is a direct consequence of Proposition \ref{Knapp_main_proposition} in Section \ref{Knapp} below
and the discussion in Subsection \ref{Knapp_Legendre}.
The necessary conditions are depicted in Figure \ref{Knapp_figure_Knapp}.

The main result of this paper is:
\begin{theorem}
\label{Introduction_theorem_FRP}
Let $\phi$ be as above and $\rho$ supported in a sufficiently small neighbourhood of $0$.
If either
\begin{enumerate}
\item[(a)] $\phi$ is adapted in its original coordinates, or
\item[(b)]$\phi$ is non-adapted, $h_{\text{lin}}(\phi) < 2$, and $\phi$ is real analytic,
\end{enumerate}
then the estimate \eqref{Introduction_FRP_mixed} holds true for all $(1/p_1',1/p_3')$
as determined by Theorem \ref{Introduction_thm_Knapp}, except for the point $(1/p_1',1/p_3') = (0,1/(2h(\phi)))$
where it is false if $\rho(0) \neq 0$ and either $h(\phi) = 1$ or $\nu(\phi) = 1$.
\end{theorem}
In case (b) we shall actually prove the claim for a more general class of functions than is stated here.

The part (a) of the above theorem follows from Proposition \ref{Adapted_proposition_adapted},
and the part (b) follows from Theorem \ref{Section_h_lin_less_than_2_main_theorem}
Let us mention that in the case $h_{\text{lin}}(\phi) < 2$ it turns out that we always have $\nu(\phi) = 0$,
which will be important for the boundary point $(1/p_1',1/p_3') = (0,1/(2h(\phi)))$.

In this article we do not deal with the non-adapted case when $h_\text{lin}(\phi) \geq 2$ in its full generality.
Let us briefly comment how one can easily get some preliminary Fourier restriction estimates.
Namely, the abstract result from \cite{KT98} by Keel and Tao implies that we automatically have the Fourier restriction estimate
for the region labeled by $KT$ in Figure \ref{Knapp_figure_Knapp} below.
For details we refer to Proposition \ref{Adapted_general_considerations}.

One can combine this result with the case $p_1=p_3$ from Theorem \ref{thm_main_result_IM16}
and get by interpolation the region labeled by $IM$ in Figure \ref{Knapp_figure_Knapp}.



\begin{figure}[!h]
\centering

\begin{tikzpicture}
\begin{axis}[
xlabel=$1/p_1'$, ylabel=$1/p_3'$,
xmin= 0, xmax=15,
ymin= 0, ymax=15,
axis lines=middle,
xtick=\empty, ytick=\empty,
scale = 2.2,
extra x ticks={13.5},
extra x tick labels={
    $1/2$
},
extra y ticks={0, 6, 7.62, 10.8, 13.5},
extra y tick labels={
    $O$,
    $1/(2 h(\phi))$,
    $1/(2 h^\text{res}(\phi))$,
    $1/(2 d(\phi))$,
    $1/2$
}
]

\addplot+ [
mark size=1pt,
gray,
dashed,
mark options={color = gray, fill = gray},
mark = *,
] coordinates {
(0,0)
(4.87, 4.87)
(13, 13)
};

\addplot+ [
no marks,
gray,
dashed,
] coordinates {
(13.5,0)
(0,10.8)
};

\addplot+ [
no marks,
gray,
dashed,
] coordinates {
(13.5,0)
(0,7.62)
};

\addplot+ [
mark size=1pt,
solid,
thick,
gray,
mark options={color = gray, fill = gray},
mark = *,
] coordinates {
(0,6)
(3,5.5)
(6,4.5)
};

\addplot+ [
no marks,
densely dashed,
gray,
] coordinates {
(6,4.5)
(7,4)
};

\addplot+ [
no marks,
densely dashed,
gray,
] coordinates {
(9,3)
(10,2.5)
};

\addplot+ [
mark size=1pt,
solid,
thick,
gray,
mark options={color = gray, fill = gray},
mark = *,
] coordinates {
(10,2.5)
(12,1.2)
(13.5,0)
};

\path[name path=along_axis] (axis cs:0,0) -- (axis cs:13.5,0);

\addplot+ [
name path = along_KT,
no marks,
solid,
thick,
color = blue,
] coordinates {
(0,6)
(13.5,0)
};

\addplot+ [
name path = along_IM,
no marks,
solid,
thick,
color = red,
] coordinates {
(0,6)
(4.87, 4.87)
(13.5,0)
};

\path[name path=along_boundary] (axis cs:0,6) -- (axis cs:3,5.5) -- (axis cs:6,4.5) -- (axis cs:10,2.5) -- (axis cs:12,1.2) -- (axis cs:13.5,0);

\addplot
fill between[
of=along_axis and along_KT,
split,
soft clip={domain=0:13.5},
every segment no 0/.style={
pattern=north west lines,
pattern color=blue,
solid,
opacity = 0.3},
];

\addplot
fill between[
of=along_KT and along_IM,
split,
soft clip={domain=0:13.5},
every segment no 1/.style={
pattern=north east lines,
pattern color=red,
solid,
opacity = 0.6},
];

\addplot
fill between[
of=along_IM and along_boundary,
split,
soft clip={domain=0:13.5},
every segment no 1/.style={pattern=horizontal lines,
                        pattern color=gray,solid},
];

\node at (axis cs:  9,  8) {$p_1 = p_3$};

\node at (axis cs:  13.5, 0.5) {$P$};
\node at (axis cs:  12.3, 1.5) {$P_{l_0}$};
\node at (axis cs:  10.6, 2.8) {$P_{l_0+1}$};
\node at (axis cs:  6.5, 4.8) {$P_{l^a-1}$};
\node at (axis cs:  2.6, 5.9) {$P_{l^a}$};
\node at (axis cs:  0.3, 6.3) {$\tilde{P}$};

\node at (axis cs:  1.5, 3) {$\mathcal{P}$};

\node at (axis cs:  4, 1.5) {$KT$};
\node at (axis cs:  6, 3.8) {$IM$};

\end{axis}
\end{tikzpicture}

\caption{Necessary conditions in the $(1/p_1',1/p_3')$-plane.}
\label{Knapp_figure_Knapp} 

\end{figure}
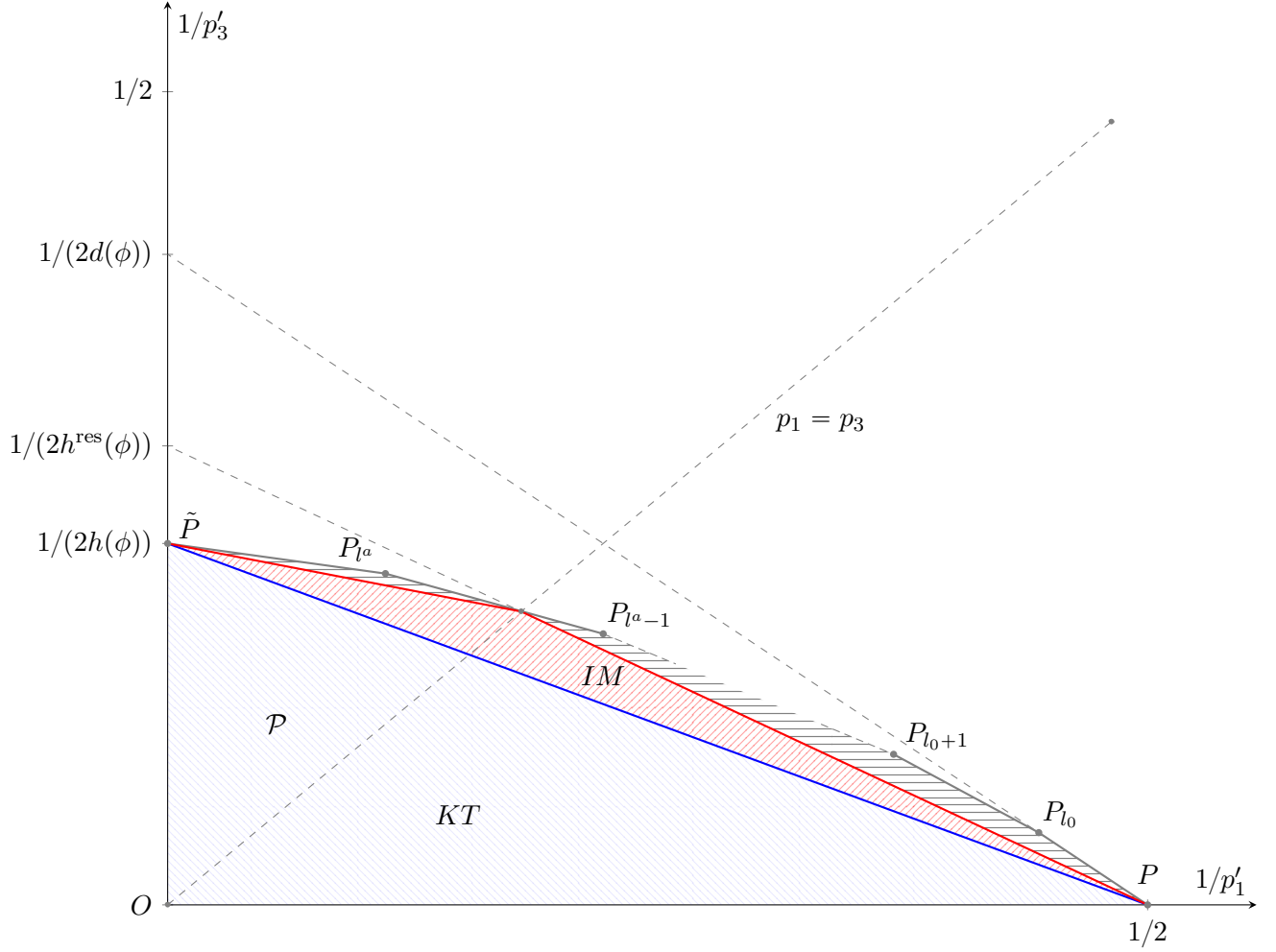



\section{Necessary conditions}
\label{Knapp}

In this section our assumptions on $\phi$ are as explained in Subsection \ref{Statement}.
Our goal is to find a complete set of necessary conditions on $p = (p_1,p_3) \in [1,\infty]^2$
for \eqref{Introduction_FRP_mixed} to hold true whenever $\rho(0) \neq 0$.
We shall reframe the conditions in several ways:
an ``explicit'' form in Subsection \ref{Knapp_explicit}, a form as in Theorem \ref{Introduction_thm_Knapp}
using the Legendre transformation of $K$ in Subsection \ref{Knapp_Legendre},
and a form when we fix the ratio $p_1'/p_3'$ in Subsection \ref{Knapp_ratio_fixed}.
In Subsection \ref{Knapp_h_lin_less_than_two} we discuss the normal forms of $\phi$ when $h_{\text{lin}}(\phi) < 2$
and determine explicitly the necessary conditions in this case.

\subsection{The explicit form}
\label{Knapp_explicit}

Let us first introduce some further notation.
If $\phi$ is linearly adapted but not adapted, then the adapted coordinate system is obtained through
\begin{align*}
y_1 = x_1, \quad y_2 = x_2 - \psi(x_1),
\end{align*}
where $\psi$ is the principal root jet.
The function $\phi$ is in the new coordinates $y$
\begin{align*}
\phi^a(y_1,y_2) \coloneqq \phi(y_1, y_2 + \psi(y_1)),
\end{align*}
i.e., $\phi^a$ represents the function $\phi$ in adapted coordinates.
We denote the vertices of $\mathcal{N}(\phi^a)$ by
\begin{align*}
(A_l, B_l) \in \N_0^2, \quad l = 0,1,2,\ldots,n,
\end{align*}
where $n \geq 0$ and we assume that the points are ordered from left to right, i.e., $A_{l-1} < A_l$ for $l = 1,2,\ldots,n$.
Next, we denote the compact edges of $\mathcal{N}(\phi^a)$ by
\begin{align*}
\gamma_l \coloneqq [(A_{l-1}, B_{l-1}), (A_l, B_l)], \quad l = 1,2,\ldots,n,
\end{align*}
and also the unbounded edges by
\begin{align*}
\gamma_{0} \coloneqq \{ (t_1,t_2) \in \R^2 : t_1 = A_{0}, t_2 \geq B_{0} \}, \\
\gamma_{n+1} \coloneqq \{ (t_1,t_2) \in \R^2 : t_1 \geq A_{n}, t_2 = B_{n} \},
\end{align*}
see Figure \ref{Knapp_figure_Newton}.
Let us denote by $L_l, \, l = 0, \ldots, n+1$, the associated lines on which these edges lie.
Each line $L_l$ is given by the equation
\begin{align*}
\kappa_1^l t_1 + \kappa_2^l t_2 = 1,
\end{align*}
where $(\kappa_1^l, \kappa_2^l) \in [0,\infty)^2$ is its associated weight.
We also introduce the quantity
\begin{align*}
a_l \coloneqq \frac{\kappa_2^l}{\kappa_1^l},
\end{align*}
which is related to the slope of $L_l$, namely, its slope is then equal to $-1/a_l$.
We obviously have $a_0 = 0$ and $a_{n+1} = \infty$.

Let us denote by $0 < m < \infty$ the leading exponent in the Taylor expansion of $\psi$.
We define $L_{\kappa}$ to be the unique line
\begin{align*}
\kappa_1 t_1 + \kappa_2 t_2 = 1
\end{align*}
satisfying $\kappa_2 = m \kappa_1$ and which is a supporting line to the Newton polyhedron $\mathcal{N}(\phi^a)$.
This line coincides with the line containing the principal face of $\mathcal{N}(\phi)$.
This follows from Varchenko's algorithm.
Next, let $l_0$ be such that
\begin{align*}
a_{l_0} > m \geq a_{l_0-1}.
\end{align*}
Note that the point $(A_{l_0-1},B_{l_0-1})$ is the right endpoint of the intersection
of $L_{\kappa}$ and $\mathcal{N}(\phi^a)$.
Varchenko's algorithm also shows that $B_{l_0-1} \geq A_{l_0-1}$.
We denote by $l^a$ the index such that $\kappa^{l^a}$ is associated to the principal face of $\mathcal{N}(\phi^a)$.
If $\pi(\phi^a)$ is a vertex, we take $l^a$ to be associated to the edge to the left of $\pi(\phi^a)$.
Note $l^a \geq l_0$.

We may now define the \emph{augmented Newton polyhedron} $\mathcal{N}^{\text{res}}(\phi^a)$ as the convex hull of the set
\begin{align*}
\mathcal{N}(\phi^a) \cup L_\kappa^+,
\end{align*}
where $L_\kappa^+$ denotes the ray
\begin{align*}
\Big\{ (t_1,t_2) \in L_{\kappa} : t_2 \geq B_{l_0-1} \Big\}.
\end{align*}

Before stating the necessary conditions analogous to \cite[Proposition 1.16]{IM16},
let us recall that in the case of the principal face being a vertex, we take $\kappa$ to determine
the line containing the edge of $\mathcal{N}(\phi)$ which has $\pi(\phi)$ as its left endpoint.
Furthermore recall that $m = \kappa_2 / \kappa_1 \geq 1$ and that
$\phi$ is linearly adapted in its original coordinates.

\begin{proposition}
\label{Knapp_main_proposition}
Let $\phi$ be as above.
Let $\rho \geq 0$, $\rho \in C_0^\infty(S)$, be a smooth compactly supported function with $\rho(0) \neq 0$,
and assume that the estimate \eqref{Introduction_FRP_mixed} holds true.
If $\phi$ is non-adapted, let us consider the nonlinear shear transformation
\begin{equation*}
y_1 \coloneqq x_1, \quad y_2 \coloneqq x_2 - \psi(x_1),
\end{equation*}
and let $\phi^a(y) \coloneqq \phi(y_1,y_2+\psi(y_1))$ be the function $\phi$ expressed in the adapted coordinates.
Then it necessarily follows that for all weights $(\tilde{\kappa}_1,\tilde{\kappa}_2)$
such that $L_{\tilde{\kappa}}$ is a supporting line to $\mathcal{N}^\text{res}(\phi^a)$ we have
\begin{align}
\label{Knapp_conditions_mixed_raw}
\frac{(1+m) \tilde{\kappa}_1}{p_1'} + \frac{1}{p_3'} \leq \frac{\tilde{\kappa}_1+\tilde{\kappa}_2}{2}.
\end{align}
This is equivalent to
\begin{align}
\label{Knapp_conditions_mixed}
\begin{split}
&\frac{(1+m) \kappa_1^l}{p_1'} + \frac{1}{p_3'} \leq \frac{\kappa_1^l+\kappa_2^l}{2},
 \qquad l = l_0,\ldots,n+1, \\
&\frac{(1+m) \kappa_1}{p_1'} + \frac{1}{p_3'} \leq \frac{\kappa_1+\kappa_2}{2}.
\end{split}
\end{align}
Furthermore, when $\phi$ is either adapted or non-adapted we have the conditions
\begin{align}
\label{Knapp_conditions_mixed_additional}
    \frac{1}{d(\phi) p_1'} + \frac{1}{p_3'} \leq \frac{1}{2 d(\phi)},
    \qquad \frac{1}{p_3'} \leq \frac{1}{2h(\phi)}.
\end{align}
In particular when $\phi$ is non-adapted
the first condition in \eqref{Knapp_conditions_mixed_additional} then coincides with the one in the second line of \eqref{Knapp_conditions_mixed}.
Moreover in this case the conditions in \eqref{Knapp_conditions_mixed} for $l > l^a$ are redundant, and
if we fix $p_3' = \infty$ (resp. $p_1' = \infty$) then all the conditions reduce to $p_1' \geq 2$ (resp. $p_3' \geq 2h(\phi)$).
\end{proposition}
\begin{proof}
We give only a sketch of the proof since it follows the same lines as in \cite{IM16}.
Let us consider any supporting line $L_{\tilde{\kappa}}$
to the augmented Newton polyhedron $\mathcal{N}^\text{res}(\phi^a)$ for some weight $(\tilde{\kappa}_1,\tilde{\kappa}_2)$.
This particularly implies by the definition of the augmented Newton diagram that $\tilde{\kappa}_2 \geq m \tilde{\kappa}_1$.

We first consider the case when $\tilde{\kappa}_1 > 0$, i.e., when the associated line $L_{\tilde{\kappa}}$ is not horizontal.
In this case for each sufficiently small $\varepsilon > 0$ we define the region
\begin{align*}
D_\varepsilon^a \coloneqq \Big\{ y \in \R^2 : |y_1| \leq \varepsilon^{\tilde{\kappa}_1}, |y_2| \leq \varepsilon^{\tilde{\kappa}_2} \Big\},
\end{align*}
which in the original coordinate system has the form
\begin{align*}
D_\varepsilon \coloneqq \Big\{ x \in \R^2 : |x_1| \leq \varepsilon^{\tilde{\kappa}_1}, |x_2-\psi(x_1)| \leq \varepsilon^{\tilde{\kappa}_2} \Big\}.
\end{align*}
Using the $\phi^a_{\tilde{\kappa}}$ part of the Taylor approximation of $\phi^a$
one easily gets that for each $y \in D_\varepsilon^a$ we have $|\phi^a(y)| \leq C \varepsilon$.
Returning to the $x$ coordinates we obtain
\begin{align*}
|\phi(x)| \leq C \varepsilon, \quad x \in D_\varepsilon.
\end{align*}
But for $x \in D_\varepsilon$ one has
\begin{align*}
|x_2| \leq \varepsilon^{\tilde{\kappa}_2} + |\psi(x_1)| \lesssim \varepsilon^{\tilde{\kappa}_2} + \varepsilon^{m \tilde{\kappa}_1}
      \lesssim \varepsilon^{m \tilde{\kappa}_1},
\end{align*}
since $|\psi(x_1)| \lesssim |x_1|^{m}$ and $\tilde{\kappa}_2 \geq m \tilde{\kappa}_1$.
Therefore the region $D_\varepsilon$ is contained in the set where
$|x_1| \leq C_1 \varepsilon^{\tilde{\kappa}_1}$ and $|x_2| \leq C_2 \varepsilon^{m \tilde{\kappa}_1}$.
Thus we choose a Schwartz function $\varphi_\varepsilon$ which has its Fourier transform of the form
\begin{align*}
\widehat{\varphi_\varepsilon}(x_1,x_2,x_3) =
   \chi_0\Big(\frac{x_1}{C_1 \varepsilon^{\tilde{\kappa}_1}}\Big) \, \chi_0\Big(\frac{x_2}{C_2 \varepsilon^{m \tilde{\kappa}_1}}\Big)
   \, \chi_0\Big(\frac{x_3}{C \varepsilon}\Big),
\end{align*}
for some smooth compactly supported function $\chi_0$ which is identically $1$ on the interval $[-1,1]$.
Then in particular we have $\widehat{\varphi_\varepsilon}(x_1,x_2,\phi(x_1,x_2)) \geq 1$ on $D_\varepsilon$.

Now on the one hand, since $\rho(0) \neq 0$, we have
\begin{align*}
\Bigg( \int_S |\widehat{\varphi_\varepsilon}|^2 \rho \mathrm{d}\sigma \Bigg)^{1/2} \gtrsim |D_\varepsilon|^{1/2} = \varepsilon^{(\tilde{\kappa}_1+\tilde{\kappa}_2)/2},
\end{align*}
and on the other
\begin{equation*}
\Vert \varphi_\varepsilon \Vert_{L^{p_3}_{x_3} (L^{p_1}_{(x_1,x_2)})} \sim \varepsilon^{\frac{(1+m) \tilde{\kappa}_1}{p_1'} + \frac{1}{p_3'}}.
\end{equation*}
Plugging these into \eqref{Introduction_FRP_mixed} and letting $\varepsilon \to 0$ one obtains
\eqref{Knapp_conditions_mixed_raw} for the non-horizontal edges.

In the horizontal case $\tilde{\kappa}_1 = 0$ one only slightly changes the argument. Namely, one defines for a sufficiently small $\delta > 0$
\begin{align*}
D_\varepsilon^a \coloneqq \Big\{ y \in \R^2 : |y_1| \leq \varepsilon^{\delta}, |y_2| \leq \varepsilon^{\tilde{\kappa}_2} \Big\}.
\end{align*}
The associated set in the $x$ coordinates $D_\varepsilon$ is then contained in the box
determined by $|x_1| \leq \varepsilon^\delta$ and $|x_2| \leq \varepsilon^{m \delta}$.
Furthermore, using a Taylor series expansion, one can easily show that for $x \in D_\varepsilon$
we have again $|\phi(x)| \leq C \varepsilon$.
Now one proceeds as in the non-horizontal case, the only difference is that after taking the limit $\varepsilon \to 0$,
one also needs to take the limit $\delta \to 0$.

Let us now briefly explain why \eqref{Knapp_conditions_mixed_raw} and \eqref{Knapp_conditions_mixed} are equivalent.
We obviously have that \eqref{Knapp_conditions_mixed_raw} implies \eqref{Knapp_conditions_mixed}.
For the reverse implication we note that the $\tilde{\kappa}$'s considered in \eqref{Knapp_conditions_mixed}
are by definition precisely those for which the lines $L_{\tilde{\kappa}}$ contain the edges of the augmented Newton diagram.
This means that all the other supporting lines touch the augmented Newton diagram at only one point.
Now one just uses the fact that the associated weight $\tilde{\kappa}$ of such a supporting line $L_{\tilde{\kappa}}$
is obtained by a convex combination of weights associated to the edges which intersect at the point
through which $L_{\tilde{\kappa}}$ passes.
Thus, all the conditions in \eqref{Knapp_conditions_mixed_raw} can be obtained as
convex combinations of conditions in \eqref{Knapp_conditions_mixed}.

The proof of \eqref{Knapp_conditions_mixed_additional} is similar to the one for \eqref{Knapp_conditions_mixed_raw}.
One considers the set $D_\varepsilon$ defined by $\{ x \in \R^2 : |x_1| \leq \varepsilon^{\kappa_1}, |x_2| \leq \varepsilon^{\kappa_2} \}$
in the case when the principal face of $\mathcal{N}(\phi)$ is compact.
If it is not compact, then one uses $\{ x \in \R^2 : |x_1| \leq \varepsilon^{\delta}, |x_2| \leq \varepsilon^{\kappa_2} \}$.
Using the Taylor approximation of $\phi(x)$ one gets that for $x \in D_\varepsilon$ we have $|\phi(x)| \lesssim \varepsilon$.
The first condition in \eqref{Knapp_conditions_mixed_additional} is then obtained by plugging
\begin{align*}
\widehat{\varphi_\varepsilon}(x_1,x_2,x_3) =
   \chi_0\Big(\frac{x_1}{\varepsilon^{\kappa_1}}\Big) \, \chi_0\Big(\frac{x_2}{\varepsilon^{\kappa_2}}\Big)
   \, \chi_0\Big(\frac{x_3}{C \varepsilon}\Big),
\end{align*}
into the estimate \eqref{Introduction_FRP_mixed} in the compact case.
In the non-compact case we just change $\varepsilon^{\kappa_1}$ to $\varepsilon^\delta$.

In the adapted case, when $d(\phi) = h(\phi)$, we also get automatically the second condition from the first one.
Finally, as was mentioned at the beginning of this section, if $\phi$ is non-adapted
and if we take $l$ such that $\kappa^l$ is associated to the principal face of $\mathcal{N}(\phi^a)$,
then we have $h(\phi) = 1/(\kappa_1^l + \kappa_2^l)$.
Therefore the associated condition to this $l$ in \eqref{Knapp_conditions_mixed}
implies the second condition in \eqref{Knapp_conditions_mixed_additional}.

Let us now prove the remaining claims.
When $p_1' = \infty$, then all the conditions indeed reduce to
\begin{align*}
    \qquad \frac{1}{p_3'} \leq \frac{1}{2h(\phi)}
\end{align*}
since $\kappa_1^l+\kappa_2^l$ is minimal precisely for the edge $\gamma_{l^a}$
which intersects the bisectrix of $\mathcal{N}(\phi^a)$.
This is a direct consequence of the fact that the augmented Newton polyhedron is obtained by the intersection of
upper half-planes which have $L_{\kappa}$ and $L_l$'s with $\kappa_2^l/\kappa_1^l > m$ (i.e., for $l \geq l_0$) as boundaries,
and of the fact that the bisectrix intersects $L_l$ at $(1/(\kappa_1^l+\kappa_2^l), 1/(\kappa_1^l+\kappa_2^l))$.

When $p_3' = \infty$, then the condition
\begin{align*}
    \qquad \frac{1}{p_1'} \leq \frac{1}{2}
\end{align*}
is the strongest one; this is a direct consequence of $\kappa_2^l/\kappa_1^l > m = \kappa_2/\kappa_1$.

We finally prove that one does not need to consider all the conditions in the first row of \eqref{Knapp_conditions_mixed},
but only for $l = l_0,\ldots,l^a$ where $l^a$ is such that $\gamma_{l^a}$ is the principal face of $\mathcal{N}(\phi^a)$.
This follows from the following two facts.
Namely, we first note that the line in the $(1/p_1',1/p_3')$-plane given by
\begin{align}
\label{Knapp_conditions_mixed_lines}
\frac{\kappa_1^l+\kappa_2^l}{2} = \frac{(1+m) \kappa_1^l}{p_1'} + \frac{1}{p_3'}
\end{align}
intersects the axis $1/p_1' = 0$ at the point which has the $1/p_3'$ coordinate equal to $(\kappa_1^l+\kappa_2^l)/2$,
which is greater than $1/(2h(\phi))$ if $l \neq l^a$, by the previous discussion in the case $p_1' = \infty$.
And secondly, as $\kappa_1^l$ decreases when $l$ increases,
the slope of the line \eqref{Knapp_conditions_mixed_lines} in the $(1/p_1',1/p_3')$-plane increases with $l$ too.
Therefore, in the $(1/p_1',1/p_3')$-plane the lines given by \eqref{Knapp_conditions_mixed_lines} and corresponding to necessary conditions
associated to any $l$ with $l > l^a$ are lying above the line associated to $l^a$ in the area where $1/p_1' \geq 0$.
\end{proof}

The necessary conditions from Proposition \ref{Knapp_main_proposition} determine a polyhedron
in the $(1/p_1',1/p_3')$-plane which we denote by $\mathcal{P}$ (see Figure \ref{Knapp_figure_Knapp}).
Let us define the lines
\begin{align*}
\begin{split}
&\tilde{L}_{l} \coloneqq \Bigg\{ \Big(\frac{1}{p_1'}, \frac{1}{p_3'}\Big) : \frac{(1+m) \kappa_1^l}{p_1'} + \frac{1}{p_3'} = \frac{\kappa_1^l+\kappa_2^l}{2} \Bigg\},
 \qquad l = l_0,\ldots,n+1, \\
&\tilde{L} \coloneqq \Bigg\{ \Big(\frac{1}{p_1'}, \frac{1}{p_3'}\Big) : \frac{(1+m) \kappa_1}{p_1'} + \frac{1}{p_3'} = \frac{\kappa_1+\kappa_2}{2} \Bigg\},
\end{split}
\end{align*}
associated to the necessary conditions.
Using arguments similar as in the proof of Proposition \ref{Knapp_main_proposition},
or the Legendre transformation from the following Subsection \ref{Knapp_Legendre}, one can show that
the polyhedron $\mathcal{P}$ is of the form
\begin{align*}
\mathcal{P} = OPP_{l_0}P_{l_0+1}\ldots P_{l^a-1}P_{l^a}\tilde{P},
\end{align*}
i.e., the polyhedron with vertices $O, P, P_{l_0}, P_{l_0+1}, \ldots, P_{l^a-1}, P_{l^a}, \tilde{P}$,
where the point $O$ is the origin and the other points are as follows.
The point $P$ is $(1/2,0)$ and the point $\tilde{P}$ is $(0,1/(2h(\phi)))$.
The point $P_{l_0}$ is the intersection of $\tilde{L}$ and $\tilde{L}_{l_0}$, and
all the other points $P_{l}$ are given as intersections of the lines $\tilde{L}_{l}$ and $\tilde{L}_{l-1}$ for $l = l_0+1, \ldots, l^a$.
Hence, an easy calculation shows
\begin{align*}
P_{l_0} &= \frac{1}{2} \Bigg(\frac{1}{m+1}\Big( 1 + \frac{\kappa_2^{l_0}-\kappa_2}{\kappa_1^{l_0}-\kappa_1} \Big),
                             \kappa_2 - \kappa_1 \frac{\kappa_2^{l_0}-\kappa_2}{\kappa_1^{l_0}-\kappa_1} \Bigg),\\
P_{l  } &= \frac{1}{2} \Bigg(\frac{1}{m+1}\Big( 1 + \frac{\kappa_2^{l}-\kappa_2^{l-1}}{\kappa_1^{l}-\kappa_1^{l-1}} \Big),
                             \kappa_2^{l-1} - \kappa_1^{l-1} \frac{\kappa_2^{l}-\kappa_2^{l-1}}{\kappa_1^{l}-\kappa_1^{l-1}} \Bigg),
           \quad l = l_0+1, \ldots, l^a.
\end{align*}

As in the $p_1=p_3$ case considered in \cite{IM16}, we expect that the conditions from Proposition \ref{Knapp_main_proposition} are sharp.
This will of course follow if we prove that the Fourier restriction estimate is true within the range they determine.
In the adapted case, when $d(\phi) = h(\phi)$, the only condition we obtained was
\begin{equation}
\label{Knapp_adapted_mixed}
\frac{1}{h(\phi)} \frac{1}{p_1'} + \frac{1}{p_3'} \leq \frac{1}{2 h(\phi)}.
\end{equation}
This condition is sharp as will be shown in Section \ref{Adapted},
though sometimes the endpoint estimate on the $1/p_3'$ axis will not hold.

\subsection{The form using the Legendre transformation}
\label{Knapp_Legendre}

As already noted, the necessary conditions can be stated as
\begin{align*}
\frac{(1+m) \tilde{\kappa}_1}{p_1'} + \frac{1}{p_3'} \leq \frac{\tilde{\kappa}_1+\tilde{\kappa}_2}{2},
\end{align*}
for all $(\tilde{\kappa}_1,\tilde{\kappa}_2)$ such that $L_{\tilde{\kappa}}$ is a supporting line to the augmented Newton polyhedron of $\phi^a$.
This can be rewritten as
\begin{align*}
\frac{1}{p_3'} \leq -\frac{1}{2} \Bigg( \Big(\frac{2+2m}{p_1'} - 1 \Big)\tilde{\kappa}_1 - \tilde{\kappa}_2\Bigg).
\end{align*}
As in Subsection \ref{Statement} we denote by $K$ the function associating to each $\tilde{\kappa}_1 \in [0,\kappa_1]$
the $\tilde{\kappa}_2$ such that $L_{\tilde{\kappa}}$ is a supporting line to the augmented Newton polyhedron of $\phi^a$,
i.e., we have $\tilde{\kappa} = (\tilde{\kappa}_1, K_f(\tilde{\kappa}_1))$.
The Legendre transformation of $K$ is given by
\begin{align*}
\mathcal{L}(K) [w] \coloneqq \sup_{u \in [0, \kappa_1]} (wu-K(u)),
\end{align*}
and thus we have
\begin{align*}
\frac{1}{p_3'} \leq -\frac{1}{2} \mathcal{L}(K) \Big[\frac{2+2m}{p_1'} - 1\Big].
\end{align*}
We have depicted the graph of $K$ in Figure \ref{Knapp_figure_Legendre}.



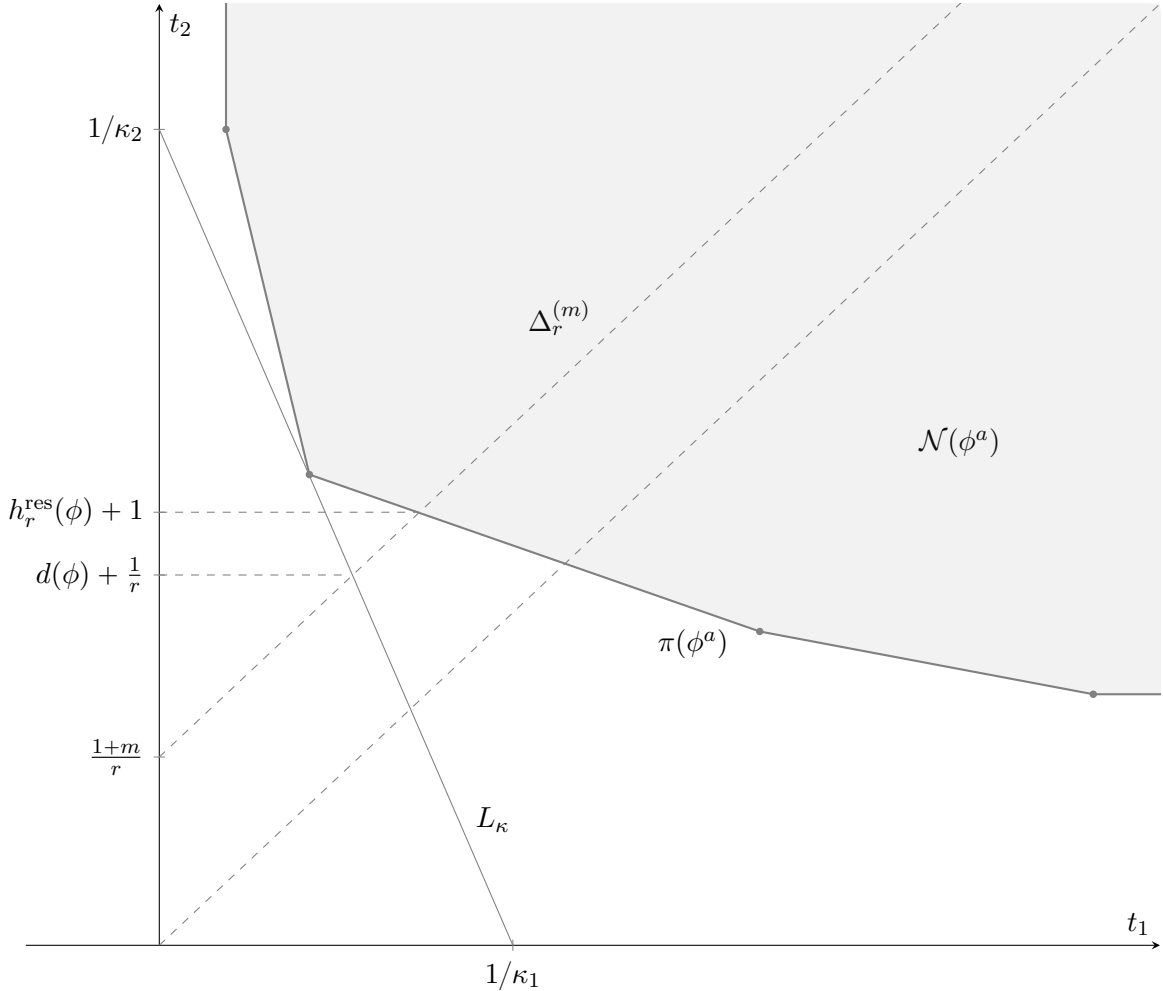
\begin{figure}[!h]
\centering

\begin{tikzpicture}
\begin{axis}[
xlabel=$t_1$, ylabel=$t_2$,
xmin= -2, xmax=15,
ymin= 0, ymax=15,
axis lines=middle,
xtick=\empty, ytick=\empty,
scale = 2.2,
extra x ticks={5.3},
extra x tick labels={
    $1/\kappa_1$
},
extra y ticks={3, 5.9, 6.9, 13},
extra y tick labels={
    $\frac{1+m}{r}$,
    $d(\phi)+\frac{1}{r}$,
    $h^{\text{res}}_r(\phi)+1$,
    $1/\kappa_2$
}
]

\addplot+ [
mark size=1pt,
solid,
thick,
gray,
mark options={color = gray, fill = gray},
mark = *,
] coordinates {
(1,16)
(1,13)
(2.25,7.5)
(9,5)
(14,4)
(16,4)
};

\addplot+ [
no marks,
solid,
gray,
mark options={color = gray, fill = gray},
mark = *,
] coordinates {
(0,13)
(5.3,0)
};

\node at (axis cs:  5, 2) {$L_{\kappa}$};

\addplot+ [
no marks,
dashed,
gray,
mark options={color = gray, fill = gray},
mark = *,
] coordinates {
(0,0)
(16,16)
};

\addplot+ [
no marks,
dashed,
gray,
mark options={color = gray, fill = gray},
mark = *,
] coordinates {
(0,3)
(13,16)
};

\node at (axis cs:  6, 10) {$\Delta^{(m)}_r$};

\addplot+ [
no marks,
dashed,
gray,
mark options={color = gray, fill = gray},
mark = *,
] coordinates {
(0,6.9)
(3.9,6.9)
};

\addplot+ [
no marks,
dashed,
gray,
mark options={color = gray, fill = gray},
mark = *,
] coordinates {
(0,5.9)
(2.8,5.9)
};

\node at (axis cs:  8, 4.8) {$\pi(\phi^a)$};

\node at (axis cs:  12,  8) {$\mathcal{N}(\phi^a)$};


\addplot+[
draw=none,
mark=none,
gray,
solid,
fill=gray, 
fill opacity=0.1]
coordinates {
(1,16)
(1,13)
(2.25,7.5)
(9,5)
(14,4)
(16,4)
(15,15)
};


\end{axis}
\end{tikzpicture}

\caption{The restriction height.}
\label{Knapp_figure_res_height} 

\end{figure}


\subsection{Conditions when the ratio is fixed}
\label{Knapp_ratio_fixed}

If we fix a ratio $r = p_1'/p_3' \in [0,\infty]$, then we are able to introduce a quantity slight more general
than the restriction height $h^{\text{res}}(\phi)$ introduced in \cite{IM16}.
We shall not use this quantity in this article, but it may prove useful when considering the mixed norm Fourier restriction
for functions $\phi$ with $h_{\text{lin}}(\phi) \geq 2$.
The cases $r \in \{0,\infty\}$ are not interesting since we shall prove the associated results in
Section \ref{Adapted} easily, so we assume that $r \in (0,\infty)$ is fixed. In this case the
conditions \eqref{Knapp_conditions_mixed} can be restated as
\begin{align*}
\frac{(1+m) \tilde{\kappa}_1}{r p_3'} + \frac{1}{p_3'} \leq \frac{\tilde{\kappa}_1+\tilde{\kappa}_2}{2},
\end{align*}
i.e.,
\begin{align*}
p_3' \geq 2 \, \frac{(1+m)\tilde{\kappa}_1+r}{r(\tilde{\kappa}_1+\tilde{\kappa}_2)},
\end{align*}
where again $\tilde{\kappa}$ is such that
$L_{\tilde{\kappa}}$ is a supporting line to the augmented Newton polyhedron $\mathcal{N}^\text{res}(\phi^f)$.
But now we notice that the number
\begin{align*}
\frac{(1+m)\tilde{\kappa}_1+r}{r(\tilde{\kappa}_1+\tilde{\kappa}_2)}
\end{align*}
is actually the $t_2$-coordinate of the intersection of the line $L_{\tilde{\kappa}}$ with the parametrised line
\begin{align*}
t \mapsto \Big( t - \frac{1+m}{r},t \Big)
\end{align*}
which we shall denote by $\Delta^{(m)}_r$.
This motivates us to define
\begin{align*}
h^l_r \coloneqq \frac{(1+m)\kappa_1^l+r}{r(\kappa_1^l+\kappa_2^l)}-1
\end{align*}
when $\kappa_2^l/\kappa_1^l > m$ (i.e., for $l \geq l_0$). Then if we define
\begin{align}
\label{Knapp_height_mixed_height}
h^{\text{res}}_r(\phi) \coloneqq \max\Bigg\{ d(\phi)+\frac{1}{r}-1, h^{l_0}_r, \ldots, h^{n+1}_r \Bigg\},
\end{align}
the conditions \eqref{Knapp_conditions_mixed} can be restated as the requirement that the inequalities
\begin{align}
\label{Knapp_conditions_r_height}
\begin{split}
p_1' &\geq 2r(1+h^{\text{res}}_r(\phi)),\\
p_3' &\geq 2(1+h^{\text{res}}_r(\phi)),
\end{split}
\end{align}
must hold necessarily true for all $r \in (0,\infty)$, along with the inequalities $p_1' \geq 2$ and $p_3' \geq 2h(\phi)$,
representing the respective cases $r = 0$ and $r = \infty$.

By definition, the restriction height $h^{\text{res}}(\phi)$ from \cite{IM16} coincides with $h^{\text{res}}_r(\phi)$ when $r = 1$,
and in the same way as in \cite{IM16} we see from \eqref{Knapp_height_mixed_height} that $h^{\text{res}}_r(\phi)+1$ can be read off
as the $t_2$-coordinate of the point where the line $\Delta^{(m)}_r$ intersects the augmented Newton diagram of $\phi^a$
(see Figure \ref{Knapp_figure_res_height}).



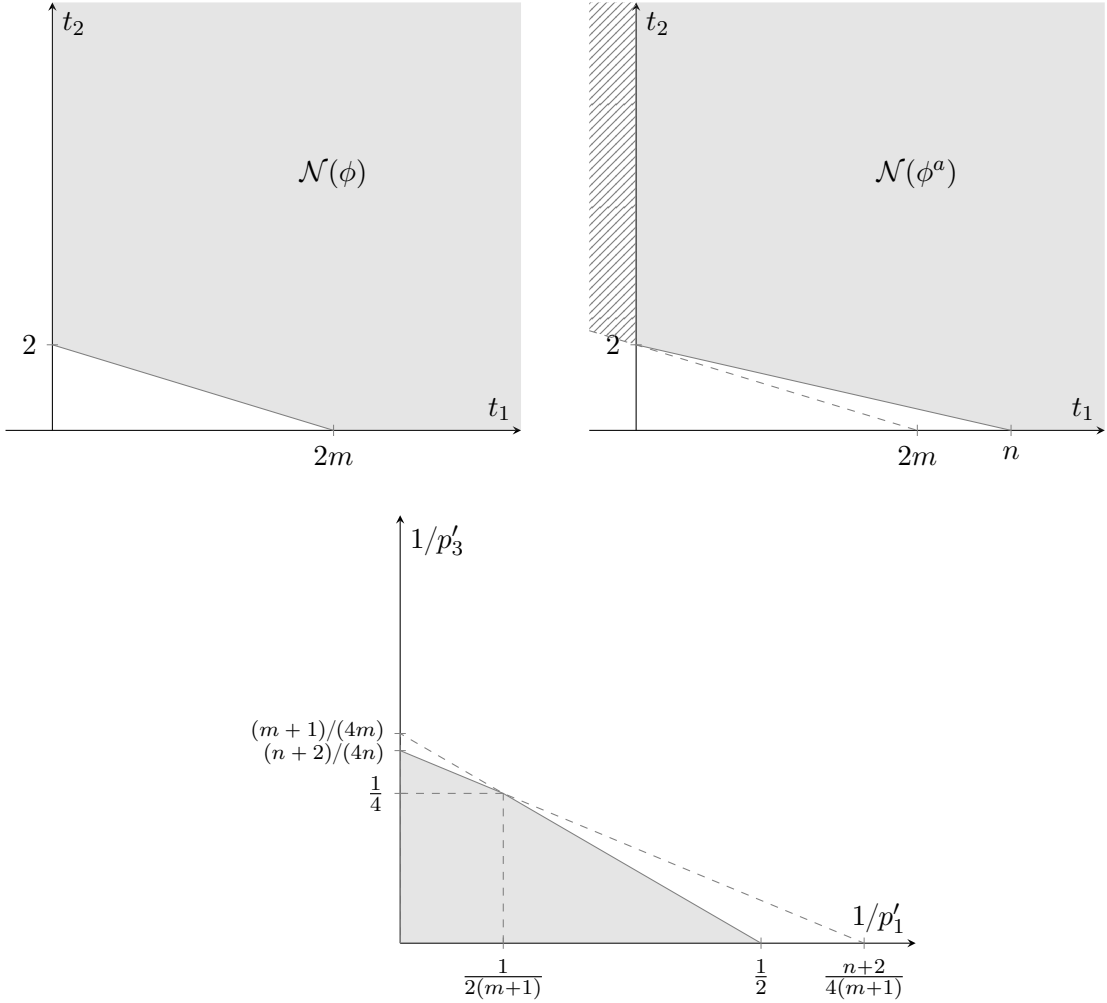
\begin{figure}[h]
\centering

\begin{tikzpicture}
\begin{axis}[
xlabel=$t_1$, ylabel=$t_2$,
xmin= -1, xmax=10,
ymin= 0, ymax=10,
axis lines=middle,
xtick=\empty, ytick=\empty,
extra x ticks={6},
extra x tick labels={
    $2m$
},
extra y ticks={2},
extra y tick labels={
    $2$
}
]

\addplot+[
mark=none,
gray,
solid,
fill=gray, 
fill opacity=0.2]
coordinates {
(0,2)
(6,0)
};

\addplot+[
draw=none,
mark=none,
gray,
solid,
fill=gray, 
fill opacity=0.2]
coordinates {
(0,10)
(0,2)
(6,0)
(10,0)
(10,10)
};

\node at (axis cs:  6,  6) {$\mathcal{N}(\phi)$};

\end{axis}
\end{tikzpicture}
\qquad
\begin{tikzpicture}
\begin{axis}[
xlabel=$t_1$, ylabel=$t_2$,
xmin= -1, xmax=10,
ymin= 0, ymax=10,
axis lines=middle,
xtick=\empty, ytick=\empty,
extra x ticks={6, 8},
extra x tick labels={
    $2m$,
    $n$
},
extra y ticks={2},
extra y tick labels={
    $2$
}
]

\addplot+[
mark=none,
gray,
solid,
fill=gray, 
fill opacity=0.2]
coordinates {
(0,2)
(8,0)
};

\addplot+[
draw=none,
mark=none,
gray,
solid,
fill=gray, 
fill opacity=0.2]
coordinates {
(0,10)
(0,2)
(8,0)
(10,0)
(10,10)
};

\addplot+[
mark=none,
gray,
dashed,
fill=gray, 
fill opacity=0.2]
coordinates {
(-1,2.33)
(6,0)
};

\addplot+[
draw=none,
mark=none,
gray,
solid,
pattern=north east lines,
pattern color=gray]
coordinates {
(-1,10)
(-1,2.33)
(0,2)
(0,10)
};

\node at (axis cs:  6,  6) {$\mathcal{N}(\phi^a)$};

\end{axis}
\end{tikzpicture}
\newline
\newline
\begin{tikzpicture}
\begin{axis}[
xlabel=$1/p'_1$, ylabel=$1/p'_3$,
xmin= 0, xmax=10,
ymin= 0, ymax=10,
axis lines=middle,
xtick=\empty, ytick=\empty,
extra x ticks={2, 7, 9},
extra x tick labels={
    $\frac{1}{2(m+1)}$,
    $\frac{1}{2}$,
    $\frac{n+2}{4(m+1)}$
},
extra y ticks={3.5, 4.5, 4.9},
extra y tick labels={
    $\frac{1}{4}$,
    {\scriptsize $(n+2)/(4n)$},
    {\scriptsize $(m+1)/(4m)$}
},
y tick label style={yshift={(\ticknum==1)*-0.1em}},
y tick label style={yshift={(\ticknum==2)*0.1em}}
]

\addplot+[
mark=none,
gray,
solid,
fill=gray, 
fill opacity=0.2]
coordinates {
(0,0)
(0,4.5)
(2,3.5)
(7,0)
};

\addplot+[
mark=none,
gray,
dashed,
] coordinates {
(0,0)
(0,4.9)
(2,3.5)
(9,0)
};

\addplot+ [
no marks,
dashed,
gray,
mark options={color = gray, fill = gray},
mark = *,
] coordinates {
(0,3.5)
(2,3.5)
};

\addplot+ [
no marks,
dashed,
gray,
mark options={color = gray, fill = gray},
mark = *,
] coordinates {
(2,0)
(2,3.5)
};

\end{axis}
\end{tikzpicture}

\caption{The Newton polyhedra associated to $A_{n-1}$ type singularity in the
(linearly adapted) original and adapted coordinates respectively, and the associated necessary conditions.}
\label{Knapp_figure_A} 

\end{figure}


\subsection{Necessary conditions when $h_\text{lin}(\phi)$ < 2}
\label{Knapp_h_lin_less_than_two}

In the case when $\phi$ is non-adapted and the linear height of $\phi$ is strictly less than $2$
it turns out that there are only two necessary conditions from Proposition \ref{Knapp_main_proposition}.
Namely, in this case we shall show that $l_0 = l^a$, and therefore the only conditions are
\begin{align*}
\begin{split}
\frac{(1+m) \kappa_1^{l^a}}{p_1'} + \frac{1}{p_3'} &\leq \frac{\kappa_1^{l^a}+\kappa_2^{l^a}}{2}, \\
\frac{(1+m) \kappa_1}{p_1'} + \frac{1}{p_3'} &\leq \frac{\kappa_1+\kappa_2}{2}.
\end{split}
\end{align*}
If we replace above the inequality signs with equality signs, we get two linear equations in $(1/p_1',1/p_3')$.
Let $(1/\mathfrak{p}_1',1/\mathfrak{p}_3')$ be the solution of this system.
We shall call $\mathfrak{p} = (\mathfrak{p}_1,\mathfrak{p}_3)$ \emph{the critical exponent}.
Then, by interpolation, it is sufficient to prove the Fourier restriction estimate \eqref{Introduction_FRP_mixed}
for the exponent $\mathfrak{p}$ and the endpoint exponents associated to the points lying on the axes, i.e., $(0,1/2)$ and $(1/(2h(\phi)),0)$.

In order to obtain what precisely the critical exponent $\mathfrak{p}$ is, we recall \cite[Proposition 2.11]{IM16}
which gives us explicit normal forms of $\phi$ in the case when $h_\text{lin}(\phi)$ < 2.
In the real analytic case these normal forms were derived in \cite{Si74} by D. Siersma.
\cite[Proposition 2.11]{IM16} states that there are two type of singularities, $A$ and $D$.

In the case of $A$ type singularity the form of the function $\phi$ is
\begin{equation}
\label{Knapp_A_form}
\phi(x_1,x_2) = b(x_1,x_2)(x_2-\psi(x_1))^2 + b_0(x_1).
\end{equation}
Here $\psi$, $b$, and $b_0$ are smooth functions such that $\psi(x_1) = cx_1^m + \mathcal{O}(x_1^{m+1})$
(with $c \neq 0$ and $m \geq 2$), $b(0,0) \neq 0$, and $b_0(x_1) = x_1^n \beta(x_1)$
(with either $\beta(0) \neq 0$ and $n \geq 2m+1$, or $b_0$ is flat, i.e., ``$n = \infty$'').
The function $\psi$ is the principal root jet of $\phi$. If $b_0$ is flat, this is $A_\infty$ type
singularity, and otherwise it is $A_{n-1}$ type singularity.
In adapted coordinates, the formula \eqref{Knapp_A_form} turns into
\begin{equation}
\label{Knapp_A_form_adapted}
\phi^a(y_1,y_2) = b^a(y_1,y_2) y_2^2 + b_0(y_1),
\end{equation}
where $b^a(y_1,y_2) = b(y_1,y_2+\psi(y_1))$, i.e., the function $b$ in $(y_1,y_2)$ coordinates.
From the formulas \eqref{Knapp_A_form} and \eqref{Knapp_A_form_adapted}
one can now determine the form of the Newton polyhedron of $\phi$ and $\phi^a$ (see Figure \ref{Knapp_figure_A}).
Reading off the Newton polyhedra we have
\begin{align*}
(\kappa_1, \kappa_2)             &= \Big( \frac{1}{2m}, \frac{1}{2} \Big), \quad\, d(\phi) = \frac{2m}{m+1}, \\
(\kappa_1^{l^a}, \kappa_2^{l^a}) &= \Big( \frac{1}{n},  \frac{1}{2} \Big), \qquad h(\phi) = \frac{2n}{n+2},
\end{align*}
and so the necessary conditions \eqref{Knapp_conditions_mixed} can be written as
\begin{align*}
\begin{split}
\frac{2}{p_1'} + \frac{4m}{m+1} \frac{1}{p_3'} &\leq 1, \\
\frac{4(m+1)}{n+2} \frac{1}{p_1'} + \frac{4n}{n+2} \frac{1}{p_3'} &\leq 1.
\end{split}
\end{align*}
Now an easy calculation shows that $(1/\mathfrak{p}_1',1/\mathfrak{p}_3') = (1/(2m+2),1/4)$,
i.e., we have determined the critical exponent.



\begin{figure}[h]
\centering

\begin{tikzpicture}
\begin{axis}[
xlabel=$t_1$, ylabel=$t_2$,
xmin= -1, xmax=10,
ymin= 0, ymax=10,
axis lines=middle,
xtick=\empty, ytick=\empty,
extra x ticks={1.3,5},
extra x tick labels={
    $1$,
    $2m+1$
},
extra y ticks={2.6,3.5,7},
extra y tick labels={
    $2$,
    $\frac{2m+1}{m}$,
    $4+k$
}
]

\addplot+[
mark=none,
gray,
solid,
]
coordinates {
(0,7)
(1.3,2.6)
(5,0)
};

\addplot+[
mark=none,
gray,
dashed,
]
coordinates {
(0,3.5)
(1.3,2.6)
};

\addplot+[
draw=none,
mark=none,
gray,
solid,
fill=gray, 
fill opacity=0.2]
coordinates {
(0,10)
(0,7)
(1.3,2.6)
(5,0)
(10,0)
(10,10)
};

\addplot+ [
no marks,
dashed,
gray,
mark options={color = gray, fill = gray},
mark = *,
] coordinates {
(0,2.6)
(1.3,2.6)
};

\addplot+ [
no marks,
dashed,
gray,
mark options={color = gray, fill = gray},
mark = *,
] coordinates {
(1.3,0)
(1.3,2.6)
};

\node at (axis cs:  6,  6) {$\mathcal{N}(\phi)$};

\end{axis}
\end{tikzpicture}
\qquad
\begin{tikzpicture}
\begin{axis}[
xlabel=$t_1$, ylabel=$t_2$,
xmin= -1, xmax=10,
ymin= 0, ymax=10,
axis lines=middle,
xtick=\empty, ytick=\empty,
extra x ticks={1.3,5,8},
extra x tick labels={
    $1$,
    $2m+1$,
    $n$
},
extra y ticks={2.6,3.5,7},
extra y tick labels={
    $2$,
    $\frac{2m+1}{m}$,
    $4+k$
}
]

\addplot+[
mark=none,
gray,
solid,
]
coordinates {
(0,7)
(1.3,2.6)
(8,0)
};

\addplot+[
mark=none,
gray,
dashed,
]
coordinates {
(0,3.5)
(1.3,2.6)
(5,0)
};

\addplot+[
draw=none,
mark=none,
gray,
solid,
fill=gray, 
fill opacity=0.2]
coordinates {
(0,10)
(0,7)
(1.3,2.6)
(8,0)
(10,0)
(10,10)
};

\addplot+[
draw=none,
mark=none,
gray,
solid,
pattern=north east lines,
pattern color=gray
] coordinates {
(-1,10)
(0,10)
(0,7)
(1.3,2.6)
(0,3.5)
(-1,4.25)
};

\addplot+ [
no marks,
dashed,
gray,
mark options={color = gray, fill = gray},
mark = *,
] coordinates {
(0,2.6)
(1.3,2.6)
};

\addplot+ [
no marks,
dashed,
gray,
mark options={color = gray, fill = gray},
mark = *,
] coordinates {
(1.3,0)
(1.3,2.6)
};

\node at (axis cs:  6,  6) {$\mathcal{N}(\phi^a)$};

\end{axis}
\end{tikzpicture}
\newline
\newline
\begin{tikzpicture}
\begin{axis}[
xlabel=$1/p'_1$, ylabel=$1/p'_3$,
xmin= 0, xmax=10,
ymin= 0, ymax=10,
axis lines=middle,
xtick=\empty, ytick=\empty,
extra x ticks={2, 7, 9},
extra x tick labels={
    $\frac{1}{4(m+1)}$,
    $\frac{1}{2}$,
    $\frac{n+1}{4(m+1)}$
},
extra y ticks={3.5, 4.5, 4.9},
extra y tick labels={
    $\frac{1}{4}$,
    {\scriptsize $(n+1)/(4n)$},
    {\scriptsize $(m+1)/(2(2m+1))$}
},
y tick label style={yshift={(\ticknum==1)*-0.1em}},
y tick label style={yshift={(\ticknum==2)*0.1em}}
]

\addplot+[
mark=none,
gray,
solid,
fill=gray, 
fill opacity=0.2]
coordinates {
(0,0)
(0,4.5)
(2,3.5)
(7,0)
};

\addplot+[
mark=none,
gray,
dashed,
] coordinates {
(0,0)
(0,4.9)
(2,3.5)
(9,0)
};

\addplot+ [
no marks,
dashed,
gray,
mark options={color = gray, fill = gray},
mark = *,
] coordinates {
(0,3.5)
(2,3.5)
};

\addplot+ [
no marks,
dashed,
gray,
mark options={color = gray, fill = gray},
mark = *,
] coordinates {
(2,0)
(2,3.5)
};

\end{axis}
\end{tikzpicture}

\caption{The Newton polyhedra associated to $D_{n+1}$ type singularity in the
(linearly adapted) original and adapted coordinates respectively, and the associated necessary conditions.}
\label{Knapp_figure_D} 

\end{figure}
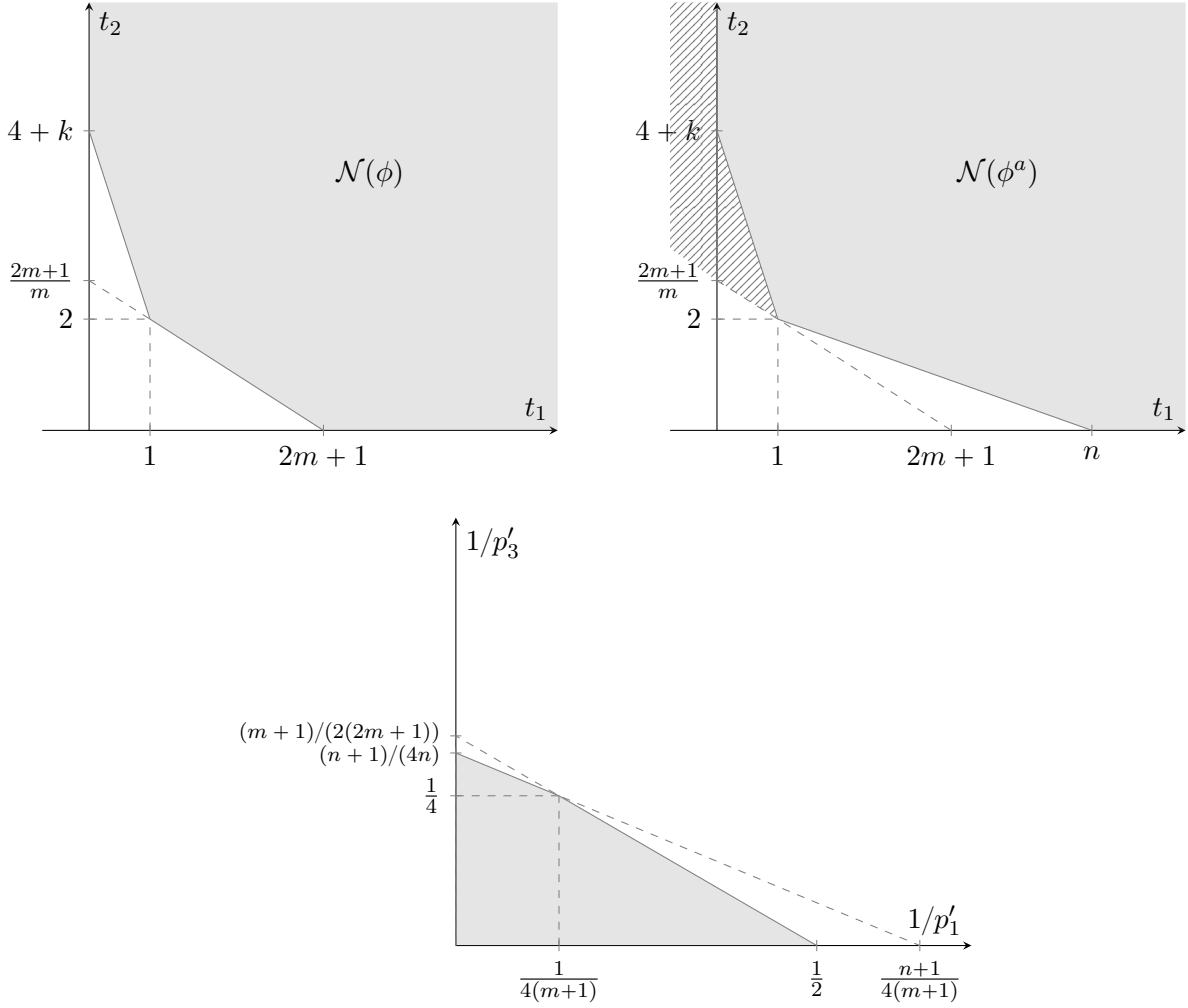


In the case of $D$ type singularity \cite[Proposition 2.11]{IM16} tells us that
\begin{align}
\label{Knapp_D_form}
\begin{split}
\phi(x_1,x_2)   &= (x_1 b_1(x_1,x_2) + x_2^2 b_2(x_2))(x_2-\psi(x_1))^2 + b_0(x_1),\\
\phi^a(y_1,y_2) &= \Big(y_1 b_1^a(y_1,y_2) + (y_2+\psi(y_1))^2 b_2(y_2+\psi(y_1))\Big) y_2^2 + b_0(y_1),
\end{split}
\end{align}
i.e., the function $b$ from $\eqref{Knapp_A_form}$ is now to be written as $b(x_1,x_2) = x_1 b_1(x_1,x_2) + x_2^2 b_2(x_2)$.
In this case we have the conditions $b_1(0,0) \neq 0$ and $b_2(x_2) = c_2 x_2^k + \mathcal{O}(x_2^{k+1})$.
Again $\psi(x_1) = cx_1^m + \mathcal{O}(x_1^{m+1})$ ($c \neq 0$, $m \geq 2$) and $b_0(x_1) = x_1^n \beta(x_1)$,
but now either $\beta(0) \neq 0$ and $n \geq 2m+2$, or $b_0$ is flat.
If $b_0$ is flat, this is $D_\infty$ type singularity, and otherwise it is $D_{n+1}$ type singularity.
The function $b_1^a$ is the function $b_1$ in $(y_1,y_2)$ coordinates.
Now one determines the form of the Newton polyhedra (see Figure \ref{Knapp_figure_D}) and reads off that
\begin{align*}
(\kappa_1, \kappa_2)             &= \Big( \frac{1}{2m+1}, \frac{m}{2m+1} \Big), \quad d(\phi) = \frac{2m+1}{m+1}, \\
(\kappa_1^{l^a}, \kappa_2^{l^a}) &= \Big( \frac{1}{n},  \frac{n-1}{2n} \Big), \,\qquad\qquad h(\phi) = \frac{2n}{n+1},
\end{align*}
Therefore, the necessary conditions can be written as
\begin{align*}
\begin{split}
\frac{2}{p_1'} + \frac{2(2m+1)}{m+1}\frac{1}{p_3'} &\leq 1, \\
\frac{4(m+1)}{n+1}\frac{1}{p_1'} + \frac{4n}{n+1}\frac{1}{p_3'} &\leq 1.
\end{split}
\end{align*}
Again, a simple calculation shows that $(1/\mathfrak{p}_1',1/\mathfrak{p}_3') = (1/(4m+4),1/4)$.

Note that in the $A_\infty$ and $D_\infty$ cases the necessary conditions form a right-angled trapezium
in the $(1/p_1',1/p_3')$-plane (easily seen by taking $n \to \infty$; one can also do a direct calculation).
As the critical exponents in the cases $A_{n-1}$ and $D_{n+1}$ do not depend on $n$,
one is easily convinced that the critical exponents of $A_\infty$ and $D_\infty$ cases are equal to
the respective critical exponents of $A_{n-1}$ and $D_{n+1}$.


\section{Auxiliary results}
\label{Technical_results}


\subsection{Reduction to the case $\nabla \phi(0) = 0$}
\label{Basic_assumptions_phi}

In Subsection \ref{Notation_and_assumptions} we mentioned that one can always reduce
the mixed normed Fourier restriction problem to the case when $\nabla \phi(0) = 0$,
despite rotational invariance not being at one's disposal.
Let us justify this. Consider the linear transformation
\begin{align*}
L(x_1,x_2,x_3) \coloneqq (x_1,x_2,x_3 + \partial_1 \phi(0) x_1 + \partial_2 \phi(0) x_2)
\end{align*}
whose inverse and transpose are
\begin{align*}
L^{-1}(x_1,x_2,x_3) = (x_1,x_2,x_3 - \partial_1 \phi(0) x_1 - \partial_2 \phi(0) x_2), \\
L^t(x_1,x_2,x_3)    = (x_1 + \partial_1 \phi(0) x_3, x_2 + \partial_2 \phi(0) x_3, x_3).
\end{align*}
Plugging in the function $f \circ L^t$ into the expression of
the mixed norm Fourier restriction estimate \eqref{Introduction_FRP_mixed} we obtain
\begin{equation*}
\Bigg( \int |\FT (f \circ L^t)|^2 (\xi,\phi(\xi)) \rho(\xi) \sqrt{1+|\nabla \phi(\xi)|^2} \mathrm{d} \xi \Bigg)^{1/2} \leq
\Const_{p} \Vert f \circ L^t \Vert_{L^{p_3}_{x_3} (L^{p_1}_{(x_1,x_2)})}.
\end{equation*}
Now one just notices that $\Vert f \circ L^t \Vert_{L^{p_3}_{x_3} (L^{p_1}_{(x_1,x_2)})} = \Vert f \Vert_{L^{p_3}_{x_3} (L^{p_1}_{(x_1,x_2)})}$,
and that
\begin{align*}
|\FT (f \circ L^t)|^2 (\xi,\phi(\xi))
   &= |\FT f|^2 (L^{-1}(\xi,\phi(\xi))) \\
   &= |\FT f|^2 (\xi,\phi(\xi) - \xi \cdot \nabla \phi(0)),
\end{align*}
since the determinant of $L$ is $1$.
Thus the estimate \eqref{Introduction_FRP_mixed} with the function $\phi$ is equivalent
(up to a slight change in amplitude due to the Jacobian factor $\sqrt{1+|\nabla \phi(\xi)|^2}$)
to the same estimate with the function $\phi$ replaced by the function $\xi \mapsto \phi(\xi) - \xi \cdot \nabla\phi(0)$,
which has gradient $0$ at the origin.


\subsection{Auxiliary results related to oscillatory sums and integrals}
\label{Oscillatory_auxiliary}

We shall often need the following two one-dimensional oscillatory integral results.
The first one is a van der Corput-type estimate used in \cite{IM16} and originating in the works of
van der Corput \cite{vdC21}, G.I. Arhipov \cite{AKC79}, and J.E. Bj\"ork (as noted in \cite{Do77}).
\begin{lemma}
\label{Oscillatory_auxiliary_van_der_corput}
Let $M \geq 2$ be an integer and let $f \in C^M(I)$ be a real-valued function on the interval $I \subset \R$.
Let us assume that either
\begin{enumerate}
\item[(i)]
$|f^{(M)}(s)| \geq 1$ for every $s \in I$, or
\item[(ii)]
f is of polynomial type $M \geq 2$, that is, $I$ is compact and there are positive constants $c_1, c_2$ such that
\begin{align*}
c_1 \leq \sum_{j=1}^M |f^{(j)}(s)| \leq c_2, \quad \text{for every} \,\, s \in I.
\end{align*}
\end{enumerate}
Then there exists a constant $C$ which depends only on $M$ in case (i), and on $M$, $c_1$, $c_2$, and $I$ in case (ii),
such that for every $\lambda \in \R$ we have

\begin{align*}
\Big| \int_I e^{i\lambda f(s)} g(s) \mathrm{d}s \Big| \leq C(\Vert g\Vert_{L^\infty(I)} + \Vert g'\Vert_{L^1(I)}) |\lambda|^{-1/M},
\end{align*}
for any $L^\infty(I)$ function $g$ with an integrable derivative on $I$.
Furthermore, if $G \in L^1(\R)$ is a nonnegative function
which is majorized by a function $H \in L^1(\R)$ such that $\widehat{H} \in L^1(\R)$,
then for the same constant $C$ as above we have
\begin{align*}
\int_I G(\lambda f(s)) \mathrm{d}s \leq C (\Vert H\Vert_{L^1(\R)}+\Vert \widehat{H}\Vert_{L^1(\R)}) |\lambda|^{-1/M}.
\end{align*}
\end{lemma}

We note that in the above lemma in case $(ii)$ we can use in both expressions $(1+|\lambda|)^{-1/M}$
instead of $|\lambda|^{-1/M}$ since the constant $C$ depends on $I$ anyway.\footnote{
There is a typo in \cite[Lemma 2.1]{IM16}:
in the estimate in part (a) the expression $(1+|\lambda|)^{-1/M}$ should be changed to $|\lambda|^{-1/M}$.
}

We also remark that we can always use $G = |\varphi|$ for a Schwartz function $\varphi$ since
the Fourier transform of $|\varphi|$ is integrable.
The proof of this (known) fact is almost straightforward.
Namely, the derivative of $|\varphi|$ can have jumps only at the points $s$ where $\varphi(s) = 0$ and $\varphi'(s) \neq 0$.
Denote the set of such points $N$ and note that it is a discrete set.
In order to estimate the Fourier transform of $|\varphi|$ at $\xi$, one integrates by parts the expression
\begin{align*}
(\FT |\varphi|)(\xi) = \int e^{- i x \xi} |\varphi|(x) \mathrm{d}x
\end{align*}
twice and gets the additional boundary terms which can be estimated by $|\xi|^{-2} \sum_{s \in N} |\varphi'(s)|$.
Using the fact that between any two neighbouring points $s_1, s_2 \in N$ there is a point $s$ inbetween such that $\varphi'(s) = 0$
one easily gets $\sum_{s \in N} |\varphi'(s)| \leq \int |\varphi''(s)| \mathrm{d}s < +\infty$ and the claim follows.

The second lemma (less general, but with a stronger implication than the one in \cite[Section 2.2]{IM16}) we need
gives us an asymptotic of an oscillatory integral of Airy type.
We shall also need some variants, but these we shall state and prove along the way when they are needed.
\begin{lemma}
\label{Oscillatory_auxiliary_airy}
For $\lambda \geq 1$ and $u \in \R$, $|u| \lesssim 1$, let us consider the integral
\begin{align*}
J(\lambda,u,s) \coloneqq \int_\R e^{i \lambda (b(t,s)t^3 -ut)} a(t,s) \mathrm{d}t,
\end{align*}
where $a, b$ are smooth and real-valued functions on an open neighbourhood of $I \times K$
for $I$ a compact neighbourhood of the origin in $\R$ and $K$ a compact subset of $\R^m$.
Let us assume that $b(t,s) \neq 0$ on $I \times K$ and that $|t| \leq \varepsilon$ on the support of $a$.
If $\varepsilon > 0$ is chosen sufficiently small and $\lambda$ sufficiently large,
then the following holds true:
\begin{enumerate}
\item[(a)]
If $\lambda^{2/3} |u| \lesssim 1$, then we can write
\begin{align*}
J(\lambda,u,s) = \lambda^{-1/3} g(\lambda^{2/3}u, \lambda^{-1/3}, s),
\end{align*}
where $g(v, \mu, s)$ is a smooth function of $(v,\mu,s)$ on its natural domain.

\item[(b)]
If $\lambda^{2/3} |u| \gg 1$, then we can write
\begin{align*}
J(\lambda,u,s) =
   &\lambda^{-1/2} |u|^{-1/4} \chi_0(u/\varepsilon) \, \sum_{\tau \in \{+,-\}} a_\tau(|u|^{1/2},s; \lambda |u|^{3/2}) e^{i \lambda |u|^{3/2} q_\tau(|u|^{1/2},s)} \\
   &+ (\lambda |u|)^{-1} E(\lambda |u|^{3/2}, |u|^{1/2}, s),
\end{align*}
where $a_{\pm}$ are smooth functions in $(|u|^{1/2},s)$ and
classical symbols\footnote{
There is a slight error in \cite[Lemma 2.2]{IM16}.
Namely, there the functions $a_{\pm}$ should also be classical symbols of order $0$ in the same variable as stated here.
}
of order $0$ in $\lambda |u|^{3/2}$,
and where $q_\pm$ are smooth functions such that $|q_\pm| \sim 1$.
The function $E$ is a smooth function satisfying
\begin{align*}
|\partial^\alpha_\mu \partial^\beta_v \partial^\gamma_s E(\mu, v, s)| \leq C_{N,\alpha,\beta,\gamma} |\mu|^{-N},
\end{align*}
for all $N, \alpha, \beta, \gamma \in \N_0$.

\end{enumerate}
\end{lemma}

\begin{proof}
For the part (a) we only sketch the proof since it is a straightforward modification of \cite[Lemma 2.2., (a)]{IM16}.
In the integral defining $J$ we substitute $t \mapsto \lambda^{-1/3} t$.
Then we can write
\begin{align*}
J(\lambda,u,s) = \lambda^{-1/3} \int_\R e^{i (b(\lambda^{-1/3}t,s)t^3 - \lambda^{2/3}ut)}
                 a(\lambda^{-1/3} t,s) \chi_0(\lambda^{-1/3}t/\varepsilon) \mathrm{d}t.
\end{align*}
We added the smooth cutoff function $\chi_0$ localised near $0$ in order to emphasize that domain of integration.
If we denote
\begin{align*}
v &= \lambda^{2/3} u, \\
\mu &= \lambda^{-1/3},
\end{align*}
then the integral can be written as
\begin{align*}
\int_\R e^{i (b(\mu t,s)t^3 - vt)}
a(\mu t,s) \chi_0(\mu t/\varepsilon) \mathrm{d}t.
\end{align*}
We split the integral into two parts, depending on whether
the integration domain is contained in $|t| \lesssim C$ or $|t| > C$ for some fixed large $C$,
by using a smooth cutoff function.
The part where $|t| \lesssim C$ is obviously smooth in all the (bounded) parameters $(v,\mu,s)$
and hence it satisfies the conclusion of the lemma.
If $C$ is sufficiently large, $\varepsilon$ sufficiently small, and $|t| > C$, then
\begin{align*}
|\partial_{t} (b(\mu t,s)t^3 - vt)| &\sim |t|^2, \\
|\partial^N_{t} \partial^\alpha (b(\mu t,s)t^3 - vt)| &\lesssim_{N,\alpha} |t|^{3+|\alpha|-N},
\end{align*}
where $\partial^\alpha$ is any derivative in the $(v,\mu,s)$ variables.
Therefore by taking derivatives of the integral in $(v,\mu,s)$, factors of polynomial growth in $t$ appear.
This can be controlled by using integration by parts a sufficient number of times since the phase
derivative is $\sim |t|^2$, and so we get the uniform estimate in this case too.

The part (b) is also a straightforward modification of \cite[Lemma 2.2, (b)]{IM16},
and so we sketch the proof.
Here we get a stronger result for the function $E$ compared to \cite[Lemma 2.2, (b)]{IM16} since we assume that there are no $t^2$ terms in the phase.
We start by substituting $t \mapsto |u|^{1/2}t$.
Then one gets
\begin{align*}
J(\lambda,u,s) = v \int_\R e^{i \mu(b(v t,s)t^3 - (\sgn u) t)}
                 a(v t,s) \chi_0(v t/\varepsilon) \mathrm{d}t,
\end{align*}
where $\mu$ denotes $\lambda |u|^{3/2}$ and $v$ denotes $|u|^{1/2}$.
If $|u| \gtrsim \varepsilon$ and if $\varepsilon$ is sufficiently small, then
the integration domain is $|t| \ll 1$, and so we may use integration by parts
and get an estimate as is required for the $E$ term in the conclusion.

Let us now assume $|u| \ll \varepsilon$, and so in particular $|v| \ll 1$. The derivative of the phase is
\begin{align*}
\partial_{t} (b(v t,s)t^3 - (\sgn u)t) = t^2 (3b(vt,s) + vt (\partial_1 b)(vt,s)) - \sgn(u).
\end{align*}
If $t$ is away from the critical points (which only exist if $u$ and $b$ are of the same sign),
then we can argue similarly as in the (a) part of the proof by using integration by parts
and get an estimate as is required for the $E$ term in the conclusion.
If $u$ and $b$ have the same sign, then there are two critical points $|t_{\pm}(v,s)| \sim 1$.
One now applies the stationary phase method at each of the critical points
and obtains the form as in the conclusion of the theorem.
\end{proof}

\medskip

Next, we state results relating the Newton polyhedron and its associated quantities with asymptotics of oscillatory integrals.
\begin{theorem}
\label{Oscillatory_auxiliary_decay}
Let $\phi : \Omega \to \R$ be a smooth function of finite type defined on an open set $\Omega \subset \R^2$ containing the origin.
If $\Omega$ is a sufficiently small neighbourhood of the origin and $\eta \in C_c^\infty(\Omega)$, then
\begin{align*}
\Bigg| \int e^{i(\xi_1 x_1 + \xi_2 x_2 + \xi_3 \phi(x))} \eta(x) \mathrm{d}x \Bigg| \leq C_\eta (1+ |\xi|)^{-1/h(\phi)} \, (\log(2+|\xi|))^{\nu(\phi)},
\end{align*}
for all $\xi \in \R^3$.
\end{theorem}
This result was proven in \cite{IM11b} and can be interpreted as a uniform estimate with respect to a linear pertubation of the phase.
The case when $h(\phi) < 2$ was considered earlier in \cite{Dui74}.
The case when $\phi$ is real analytic and there is no pertubation (i.e., $\xi_1=\xi_2=0$)
the above result goes back to Varchenko \cite{Var76}.
In the case of a real analytic function $\phi$ one actually has a uniform estimate with respect to analytic pertubations
(this was proved by Karpushkin in \cite{Kar84}).

We also have the following result from \cite{IM11b} which gives us sharpness of Theorem \ref{Oscillatory_auxiliary_decay}
in the case when $\xi_1=\xi_2=0$.
\begin{theorem}
\label{Oscillatory_auxiliary_decay_sharpness}
Let $\phi$ be as in Theorem \ref{Oscillatory_auxiliary_decay} and let us define for $\lambda > 0$ the function
\begin{align*}
J_{\pm}(\lambda) \coloneqq
   \int e^{\pm i \lambda \phi(x)} \eta(x) \mathrm{d}x
\end{align*}
for an $\eta \in C_c^\infty(\Omega)$.
If the principal face $\pi(\phi^a)$ of $\mathcal{N}(\phi^a)$ is a compact face,
and if $\Omega$ is a sufficiently small neighbourhood of the origin, then
\begin{align*}
\lim_{\lambda \to +\infty} \frac{\lambda^{1/h(\phi)}}{(\log \lambda)^{\nu(\phi)}} J_{\pm}(\lambda) = c_{\pm} \eta(0),
\end{align*}
where $c_{\pm}$ are nonzero constants depending on the phase $\phi$ only.
\end{theorem}
An analogous result was proved earlier by Greenblatt in \cite{Grb09} for real analytic phase functions $\phi$.
When the principal face is not compact, Theorem \ref{Oscillatory_auxiliary_decay_sharpness} may fail in general
(for an example of this see \cite{IS97}).

\medskip

Finally, we state three lemmas which we shall often use in conjunction with Stein's complex interpolation theorem.
The proofs of the first and third lemma can be found in \cite[Section 2.5]{IM16},
while we only give a brief note on the proof of the second lemma since it is a direct modification of the first one.
The proof of all of them are elementary, though the proof of the third one is quite technical.
\begin{lemma}
\label{Oscillatory_auxiliary_one_parameter}
Let $Q = \prod_{k=1}^n[-R_k,R_k]$ be a compact cube in $\R^n$ for some real numbers $R_k > 0$, $k = 1,\ldots,n$,
and let $\alpha,\beta^1,\ldots,\beta^n$ be some fixed nonzero real numbers.
For a $C^1$ function $H$ defined on an open neighbourhood of $Q$, nonzero real numbers $a_1, \ldots, a_n$, and
$M$ a positive integer we define
\begin{align*}
F(t) \coloneqq \sum_{l=0}^M 2^{i \alpha l t} (H \chi_Q) (2^{\beta^1 l} a_1, \ldots, 2^{\beta^n l} a_n)
\end{align*}
for $t \in \R$.
Then there is a constant $C$ which depends only on $Q$ and the numbers $\alpha$ and $\beta^k$'s,
but not on $H$, $a_k$'s, $M$, and $t$, such that
\begin{align*}
|F(t)| \leq C \frac{\Vert H \Vert_{C^1(Q)}}{|2^{i \alpha t}-1|}
\end{align*}
for all $t \in \R$.
\end{lemma}
We shall often use this lemma in combination with the holomorphic function
\begin{align}
\label{Oscillatory_auxiliary_one_parameter_gamma}
\gamma(\zeta) \coloneqq \frac{2^{\alpha (\zeta-1)}-1}{2^{\alpha (\theta-1)}-1}
\end{align}
when applying complex interpolation. This function has the property that
\begin{align*}
\Big| \gamma(1+it) F(t) \Big| \leq C_\theta
\end{align*}
for a positive constant $C_\theta < +\infty$, and $\gamma(\theta) = 1$.

The following lemma is a slight variation of what was written in \cite[Remark 2.8]{IM16}.
\begin{lemma}
\label{Oscillatory_auxiliary_one_parameter_hoelder}
Let $Q = \prod_{k=1}^n[-R_k,R_k]$ be a compact cube in $\R^n$ for some real numbers $R_k > 0$, $k = 1,\ldots,n$,
let $\alpha,\beta^1,\ldots,\beta^n$ be some fixed nonzero real numbers, and let $0 < \epsilon < 1$.
For a $C^1$ function $H$ on a neighbourhood of $Q$, nonzero real numbers $a_1, \ldots, a_n$, and
$M$ a positive integer we define
\begin{align*}
F(t) \coloneqq \sum_{l=0}^M 2^{i \alpha l t} (H \chi_Q) (2^{\beta^1 l} a_1, \ldots, 2^{\beta^n l} a_n)
\end{align*}
for $t \in \R$.
Then there is a constant $C$ which depends only on $Q$ and the numbers $\alpha$, $\beta^k$'s, and $\epsilon$,
but not on $H$, $a_k$'s, $M$, and $t$, such that
\begin{align*}
|F(t)| \leq C \frac{|H(0)| + \sum_{k=1}^n C_k}{|2^{i \alpha t}-1|}
\end{align*}
for all $t \in \R$.
The constants $C_k$ are given as
\begin{align*}
C_1 &\coloneqq \sup_{y_1 \in R_1} |y_1|^{1-\epsilon} \int_0^1 |(\partial_1 H)(sy_1,0,\ldots,0)| \mathrm{d}s,\\
C_k &\coloneqq \sup_{y_1, \ldots, y_k} |y_k|^{1-\epsilon} \int_0^1 |(\partial_k H)(y_1,\ldots,y_{k-1},sy_k,0,\ldots,0)| \mathrm{d}s, \quad k > 1,
\end{align*}
where the supremum goes over the set $\prod_{j=1}^k[-R_j,R_j]$.
\end{lemma}
The only difference compared to the proof of \cite[Lemma 2.7]{IM16} is that one now writes
\begin{align}
\label{Oscillatory_auxiliary_one_parameter_hoelder_proof_note}
H(y) &= H(0) + |y_1|^\epsilon \frac{H(y_1,0,\ldots,0)-H(0)}{|y_1|^\epsilon} \nonumber \\
     &+ \sum_{k=1}^n |y_k|^\epsilon \frac{H(y_1,\ldots,y_{k-1},y_k,0,\ldots,0) - H(y_1,\ldots,y_{k-1},0,0,\ldots,0)}{|y_k|^\epsilon},
\end{align}
and notes that the fractions are bounded by their respective $C_k$'s.

In the above lemma we could have directly defined $C_k$'s as the H\"older quotients appearing in
\eqref{Oscillatory_auxiliary_one_parameter_hoelder_proof_note},
but the formulas used in Lemma \ref{Oscillatory_auxiliary_one_parameter_hoelder} turn out to be more practical.
One can easily construct an example though where using the H\"older quotients is more appropriate.
One example is when one has an oscillatory factor such as in $H(y_1) = y_1^\epsilon \, e^{i y_1^{-1}}$ , $0 < y_1 < 1$
(cf. the Riemann singularity as in \cite[Chapter VIII, Subsection 1.4.2]{Ste93}).
This function is $\epsilon$-H\"older continuous at $0$ and satisfies the conclusion of Lemma \ref{Oscillatory_auxiliary_one_parameter_hoelder}
in the sense that $|F(t)| \leq C/|2^{i\alpha t} - 1|$,
but one can show without too much effort that the integral defining $C_1$ in Lemma \ref{Oscillatory_auxiliary_one_parameter_hoelder} is infinite.

The third lemma is a two parameter version of the first one.
\begin{lemma}
\label{Oscillatory_auxiliary_two_parameter}
Let $Q = \prod_{k=1}^n[-R_k,R_k]$ be a compact cube in $\R^n$ for some real numbers $R_k > 0$, $k = 1,\ldots,n$,
and let $\alpha_1,\alpha_2 \in \Q^{\times}$, and $\beta^k_1, \beta^k_2 \in \Q$, $k = 1,\ldots,n$, be fixed numbers
such that
\begin{align*}
\alpha_1 \beta_2^k - \alpha_2 \beta_1^k \neq 0,
\end{align*}
for all $k$ (i.e., the vector $(\alpha_1, \alpha_2)$ is linearly independent from $(\beta_1^k, \beta_2^k)$).
For a $C^2$ function $H$ defined on an open neighbourhood of $Q$, nonzero real numbers $a_1, \ldots, a_n$, and
$M_1, M_2$ positive integers we define
\begin{align*}
F(t) \coloneqq \sum_{l_1=0}^{M_1} \sum_{l_2=0}^{M_2} 2^{i (\alpha_1 l_1 + \alpha_2 l_2) t}
   (H \chi_Q) (2^{(\beta^1_1 l_1+ \beta^1_2 l_2)} a_1, \ldots, 2^{(\beta^n_1 l_1 + \beta^n_2 l_2)} a_n)
\end{align*}
for $t \in \R$.
Then there is a constant $C$ which depends only on $Q$ and the numbers $\alpha_1, \alpha_2$, $\beta_1^k$'s, $\beta_2^k$'s,
but not on $H$, $a_k$'s, $M_1$, $M_2$, and $t$, such that
\begin{align*}
|F(t)| \leq C \frac{\Vert H \Vert_{C^2(Q)}}{|\rho(t)|}
\end{align*}
for all $t \in \R$.
The function $\rho$ is defined by $\rho(t) \coloneqq \prod_{\nu = 1}^N \tilde{\rho}(\nu t) \tilde{\rho}(-\nu t)$, where
\begin{align*}
\tilde{\rho}(t) \coloneqq (2^{i \alpha_1 t} - 1)(2^{i \alpha_2 t} - 1) \prod_{k=1}^n (2^{i(\alpha_1 \beta_2^k - \alpha_2 \beta_1^k)t}-1),
\end{align*}
and $N$ is a positive integer depending on the $\beta_1^k$'s and $\beta_2^k$'s.
\end{lemma}

For future reference, we also note the following construction from \cite[Remark 2.10]{IM16} of a complex function $\gamma$ on the strip
$\Sigma \coloneqq \{\zeta \in \C : 0 \leq \Re \zeta \leq 1\}$
which shall be used in the context of complex interpolation together with the above two parameter lemma.
If we are given $0 < \theta < 1$ and the exponents $\alpha_1, \alpha_2$, and $\beta_1^k$'s, $\beta_2^k$'s as above,
we define
\begin{align}
\label{Oscillatory_auxiliary_two_parameter_gamma}
\gamma(\zeta) \coloneqq \prod_{\nu=1}^N \frac{\tilde{\gamma}(\nu(\zeta-1)) \tilde{\gamma}(-\nu(\zeta-1))}{\tilde{\gamma}(\nu(\theta-1)) \tilde{\gamma}(-\nu(\theta-1))},
\end{align}
where
\begin{align*}
\tilde{\gamma}(\zeta) \coloneqq (2^{\alpha_1 \zeta}-1)(2^{\alpha_2 \zeta}-1) \prod_{k=1}^n (2^{(\alpha_1 \beta_2^k - \alpha_2 \beta_1^k)\zeta}-1).
\end{align*}
The function $\gamma$ has the following two key properties.
It is an entire analytic function uniformly bounded on the strip $\Sigma$,
and for the function $F$ as in Lemma \ref{Oscillatory_auxiliary_two_parameter} there is a positive constant $C_\theta < +\infty$
such that for all $t \in \R$
\begin{align*}
\Big| \gamma(1+it) F(t) \Big| \leq C_\theta.
\end{align*}
It also has the property that $\gamma(\theta) = 1$.


\subsection{Auxiliary results related to mixed $L^p$-norms}
\label{Mixed_norm_auxiliary}

In this subsection $R$ shall denote the Fourier restriction operator $L^p(\R^3) \to L^2(\mathrm{d}\mu)$ for a positive finite Radon measure $\mu$,
and all functions and measures will have $\R^3$ as their domain, unless stated otherwise.
Recall that we assume $p = (p_1,p_3)$.

We first recall what happens in the simple case when $p=(2,1)$ and $\mu$ has the form
\begin{align*}
\langle \mu, f \rangle = \int_{\Omega} f(x,\phi(x)) \, \eta(x) \mathrm{d}x,
\end{align*}
where $\phi$ is any measurable function on an open set $\Omega$ and $\eta \in C_c^\infty(\Omega)$ is a nonnegative function.
In this case the form of the adjoint of $R$ is
\begin{align*}
(R^* f)(x_1,x_2,x_3) = \int_{\R^2} e^{i (x_1 \xi_1 + x_2 \xi_2 + x_3 \phi(\xi ))} f(\xi) \eta(\xi) \mxi,
\end{align*}
and it is called the extension operator.
Using Plancherel for each fixed $x_3$, we easily get boundedness of $R^* : L^2(\mathrm{d}\mu) \to L^\infty_{x_3}(L^2_{(x_1,x_2)})$.
Note that the operator bound depends only on the $L^\infty$ norm of $\eta$.
In particular we know that $R : L^1_{x_3}(L^2_{(x_1,x_2)}) \to L^2(\mathrm{d}\mu)$ is bounded.

When considering the $L^p-L^2$ Fourier restriction problem for other $p$'s, it is advantageous to reframe the problem
using the so called ``$R^* R$'' method.
The boundedness of the restriction operator $R : L^p \to L^2(\mathrm{d}\mu)$ is equivalent to the boundedness of the operator $T = R^* R$, which can be written as
\begin{equation}
\label{Auxiliary_RstarR}
Tf (y) \coloneqq \int_{\R^3} \int_{\R^3} f(y-x) e^{i \xi \cdot x} \, \mathrm{d}\mu(\xi) \, \mathrm{d}x
 = f * \widecheck{\mu} (y), \qquad f \in \mathcal{S}(\R^3),
\end{equation}
in the pair of spaces $L^p \to L^{p'}$, where $p'$ denotes the Young conjugate exponents $(p_1',p_3')$.
Note that the operator $T$ is linear in $\mu$ and it even makes sense for a complex $\mu$ (unlike the restriction operator $R$).
This enables us to decompose the measure $\mu$ into a sum of complex measures,
each having an associated operator of the same form as in \eqref{Auxiliary_RstarR}.

The following few lemmas give us information on the boundedness of convolution operators such as in \eqref{Auxiliary_RstarR}.
\begin{lemma}
\label{Mixed_norm_auxiliary_lemma}
Let us consider the convolution operator $T : f \mapsto f * \widehat{\mu}$ for a tempered Radon measure $\mu$
(i.e., a Radon measure which is a tempered distribution).
\begin{enumerate}
\item[(i)]
If $\widehat{\mu}$ is a measurable function which satisfies
\begin{align}
\label{Auxiliary_mu_decay}
|\widehat{\mu}(x_1,x_2,x_3)| \lesssim A (1+|x_3|)^{-\tilde{\sigma}}
\end{align}
for some $\tilde{\sigma} \in [0,1)$, then the operator norm of $T : L^p \to L^{p'}$ for $(1/p_1',1/p_3') = (0,\tilde{\sigma}/2)$
is bounded (up to a multiplicative constant) by $A$.
\item[(ii)]
If $\mu$ is a bounded function such that $\Vert \mu \Vert_{L^\infty} \lesssim B$, then
the operator norm of $T : L^2 \to L^{2}$ is bounded (up to a multiplicative constant) by $B$.
\end{enumerate}
\end{lemma}

\begin{proof}
One can easily show by integrating \eqref{Auxiliary_RstarR} in $(x_1,x_2)$ variables that
\begin{align*}
\Vert Tf (\cdot,y_3) \Vert_{L^\infty_{(y_1,y_2)}} \lesssim
  A \int_\R \Vert f(\cdot, y_3-x_3) \Vert_{L^1_{(x_1,x_2)}} (1+|x_3|)^{-\tilde{\sigma}} \mx_3,
\end{align*}
and therefore we can now apply the (one-dimensional) Hardy-Littlewood-Sobolev inequality and obtain the claim in the first case.
The second case when $p_1=p_3=2$ is a well known classical result for multipliers.
\end{proof}
For a more abstract approach to the above lemma see \cite{GV92} and \cite{KT98}.
There one also obtains an appropriate result for $\tilde{\sigma} = 1$ when $1/p_1' > 0$,
but shall not need this.

A particular useful application of the above lemma is the following.
\begin{lemma}
\label{Mixed_norm_auxiliary_lemma_local}
Let us consider $T : f \mapsto f * \widehat{\mu}$ for a tempered Radon measure $\mu$ which is now localised in the frequency space:
\begin{align*}
\supp \widehat{\mu} \subset \R^2 \times [-\lambda_3,\lambda_3]
\end{align*}
for a $\lambda_3 \gtrsim 1$.
Let us assume that $\mu$ and $\widehat{\mu}$ are measurable functions satisfying
\begin{align}
\label{Auxiliary_AB_estimates}
\begin{split}
\Vert \widehat{\mu} \Vert_{L^\infty} &\lesssim A,\\
\Vert \mu \Vert_{L^\infty} &\lesssim B.
\end{split}
\end{align}
Then $T$ is a bounded operator for $(\frac{1}{p_1'},\frac{1}{p_3'}) = (0,\frac{\tilde{\sigma}}{2})$ for all $\tilde{\sigma} \in [0,1)$,
with the associated operator norm being at most (up to a multiplicative constant) $A \, \lambda_3^{\tilde{\sigma}}$.
The operator norm of $T : L^2 \to L^{2}$ is bounded (up to a multiplicative constant) by $B$.
\end{lemma}

\begin{proof}
We only need to obtain the decay estimate \eqref{Auxiliary_mu_decay}.
We note that since $\widehat{\mu}$ has $x_3$ support bounded by $\lambda_3$, it follows
\begin{align*}
|\widehat{\mu} (x_1,x_2,x_3)| &\lesssim A \, (1+ \lambda_3^{-1} |x_3|)^{-\tilde{\sigma}}\\
   &\lesssim A \, \lambda_3^{\tilde{\sigma}} \, (1+ |x_3|)^{-\tilde{\sigma}}
\end{align*}
for all $\tilde{\sigma} \in [0,1)$.
\end{proof}

At the end of this subsection we note the following simple result which tells us
that the conclusion of Lemma \ref{Mixed_norm_auxiliary_lemma} is in a sense quite sharp.
We remark that the last conclusion in the lemma below is consistent with the condition $\tilde{\sigma}<1$ in \eqref{Auxiliary_mu_decay}.
\begin{lemma}
\label{Mixed_norm_auxiliary_lemma_sharpness}
Consider the convolution operator $T : f \mapsto f * \widehat{\mu}$ for a tempered Radon measure $\mu$
whose Fourier transform $\widehat{\mu}$ is continuous.
Let $\varphi : [0,+\infty) \to (0,+\infty)$ be an increasing and unbounded continuous function
and assume that at least one of the limits
\begin{align*}
\lim_{x_3 \to - \infty} \widehat{\mu}(0,0,x_3) \frac{(1+|x_3|)^{\tilde{\sigma}}}{\varphi(|x_3|)}
\qquad \text{or} \qquad
\lim_{x_3 \to + \infty} \widehat{\mu}(0,0,x_3) \frac{(1+|x_3|)^{\tilde{\sigma}}}{\varphi(|x_3|)}
\end{align*}
exists for some $\tilde{\sigma} \in (0,1)$, with the limiting value being a nonzero number.
Then $T : L^p \to L^{p'}$ is not a bounded operator for $(1/p_1',1/p_3') = (0,\tilde{\sigma}/2)$.
The conclusion also holds in the case when $\varphi$ is the constant function $1$, $\tilde{\sigma} = 1$,
and if we additionally assume that $\widehat{\mu}$ is an $L^\infty(\R^3)$ function and that both of the above limits exist and are equal,
with the limiting value being a nonzero number.
\end{lemma}
\begin{proof}
Let us begin the proof by assuming that the operator
\begin{align*}
T : L^{\frac{2}{2-\tilde{\sigma}}}_{x_3} (L^1_{(x_1,x_2)}) &\to L^{2/\tilde{\sigma}}_{x_3} (L^\infty_{(x_1,x_2)})
\end{align*}
is bounded.
Since $\widehat{\mu}$ is continuous, without loss of generality we can assume that
\begin{align}
\label{Mixed_norm_auxiliary_lemma_sharpness_limit_approx}
\widehat{\mu}(x) \sim  |x_3|^{-\tilde{\sigma}} \, \varphi(|x_3|)
\end{align}
for all $x$ in the open set $U$ of the form
\begin{align*}
\{ x \in \R^3 : x_3 > K, |(x_1,x_2)| < \epsilon_U(x_3)\},
\end{align*}
where $K > 0$ and $\epsilon_U$ is a continuous and strictly positive function on $\R$.

Now consider the function
$$f(x) = \varepsilon^{-2} \chi_0 \Big(\frac{x_1}{\varepsilon}\Big) \chi_0 \Big(\frac{x_2}{\varepsilon}\Big) \chi_0 \Big(\frac{x_3}{M} \Big),$$
where
$\chi_0$ is smooth, identically $1$ in the interval $[-1,1]$, and supported within the interval $[-2,2]$.
Then
$$
\Vert f\Vert_{L^{\frac{2}{2-\tilde{\sigma}}}_{x_3} (L^1_{(x_1,x_2)})} \sim M^{1-\frac{\tilde{\sigma}}{2}},
$$
and if we assume $\varepsilon$ to be sufficiently small and $M$ sufficiently large, one obtains by a simple calculation that
$$
Tf(0,0,x_3) \sim \Bigg( \chi_0 \Big(\frac{ \cdot }{M}\Big) * \Big( |\cdot|^{-\tilde{\sigma}} \, \varphi(|\cdot|) \Big) \Bigg)(x_3)
$$
for all $x_3$ such that $4M < x_3 < C(M, \varepsilon)$, where $C(M, \varepsilon) \to \infty$ when $\varepsilon \to 0$ and $M$ is fixed.
If in addition we know say $x_3 \leq 5M < C(M, \varepsilon)$, then 
$$
Tf(0,0,x_3) \gtrsim M^{1-\tilde{\sigma}} \, \varphi(|M|),
$$
and the lower bound on the norm is
\begin{align*}
\Vert Tf \Vert_{L^{2/\tilde{\sigma}}_{x_3} (L^\infty_{(x_1,x_2)})}
  &\gtrsim \Big( M^{1-\tilde{\sigma}} \, \varphi(|M|) \Big) \, M^{\tilde{\sigma}/2}\\
  &= M^{1-\tilde{\sigma}/2} \, \varphi(|M|).
\end{align*}
But now by the boundedness assumption we obtain
\begin{align*}
M^{1-\tilde{\sigma}/2} \, \varphi(|M|) \lesssim M^{1-\tilde{\sigma}/2} \sim \Vert f\Vert_{L^{\frac{2}{2-\tilde{\sigma}}}_{x_3} (L^1_{(x_1,x_2)})},
\end{align*}
i.e., $\varphi(|M|) \lesssim 1$. This is impossible in general since we can take $M \to \infty$.

In the case when the limits are equal, $\tilde{\sigma} = 1$, and $\varphi$ is the constant function $1$,
we can take \eqref{Mixed_norm_auxiliary_lemma_sharpness_limit_approx} to be true for $x \in U$ too.
If we use the same $f$ as above, then for any $x_3 \in [-M/2,M/2]$ we easily obtain from the definition of $T$ that
$$
Tf(0,0,x_3) \gtrsim \int_K^{M/2} |t|^{-1} \mathrm{d}t - K \Vert \widehat{\mu} \Vert_{L^\infty} \gtrsim \ln M
$$
for an $M$ sufficiently large and $\varepsilon$ sufficiently small.
Thus the norm $\Vert Tf \Vert_{L^{2}_{x_3} (L^\infty_{(x_1,x_2)})}$ is bounded below by $M^{1/2} \ln M$,
while $\Vert f\Vert_{L^{2}_{x_3} (L^1_{(x_1,x_2)})}$ is of size $M^{1/2}$.
This is impossible if $T$ is bounded.
\end{proof}

In the case $\tilde{\sigma} = 1$ and when $\varphi$ is identically equal to a nonzero constant the above proof
does not work if the limits have the same absolute value but opposite signs.
This is related to the fact that an operator given as a convolution against $x \mapsto x/(1+x^2)$ is bounded $L^2(\R) \to L^2(\R)$
since the Fourier transform of $x \mapsto x/(1+x^2)$ is up to a constant $\xi \mapsto e^{-|\xi|} \sgn \xi$.


\section{The adapted case and reduction to restriction estimates near the principal root jet}
\label{Adapted}

Here we mimic \cite[Chapter 3]{IM16} and the last section of \cite{IM11b}, where the adapted case for $p_1=p_3$ was considered.
In this section we shall be concerned with measures of the form
\begin{equation}
\label{Adapted_measure_form}
\langle \mu, f\rangle = \int f(x,\phi(x)) \eta(x) \mx,
\end{equation}
where $\phi(0) = 0$, $\nabla \phi(0) = 0$, and $\eta$ is a smooth nonnegative function
with support contained in a sufficiently small neighbourhood of $0$.
We assume that $\phi$ is of finite type on the support of $\eta$.
The associated Fourier restriction problem is
\begin{equation}
\label{FRP_mixed}
\Bigg( \int |\widehat{f}|^2 \mathrm{d}\mu \Bigg)^{1/2} \leq
\Const \Vert f \Vert_{L^{p_3}_{x_3} (L^{p_1}_{(x_1,x_2)})}, \qquad f \in \mathcal{S}(\R^3),
\end{equation}
for any $\eta$ with support contained in a sufficiently small neighbourhood of $0$.

The following proposition will be useful in this section.
\begin{proposition}
\label{Adapted_general_considerations}
Let $\mu$, $\phi$, and $\eta$ be as above.
Then the mixed norm Fourier restriction estimate \eqref{FRP_mixed} holds true for the point $(1/p_1',1/p_3') = (1/2,0)$.
Furthermore we have the following two cases.
\begin{enumerate}
\item[(i)]
If either $h(\phi) = 1$ or $\nu(\phi) = 1$, then the estimate \eqref{FRP_mixed} holds true
for $1/p_1' = 0$ and $1/p_3' <1/(2h(\phi))$.
In this case the estimate for $(1/p_1',1/p_3') = (0,1/(2h(\phi)))$ does not hold if $\eta(0) \neq 0$.
\item[(ii)]
If $h(\phi)>1$ and $\nu(\phi)=0$, then the estimate \eqref{FRP_mixed} holds true for $(1/p_1',1/p_3') = (0,1/(2h(\phi)))$.
\end{enumerate}
\end{proposition}
\begin{proof}
The claim for $(1/p_1',1/p_3') = (1/2,0)$ follows from considerations at the beginning of Subsection \ref{Mixed_norm_auxiliary}.

Let us now recall what happens in the non-degenerate case, i.e., when the determinant of the Hessian $\det \mathcal{H}_\phi(0,0) \neq 0$.
This is equivalent to $h(\phi) = 1$ and in this case $\phi$ is adapted in any coordinate system.
Here we have the bound \eqref{FRP_mixed} for all of the $(1/p_1',1/p_3')$ given in the necessary condition \eqref{Knapp_adapted_mixed},
except for the point $(0,1/2)$, for which it does not hold.
This fact is actually true globally, i.e., the Strichartz estimates hold (see \cite{GV92, KT98} and references therein) in the same range,
and one can easily convince oneself that the same proof as in say \cite{KT98} goes through in our local case.
For the negative results at the point $(0,1/2)$ in the case of Strichartz estimates see \cite{KM93} and \cite{Mo98}.
We can also get a negative result at the point $(0,1/2)$ directly in our case
by applying Lemma \ref{Mixed_norm_auxiliary_lemma_sharpness} for the case $\tilde{\sigma} = 1$ and $\varphi$ is identically equal to $1$.
The limits in Lemma \ref{Mixed_norm_auxiliary_lemma_sharpness} are obtained by a simple application of the two dimensional stationary phase method.
Furthermore, since the Hessian does not change its sign when changing the phase $\phi \mapsto -\phi$, the limits in both directions are equal.

The claims for the case when $h(\phi) > 1$ follow easily by applying Theorems \ref{Oscillatory_auxiliary_decay} and \ref{Oscillatory_auxiliary_decay_sharpness}
to Lemmas \ref{Mixed_norm_auxiliary_lemma} and \ref{Mixed_norm_auxiliary_lemma_sharpness} respectively.
In Lemma \ref{Mixed_norm_auxiliary_lemma_sharpness} we take $\varphi$ to be the logarithmic function $x \mapsto \log(2+x)$.
\end{proof}

\subsection{The adapted case}
\label{Adapted_adapted_case}

The following proposition tells us precisely when the Fourier restriction estimate holds in the adapted case.
\begin{proposition}
\label{Adapted_proposition_adapted}
Let us assume that $\mu$, $\phi$, and $\eta$ are as explained at the beginning of this section, and
let us assume that $\phi$ is adapted.
\begin{itemize}
\item[(i)]
If $h(\phi)=1$ or $\nu(\phi)=1$, then the full range Fourier restriction estimate given by the necessary condition
\eqref{Knapp_adapted_mixed} holds true, except for the point $(1/p_1',1/p_3') = (0,1/(2h(\phi)))$ where it is false if $\eta(0) \neq 0$.
\item[(ii)]
If $h(\phi)>1$ and $\nu(\phi)=0$, then the full range Fourier restriction estimate given by the necessary condition \eqref{Knapp_adapted_mixed} holds true,
including the point $(1/p_1',1/p_3') = (0,1/(2h(\phi)))$.
\end{itemize}
\end{proposition}

\begin{proof}
The case when $h(\phi) = 1$ is the classical known case and it was already discussed in the proof of Proposition \ref{Adapted_general_considerations}.
The case when $h(\phi) > 1$ and $\nu(\phi) = 0$ follows from Proposition \ref{Adapted_general_considerations} by interpolation.

Let us now consider the remaining case when $h(\phi) > 1$ and $\nu(\phi) = 1$.
Then if we would use Proposition \ref{Adapted_general_considerations} and interpolation as in the previous case,
we would miss all the boundary points determined by the line of the necessary condition \eqref{Knapp_adapted_mixed}
\begin{align*}
\frac{1}{h(\phi)} \frac{1}{p_1'} + \frac{1}{p_3'} = \frac{1}{2 h(\phi)},
\end{align*}
except the point $(1/2,0)$ where we know that the estimate always holds.
Recall that this is essentially because we have the logarithmic factor in the decay of the Fourier transform of $\mu$.
Instead, one can use the strategy from \cite[Section 4]{IM11b} to avoid this problem.
We only briefly sketch the argument. One decomposes
\begin{equation*}
\mu = \sum_{k\geq k_0} \mu_k,
\end{equation*}
where $\mu_k$ are supported within ellipsoid annuli centered at $0$ and closing in to $0$.
This is done by considering the partition of unity
\begin{align*}
\eta(x) = \sum_{k \geq k_0} \eta_k(x) = \sum_{k \geq k_0} \eta(x) \chi \circ \delta_{2^k} (x),
\end{align*}
where $\chi$ is an appropriate $C_c^\infty(\R^2)$ function supported away from the origin and
\begin{align*}
\delta_{r}(x) = (r^{\kappa_1} x_1, r^{\kappa_2} x_2),
\end{align*}
where $\kappa = (\kappa_1, \kappa_2)$ is the weight associated to the principal face of $\mathcal{N}(\phi)$.
Next, one rescales the measures $\mu_k$ and obtains measures $\mu_{0,(k)}$ having the form \eqref{Adapted_measure_form}.
These new measures have uniformly bounded total variation and Fourier decay estimate with constants uniform in $k$:
\begin{align*}
\big|\widehat{\mu_{0,(k)}}(\xi)\big| \lesssim (1+|\xi|)^{-1/h(\phi)}.
\end{align*}
Note that there is no logarithmic factor anymore.
Now we can use Proposition \ref{Adapted_general_considerations} and interpolation to obtain
the mixed norm Fourier restriction estimate within the range \eqref{Knapp_adapted_mixed} for each $\mu_{0,(k)}$.
As in \cite[Section 4]{IM11b}, one now easily obtains the bound\footnote{In the equation right above \cite[Equation (4.7)]{IM11b} there is a typo.
Instead of $(|\kappa|/2+1)k$ in the exponent, it should be $(|\kappa|+2)k$.}
\begin{equation*}
\int |\widehat{f}|^2 \mathrm{d}\mu_k \lesssim
2^{(|\kappa|+2)k} \Vert f \circ \delta^e_{2^k} \Vert^2_{L^{p_3}_{x_3} (L^{p_1}_{(x_1,x_2)})}, \qquad f \in \mathcal{S}(\R^3),
\end{equation*}
where $\delta^e_{r}(x_1,x_2,x_3) = (r^{\kappa_1}x_1,r^{\kappa_2}x_2,rx_3)$.
The scaling in our mixed norm case is
\begin{equation*}
\Vert f \circ \delta^e_{2^k} \Vert_{L^{p_3}_{x_3} (L^{p_1}_{(x_1,x_2)})} =
2^{-k(\frac{\kappa_1+\kappa_2}{p_1} + \frac{1}{p_3})} \Vert f \Vert_{L^{p_3}_{x_3} (L^{p_1}_{(x_1,x_2)})},
\end{equation*}
and therefore
\begin{align*}
\int |\widehat{f}|^2 \mathrm{d}\mu_k
 &\lesssim 2^{k(|\kappa|+2)-2k(\frac{\kappa_1+\kappa_2}{p_1} + \frac{1}{p_3})} \Vert f \Vert^2_{L^{p_3}_{x_3} (L^{p_1}_{(x_1,x_2)})}\\
 &= 2^{k(|\kappa|+2 -\frac{2|\kappa|}{p_1} - \frac{2}{p_3})} \Vert f \Vert^2_{L^{p_3}_{x_3} (L^{p_1}_{(x_1,x_2)})}\\
 &= 2^{2k|\kappa|(-\frac{1}{2}+\frac{1}{p'_1} + \frac{1}{|\kappa|p'_3})} \Vert f \Vert^2_{L^{p_3}_{x_3} (L^{p_1}_{(x_1,x_2)})}\\
 &\leq \Vert f \Vert^2_{L^{p_3}_{x_3} (L^{p_1}_{(x_1,x_2)})},
\end{align*}
by the necessary condition
\begin{equation*}
\frac{1}{p'_1} + \frac{1}{|\kappa|p'_3} \leq \frac{1}{2},
\end{equation*}
and the equalities $d(\phi) |\kappa| = h(\phi) |\kappa| = 1$. The rest of the proof is the same as in \cite{IM11b}
if we assume $p_1 > 1$, since then one can use the Littlewood-Paley theorem\footnote{
Here we don't need a mixed norm Littlewood-Paley theorem since the decomposition is only in the tangential direction where $p_1=p_2$.
Note that the ordering of the mixed norm is important, namely that the outer norm is associated to the normal direction.}
and the Minkowski inequality (which we can apply since $p_1=p_2\leq2$ and $p_3\leq2$)
to sum the above inequality in $k$.
The proof of Proposition \ref{Adapted_proposition_adapted} is done.
\end{proof}

\subsection{Reduction to the principal root jet}
\label{Adapted_reduction_principal_root}

In this subsection we make some preliminary reductions for the case when $\phi$ is not adapted.
Recall that we may assume that $\phi$ is linearly adapted and that
we denote by $\psi$ the principal root of $\phi$.
Then we can obtain the adapted coordinates $y$ (after possibly interchanging the coordinates $x_1$ and $x_2$) through
\begin{align*}
y_1 &= x_1,\\
y_2 &= x_2-\psi(x_1).
\end{align*}
Before stating the last proposition of this section (analogous to \cite[Proposition 3.1]{IM16})
let us recall some notation from \cite{IM16}.
We write
\begin{align*}
\psi(x_1) = b_1 x_1^m + \mathcal{O}(x_1^{m+1}),
\end{align*}
where $b_1 \neq 0$ and $m \geq 2$ by linear adaptedness (see \cite[Proposition 1.7]{IM16}).
If $F$ is an integrable function on the domain of $\eta$, say $\Omega \subseteq \R^2$, then we denote
$$ \mu^F \coloneqq (F \otimes 1) \mu.$$
If $\chi_0$ denotes a $C_c^\infty(\R)$ function equal to $1$ in a neighbourhood of the origin,
we may define
$$ \rho_1(x_1,x_2) = \chi_0\Big( \frac{x_2-b_1x_1^m}{\varepsilon x_1^m} \Big),$$
where $\varepsilon$ is an arbitrarily small parameter.
The domain of $\rho_1$ is a $\kappa$-homogeneous subset of $\Omega$ which contains
the principal root jet $x_2 = \psi(x_1)$ of $\phi$ when $\Omega$ is contained in a sufficiently small neighbourhood of $0$.

\begin{proposition}
\label{Adapted_proposition_reduction}
Assume $\phi$ is of finite type on $\Omega$, non-adapted, and linearly adapted (i.e., $d(\phi) = h_{\text{lin}}(\phi)$).
Let $\varepsilon > 0$ be sufficiently small and
let $\mu^{1-\rho_1}$ have support contained in a sufficiently small neighbourhood of $0$.
Then the mixed norm Fourier restriction estimate \eqref{FRP_mixed} with respect to the measure $\mu^{1-\rho_1}$ holds true
for all $(1/p_1',1/p_3')$ which satisfy
\begin{align*}
\frac{1}{d(\phi)} \frac{1}{p_1'} + \frac{1}{p_3'} &\leq \frac{1}{2 d(\phi)},\\
p_1 &> 1,
\end{align*}
i.e., within the range determined by the necessary condition associated to the principal face of $\mathcal{N}(\phi)$,
except maybe the boundary points of the form $(0,1/p_3')$. In particular, it also holds true within the narrower range
determined by all of the necessary conditions, excluding maybe the boundary points of the form $(0,1/p_3')$.
\end{proposition}

We just briefly mention that the proof of the Proposition \ref{Adapted_proposition_reduction} is trivial
as soon as one uses the results from \cite[Chapter 3]{IM16}. Analogously to the previous subsection,
one decomposes the measure $\mu^{1-\rho_1}$ by using the $\kappa$ dilations associated to the principal face of $\mathcal{N}(\phi)$.
The measures $\nu_k$ obtained by rescaling are of the form \eqref{Adapted_measure_form}, have uniformly bounded total variation, and 
have the Fourier transform decay (with constants uniform in $k$)
$$ |\widehat{\nu}_k(\xi)| \lesssim (1+|\xi|)^{-d(\phi)}. $$
All of this was proven in \cite[Chapter 3]{IM16}.
Therefore we have the Fourier restriction estimate for each $\nu_k$ for the points $(1/p_1',1/p_3') = (1/2,0)$ and $(1/p_1',1/p_3') = (0,1/(2d(\phi)))$.
Now one uses again interpolation, the Minkowski inequality, and the Littlewood-Paley theorem, to obtain the claim.

Note that the estimates for the boundary points of the form $(0,1/p_3')$
can be directly solved for the original measure $\mu$ through Proposition \ref{Adapted_general_considerations}.



\section{The case $h_{\text{lin}}(\phi)$ < 2}
\label{Section_h_lin_less_than_2}

In the remainder of this article we shall we concerned with the proof of:
\begin{theorem}
\label{Section_h_lin_less_than_2_main_theorem}
Let $\phi : \R^2 \to \R$ be a smooth function of finite type defined on a sufficiently small neighbourhood $\Omega$ of the origin,
satisfying $\phi(0) = 0$ and $\nabla \phi(0) = 0$.
Let us assume that $\phi$ is linearly adapted, but not adapted, and that $h_{\text{lin}}(\phi) < 2$.
We additionally assume that the following holds:
Whenever the function $b_0$ appearing in \eqref{Knapp_A_form}, \eqref{Knapp_A_form_adapted}, \eqref{Knapp_D_form} is flat
(i.e., when $\phi$ is $A_\infty$ or $D_\infty$ type singularity),
then it is necessarily identically equal to $0$.
In this case, for all smooth $\eta \geq 0$ with support in a sufficiently small neighbourhood of the origin the Fourier restriction estimate
\eqref{FRP_mixed} holds for all $p$ given by the necessary conditions determined in Subsection \ref{Knapp_h_lin_less_than_two}.
\end{theorem}

The above condition on the function $b_0$ is implied by the Condition (R) from \cite{IM16} (see \cite[Remark 2.12. (c)]{IM16}).

We begin with some preliminaries.
As one can see from the Newton diagrams in Subsection \ref{Knapp_h_lin_less_than_two}, the assumption in our case
$h_{\text{lin}}(\phi)$ < 2 implies that $h(\phi) \leq 2$.
Additionally, we see that $h(\phi) = 2$ implies that we either have $A_\infty$ or $D_\infty$ type singularity.
As mentioned in Subsection \ref{Notation_and_assumptions},
the Varchenko exponent is $0$, i.e., $\nu(\phi) = 0$, if $h(\phi) < 2$.
When $h(\phi) = 2$ the equality $\nu(\phi) = 0$ also holds true in our case since the principal faces are non-compact.
We conclude that if $h_{\text{lin}}(\phi)$ < 2, then by Proposition \ref{Adapted_general_considerations}
we have the mixed norm Fourier restriction estimate \eqref{FRP_mixed}
for both of the points $(1/p_1',1/p_3') = (1/2,0)$ and $(1/p_1',1/p_3') = (0,1/(2h(\phi)))$.
Therefore, according to Subsection \ref{Knapp_h_lin_less_than_two}, by interpolation it remains
to prove the estimate \eqref{FRP_mixed} for the respective critical exponents given by
\begin{align}
\label{LowerThanTwo_points}
\begin{split}
\Big( \frac{1}{\mathfrak{p}_1'}, \frac{1}{\mathfrak{p}_3'} \Big)
  &= \Big( \frac{1}{2(m+1)}, \frac{1}{4} \Big) \qquad \text{in case of $A$ type singularity,}\\
\Big( \frac{1}{\mathfrak{p}_1'}, \frac{1}{\mathfrak{p}_3'} \Big)
  &= \Big( \frac{1}{4(m+1)}, \frac{1}{4} \Big) \qquad \text{in case of $D$ type singularity,}
\end{split}
\end{align}
where $m \geq 2$ is the principal exponent of $\psi$ from Subsection \ref{Knapp_h_lin_less_than_two}.

Recall that according to Proposition \ref{Adapted_proposition_reduction} we may concentrate on the piece of the measure $\mu$
located near the principal root jet:
\begin{align*}
\langle \mu^{\rho_1}, f \rangle = \int_{x_1 \geq 0} f(x,\phi(x)) \, \eta(x) \, \rho_1(x) \mathrm{d}x,
\end{align*}
where
\begin{align}
\label{LowerThanTwo_rho_1_def}
\rho_1(x) = \chi_0  \Big(\frac{x_2- \omega(0) x_1^m}{\varepsilon x_1^m}\Big)
\end{align}
for an arbitrarily small $\varepsilon$ and $\omega(0) x_1^m$ the first term in the Taylor expansion of
\begin{align*}
\psi(x_1) = x_1^m \omega(x_1),
\end{align*}
where $\omega$ is a smooth function such that $\omega(0) \neq 0$.

As we use the same decompositions of the measure $\mu^{\rho_1}$ as in \cite{IM16}, we shall only briefly outline the decomposition procedure.

\subsection{Basic estimates}
\label{Section_h_lin_less_than_2_Basic_estimates}

Before we outline the further decompositions and rescalings of $\mu^{\rho_1}$,
we first describe here the general strategy for proving the Fourier restriction estimates for the pieces obtained through these decompositions.
All of the pieces $\nu$ of the measure $\mu^{\rho_1}$ will essentially be of the form
\begin{align*}
\langle \nu, f \rangle = \int f \circ \Phi(x) \, a(x) \mathrm{d}x,
\end{align*}
where $\Phi$ is a phase function and $a \geq 0$ an amplitude.
The amplitude will usually be compactly supported with support away from the origin.
Both $\Phi$ and $a$ will depend on various decomposition related parameters.
We shall need to prove the Fourier restriction estimate with respect to these measures
with estimates being uniform in a certain sense with respect to the appearing decomposition parameters.

At this point one uses the ``$R^* R$'' method applied to the measure $\nu$.
The resulting operator is $T_\nu$ which acts by convolution against the Fourier transform of $\nu$.
Now one considers the spectral decomposition $(\nu^\lambda)_\lambda$ of the measure $\nu$ so that
each functions $\nu^\lambda$ is localised in the frequency space at $\lambda = (\lambda_1, \lambda_2, \lambda_3)$,
where $\lambda_i \geq 1$ are dyadic numbers for $i=1, 2, 3$.
For such functions $\nu_\lambda$ we shall obtain bounds of the form \eqref{Auxiliary_AB_estimates}.
By Lemma \ref{Mixed_norm_auxiliary_lemma_local} then we have the bounds
on their associated convolution operators $T^\lambda_\nu$:
\begin{align}
\label{LowerThanTwo_ABestimates}
\begin{split}
\Vert T_\nu^\lambda \Vert_{ L^{2/(2-\tilde{\sigma})}_{x_3} (L^1_{(x_1,x_2)}) \to L^{2/\tilde{\sigma}}_{x_3} (L^\infty_{(x_1,x_2)}) }
   &\lesssim A \lambda_3^{\tilde{\sigma}},\\
\Vert T_\nu^\lambda \Vert_{L^2\to L^2}
   &\lesssim B,
\end{split}
\end{align}
for all $\tilde{\sigma} \in [0,1)$. $A$ and $B$ shall again depend on various decomposition related parameters.
If we now define
\begin{align}
\label{LowerThanTwo_theta_and_sigma}
(\theta,\tilde{\sigma}) &\coloneqq \Big(\frac{1}{m+1},\frac{m-1}{2m}\Big), \qquad\qquad \text{in case of $A$ type singularity,}\\
(\theta,\tilde{\sigma}) &\coloneqq \Big(\frac{1}{2(m+1)},\frac{m}{2m+1}\Big), \,\qquad \text{in case of $D$ type singularity,}
\end{align}
then interpolating \eqref{LowerThanTwo_ABestimates} ($\theta$ being the interpolation coefficient) we get
precisely the estimate for the critical exponent in \eqref{LowerThanTwo_points} with the bound
\begin{align}
\label{LowerThanTwo_basic_estimate}
\Vert T_\nu^\lambda \Vert_{ L^{\mathfrak{p}_3}_{x_3} (L^{\mathfrak{p}_1}_{(x_1,x_2)}) \to L^{\mathfrak{p}_3'}_{x_3} (L^{\mathfrak{p}_1'}_{(x_1,x_2)}) }
   &\lesssim A^{1-\theta} B^\theta \lambda_3^{\frac{1}{2}-\theta}.
\end{align}
Now it remains to sum over $\lambda$.

When $\theta < 1/4$, we shall be able to always sum absolutely.
In the cases when $\theta = 1/4$ and particularly $\theta = 1/3$ (note that both appear only in $A$ type singularity with $m = 3$
and $m = 2$ respectively) we shall need the complex interpolation method developed in \cite{IM16}.

\subsection{First decompositions and rescalings of $\mu^{\rho_1}$}
\label{Section_h_lin_less_than_2_First_decompositions}

As in Section \ref{Adapted}, we use the $\kappa$ dilatations associated to the principal face of $\mathcal{N}(\phi)$,
and subsequently a Littlewood-Paley argument.
Then it remains to prove the Fourier restriction estimate for the renormalised measures $\nu_k$ of the form
\begin{align*}
\langle \nu_k, f \rangle = \int f(x,\phi(x,\delta)) \, a(x,\delta) \, \mathrm{d}x,
\end{align*}
uniformly in $k$.
As was shown in \cite[Section 4.1]{IM16}, the function $\phi(x,\delta)$ has the form
\begin{align*}
\phi(x,\delta) \coloneqq \tilde{b}(x_1,x_2,\delta_1,\delta_2) (x_2 - x_1^m \omega(\delta_1 x_1))^2 + \delta_3 x_1^n \beta(\delta_1 x_1),
\end{align*}
where
\begin{align*}
\delta = (\delta_1,\delta_2,\delta_3) \coloneqq (2^{-\kappa_1 k}, 2^{-\kappa_2 k}, 2^{-(n\kappa_1-1)k}),
\end{align*}
and
\begin{align*}
  \tilde{b}(x_1,x_2,\delta_1,\delta_2) =
  \begin{cases}
    b(\delta_1 x_1, \delta_2 x_2), & \text{in case of A type singularity}, \\
    x_1 b_1(\delta_1 x_1, \delta_2 x_2) + \delta_1^{2m-1} x_2^2 b_2(\delta_2 x_2), & \text{in case of D type singularity}.
  \end{cases}
\end{align*}
Above the functions $b$, $b_1$, $b_2$, $\beta$, and the quantity $n$ are as in Subsection \ref{Knapp_h_lin_less_than_two}.
Recall that $m = \kappa_2/\kappa_1 \geq 2$ and so $\delta_2 = \delta_1^m$.
The amplitude $a(x,\delta) \geq 0$ is a smooth function of $(x,\delta)$ supported at
\begin{align*}
x_1 \sim 1 \sim |x_2|.
\end{align*}
Furthermore, due to the $\rho_1$ cutoff function which has a $\kappa$-homogeneous domain, we may assume $|x_2 - x_1^m \omega(0)| \ll 1$.

Since we can take $k$ arbitrarily large, the parameter $\delta$ approaches $0$.
This implies that on the domain of integration of $a$ we have that $\tilde{b}(x_1,x_2,\delta_1,\delta_2)$
converges as a function of $(x_1,x_2)$ to $b(0,0)$ (resp. $b_1(0,0)x_1$) in $C^\infty$
when $k \to \infty$ and $\phi$ has A type singularity (resp. D type singularity).
The amplitude $a(x,\delta)$ converges in $C_c^\infty$ to $a(x,0)$.
We also recall that according to the assumption in Theorem \ref{Section_h_lin_less_than_2_main_theorem},
we may assume that $\delta_3 = 0$ if ``$n = \infty$'', i.e., if $b_0$ is flat in the normal form of $\phi$.

The next step is to decompose the (compactly) supported amplitude $a$ into finitely many parts,
each localised near a point $v=(v_1,v_2)$ for which we may assume that it satisfies $v_2 = v_1^m \omega(0)$
(by compactness and since in \eqref{LowerThanTwo_rho_1_def} we can take $\varepsilon$ arbitrarily small).
The newly obtained measures we denote by $\nu_\delta$ and their new amplitudes by the same symbol $a(x,\delta) \geq 0$:
\begin{align*}
\langle \nu_\delta, f \rangle = \int f(x,\phi(x,\delta)) \, a(x,\delta) \, \mathrm{d}x,
\end{align*}
where now the support of $a(\cdot,\delta)$ is contained in the set $|x-v| \ll 1$.

Since we can use Littlewood-Paley decompositions in the mixed norm case (see \cite[Theorem 2]{Liz70}, and also \cite{BCP62,Fer87}),
we can now decompose the measure $\nu_\delta$ in the $x_3$ direction in the same way as in \cite[Section 4.1]{IM16}.
This is achieved by using the cutoff functions $\chi_1(2^{2j}\phi(x,\delta))$
in order to localise near the part where $|\phi(x,\delta)| \sim 2^{-2j}$.
Then it remains to prove the mixed norm estimate \eqref{FRP_mixed} for measures $\nu_{\delta,j}$
with bounds uniform in paramteres $j \in \N$ and $\delta = (\delta_1,\delta_2,\delta_3) \in \R^3$, $\delta_i \geq 0$, $i = 1,2,3$,
where the measures $\nu_{\delta,j}$ are defined through
\begin{align}
\label{LowerThanTwo_TheFirstDecomposition_original_measure}
\langle \nu_{\delta, j}, f \rangle \coloneqq
   \int_{x_1 \geq 0} f(x, \phi(x, \delta)) \chi_1(2^{2j}\phi(x,\delta)) a(x,\delta) \mx,
\end{align}
where $j$ can be taken sufficiently large and $\delta$ sufficiently small.
The function $2^{2j}\phi(x,\delta)$ can be written as
\begin{align*}
2^{2j} \phi(x,\delta) =
  2^{2j} \tilde{b}(x_1,x_2,\delta_1,\delta_2)\Big(x_2 - x_1^m \omega(\delta_1 x_1)\Big)^2 + 2^{2j} \delta_3 x_1^n \beta(\delta_1 x_1).
\end{align*}
Following \cite{IM16}, we distinguish three cases: $2^{2j} \delta_3 \ll 1$, $2^{2j} \delta_3 \gg 1$, and the most involed $2^{2j} \delta_3 \sim 1$.

\subsection{The case  $2^{2j} \delta_3 \gg 1$}

As was done in \cite[Subsection 4.1.1]{IM16},
we change coordinates from $(x_1,x_2)$ to $(x_1, 2^{2j} \phi(x,\delta))$
and subsequently perform a rescaling (which we adjust to our mixed norm case).
Then one obtains that the mixed norm Fourier restriction for $\nu_{\delta, j}$ is equivalent to the estimate
\begin{equation*}
\int |\widehat{f}|^2 \mathrm{d}\tilde{\nu}_{\delta, j} \leq
  \Const \sqrt{\delta_3} 2^{2j(1-2/p_3')} \Vert f \Vert^2_{L^{\mathfrak{p}}(\R^3)}, \qquad f \in \mathcal{S}(\R^3),
\end{equation*}
that is, since $\mathfrak{p}_3' = 4$,
\begin{equation}
\label{LowerThanTwo_TheFirstDecomposition_Estimate}
\int |\widehat{f}|^2 \mathrm{d}\tilde{\nu}_{\delta, j} \leq
  \Const \delta^{\frac{1}{2}}_3 2^{j} \Vert f \Vert^2_{L^{\mathfrak{p}}(\R^3)}, \qquad f \in \mathcal{S}(\R^3),
\end{equation}
where $\tilde{\nu}_{\delta,j}$ is the rescaled measure
\begin{align*}
\langle \tilde{\nu}_{\delta, j}, f \rangle \coloneqq
   \int f(x_1, \phi(x, \delta, j), x_2) a(x, \delta, j) \chi_1(x_1) \chi_1(x_2) \mx.
\end{align*}
The function $a(x, \delta, j)$ has in $\delta$ and $j$ uniformly bounded $C^l$ norms for an arbitrarily large $l \geq 0$,
and the phase function is given by
\begin{align}
\label{LowerThanTwo_TheFirstDecomposition_Phase}
\begin{split}
\phi(x, \delta, j) \coloneqq
   &\tilde{b}_1 \Big( x_1, \sqrt{2^{-2j}x_2 + \delta_3 x_1^n \tilde{\beta}(\delta_1 x_1)}, \delta_1, \delta_2\Big) \\
   &\times \sqrt{2^{-2j}x_2 + \delta_3 x_1^n \tilde{\beta}(\delta_1 x_1)} + x_1^m \omega(\delta_1 x_1),
\end{split}
\end{align}
where $x_1\sim1$, $x_2\sim1$, and without loss of generality we may assume $\tilde{b}_1(x_1,x_2,0,0) \sim 1$ and $\tilde{\beta}(0) \sim 1$;
for details see \cite[Subsection 4.1.1]{IM16}.
There the phase function $\phi(x, \delta, j)$ was obtained by solving the equation
\begin{align*}
2^{2j} \phi(y,\delta) = 2^{2j} \tilde{b}(y_1,y_2,\delta_1,\delta_2) (y_2 - y_1^m \omega(\delta_1 y_1))^2 - 2^{2j} \delta_3 y_1^n \tilde{\beta}(\delta_1 y_1)
\end{align*}
in $y_2$ after substituting $x_1 = y_1$ and $x_2 = 2^{2j} \phi(y,\delta)$.

By using the implicit function theorem one can show that when $\delta \to 0$,
then we have the following $C^\infty$ convergence in the $(x_1,x_2)$ variables:
\begin{align}
\label{LowerThanTwo_TheFirstDecomposition_tilde_b_1_convergence}
  \begin{cases}
    \tilde{b}_1(x_1,x_2,\delta_1,\delta_2) \to
    b(0,0)^{-1/2}, & \text{in case of A type singularity}, \\
    \tilde{b}_1(x_1,x_2,\delta_1,\delta_2) \to
    (b_1(0,0) x_1)^{-1/2}, & \text{in case of D type singularity}.
  \end{cases}
\end{align}
In both the A and D type singularity cases we see that $\tilde{b}_1$ does not depend on $x_2$ in an essential way.

Now we proceed to perform a spectral decomposition of $\tilde{\nu}_{\delta, j}$,
i.e., for $(\lambda_1, \lambda_2, \lambda_3)$ dyadic numbers with $\lambda_i \geq 1$,
$i = 1,2,3$, we define the spectrally localised measures $\nu_j^\lambda$ through
\begin{align}
\begin{split}
\label{LowerThanTwo_TheFirstDecomposition_FourierForm}
\widehat{\nu_j^\lambda}(\xi_1, \xi_2, \xi_3)
  \coloneqq &\chi_1 \Big(\frac{\xi_1}{\lambda_1}\Big) \chi_1 \Big(\frac{\xi_2}{\lambda_2}\Big)
             \chi_1 \Big(\frac{\xi_3}{\lambda_3}\Big) \widehat{\tilde{\nu}}_{\delta,j}(\xi) \\
   = &\chi_1\Big(\frac{\xi_1}{\lambda_1}\Big) \, \chi_1\Big(\frac{\xi_2}{\lambda_2}\Big) \, \chi_1\Big(\frac{\xi_3}{\lambda_3}\Big) \\
     &\times \int e^{-i( \xi_2 \phi(x, \delta, j) + \xi_3 x_2 + \xi_1 x_1 )} \, a(x,\delta,j) \, \chi_1(x_1) \, \chi_1(x_2) \, \mx.
\end{split}
\end{align}
We slightly abuse notation in the following way.
Whenever $\lambda_i = 1$, then the appropriate factor $ \chi_1(\frac{\xi_i}{\lambda_i})$ in the above expression
should be considered as a localisation to $|\xi_i| \lesssim 1$, instead of $|\xi_i| \sim 1$.

If we define the operators
\begin{align*}
\tilde{T}_{\delta, j} f \coloneqq f * \widehat{\tilde{\nu}}_{\delta, j}, \qquad \qquad
T^\lambda_j f \coloneqq f * \widehat{\nu^\lambda_j},
\end{align*}
then we formally have
\begin{align*}
\tilde{T}_{\delta, j} = \sum_{\lambda} T^\lambda_j,
\end{align*}
and according to \eqref{LowerThanTwo_TheFirstDecomposition_Estimate} and by applying the ``$R^*R$'' technique we need to prove
\begin{align}
\label{LowerThanTwo_TheFirstDecomposition_Estimate_rephrased}
\Vert \tilde{T}_{\delta, j} \Vert_{L^{\mathfrak{p}} \to L^{\mathfrak{p}'}} \lesssim \delta_3^{\frac{1}{2}} 2^{j}.
\end{align}
In case when we are able to obtain this estimate by summing absolutely the operator pieces $T^\lambda_j$
we shall proceed as explained in Subsection \ref{Section_h_lin_less_than_2_Basic_estimates}.
In this case in order to obtain the \eqref{LowerThanTwo_ABestimates} estimates
we need an $L^\infty$ bound for $\widehat{\nu_j^\lambda}$, which
we shall get from the expression \eqref{LowerThanTwo_TheFirstDecomposition_FourierForm},
and an $L^\infty$ bound for $\nu_j^\lambda$, which we shall derive next.

Using the equation \eqref{LowerThanTwo_TheFirstDecomposition_FourierForm} we get by Fourier inversion
\begin{align}
\begin{split}
\label{LowerThanTwo_TheFirstDecomposition_SpaceForm}
\nu_j^\lambda (x_1, x_2, x_3) = &\lambda_1 \lambda_2 \lambda_3 \int \widecheck{\chi}_1 (\lambda_1 (x_1-y_1)) \, \widecheck{\chi}_1 (\lambda_2 (x_2-\phi(y, \delta, j)))\\
   &\times \widecheck{\chi}_1(\lambda_3 (x_3-y_2)) \, a(y, \delta, j) \, \chi_1(y_1) \, \chi_1(y_2) \, \my.
\end{split}
\end{align}
Here we immediately obtain that the $L^{\infty}$ bound on $\nu_j^\lambda$ is
up to a multiplicative constant $\lambda_2$ using the first and the third factor within the integral
by substituting $\lambda_1 y_1$ and $\lambda_3 y_2$.
On the other hand, one can easily verify that
$$\partial_{y_2} \phi(y, \delta, j) \sim \delta_3^{-1/2} \, 2^{-2j} \ll 1,$$
and hence
by substituting $z_1 = \lambda_1 y_1, z_2 = \lambda_2 \phi(y, \delta, j)$, and utilising the first two factors within the integral,
we obtain
$$ \Vert \nu_j^\lambda \Vert_{L^\infty} \lesssim \delta_3^{1/2} \, 2^{2j} \lambda_3, $$
and therefore combining these two estimates we get
\begin{equation}
\label{LowerThanTwo_TheFirstDecomposition_SpaceBound}
\Vert \nu_j^\lambda \Vert_{L^\infty} \lesssim \min\{ \lambda_2, \, \delta_3^{1/2} \, 2^{2j} \lambda_3 \}.
\end{equation}

It remains to estimate the Fourier side;
for this we shall need to consider several cases depending on the relation between $\lambda_1$, $\lambda_2$, and $\lambda_3$.
Let us mention that as in \cite{IM16},
here we shall have no problems when absolutely summing the ``diagonal'' pieces
where $\lambda_1 \sim \lambda_2 \sim \delta_3^{1/2} 2^{2j} \lambda_3$.
However, unlike in \cite{IM16}, a case appears which is not absolutely summable.
This will be a recurring theme in this article.
It will also indicate that we should take care even when estimates are obtained by integration by parts.

{\bf{Case 1. $\lambda_1 \ll \lambda_2$ or $\lambda_1 \gg \lambda_2$, and $\lambda_3 \gg \lambda_2$.}}
In this case we can use integration by parts in both $x_1$ and $x_2$ in \eqref{LowerThanTwo_TheFirstDecomposition_FourierForm} to obtain
$$ \Vert \widehat{\nu_j^\lambda} \Vert_{L^\infty} \lesssim
   \Big( \lambda_3 \, \max\{\lambda_1, \lambda_2\} \Big)^{-N}, $$
for any nonnegative integer $N$.
Therefore, after plugging this estimate and the estimate \eqref{LowerThanTwo_TheFirstDecomposition_SpaceBound}
into \eqref{LowerThanTwo_ABestimates} and \eqref{LowerThanTwo_basic_estimate},
we may sum in all three parameters $\lambda_1$, $\lambda_2$, and $\lambda_3$,
after which one obtains an admissible estimate for \eqref{LowerThanTwo_TheFirstDecomposition_Estimate_rephrased}.

{\bf{Case 2. $\lambda_1 \ll \lambda_2$ or $\lambda_1 \gg \lambda_2$, and $\lambda_3 \lesssim \lambda_2$.}}
Here it is sufficient to use integration by parts in $x_1$. Therefore, we have
$$ \Vert \widehat{\nu_j^\lambda} \Vert_{L^\infty}
   \lesssim \Big( \max\{\lambda_1, \lambda_2\} \Big)^{-N}, $$
for any nonnegative integer $N$.
Again, after interpolating summation of operators $T^\lambda_j$ is possible in all three parameters.

{\bf{Case 3. $\lambda_1 \sim \lambda_2$ and $\lambda_3 \ll \delta_3^{-1/2} \, 2^{-2j} \lambda_2$.}}
In this case we see that necessarily $\lambda_1 \gtrsim \delta_3^{\frac{1}{2}} 2^{2j}$.
Also we note that if we fix say $\lambda_1$, then there are only finitely many dyadic numbers $\lambda_2$ such that $\lambda_1 \sim \lambda_2$,
and therefore we essentially need to sum in only two parameters in this case.
By stationary phase (and integration by parts when away from the critical point) in $x_1$ and integration by parts in $x_2$ we get
$$ \Vert \widehat{\nu_j^\lambda} \Vert_{L^\infty} \lesssim \lambda_1^{-\frac{1}{2}} \, (\delta_3^{-\frac{1}{2}} \, 2^{-2j} \, \lambda_1)^{-N}. $$
The better bound in \eqref{LowerThanTwo_TheFirstDecomposition_SpaceBound} is $\delta_3^{1/2} \, 2^{2j} \lambda_3$.
Therefore \eqref{LowerThanTwo_basic_estimate} becomes in our case
\begin{align*}
\Vert T^\lambda_j \Vert_{ L^{\mathfrak{p}_3}_{x_3} (L^{\mathfrak{p}_1}_{(x_1,x_2)}) \to L^{\mathfrak{p}_3'}_{x_3} (L^{\mathfrak{p}_1'}_{(x_1,x_2)}) }
  &\lesssim
  \lambda_1^{(\theta-1)(N+1/2)} \, (\delta_3^{\frac{1}{2}} \, 2^{2j})^{N(1-\theta)} \,
  (\delta_3^{\frac{1}{2}} \, 2^{2j})^{\theta} \, \lambda_3^{\theta} \, \lambda_3^{\frac{1}{2}-\theta}\\
  &\lesssim
  \lambda_1^{(\theta-1)(N+1/2)} \, \lambda_3^{\frac{1}{2}} \, (\delta_3^{\frac{1}{2}} \, 2^{2j})^{N - (N-1)\theta},
\end{align*}
and hence by summation in $\lambda_3$ and taking $N=1$ we get
\begin{align*}
\sum_{\lambda_3 \lesssim \delta_3^{-1/2} \, 2^{-2j} \lambda_1}
  \, \Vert T^\lambda_j \Vert_{ L^{\mathfrak{p}_3}_{x_3} (L^{\mathfrak{p}_1}_{(x_1,x_2)}) \to L^{\mathfrak{p}_3'}_{x_3} (L^{\mathfrak{p}_1'}_{(x_1,x_2)}) }
  &\lesssim
  \lambda_1^{\theta(N+1/2) - N} \, (\delta_3^{\frac{1}{2}} \, 2^{2j})^{N - (N-1)\theta - \frac{1}{2}}\\
  &\lesssim
  \lambda_1^{3\theta/2 - 1} \, (\delta_3^{\frac{1}{2}} \, 2^{2j})^{\frac{1}{2}}\\
  &\lesssim
  \lambda_1^{-\frac{1}{2}} \, (\delta_3^{\frac{1}{2}} \, 2^{2j})^{\frac{1}{2}}.
\end{align*}
Now we obviously get the desired result by summation over $\lambda_1 \gtrsim \delta_3^{\frac{1}{2}} 2^{2j}$.  

{\bf{Case 4. $\lambda_1 \sim \lambda_2$ and $\lambda_3 \sim \delta_3^{-1/2} \, 2^{-2j} \lambda_2$.}}
Here we essentially sum in only one parameter. Let us first determine the estimate in \eqref{LowerThanTwo_basic_estimate}. 

\noindent
{\bf{Subcase a). $1 \leq \lambda_1 \lesssim \delta_3^{\frac{3}{2}} 2^{4j}$.}}
Here we have by stationary phase in $x_1$
$$ \Vert \widehat{\nu_j^\lambda} \Vert_{L^\infty}
   \lesssim \lambda_1^{-1/2}. $$
Therefore by \eqref{LowerThanTwo_basic_estimate} we obtain
\begin{align*}
\Vert T^\lambda_j \Vert_{ L^{\mathfrak{p}_3}_{x_3} (L^{\mathfrak{p}_1}_{(x_1,x_2)}) \to L^{\mathfrak{p}_3'}_{x_3} (L^{\mathfrak{p}_1'}_{(x_1,x_2)}) }
  &\lesssim
  \lambda_1^{\frac{1}{2}(\theta-1)} \, \lambda_1^{\theta} \, \lambda_3^{\frac{1}{2}-\theta}\\
  &=
  \lambda_1^{\frac{3}{2}\theta - \frac{1}{2}} \,
  \Big( \delta_3^{-\frac{1}{2}} 2^{-2j} \lambda_1 \Big)^{\frac{1}{2}-\theta}\\
  &=
  \delta_3^{\frac{1}{2}\theta-\frac{1}{4}} \, 2^{2j\theta -j} \, \lambda_1^{\frac{\theta}{2}}.
\end{align*}
{\bf{Subcase b). $\lambda_1 \gg \delta_3^{\frac{3}{2}} 2^{4j}$.}}
In this case we have by stationary phase in $x_1$ and subsequently
by the van der Corput lemma (Lemma \ref{Oscillatory_auxiliary_van_der_corput}, $(i)$, with $M = 2$) in the second
$$ \Vert \widehat{\nu_j^\lambda} \Vert_{L^\infty}
   \lesssim \delta_3^{3/4} 2^{2j} \lambda_1^{-1}, $$
and hence
\begin{align*}
\Vert T^\lambda_j \Vert_{ L^{\mathfrak{p}_3}_{x_3} (L^{\mathfrak{p}_1}_{(x_1,x_2)}) \to L^{\mathfrak{p}_3'}_{x_3} (L^{\mathfrak{p}_1'}_{(x_1,x_2)}) }
  &\lesssim
  \delta_3^{\frac{3}{4}-\frac{3}{4}\theta} \, 2^{2j-2j\theta} \, \lambda_1^{\theta-1} \, \lambda_1^{\theta} \,
  \delta_3^{\frac{1}{2}\theta-\frac{1}{4}} \, 2^{2j\theta-j} \, \lambda_1^{\frac{1}{2}-\theta}\\
  &= \delta_3^{\frac{1}{2}-\frac{1}{4}\theta} \, 2^j \, \lambda_1^{\theta-\frac{1}{2}}.
\end{align*}

Now we sum in $\lambda_1$ using the estimates obtained in calculations in Subcases a) and b):
\begin{align*}
\sum_{\lambda_1 \geq 1} \Vert T^\lambda_j \Vert_{ L^{\mathfrak{p}_3}_{x_3} (L^{\mathfrak{p}_1}_{(x_1,x_2)}) \to L^{\mathfrak{p}_3'}_{x_3} (L^{\mathfrak{p}_1'}_{(x_1,x_2)}) }
  &\lesssim
  \delta_3^{\frac{1}{2}\theta-\frac{1}{4}} \, 2^{2j\theta -j} \, (\delta_3^{\frac{3}{2}} 2^{4j})^{\frac{\theta}{2}}
  +
  \delta_3^{\frac{1}{2}-\frac{1}{4}\theta} \, 2^j \, (\delta_3^{\frac{3}{2}} 2^{4j})^{\theta-\frac{1}{2}}\\
  &=
  2 \,\cdot\, \delta_3^{\frac{5\theta-1}{4}} \, 2^{4j\theta -j},
\end{align*}
and therefore it remains to see whether this is admissible for \eqref{LowerThanTwo_TheFirstDecomposition_Estimate_rephrased}:
\begin{align*}
  & \quad \delta_3^{\frac{5\theta-1}{4}} \, 2^{4j\theta -j} \lesssim \delta_3^{\frac{1}{2}} \, 2^j\\
\Longleftrightarrow & \quad\quad\quad\, \delta_3^{-\frac{3-5\theta}{4}}\lesssim (2^{2j})^{1-2\theta}.
\end{align*}
But recall that $2^{2j} \delta_3 \gg 1$, i.e., $\delta_3^{-1} \ll 2^{2j}$, and
notice that $0<\theta \leq 1/3$ implies $0< (3-5\theta)/4 \leq 1-2\theta$.
Hence, it is indeed admissible and we are done with this case.

{\bf{Case 5. $\lambda_1 \sim \lambda_2$ and $\lambda_3 \gg \delta_3^{-1/2} \, 2^{-2j} \lambda_2$.}}
Here we have by the stationary phase method in $x_1$ and integration by parts in $x_2$
$$ \Vert \widehat{\nu_j^\lambda} \Vert_{L^\infty} \lesssim \lambda_1^{-\frac{1}{2}} \, (\lambda_3)^{-N}, $$
and the bound in \eqref{LowerThanTwo_TheFirstDecomposition_SpaceBound} is $\lambda_1 \sim \lambda_2$.
Interpolating, we obtain (with a different $N$)
\begin{align}
\label{LowerThanTwo_TheFirstDecomposition_Basic_estimate_Case_5}
\Vert T^\lambda_j \Vert_{ L^{\mathfrak{p}_3}_{x_3} (L^{\mathfrak{p}_1}_{(x_1,x_2)}) \to L^{\mathfrak{p}_3'}_{x_3} (L^{\mathfrak{p}_1'}_{(x_1,x_2)}) }
  &\lesssim \lambda_1^{(3\theta-1)/2} \lambda_3^{-N}.
\end{align}
Now if $\theta < 1/3$, then we can easily sum in both $\lambda_1$ and $\lambda_3$.
Therefore, we assume in the following that $\theta = 1/3$.

{\bf{Subcase a). $\lambda_1 \gtrsim \delta_3^{\frac{1}{2}} 2^{2j}$.}}
Summing here in $\lambda_1$ between $\delta_3^{1/2} \, 2^{2j}$ and $\delta_3^{1/2} \, 2^{2j} \, \lambda_3$,
both up to a multiplicative constant, we get
\begin{align*}
\sum_{\delta_3^{\frac{1}{2}} 2^{2j} \lesssim \lambda_1 \lesssim \delta_3^{\frac{1}{2}} 2^{2j} \lambda_3}
  \Vert T^\lambda_j \Vert_{ L^{\mathfrak{p}_3}_{x_3} (L^{\mathfrak{p}_1}_{(x_1,x_2)}) \to L^{\mathfrak{p}_3'}_{x_3} (L^{\mathfrak{p}_1'}_{(x_1,x_2)}) }
  &\lesssim \lambda_3^{-N} \, \log_2 \Big( \frac{\delta_3^{1/2} \, 2^{2j} \, \lambda_3}{\delta_3^{1/2} \, 2^{2j}} \Big)\\
  &\lesssim \lambda_3^{-N+1}.
\end{align*}
Now we may sum in $\lambda_3$ to get the desired result.

{\bf{Subcase b). $1 \leq \lambda_1 \ll \delta_3^{\frac{1}{2}} 2^{2j}$.}}
Note that here we sum $\lambda_3$ over all the dyadic numbers greater than or equal to $1$.
We can also assume that $\lambda_1 \gg \delta_3^{\frac{1}{2}} 2^{j}$ since
summation in $\lambda_1$ in \eqref{LowerThanTwo_TheFirstDecomposition_Basic_estimate_Case_5} up to $\delta_3^{\frac{1}{2}} 2^{j}$
gives the bound $\lambda_3^{-N+1/2} \, \log_2(\delta_3^{\frac{1}{2}} 2^{j})$ which we can
sum in $\lambda_3$ and then estimate by $\delta_3^{\frac{1}{2}} 2^{j}$.
This is admissible for \eqref{LowerThanTwo_TheFirstDecomposition_Estimate_rephrased}.

In order to obtain the required bound in the remaining range:
\begin{align*}
\delta_3^{\frac{1}{2}} 2^{j} \ll \lambda_1 \ll \delta_3^{\frac{1}{2}} 2^{2j}, \qquad \qquad
1 \leq \lambda_3,
\end{align*}
we need to use the complex interpolation technique developed in \cite{IM16}.
For simplicity we assume that $\lambda_1 = \lambda_2$
(we can do this without losing much on generality since for a fixed $\lambda_1$ there are only finitely many dyadic numbers $\lambda_2$
such that $\lambda_1 \sim \lambda_2$).

We need to consider the following function parametrised by the complex number $\zeta$ and the dyadic number $\lambda_3$:
\begin{align}
\label{LowerThanTwo_TheFirstDecomposition_ComplexInterpolationMeasure}
\mu_\zeta^{\lambda_3} = \gamma(\zeta) \, (\delta_3^{-3/2}\,2^{-3j})^\zeta \, \sum_{\delta_3^{1/2}2^{j} \ll \lambda_1 \ll \delta_3^{1/2}2^{2j}}
    (\lambda_1)^{\frac{1-3\zeta}{2}} \, \nu_j^\lambda,
\end{align}
where
\begin{align*}
\gamma(\zeta) = 2^{-3(\zeta-1)/2} - 1.
\end{align*}
The associated convolution operator (given by convolution against the Fourier transform of the function $\mu_\zeta^{\lambda_3}$)
we denote by $T_\zeta^{\lambda_3}$.

At this point let us mention that whenever we use complex interpolation
we shall generically denote by $\mu_\zeta$ the considered measure parametrised by the complex number $\zeta$,
sometimes with an additional superscript, as is in the current case.
Similarily, the associated operator shall be denoted by $T_\zeta$, up to possible appearing superscripts.

For $\zeta = 1/3$ we see that
\begin{align*}
\delta_3^{1/2}\,2^{j} \, \mu_\zeta^{\lambda_3} = \sum_{\delta_3^{1/2}2^{j} \ll \lambda_1 \ll \delta_3^{1/2}2^{2j}} \nu_j^\lambda,
\end{align*}
which means, by Stein's interpolation theorem, that it is sufficient to prove
\begin{align}
\label{LowerThanTwo_TheFirstDecomposition_ComplexInterpolation_Basic_estimates}
\begin{split}
\Vert T_{it}^{\lambda_3} \Vert_{ L^{2/(2-\tilde{\sigma})}_{x_3} (L^1_{(x_1,x_2)}) \to L^{2/\tilde{\sigma}}_{x_3} (L^\infty_{(x_1,x_2)}) }
   &\lesssim \lambda_3^{-N},\\
\Vert T_{1+it}^{\lambda_3} \Vert_{L^2\to L^2}
   &\lesssim 1,
\end{split}
\end{align}
for some $N > 0$, with constants uniform in $t \in \R$, and where $\tilde{\sigma} = 1/4$ since $m = 2$, i.e., $\theta = 1/3$
(see \eqref{LowerThanTwo_theta_and_sigma}).

The first estimate is trivial in \eqref{LowerThanTwo_TheFirstDecomposition_ComplexInterpolation_Basic_estimates}.
Namely, since $\widehat{\nu_j^\lambda}$ have essentially disjoint supports, it follows
from the formula \eqref{LowerThanTwo_TheFirstDecomposition_ComplexInterpolationMeasure} and
the estimate on the Fourier transform of $\nu_j^\lambda$ that
$$ \Vert \widehat{\mu^{\lambda_3}_{it}} \Vert_{L^\infty} \lesssim \lambda_3^{-N}, $$
for any $N \in \N$, the implicit constant depending of course on $N$.
Now one just uses the results from Subsection \ref{Mixed_norm_auxiliary}.

In order to prove the second estimate in \eqref{LowerThanTwo_TheFirstDecomposition_ComplexInterpolation_Basic_estimates}
we shall need to use the oscillatory sum result Lemma \ref{Oscillatory_auxiliary_one_parameter}.
It turns out that the term $(\delta_3^{-3/2}\,2^{-3j})^\zeta$ in the definition of $\mu_\zeta^{\lambda_3}$ is redundant,
and that we can actually prove the stronger estimate
\begin{align*}
\Bigg\Vert \gamma(1+it) \sum_{\delta_3^{1/2}2^{j} \ll \lambda_1 \ll \delta_3^{1/2}2^{2j}} (\lambda_1)^{-1-\frac{3}{2}it} \, \nu_j^\lambda \Bigg\Vert_{L^\infty}
   \lesssim 1,
\end{align*}
that is
\begin{align}
\label{LowerThanTwo_TheFirstDecomposition_ComplexInterpolationFinal}
\Bigg\Vert \sum_{\delta_3^{1/2}2^{j} \ll \lambda_1 \ll \delta_3^{1/2}2^{2j}} (\lambda_1)^{-1-\frac{3}{2}it} \, \nu_j^\lambda \Bigg\Vert_{L^\infty}
   \lesssim \frac{1}{\Big| 2^{-\frac{3}{2}it} - 1 \Big|},
\end{align}
uniformly in $t$.

We start by substituting $\lambda_1 y_1 \mapsto y_1$ and $\lambda_3 y_2 \mapsto y_2$ in the expression \eqref{LowerThanTwo_TheFirstDecomposition_SpaceForm}
and plugging the obtained expression into the sum on the left hand side of \eqref{LowerThanTwo_TheFirstDecomposition_ComplexInterpolationFinal}:
\begin{align*}
\sum_{\delta_3^{1/2}2^{j} \ll \lambda_1 \ll \delta_3^{1/2}2^{2j}}
   (\lambda_1)^{-\frac{3}{2}it}
   &\iint \widecheck{\chi}_1 (\lambda_1 x_1-y_1) \, \widecheck{\chi}_1 (\lambda_1 x_2- \lambda_1 \phi(y_1/\lambda_1, y_2/\lambda_3, \delta, j))\\
   &\times \widecheck{\chi}_1(\lambda_3 x_3-y_2) \, a(y_1/\lambda_1, y_2/\lambda_3, \delta, j) \, \chi_1(y_1/\lambda_1) \, \chi_1(y_2/\lambda_3) \, \my_1\my_2.
\end{align*}
Recall that here $y_1 \sim \lambda_1$ and $y_2 \sim \lambda_3$ are both positive, and that $|\phi(\lambda^{-1}_1 y_1, \lambda^{-1}_3 y_2, \delta, j)| \sim 1$.
Therefore we can assume $|(x_1,x_2)| \leq C$ for some large constant $C$,
since otherwise we can use the first two factors within the integral to gain a factor $\lambda_1^{-N}$.
As the dominant term in $\phi$ is in the $y_1$ variable and as $\lambda_3$ is fixed,
we shall only concentrate on the $y_1$ integration and consider $y_2/\lambda_3 \sim 1$ as a bounded parameter.
Therefore the inner $y_1$ integration, after substituting $\lambda_1 x_1-y_1 \mapsto y_1$, becomes
\begin{align*}
\sum_{\delta_3^{1/2}2^{j} \ll \lambda_1 \ll \delta_3^{1/2}2^{2j}}
   (\lambda_1)^{-\frac{3}{2}it}
   &\int \widecheck{\chi}_1 (y_1) \, \widecheck{\chi}_1 (\lambda_1 x_2- \lambda_1 \phi(x_1-\lambda^{-1}_1 y_1, \lambda^{-1}_3 y_2, \delta, j))\\
   &\times a(x_1 - \lambda^{-1}_1 y_1, \lambda^{-1}_3 y_2 , \delta, j) \, \chi_1(x_1 - \lambda_1^{-1} y_1) \my_1,
\end{align*}
where now $x_1 - \lambda_1^{-1} y_1 \sim 1$, and therefore $|y_1| \lesssim \lambda_1$.

Next, we can restrict ourselves, by using a smooth cutoff function,
to the discussion of the integration domain where $|y_1| \ll \lambda_1^{\varepsilon}$ for some small $\varepsilon$,
since in the other part by using the first factor in the integral we could gain a factor of $\lambda_1^{-N\varepsilon}$.
Since $\lambda_1 \gg \delta_3^{1/2} 2^j \gg 1$ can be taken arbitrarily large, and hence $\lambda_1^{-1} y_1$ arbitrarily small,
the relation $x_1 - \lambda_1^{-1} y_1 \sim 1$ implies $x_1 \sim 1$.
Therefore by applying a Taylor expansion to the function $\phi(x_1-\lambda^{-1}_1 y_1, \lambda^{-1}_3 y_2, \delta, j)$
in the first variable, we obtain
\begin{align*}
\sum_{\delta_3^{1/2}2^{j} \ll \lambda_1 \ll \delta_3^{1/2}2^{2j}}
   (\lambda_1)^{-\frac{3}{2}it}
   &\int \widecheck{\chi}_1 (y_1) \,
     \widecheck{\chi}_1 (\lambda_1 Q(x_1,x_2, \lambda^{-1}_3 y_2, \delta, j) + y_1 \, r(\lambda^{-1}_1 y_1, x_1, \lambda^{-1}_3 y_2, \delta, j))\\
   &\times a(x_1 - \lambda^{-1}_1 y_1, \lambda^{-1}_3 y_2 , \delta, j) \, \chi_1(x_1 - \lambda_1^{-1} y_1) \, \chi_0(\lambda_1^{-\varepsilon} \, y_1) \my_1,
\end{align*}
where $|\partial_1^N r| \sim 1$ for any $N \geq 0$, and $Q(x_1,x_2, \lambda^{-1}_3 y_2, \delta, j) = x_2 - \phi(x_1, \lambda^{-1}_3 y_2, \delta, j)$.

Now we note that the first two factors in the integral are essentially a convolution, and therefore,
by using this two factors, one easily obtains that the bound on the integral is $|\lambda_1 Q|^{-N}$.
If $|\lambda_1 Q| \gg 1$, $|\lambda_1 Q|^{-N}$ is a geometric series summable in $\lambda_1$, and if $|\lambda_1 Q| \lesssim 1$,
then we are actually within the scope of Lemma \ref{Oscillatory_auxiliary_one_parameter}.
Namely, we define the function $H$ as
\begin{align*}
H(z_1, z_2, z_3; \lambda^{-1}_3 y_2, x_1, x_2 , \delta, 2^{-j}) \coloneqq
   &\int \widecheck{\chi}_1 (y_1) \,
     \widecheck{\chi}_1 (z_1 + y_1 \, r(z_2^{1/\varepsilon} y_1, x_1, \lambda^{-1}_3 y_2, \delta, j))\\
   &\times a(x_1 - z_2^{1/\varepsilon} y_1, \lambda^{-1}_3 y_2 , \delta, j) \, \chi_1(x_1 - z_2^{1/\varepsilon} y_1) \, \chi_0(z_2 \, y_1) \my_1.
\end{align*}
Note that $H$ does not actually depend on $z_3$,
but we need to use it in order to implement the lower bound on $\lambda_1$ in the summation
(this is realised through the characteristic function $\chi_Q$ in the definition of $F(t)$ in Lemma \ref{Oscillatory_auxiliary_one_parameter}).
Tracing back, we note that all the dependencies in $j$ are actually dependencies in $2^{-j}$.
All the parameters $(\lambda^{-1}_3 y_2, x_1, x_2 , \delta, 2^{-j})$ are now restrained to a bounded set
and the $C^1$ norm of $H$ in $(z_1, z_2, z_3)$ is bounded uniformly in all the (bounded) parameters if $(z_1, z_2, z_3)$ are
contained in a bounded set.
Therefore by taking
\begin{align*}
(z_1, z_2, z_3) = (\lambda_1 Q(x_1, x_2, \lambda^{-1}_3 y_2, \delta, j), \lambda_1^{-\varepsilon}, \delta_3^{1/2}2^{j} \lambda_1^{-1})
\end{align*}
and applying Lemma \ref{Oscillatory_auxiliary_one_parameter} with $\alpha = -3/2$, $\lambda_1 = 2^l$,
$M = c \delta_3^{1/2}2^{2j}$ for a small $c > 0$ determined by the implicit constant in the summation condition $\lambda_1 \ll \delta_3^{1/2}2^{2j}$,
and with
\begin{align*}
(\beta^1, \beta^2, \beta^3) &= (1, -\varepsilon, -1),\\
(a_1, a_2, a_3) &= (Q(x_1, x_2, \lambda^{-1}_3 y_2, \delta, j), 1, \delta_3^{1/2}2^{j}),
\end{align*}
we obtain the bound \eqref{LowerThanTwo_TheFirstDecomposition_ComplexInterpolationFinal}.
Note that the lower bound on $\lambda_1$ in the summation in \eqref{LowerThanTwo_TheFirstDecomposition_ComplexInterpolationFinal}
is realised by taking $|z_3| \ll 1$.
We are done with the case $2^{2j} \delta_3 \gg 1$.

\subsection{The setting when $2^{2j} \delta_3 \lesssim 1$}

As explained in Section \cite[Subsection 4.2]{IM16},
in this case we use the change of coordinates $(x_1, x_2) \mapsto (x_1, 2^{-j}( x_2 + x_1^m \omega(\delta_1 x_1)))$
in the expression \eqref{LowerThanTwo_TheFirstDecomposition_original_measure} for $\nu_{\delta, j}$.
After renormalising the measure $\nu_{\delta, j}$
we obtain that the mixed norm Fourier restriction estimate for $\nu_{\delta, j}$ is equivalent to
\begin{equation*}
\int |\widehat{f}|^2 \mathrm{d}\tilde{\nu}_{\delta, j} \leq
  \Const 2^{j(1-4/p_3')} \Vert f \Vert^2_{L^\mathfrak{p}(\R^3)}, \qquad f \in \mathcal{S}(\R^3),
\end{equation*}
that is, since $\mathfrak{p}_3' = 4$,
\begin{equation*}
\int |\widehat{f}|^2 \mathrm{d}\tilde{\nu}_{\delta, j} \leq
  \Const \Vert f \Vert^2_{L^\mathfrak{p}(\R^3)}, \qquad f \in \mathcal{S}(\R^3),
\end{equation*}
where $\tilde{\nu}_{\delta,j}$ is the rescaled measure
\begin{align*}
\langle \tilde{\nu}_{\delta, j}, f \rangle \coloneqq
   \int f(x_1, 2^{-j}x_2 + x_1^m \omega(\delta_1 x_1), \phi^a(x, \delta, j)) a(x, \delta, j) \mx.
\end{align*}
The function $a(x, \delta, j)$ has the form
\begin{align*}
a(x, \delta, j) \coloneqq
   \chi_1(\phi^a(x, \delta, j)) a(x_1, 2^{-j}x_2 + x_1^m \omega(\delta_1 x_1), \delta)
\end{align*}
and the phase function is given by
\begin{align}
\label{LowerThanTwo_TheSecondDecomposition_Phase}
\phi^a(x, \delta, j) \coloneqq
   &\tilde{b}(x_1, 2^{-j}x_2 + x_1^m \omega(\delta_1 x_1), \delta_1, \delta_2) x_2^2 + 2^{2j} \delta_3 x_1^n \beta(\delta_1 x_1),
\end{align}
where $|\tilde{b}(x_1,x_2,0,0)| \sim 1$ and $|\beta(0)| \sim 1$.

Also, we recall that when $\delta \to 0$, then $\tilde{b}(x_1, x_2, \delta_1, \delta_2)$
converges in $C^\infty$ to a nonzero constant if $\phi$ has $A$ type singularity,
and that it converges up to a multiplicative constant to $x_1$ if $\phi$ has $D$ type singularity.
We shall assume without loss of generality that $\tilde{b}(x_1,x_2,\delta_1,\delta_2) > 0$
since one can just reflect the third coordinate of $f$ in the expression for the measure $\tilde{\nu}_{\delta, j}$.

Support assumptions on $a(\cdot, \delta)$ from Subsection \ref{Section_h_lin_less_than_2_First_decompositions}
(namely, that the support is contained in a small neighbourhood of the point $(v_1,v_1^m \omega(0))$ for some $v_1 > 0$)
imply that $a(\cdot, \delta, j)$ is supported in a set where $x_1 \sim 1$ and $|x_2| \lesssim 1$.

We again perform a spectral decomposition of $\tilde{\nu}_{\delta, j}$,
i.e., for $(\lambda_1, \lambda_2, \lambda_3)$ dyadic numbers with $\lambda_i \geq 1$, $i = 1,2,3$,
we consider localised measures $\nu_j^\lambda$ defined through
\begin{align}
\begin{split}
\label{LowerThanTwo_TheSecondDecomposition_FourierForm}
\widehat{\nu_j^\lambda} (\xi) =
   &\chi_1\Big(\frac{\xi_1}{\lambda_1}\Big) \, \chi_1\Big(\frac{\xi_2}{\lambda_2}\Big) \, \chi_1\Big(\frac{\xi_3}{\lambda_3}\Big)\\
   &\times \int e^{-i \Phi(x, \delta, j, \xi)} \, a(x,\delta,j) \, \chi_1(x_1) \, \chi_1(x_2) \, \mx,
\end{split}
\end{align}
with the complete phase function $\Phi$ being
\begin{align*}
\Phi(x, \delta, j, \xi) \coloneqq
\xi_3 \phi^a(x, \delta, j) + 2^{-j} \xi_2 x_2 + \xi_2 x_1^m \omega(\delta_1 x_1) + \xi_1 x_1.
\end{align*}
We also introduce the operators $\tilde{T}_{\delta, j} f \coloneqq f * \widehat{\tilde{\nu}}_{\delta, j}$
and $T^\lambda_j f \coloneqq f * \widehat{\nu^\lambda_j}$.
Then we need to prove:
\begin{align}
\Vert \tilde{T}_{\delta, j} \Vert_{L^{\mathfrak{p}} \to L^{\mathfrak{p}'}} \lesssim 1.
\end{align}
In most of the cases this will be done in a similar manner as in the previous subsection.
In the case when $2^{2j} \delta_3 \sim 1$, $\theta = 1/3$, and $\lambda_1 \sim \lambda_2 \sim \lambda_3$,
with which we shall deal in the next Section,
we shall need to perform a finer analysis.

\subsection{The case $2^{2j} \delta_3 \ll 1$}

Here we have the stronger bounds $x_1 \sim 1$ and $|x_2| \sim 1$ since $\phi^a(x,\delta,j) \sim 1$
by \eqref{LowerThanTwo_TheSecondDecomposition_Phase} and the assumption $2^{2j} \delta_3 \ll 1$.
We also have $|\partial_{x_2} \phi^a(x,\delta,j)| \sim 1$ since
$\phi^a(x, \delta, j)$ is a small pertubation of $b(0,0) x_2^2$ in case of $A$ type singularity,
and a small pertubation of $b_1(0,0) x_1 x_2^2$ in case of $D$ type singularity.

Taking the inverse transform of \eqref{LowerThanTwo_TheSecondDecomposition_FourierForm} we get
\begin{align}
\begin{split}
\label{LowerThanTwo_TheSecondDecomposition_SpaceForm}
\nu_j^\lambda (x) = &\lambda_1 \lambda_2 \lambda_3 \int \widecheck{\chi}_1 (\lambda_1 (x_1-y_1)) \, \widecheck{\chi}_1 (\lambda_2 (x_2-2^{-j}y_2-y_1^m\omega(\delta_1 y_1)))\\
   &\times \widecheck{\chi}_1(\lambda_3 (x_3-\phi^a(y,\delta,j))) \, a(y, \delta, j) \, \chi_1(y_1) \, \chi_1(y_2) \, \my.
\end{split}
\end{align}
Similarily as in the case $2^{2j} \delta_3 \gg 1$, we can consider either the substitution $(z_1, z_2) = (\lambda_1 y_1, \lambda_2 2^{-j} y_2)$,
or the substitution $(z_1, z_2) = (\lambda_1 y_1, \lambda_3 \phi^a(y, \delta, j))$
(in order to carry this out one needs to consider the cases $y_2 \sim 1$ and $y_2 \sim -1$ separately).
Then one can easily obtain
\begin{equation}
\label{LowerThanTwo_TheSecondDecomposition_SpaceBound}
\Vert \nu_j^\lambda \Vert_{L^\infty} \lesssim \min\{2^{j} \lambda_3, \lambda_2 \}.
\end{equation}
Next we calculate the $L^\infty$ bounds on the Fourier transform
by using the expression \eqref{LowerThanTwo_TheSecondDecomposition_FourierForm}.

{\bf{Case 1. $\lambda_1 \ll \lambda_2$ or $\lambda_1 \gg \lambda_2$, and $\lambda_3 \ll \max\{\lambda_1, \lambda_2\}$.}}
By integration by parts in $x_1$ one has
$$ \Vert \widehat{\nu_j^\lambda} \Vert_{L^\infty} \lesssim
   \Big( \max\{\lambda_1, \lambda_2\} \Big)^{-N}. $$
The operators $T^\lambda_j$ are now summable which can be seen
by using the estimate in \eqref{LowerThanTwo_basic_estimate} obtained by interpolation.

{\bf{Case 2. $\lambda_1 \ll \lambda_2$ or $\lambda_1 \gg \lambda_2$, and $\lambda_3 \gtrsim \max\{\lambda_1,\lambda_2\}$.}}
Here we use integration by parts in $x_2$ only and so we have the bound
$$ \Vert \widehat{\nu_j^\lambda} \Vert_{L^\infty} \lesssim
   \lambda_3^{-N}. $$
After interpolating we can again sum operators $T^\lambda_j$ in all three paramteres.

{\bf{Case 3. $\lambda_1 \sim \lambda_2$ and $\lambda_3 \ll 2^{-j}\lambda_2$.}}
Note that necessarily $\lambda_2 \geq 2^j$.
Here we use stationary phase in $x_1$ and integration by parts in $x_2$.
Then one gets the estimate
$$ \Vert \widehat{\nu_j^\lambda} \Vert_{L^\infty} \lesssim
   \lambda_1^{-1/2} (2^{-j} \lambda_2)^{-N}. $$
The better bound in \eqref{LowerThanTwo_TheSecondDecomposition_SpaceBound} is $2^j \lambda_3$.
Therefore \eqref{LowerThanTwo_basic_estimate} becomes
\begin{align*}
\Vert T^\lambda_j \Vert_{ L^{\mathfrak{p}_3}_{x_3} (L^{\mathfrak{p}_1}_{(x_1,x_2)}) \to L^{\mathfrak{p}_3'}_{x_3} (L^{\mathfrak{p}_1'}_{(x_1,x_2)}) }
   &\lesssim (\lambda_1^{-1/2} (2^{-j} \lambda_2)^{-N})^{1-\theta} (2^j \lambda_3)^\theta \lambda_3^{\frac{1}{2}-\theta}.
\end{align*}
If $\theta < 1/3$, then we can rewrite
\begin{align*}
\Vert T^\lambda_j \Vert_{ L^{\mathfrak{p}_3}_{x_3} (L^{\mathfrak{p}_1}_{(x_1,x_2)}) \to L^{\mathfrak{p}_3'}_{x_3} (L^{\mathfrak{p}_1'}_{(x_1,x_2)}) }
   &\lesssim (\lambda_1^{-1/2} \lambda_3^{-N})^{1-\theta} \lambda_1^\theta \lambda_3^{\frac{1}{2}-\theta},
\end{align*}
we note that one can now easily sum in both $\lambda_1$ and $\lambda_3$.
If $\theta = 1/3$, then the first inequality for $T^\lambda_j$ can be rewritten as
\begin{align*}
\Vert T^\lambda_j \Vert_{ L^{\mathfrak{p}_3}_{x_3} (L^{\mathfrak{p}_1}_{(x_1,x_2)}) \to L^{\mathfrak{p}_3'}_{x_3} (L^{\mathfrak{p}_1'}_{(x_1,x_2)}) }
   &\lesssim (2^{-j} \lambda_1)^{-N} \lambda_3^{1/2},
\end{align*}
for some different $N$.
Now we first sum in $\lambda_3$ up to $2^{-j}\lambda_1$, and then we sum in $\lambda_1 \geq 2^j$.

{\bf{Case 4. $\lambda_1 \sim \lambda_2$ and $\lambda_3 \sim 2^{-j}\lambda_2$.}}
Again necessarily $\lambda_2 \gtrsim 2^j$.
One uses in both $x_1$ and $x_2$ the stationary phase method and gets
$$ \Vert \widehat{\nu_j^\lambda} \Vert_{L^\infty} \lesssim
   2^{j/2} \lambda_1^{-1}. $$
The estimate for $\Vert \nu_j^\lambda \Vert_{L^\infty}$ from
\eqref{LowerThanTwo_TheSecondDecomposition_SpaceBound} is $\lesssim \lambda_2$.
Hence, we get the estimate
\begin{align*}
\Vert T^\lambda_j \Vert_{ L^{\mathfrak{p}_3}_{x_3} (L^{\mathfrak{p}_1}_{(x_1,x_2)}) \to L^{\mathfrak{p}_3'}_{x_3} (L^{\mathfrak{p}_1'}_{(x_1,x_2)}) }
   &\lesssim (2^{j/2} \lambda_1^{-1})^{1-\theta} \lambda_1^\theta \lambda_3^{1/2-\theta}\\
   &\lesssim 2^{j\theta/2}  \lambda_1^{\theta-1/2}.
\end{align*}
By summation in $\lambda_1 \gtrsim 2^j$ we obtain the bound
$$
2^{3j\theta/2 - j/2}.
$$
Now since $\theta \leq 1/3$, we get the desired result.

{\bf{Case 5. $\lambda_1 \sim \lambda_2$ and $\lambda_3 \gtrsim \lambda_2$.}}
Here it suffices to use integration by parts in $x_2$ only. One easily gets
\begin{align*}
\Vert T^\lambda_j \Vert_{ L^{\mathfrak{p}_3}_{x_3} (L^{\mathfrak{p}_1}_{(x_1,x_2)}) \to L^{\mathfrak{p}_3'}_{x_3} (L^{\mathfrak{p}_1'}_{(x_1,x_2)}) }
   &\lesssim \lambda_3^{-N},
\end{align*}
and one can now sum in both $\lambda_1$ and $\lambda_3$.

{\bf{Case 6. $\lambda_1 \sim \lambda_2$ and $2^{-j}\lambda_2 \ll \lambda_3 \ll \lambda_2$.}}
By the stationary phase method in $x_1$ and integration by parts in $x_2$
$$ \Vert \widehat{\nu_j^\lambda} \Vert_{L^\infty} \lesssim \lambda_1^{-\frac{1}{2}} \, (\lambda_3)^{-N}, $$
and the better bound in \eqref{LowerThanTwo_TheSecondDecomposition_SpaceBound} is $\lambda_2$.

Similarily as in the case $2^{2j} \delta_3 \gg 1$ one easily sees that, unless $\theta = 1/3$, one can sum in both parameters.
Henceforth we shall assume $\theta = 1/3$ and use complex interpolation in order to deal with this case.
Here we know that $\phi$ has $A$ type singularity and $\tilde{\sigma} = 1/4$.
For simplicity we shall again assume that $\lambda_1 = \lambda_2$.

We consider the following function parametrised by the complex number $\zeta$ and the dyadic number $\lambda_3$:
\begin{align*}
\mu_\zeta^{\lambda_3} = \gamma(\zeta) \, \sum_{\lambda_3 \ll \lambda_1 \ll 2^j \lambda_3}
    (\lambda_1)^{\frac{1-3\zeta}{2}} \, \nu_j^\lambda,
\end{align*}
where
\begin{align*}
\gamma(\zeta) = 2^{-3(\zeta-1)/2} - 1.
\end{align*}
We denote the associated convolution operator by $T_\zeta^{\lambda_3}$. For $\zeta = 1/3$ we see that
\begin{align*}
\mu_\zeta^{\lambda_3} = \sum_{\lambda_3 \ll \lambda_1 \ll 2^j \lambda_3} \nu_j^\lambda.
\end{align*}
Hence, by interpolation it suffices to prove
\begin{align*}
\begin{split}
\Vert T_{it}^{\lambda_3} \Vert_{ L^{2/(2-\tilde{\sigma})}_{x_3} (L^1_{(x_1,x_2)}) \to L^{2/\tilde{\sigma}}_{x_3} (L^\infty_{(x_1,x_2)}) }
   &\lesssim \lambda_3^{-N},\\
\Vert T_{1+it}^{\lambda_3} \Vert_{L^2\to L^2}
   &\lesssim 1,
\end{split}
\end{align*}
for some $N > 0$, with constants uniform in $t \in \R$.

The first estimate follows right away since $\widehat{\nu_j^\lambda}$ have essentially disjoint supports,
and so the $L^\infty$ estimate for $\widehat{\nu_j^\lambda}$ implies
$$ \Vert \widehat{\mu^{\lambda_3}_{it}} \Vert_{L^\infty} \lesssim \lambda_3^{-N}, $$
for any $N \in \N$.

We prove the second estimate using Lemma \ref{Oscillatory_auxiliary_one_parameter}. We need to prove
\begin{align}
\label{LowerThanTwo_TheSecondDecomposition_ComplexInterpolationFinal}
\Bigg\Vert \sum_{\lambda_3 \ll \lambda_1 \ll 2^j \lambda_3} (\lambda_1)^{-1-\frac{3}{2}it} \, \nu_j^\lambda \Bigg\Vert_{L^\infty}
   \lesssim \frac{1}{\Big| 2^{-\frac{3}{2}it} - 1 \Big|},
\end{align}
uniformly in $t$.

We first use the substitution $(z_1,z_2) = (y_1,\phi^a(y_1,y_2,\delta,j))$ in the expression
\eqref{LowerThanTwo_TheSecondDecomposition_SpaceForm}, considering the cases $y_2 \sim 1$ and $y_2 \sim -1$ separately.
In order to solve for $(y_1,y_2)$ in terms of $(z_1,z_2)$,
we introduce for a moment intermediary coordinates $(\tilde{y}_1, \tilde{y}_2) = (y_1,2^{-j}y_2+y_1^m \omega(\delta_1 y_1))$.
In coordinates $(\tilde{y}_1, \tilde{y}_2)$ the expression for $\phi^a = z_2$ becomes
\begin{align*}
2^{2j} \tilde{b}(\tilde{y}_1, \tilde{y}_2, \delta_1, \delta_2) (\tilde{y}_2 - \tilde{y}_1^m \omega(\delta_1 \tilde{y}_1))^2 + 2^{2j} \delta_3 \tilde{y}_1^n \beta(\delta_1 \tilde{y}_1).
\end{align*}
Then one can easily see that by solving for $\tilde{y}_2$ in terms of $(z_1,z_2)$, one gets precisely the expression
\eqref{LowerThanTwo_TheFirstDecomposition_Phase} as in the case $2^{2j} \delta_3 \gg 1$.
Therefore by solving for $y_2$ in terms of $(z_1, z_2)$ one gets
\begin{align*}
y_2 = \pm
   \tilde{b}_1 \Big( z_1, \sqrt{2^{-2j}z_2 - \delta_3 z_1^n \beta(\delta_1 z_1)}, \delta_1, \delta_2\Big)
   \times \sqrt{z_2 - 2^{2j} \delta_3 z_1^n \beta(\delta_1 z_1)},
\end{align*}
where now both $z_1$ and $z_2$ are positive. We shall from now on consider $y_2$ as a function of $(z_1, z_2)$.
On the limit $j \to \infty$ and $\delta \to 0$ the function $y_2 = y_2(z_1,z_2,\delta,j)$ converges to $\pm C \sqrt{z_2}$ for some constant $C \neq 0$
since we are in the $\theta = 1/3$ case (i.e., $A$ type singularity case);
see \eqref{LowerThanTwo_TheFirstDecomposition_tilde_b_1_convergence}.

After applying the just introduced substitution to the expression \eqref{LowerThanTwo_TheSecondDecomposition_SpaceForm} we get
\begin{align*}
\begin{split}
\nu_j^\lambda (x) = &\lambda_1 \lambda_2 \lambda_3 \int \widecheck{\chi}_1 (\lambda_1 (x_1-z_1)) \,
    \widecheck{\chi}_1 (\lambda_2 (x_2-2^{-j}y_2(z_1,z_2,\delta,j)-z_1^m\omega(\delta_1 z_1)))\\
   &\times \widecheck{\chi}_1(\lambda_3 (x_3-z_2)) \, \tilde{a}_1(z, \delta, j) \, \chi_1(z_1) \, \chi_1(y_2(z_1,z_2,\delta,j)) \, \mz,
\end{split}
\end{align*}
where $\tilde{a}_1$ is the function $a$ multiplied by the Jacobian of the change of variables.
Since $|y_2| \sim 1$ is equivalent to $|z_2| \sim 1$, we may rewrite again the above expression as
\begin{align}
\begin{split}
\label{LowerThanTwo_TheSecondDecomposition_SpaceForm_after_sub}
\nu_j^\lambda (x) = &\lambda_1 \lambda_2 \lambda_3 \int \widecheck{\chi}_1 (\lambda_1 (x_1-z_1)) \,
    \widecheck{\chi}_1 (\lambda_2 (x_2-2^{-j}y_2(z_1,z_2,\delta,j)-z_1^m\omega(\delta_1 z_1)))\\
   &\times \widecheck{\chi}_1(\lambda_3 (x_3-z_2)) \, \tilde{a}(z, \delta, j) \, \chi_1(z_1) \, \chi_1(z_2) \, \mz.
\end{split}
\end{align}
Now we substitute $\lambda_1 z_1 \mapsto z_1$ and $\lambda_3 z_2 \mapsto z_2$ in the expression \eqref{LowerThanTwo_TheSecondDecomposition_SpaceForm_after_sub},
plug it into the sum \eqref{LowerThanTwo_TheSecondDecomposition_ComplexInterpolationFinal}, and obtain
\begin{align*}
\sum_{\lambda_3 \ll \lambda_1 \ll 2^j \lambda_3}
   (\lambda_1)^{-\frac{3}{2}it}
   &\int \widecheck{\chi}_1 (\lambda_1 x_1 - z_1) \\
   &\times \widecheck{\chi}_1 ( \lambda_1 x_2 - 2^{-j} \lambda_1 y_2 (\lambda_1^{-1} z_1, \lambda_3^{-1} z_2, \delta, j) - \lambda_1^{-m+1} z_1^m\omega(\delta_1 \lambda_1^{-1} z_1))\\
   &\times \widecheck{\chi}_1 (\lambda_3 x_3 - z_2)
    \times \, \tilde{a}(\lambda_1^{-1} z_1, \lambda_3^{-1} z_2, \delta, 2^{-j}) \, \chi_1(\lambda_1^{-1} z_1) \, \chi_1(\lambda_3^{-1} z_2) \, \mz.
\end{align*}
Now we have $z_1 \sim \lambda_1$, $z_2 \sim \lambda_3$, and $|y_2(\lambda^{-1}_1 z_1, \lambda^{-1}_3 z_2, \delta, j)| \sim 1$.

We can assume $|(x_1, x_2)| \leq C$ for some large constant $C$, since otherwise we can use the first two factors within the integral
and gain a factor of $\lambda_1^{-N}$.
Similarily as in the case $2^{2j} \delta_3 \gg 1$ we shall consider integration in $z_1$ only
(and $\lambda_3^{-1} z_2$ shall be a bounded parameter), and one can also use the substitution $z_1 \mapsto \lambda_1 x_1 - z_1$
to reduce the problem to when $|z_1| \ll \lambda_1^{\varepsilon}$ and $x_1 \sim 1$.
We also introduce $\psi_\delta(x_1) = x_1^m \omega(\delta_1 x_1)$.
Then it remains to estimate
\begin{align*}
\sum_{\lambda_3 \ll \lambda_1 \ll 2^j \lambda_3}
   (\lambda_1)^{-\frac{3}{2}it}
   &\int \widecheck{\chi}_1 (z_1)
    \widecheck{\chi}_1 ( \lambda_1 (x_2-\psi_\delta(x_1-\lambda_1^{-1}z_1) - 2^{-j} \lambda_1 y_2 (x_1-\lambda_1^{-1} z_1, \lambda_3^{-1} z_2, \delta, j)))\\
   &\times \, \tilde{a}(x_1-\lambda_1^{-1} z_1, \lambda_3^{-1} z_2, \delta, 2^{-j}) \, \chi_1(x_1-\lambda_1^{-1} z_1) \, \chi_0(z_1 \lambda_1^{-\varepsilon}) \, \mathrm{d}z_1.
\end{align*}
Within the second factor in the integral we can use a Taylor approximation at $x_1$ and obtain
\begin{align*}
\sum_{\lambda_3 \ll \lambda_1 \ll 2^j \lambda_3}
   (\lambda_1)^{-\frac{3}{2}it}
   &\int \widecheck{\chi}_1 (z_1)
    \widecheck{\chi}_1 ( \lambda_1 Q(x_1,x_2,\lambda_3^{-1}z_2,\delta,2^{-j}) + z_1 r(\lambda_1^{-1}z_1,x_1,\lambda_3^{-1}z_2,\delta,2^{-j}) )\\
   &\times \, \tilde{a}(x_1-\lambda_1^{-1} z_1, \lambda_3^{-1} z_2, \delta, 2^{-j}) \, \chi_1(x_1-\lambda_1^{-1} z_1) \, \chi_0(z_1 \lambda_1^{-\varepsilon}) \, \mathrm{d}z_1,
\end{align*}
where $|\partial_1^N r| \sim 1$ for $N \geq 0$ since the term $\psi_\delta$ is dominant, and $Q$ is a smooth function with uniform bounds.
Now we notice that this form is the same as in the case $2^{2j} \delta_3 \gg 1$ in the part where we used complex interpolation, and hence
the same proof using the oscillatory sum lemma can be applied, up to obvious changes such as changing the summation bounds.


\subsection{The case $2^{2j} \delta_3 \sim 1$}

As in \cite{IM16} we denote 
$$\sigma \coloneqq 2^{2j} \delta_3, \qquad \qquad b^{\#}(x,\delta,j) \coloneqq \tilde{b}(x_1,2^{-j}x_2+x_1^m\omega(\delta_1 x_1),\delta),$$
and so $\sigma \sim 1$ and $|b^{\#}(x,\delta,j)| \sim 1$.
Therefore the complete phase can be rewritten as
\begin{align}
\label{LowerThanTwo_TheSecondDecomposition_CompletePhase_rewritten}
\begin{split}
\Phi(x, \delta, j, \xi) \coloneqq
  & \xi_1 x_1 + \xi_2 x_1^m \omega(\delta_1 x_1) + \xi_3 \sigma x_1^n \beta(\delta_1 x_1) \\
  & + 2^{-j} \xi_2 x_2 + \xi_3 b^{\#}(x,\delta,j) x_2^2. \\
\end{split}
\end{align}
Recall also that in this case we have the weaker conditions $x_1 \sim 1$ and $|x_2| \lesssim 1$
for the domain of integration in the integral in \eqref{LowerThanTwo_TheSecondDecomposition_FourierForm}.

We furthermore slightly modify the notation in this case, as it was done in \cite{IM16}.
Namely, $\delta$ shall denote in this subsection $(\delta_1, \delta_2)$ since $\delta_3$ appears only in $\sigma$.
We also note that in this case there is no $A_\infty$ nor $D_\infty$ type singularity.

Let us introduce the notation
\begin{align*}
\psi_\omega(y_1) = y_1^m \omega(\delta_1 y_1), \qquad \qquad
\psi_\beta(y_1) = \sigma y_1^n \beta(\delta_1 y_1).
\end{align*}
Then, after applying the inverse Fourier transform to \eqref{LowerThanTwo_TheSecondDecomposition_FourierForm}, we may write
\begin{align}
\begin{split}
\label{LowerThanTwo_TheSecondDecomposition_SpaceForm_rewritten}
\nu_j^\lambda (x) = \lambda_1 \lambda_2 \lambda_3 &\int \widecheck{\chi}_1 (\lambda_1 (x_1-y_1)) \, \widecheck{\chi}_1 (\lambda_2 (x_2-2^{-j}y_2-\psi_\omega(y_1)))\\
   &\times \widecheck{\chi}_1(\lambda_3 (x_3 -  b^{\#}(y,\delta,j) y_2^2 - \psi_\beta(y_1))) \\
   &\times a(y, \delta, j) \, \chi_1(y_1) \, \chi_0(y_2) \, \my.
\end{split}
\end{align}
As was noted in \cite[Subsection 4.2.2.]{IM16}, here we have the bounds
\begin{equation}
\label{LowerThanTwo_TheSecondDecomposition_SpaceBound_last_case}
\Vert \nu_j^\lambda \Vert_{L^\infty} \lesssim \lambda_3^{1/2} \min\{2^{j} \lambda^{1/2}_3, \lambda_2 \}.
\end{equation}
Namely, in the first factor within the integral in \eqref{LowerThanTwo_TheSecondDecomposition_SpaceForm_rewritten}
we can substitute $\lambda_1 y_1 \mapsto y_1$, and afterwards either substitute $\lambda_2 2^{-j} y_2 \mapsto y_2$ in the second factor, or
use the van der Corput lemma (i.e., Lemma \ref{Oscillatory_auxiliary_van_der_corput}, $(i)$) in the third factor
with respect to the $y_2$ variable.

As can easily be seen from \eqref{LowerThanTwo_TheSecondDecomposition_CompletePhase_rewritten} by using integration by parts in $x_1$,
if one of $\lambda_1, \lambda_2$ is considerably larger than any other $\lambda_i, i =1,2,3$,
then we can easily gain a sufficiently strong estimate with which
one can sum absolutely in all three parameters $\lambda_i, i =1,2,3$, the operators $T_j^\lambda$.

If $\lambda_3$ is significantly larger than both $\lambda_1$ and $\lambda_2$ and $\phi$ is of type $A$,
we can also use integration by parts in $x_1$ in order to get a sufficiently strong estimate.
In the case when $\lambda_3$ is the largest and $\phi$ is of type $D$, then $b^{\#}(x,\delta,j)$ is approximately $x_1$ in the $C^\infty$ sense,
and so in this case and when $|x_2| \sim 1$, we use integration by parts in $x_2$, and when $|x_2| \ll 1$ integration by parts in $x_1$.
In both parts we get the bound $\lambda_3^{-N}$ with which we can obtain a summable estimate for $T_j^\lambda$ in all three parameters.

As it turns out, in almost all the other possible relations between $\lambda_i$, $i=1,2,3$, we shall need complex interpolation if $\theta = 1/3$,
or if $\theta = 1/4$ and it is the ``diagonal'' case, i.e., all the $\lambda_i$, $i=1,2,3$, are of approximately the same size.
If $\theta = 1/3$ and $\lambda_i$, $i=1,2,3$, are of approximately the same size we shall actually need a finer analysis
where estimates on Airy integrals are needed. This will be done in the next section.

{\bf{Case 1.1. $\lambda_1 \sim \lambda_3$, $\lambda_2 \ll \lambda_1$, and $\lambda_2 \leq 2^j \lambda_1^{1/2}$.}}
On the part where $|x_2| \sim 1$ we can use integration by parts in $x_2$ and
obtain much stroger estimates sufficient for absolute summation.
When $|x_2| \ll 1$ we use stationary phase in both variables, and so
\begin{align*}
\Vert \widehat{\nu_j^\lambda} \Vert_{L^\infty} \lesssim \lambda_1^{-1},
\qquad \qquad
\Vert \nu_j^\lambda \Vert_{L^\infty} \lesssim \lambda_1^{1/2} \lambda_2,
\end{align*}
from which one can calculate that
\begin{align*}
\Vert T^\lambda_j \Vert_{ L^{\mathfrak{p}_3}_{x_3} (L^{\mathfrak{p}_1}_{(x_1,x_2)}) \to L^{\mathfrak{p}_3'}_{x_3} (L^{\mathfrak{p}_1'}_{(x_1,x_2)}) }
   &\lesssim \lambda_1^{(\theta-1)/2} \lambda_2^\theta.
\end{align*}
Let us denote by $T_{\delta,j}^I$ the sum of the operator pieces $T_j^\lambda$ in this case.
We need to separate the sum in $\lambda_1$ into two subcases $\lambda_1 \leq 2^{2j}$ and $\lambda_1 > 2^{2j}$:
\begin{align*}
\Vert T_{\delta,j}^I \Vert_{ L^{\mathfrak{p}_3}_{x_3} (L^{\mathfrak{p}_1}_{(x_1,x_2)}) \to L^{\mathfrak{p}_3'}_{x_3} (L^{\mathfrak{p}_1'}_{(x_1,x_2)}) }
   &\lesssim
      \sum_{\lambda_1=1}^{2^{2j}} \sum_{\lambda_2=1}^{\lambda_1} \lambda_1^{(\theta-1)/2} \lambda_2^\theta +
      \sum_{\lambda_1=2^{2j+1}}^{\infty} \sum_{\lambda_2=1}^{2^j\lambda_1^{1/2}} \lambda_1^{(\theta-1)/2} \lambda_2^\theta \\
   &\lesssim
      \sum_{\lambda_1=1}^{2^{2j}} \lambda_1^{(3\theta-1)/2} +
      \sum_{\lambda_1=2^{2j+1}}^{\infty} 2^{j\theta} \lambda_1^{(2\theta-1)/2} \\
   &\lesssim
      \sum_{\lambda_1=1}^{2^{2j}} \lambda_1^{(3\theta-1)/2} + 2^{j(3\theta-1)}.
\end{align*}
Therefore if $\theta < 1/3$, then we obtain the desired result, and if $\theta = 1/3$,
we need to use complex interpolation for the first sum where $\lambda_1 \leq 2^{2j}$.
For $\theta = 1/3$, we have
\begin{align*}
\Vert T^\lambda_j \Vert_{ L^{\mathfrak{p}_3}_{x_3} (L^{\mathfrak{p}_1}_{(x_1,x_2)}) \to L^{\mathfrak{p}_3'}_{x_3} (L^{\mathfrak{p}_1'}_{(x_1,x_2)}) }
   &\lesssim (\lambda_1 \lambda_2^{-1})^{-1/3},
\end{align*}
and one is easily convinced that we may restrict ourselves to the case 
\begin{align*}
1 \ll \lambda_1 \ll 2^{2j}, \qquad \qquad 1 \ll \lambda_2 \ll \lambda_1.
\end{align*}
The bound on the operator norm motivates us to define $k$ through $2^k \coloneqq \lambda_1 \lambda_2^{-1} = 2^{k_1-k_2}$,
where $2^{k_1} = \lambda_1$ and $2^{k_2} = \lambda_2$.
Our goal is to prove that for each $k$ within the range $1 \ll 2^k \ll 2^{2j}$ we have
\begin{align*}
\Bigg\Vert \sum_{\lambda_1 \lambda_2^{-1} = 2^k}
   T^\lambda_j \Bigg\Vert_{ L^{\mathfrak{p}_3}_{x_3} (L^{\mathfrak{p}_1}_{(x_1,x_2)}) \to L^{\mathfrak{p}_3'}_{x_3} (L^{\mathfrak{p}_1'}_{(x_1,x_2)}) }
   &\lesssim 2^{-k/3},
\end{align*}
since then we obtain the desired estimate by summation in $k$.

We shall slightly simplify the proof by assuming that $\lambda_1 = \lambda_3$.
Let us consider the following function parametrised by the complex number $\zeta$ and the integer $k$:
\begin{align*}
\mu_\zeta^{k} = 2^{k\frac{3\zeta-1}{2}} \gamma(\zeta) \, \sum_{\lambda_1 \lambda_2^{-1} = 2^k}
    (\lambda_1)^{\frac{3-9\zeta}{4}} \, \nu_j^\lambda,
\end{align*}
where
\begin{align*}
\gamma(\zeta) = \frac{2^{-9(\zeta-1)/4} - 1}{2^{\frac{3}{2}}-1}.
\end{align*}
The associated convolution operator (convolution again the Fourier transform of $\mu_\zeta^{k}$)
we denote by $T_\zeta^{k}$. For $\zeta = 1/3$ we see that
\begin{align*}
\mu_\zeta^{k} = \sum_{\lambda_1 \lambda_2^{-1} = 2^k} \nu_j^\lambda.
\end{align*}
Therefore, it is sufficient to prove
\begin{align*}
\begin{split}
\Vert T_{it}^{k} \Vert_{ L^{2/(2-\tilde{\sigma})}_{x_3} (L^1_{(x_1,x_2)}) \to L^{2/\tilde{\sigma}}_{x_3} (L^\infty_{(x_1,x_2)}) }
   &\lesssim 2^{-k/2},\\
\Vert T_{1+it}^{k} \Vert_{L^2\to L^2}
   &\lesssim 1,
\end{split}
\end{align*}
with constants uniform in $t \in \R$.
Recall that $\tilde{\sigma} = 1/4$ since $m = 2$ and $\theta =  1/3$.

The first estimate follows right away.
Namely, since $\widehat{\nu_j^\lambda}$ have supports located at $\lambda$,
then by the estimate for the $L^\infty$ norm of the function $\widehat{\nu_j^\lambda}$ we have
$$ | \widehat{\mu^{k}_{it}}(\xi)| \lesssim \frac{2^{-k/2}}{(1+|\xi_3|)^{1/4}}, $$
and now one needs to recall Lemma \ref{Mixed_norm_auxiliary_lemma}.

We prove the second estimate by using Lemma \ref{Oscillatory_auxiliary_one_parameter}. We need to prove
\begin{align}
\label{LowerThanTwo_CASE_1_1_ComplexInterpolationFinal}
\Bigg\Vert \sum_{\lambda_1 \lambda_2^{-1} = 2^k} 2^k (\lambda_1)^{-\frac{3}{2}-\frac{9}{4}it} \, \nu_j^\lambda \Bigg\Vert_{L^\infty}
   =
\Bigg\Vert \sum_{2^k \ll \lambda_1 \ll 2^{2j}}
\lambda_1^{-1/2} \lambda_2^{-1} {\lambda_1}^{-\frac{9}{4}it} \, \nu_j^{(\lambda_1,\lambda_1 2^{-k},\lambda_1)} \Bigg\Vert_{L^\infty}
   \lesssim \frac{1}{\Big| 2^{-\frac{9}{4}it} - 1 \Big|},
\end{align}
uniformly in $t$.

After substituting $\lambda_1 y_1 \mapsto y_1$ and $\lambda_1^{1/2} y_2 \mapsto y_2$
in the expression \eqref{LowerThanTwo_TheSecondDecomposition_SpaceForm_rewritten},
we get that the sum on the left hand side of \eqref{LowerThanTwo_CASE_1_1_ComplexInterpolationFinal} is
\begin{align*}
\sum_{2^k \ll \lambda_1 \ll 2^{2j}} \lambda_1^{-\frac{9}{4}it} &\int
    \widecheck{\chi}_1 (\lambda_1 x_1-y_1) \,
    \widecheck{\chi}_1 (2^{-k}\lambda_1 x_2- 2^{-j-k} \lambda_1^{1/2} y_2- 2^{-k}\lambda_1 \psi_\omega(\lambda_1^{-1}y_1))\\
   &\times \widecheck{\chi}_1(\lambda_1 x_3 -  b^{\#}(\lambda_1^{-1}y_1,\lambda_1^{-1/2}y_2,\delta,j) y_2^2 - \lambda_1 \psi_\beta(\lambda_1^{-1}y_1)) \\
   &\times a(\lambda_1^{-1}y_1,\lambda_1^{-1/2}y_2, \delta, j) \, \chi_1(\lambda_1^{-1}y_1) \, \chi_0(\lambda_1^{-1/2}y_2) \, \my.
\end{align*}
Using the first three factors we can reduce the problem to the case $|x| \leq C$ for some large constant $C$.
Now, as we have done in previous instances of complex interpolation, we use the substitution $\lambda_1 x_1 - y_1 \mapsto y_1$,
conclude that it is sufficient to consider the part of the integration domain where $|y_1| \leq \lambda_1^\varepsilon$.
In particular then $x_1 \sim 1$ and we can use Taylor approximation for $\psi_\omega$ and $\psi_\beta$ at $x_1$. Then one gets
\begin{align*}
\sum_{2^k \ll \lambda_1 \leq 2^{2j}} \lambda_1^{-\frac{9}{4}it} &\int
    \widecheck{\chi}_1 (y_1) \,
    \widecheck{\chi}_1 (2^{-k} \lambda_1 Q_\omega(x_1,x_2,\delta_1) - 2^{-k} y_1 r_\omega(\lambda_1^{-1}y_1,x_1,\delta_1)-2^{-j-k}\lambda_1^{1/2} y_2)\\
   &\times \widecheck{\chi}_1(\lambda_1 Q_\beta(x_1,x_3,\delta_1) - y_1 r_\beta(\lambda_1^{-1}y_1,x_1,\delta_1) -  b^{\#}(x_1-\lambda_1^{-1}y_1,\lambda_1^{-1/2}y_2,\delta,j) y_2^2) \\
   &\times a(x_1-\lambda_1^{-1}y_1,\lambda_1^{-1/2}y_2, \delta, j) \, \chi_1(x_1-\lambda_1^{-1}y_1) \, \chi_0(\lambda_1^{-1/2}y_2) \, \chi_0(\lambda_1^{-\varepsilon} y_1) \, \my,
\end{align*}
where $|\partial_1^N r_\omega| \sim 1$ and $|\partial_1^N r_\beta| \sim 1$ for any $N \geq 0$.
Also note that $2^{-j}\lambda_1^{1/2} \ll 1$.

We may now conclude that it is sufficient to consider the cases when either $|A| \gg 1$ or $|B| \gg 1$, where
\begin{align*}
A \coloneqq 2^{-k} \lambda_1 Q_\omega(x_1,x_2,\delta_1), \qquad \qquad
B \coloneqq \lambda_1 Q_\beta(x_1,x_3,\delta_1),
\end{align*}
since otherwise, when both $|A|$ and $|B|$ are bounded, we could apply Lemma \ref{Oscillatory_auxiliary_one_parameter},
similarily as in the case $2^{2j} \delta_3 \gg 1$, to the function
\begin{align*}
H(z_1, z_2, z_3, z_4, z_5; x, \delta, \sigma) \coloneqq &\int
    \widecheck{\chi}_1 (y_1) \,
    \widecheck{\chi}_1 (z_1 - 2^{-k} y_1 r_\omega(z_4^{1/\varepsilon} y_1,x_1,\delta_1)-2^{-k}z_3 y_2)\\
   &\times \widecheck{\chi}_1(z_2 - y_1 r_\beta(z_4^{1/\varepsilon} y_1,x_1,\delta_1) -  b^{\#}(x_1-z_4^{1/\varepsilon}y_1,z_4^{1/(2\varepsilon)}y_2,\delta,j) y_2^2) \\
   &\times a(x_1-z_4^{1/\varepsilon}y_1,z_4^{1/(2\varepsilon)}y_2, \delta, j) \, \chi_1(x_1-z_4^{1/\varepsilon} y_1) \, \chi_0(z_4^{1/(2\varepsilon)} y_2) \, \chi_0(z_4 y_1) \, \my,
\end{align*}
where we would plug in
\begin{align*}
(z_1,z_2,z_3,z_4,z_5) = (2^{-k} \lambda_1 Q_\omega(x_1,x_2,\delta_1),\lambda_1 Q_\beta(x_1,x_3,\delta_1),2^{-j}\lambda_1^{1/2},\lambda_1^{-\varepsilon}, 2^k \lambda_1^{-1}).
\end{align*}
Note that the upper bounds on $z_4$ and $z_5$ are given by the summation bounds for the parameter $\lambda_1$,
and that the function $H$ does not depend on $z_5$.
Furthermore, the $C^1$ norm of $H$ in $(z_1,z_2,z_3,z_4,z_5)$ is bounded since derivatives of Schwartz functions are Schwartz
and only factors of polynomial growth in $y_1$ and $y_2$ appear when taking the derivatives.
The polynomial growth in $y_1$ can be dealt with by using the first factor.
For the polynomial growth in $y_2$ one has to consider the cases $|y_2| \lesssim |y_1|^N$ and $|y_2| \gg |y_1|^N$ separately.
In the first case we can obviously again use the first factor, and in the second case we use the third factor
inside which the term $b^{\#} y_2^2$ is now dominant.

Let us now first assume $|B| \gg 1$.
The first three factors within the integral are behaving essentially like
$$
\widecheck{\chi}_1 (y_1) \widecheck{\chi}_1 (A - 2^{-k}y_1 - 2^{-j-k}\lambda_1^{1/2} y_2) \widecheck{\chi}_1 (B - y_1 - y_2^2).
$$
We may reduce ourselves to the discussion of the part of the integration domain where $|y_1| \ll |B|^{\varepsilon_B}$
since otherwise, when $|y_1| \gtrsim |B|^{\varepsilon_B}$, we could use the first factor, obtain the estimate $|B|^{-N\varepsilon_B}$ for the integral,
and then sum this geometric series in $\lambda_1$.
Then $|B-y_1 r_\beta| \sim |B|$, and the integral we need to estimate is bounded by
$$
\int |\widecheck{\chi}_1(B-y_1 r_\beta - y_2^2 b^{\#})| \my_2 \lesssim \int |\widecheck{\chi}_1(B-y_1 r_\beta -t)| \, |t|^{-1/2} \mathrm{d}t \leq C |B|^{-1/2},
$$
for some constant $C$. Now one can again sum in $\lambda_1$.

Let us now assume $|B| \leq C_B$ for some large, but fixed constant $C_B$, and let $|A| \gg C_B$.
Again, we can reduce ourselves to the part where $|y_1| \ll |A|^{\varepsilon_A}$, and so $|A-2^{-k}y_1 r_\omega| \sim |A|$. Therefore if $|y_2| \leq |A|^{1/2}$,
then using the second factor we get that the integral is bounded (up to a constant) by $|A|^{-N}$. If $|y_2| > |A|^{1/2}$, then
$|B-y_1 r_\beta - y_2^2 b^{\#}| \gtrsim |A|$ and so we can use the third factor, and sum in $\lambda_1$.

{\bf{Case 1.2. $\lambda_1 \sim \lambda_3$, $\lambda_2 \ll \lambda_1$, and $\lambda_2 > 2^j \lambda_1^{1/2}$.}}
In this case we have the same bound for the Fourier transform. Hence
\begin{align*}
\Vert \widehat{\nu_j^\lambda} \Vert_{L^\infty} \lesssim \lambda_1^{-1},
\qquad \qquad
\Vert \nu_j^\lambda \Vert_{L^\infty} \lesssim 2^j \lambda_1,
\end{align*}
from which one can calculate that
\begin{align*}
\Vert T^\lambda_j \Vert_{ L^{\mathfrak{p}_3}_{x_3} (L^{\mathfrak{p}_1}_{(x_1,x_2)}) \to L^{\mathfrak{p}_3'}_{x_3} (L^{\mathfrak{p}_1'}_{(x_1,x_2)}) }
   &\lesssim \lambda_1^{\theta-1/2} 2^{j \theta}.
\end{align*}
If we denote by $T_{\delta,j}^{II}$ the sum of the operator pieces in this case, then we have:
\begin{align*}
\Vert T_{\delta,j}^{II} \Vert_{ L^{\mathfrak{p}_3}_{x_3} (L^{\mathfrak{p}_1}_{(x_1,x_2)}) \to L^{\mathfrak{p}_3'}_{x_3} (L^{\mathfrak{p}_1'}_{(x_1,x_2)}) }
   &\lesssim
      2^{j\theta} \sum_{\lambda_1=2^{2j}}^{\infty} \sum_{\lambda_2=2^j\lambda_1^{1/2}}^{\lambda_1} \lambda_1^{\theta-1/2}\\
   &\lesssim
      2^{j\theta} \sum_{\lambda_1=2^{2j}}^{\infty} (\log_2 \lambda_1 - 2j) \lambda_1^{\theta-1/2}\\
   &\lesssim
      2^{j(3\theta-1)} \lesssim 1.
\end{align*}

{\bf{Case 2.1. $\lambda_2 \sim \lambda_3$, $\lambda_1 \ll \lambda_2$, and $\lambda_2 \leq 2^{2j}$.}}
Here again we may use stationary phase in both variables (and when $|x_2| \sim 1$ even integration by parts in $x_2$).
The estimates are
\begin{align*}
\Vert \widehat{\nu_j^\lambda} \Vert_{L^\infty} \lesssim \lambda_2^{-1},
\qquad \qquad
\Vert \nu_j^\lambda \Vert_{L^\infty} \lesssim \lambda_2^{3/2},
\end{align*}
and therefore independent of $\lambda_1$. As in \cite{IM16} we define
\begin{align*}
\sigma_j^{\lambda_2,\lambda_3} = \sum_{\lambda_1 \ll \lambda_2} \nu_j^\lambda,
\end{align*}
and note that then we can write
\begin{align*}
\widehat{\sigma_j^{\lambda_2,\lambda_3}} =
  \chi_0\Big(\frac{\xi_1}{\lambda_2}\Big) \chi_1 \Big(\frac{\xi_2}{\lambda_2}\Big) \chi_1 \Big(\frac{\xi_3}{\lambda_3}\Big) \widehat{\tilde{\nu}}_{\delta,j},
\end{align*}
where $\chi_0$ is a smooth cutoff function supported in a sufficiently small neighbourhood of $0$.
Therefore, one easily sees that using the same argumentation as for $\nu_j^\lambda$ we have
\begin{align*}
\Vert \widehat{\sigma_j^{\lambda_2,\lambda_3}} \Vert_{L^\infty} \lesssim \lambda_2^{-1},
\qquad \qquad
\Vert \sigma_j^{\lambda_2,\lambda_3} \Vert_{L^\infty} \lesssim \lambda_2^{3/2}.
\end{align*}
The operator norm bound is
\begin{align*}
\Vert T^{\lambda_2,\lambda_3}_j \Vert_{ L^{\mathfrak{p}_3}_{x_3} (L^{\mathfrak{p}_1}_{(x_1,x_2)}) \to L^{\mathfrak{p}_3'}_{x_3} (L^{\mathfrak{p}_1'}_{(x_1,x_2)}) }
   &\lesssim \lambda_2^{(3\theta-1)/2}.
\end{align*}
Hence, if $\theta < 1/3$, then we obtain the desired result by summing the geometric series,
and if $\theta = 1/3$, we need to use complex interpolation.

As usual, we consider only the case $\lambda_2 = \lambda_3$.
Also note that we may reduce ourselves to the summation over $\lambda_2 \ll 2^{2j}$ instead of $\lambda_2 \leq 2^{2j}$.
We define the following function parametrised by the complex number $\zeta$:
\begin{align*}
\mu_\zeta = \gamma(\zeta) \, \sum_{1 \ll \lambda_2 \ll 2^{2j}}
    (\lambda_2)^{\frac{3-9\zeta}{4}} \, \sigma_j^{\lambda_2,\lambda_2},
\end{align*}
where
\begin{align*}
\gamma(\zeta) = \frac{2^{-9(\zeta-1)/4} - 1}{2^{\frac{3}{2}}-1}.
\end{align*}
The associated convolution operator we denote by $T_\zeta$. For $\zeta = 1/3$ we see that
\begin{align*}
\mu_\zeta = \sum_{1 \ll \lambda_2 \ll 2^{2j}} \sigma_j^{\lambda_2,\lambda_2},
\end{align*}
and so it is sufficient to prove
\begin{align*}
\begin{split}
\Vert T_{it} \Vert_{ L^{2/(2-\tilde{\sigma})}_{x_3} (L^1_{(x_1,x_2)}) \to L^{2/\tilde{\sigma}}_{x_3} (L^\infty_{(x_1,x_2)}) }
   &\lesssim 1,\\
\Vert T_{1+it} \Vert_{L^2\to L^2}
   &\lesssim 1,
\end{split}
\end{align*}
with constants uniform in $t \in \R$.

The first estimate follows right away since $\sigma_j^{\lambda_2,\lambda_2}$ have $\xi_3$-supports located around $\lambda_2$,
which implies by the $L^\infty$ estimate for the Fourier transform of $\nu_j^\lambda$ that
$$ |\widehat{\mu_{it}}(\xi)| \lesssim \frac{1}{(1+|\xi_3|)^{1/4}}. $$
Now we can apply Lemma \ref{Mixed_norm_auxiliary_lemma}.

We prove the second estimate by using the oscillatory sum lemma (Lemma \ref{Oscillatory_auxiliary_one_parameter}).
We need to prove
\begin{align}
\label{LowerThanTwo_CASE_2_1__ComplexInterpolationFinal}
\Bigg\Vert \sum_{1 \ll \lambda_2 \ll 2^{2j}} (\lambda_1)^{-\frac{3}{2}-\frac{9}{4}it} \, \sigma_j^{\lambda_2,\lambda_2} \Bigg\Vert_{L^\infty}
   \lesssim \frac{1}{\Big| 2^{-\frac{9}{4}it} - 1 \Big|},
\end{align}
uniformly in $t$.

First note that since we obtain the function $\sigma_j^{\lambda_2,\lambda_2}$ by summation in $\lambda_1$,
the expression \eqref{LowerThanTwo_TheSecondDecomposition_SpaceForm_rewritten} has to be replaced by
\begin{align}
\begin{split}
\label{LowerThanTwo_CASE_2_1_SpaceForm_rewritten}
\sigma_j^{\lambda_2,\lambda_2} =
   \lambda_2^3 &\int \widecheck{\chi}_0 (\lambda_2 (x_1-y_1)) \,
   \widecheck{\chi}_1 (\lambda_2 (x_2-2^{-j}y_2- \psi_\omega(y_1)))\\
   &\times \widecheck{\chi}_1(\lambda_2 (x_3 -  b^{\#}(y,\delta,j) y_2^2 - \psi_\beta(y_1))) \\
   &\times a(y, \delta, j) \, \chi_1(y_1) \, \chi_0(y_2) \, \my.
\end{split}
\end{align}
Recall that the function $\chi_0$ of the first factor within the integral has support contained in $[-\epsilon,\epsilon]$
where the small constant $\epsilon$ depends on the implicit constant in the relation $\lambda_1 \ll \lambda_2$.

After substituting $\lambda_2 y_1 \mapsto y_1$ and $\lambda_2^{1/2} y_2 \mapsto y_2$ in the expression \eqref{LowerThanTwo_CASE_2_1_SpaceForm_rewritten},
we get that the sum on the left hand side of \eqref{LowerThanTwo_CASE_2_1__ComplexInterpolationFinal} is
\begin{align*}
\sum_{1 \ll \lambda_2 \ll 2^{2j}} \lambda_2^{-\frac{9}{4}it} &\int
    \widecheck{\chi}_1 (\lambda_2 x_1-y_1) \,
    \widecheck{\chi}_1 (\lambda_2 x_2-2^{-j}\lambda_2^{1/2} y_2- \lambda_2 \psi_\omega(\lambda_2^{-1}y_1))\\
   &\times \widecheck{\chi}_1(\lambda_2 x_3 -  b^{\#}(\lambda_2^{-1}y_1,\lambda_2^{-1/2}y_2,\delta,j) y_2^2 - \lambda_2 \psi_\beta(\lambda_2^{-1}y_1)) \\
   &\times a(\lambda_2^{-1}y_1,\lambda_2^{-1/2}y_2, \delta, j) \, \chi_1(\lambda_2^{-1}y_1) \, \chi_0(\lambda_2^{-1/2}y_2) \, \my.
\end{align*}
Since otherwise we could use the first three factors within the integral to gain a factor of $\lambda_2^{-N}$,
we may assume that $|x| \leq C$ for some large constant $C$.

Now again we use the substitution $\lambda_2 x_1 - y_1 \mapsto y_1$, conclude that it is sufficient to consider
the part of the integration domain where $|y_1| \leq \lambda_2^\varepsilon$,
which implies $x_1 \sim 1$, and so we may use Taylor approximation for $\psi_\omega$ and $\psi_\beta$ at $x_1$.
Then one gets
\begin{align*}
\sum_{1 \ll \lambda_2 \ll 2^{2j}} \lambda_2^{-\frac{9}{4}it} &\int
    \widecheck{\chi}_1 (y_1) \,
    \widecheck{\chi}_1 (\lambda_2 Q_\omega(x_1,x_2,\delta_1) - y_1 r_\omega(\lambda_2^{-1}y_1,x_1,\delta_1)-2^{-j}\lambda_2^{1/2} y_2)\\
   &\times \widecheck{\chi}_1(\lambda_2 Q_\beta(x_1,x_3,\delta_1) - y_1 r_\beta(\lambda_2^{-1}y_1,x_1,\delta_1) -  b^{\#}(x_1-\lambda_2^{-1}y_1,\lambda_2^{-1/2}y_2,\delta,j) y_2^2) \\
   &\times a(x_1-\lambda_2^{-1}y_1,\lambda_2^{-1/2}y_2, \delta, j) \, \chi_1(x_1-\lambda_2^{-1}y_1) \, \chi_0(\lambda_2^{-1/2}y_2) \, \chi_0(\lambda_2^{-\varepsilon} y_1) \, \my,
\end{align*}
where $|\partial_1^N r_\omega| \sim 1$ and $|\partial_1^N r_\beta| \sim 1$ for any $N \geq 0$.
Note that $2^{-j}\lambda_2^{1/2} \ll 1$.

Now we may restrict ourselves to cases when either $|A| \gg 1$ or $|B| \gg 1$, where
\begin{align*}
A \coloneqq \lambda_2 Q_\omega(x_1,x_2,\delta_1), \qquad \qquad B \coloneqq \lambda_2 Q_\beta(x_1,x_3,\delta_1),
\end{align*}
since otherwise we could apply the oscillatory sum lemma similarily as in Case 1.1.

The first three factors within the integral are behaving essentially like
$$
\widecheck{\chi}_1 (y_1) \widecheck{\chi}_1 (A - y_1 - 2^{-j}\lambda_2^{1/2} y_2) \widecheck{\chi}_1 (B - y_1 - y_2^2).
$$
Let us first consider $|B| \gg 1$, as in Case 1.1.
As usual, we may restrict ourselves to the part of the integration domain where $|y_1| \ll |B|^{\varepsilon_B}$.
Therefore there we have $|B-y_1| \sim |B|$, and the integral is bounded by
$$
\int |\widecheck{\chi}_1(B-y_1 r_\beta - y_2^2 b^{\#})| \my_2 \lesssim \int |\widecheck{\chi}_1(B-y_1 r_\beta   -t)| \, |t|^{-1/2} \mathrm{d}t \leq C |B|^{-1/2},
$$
for some constant $C$. Now one can sum in $\lambda_2$.

Let us now assume $|B| \leq C_B$ for some large, but fixed constant, and $|A| \gg C_B$.
Again, we may consider only the part of the integration domain where $|y_1| \ll |A|^{\varepsilon_A}$, and so here we have $|A-y_1| \sim |A|$.
Therefore, if $|y_2| \leq |A|^{1/2}$, then using the second factor
we get that the integral is bounded (up to a constant) by $|A|^{-N}$.
If $|y_2| > |A|^{1/2}$, then $|B-y_1 r_\beta - y_2^2 b^{\#}| \gtrsim |A|$ and so
we can use the third factor to gain $|A|^{-N}$, and sum in $\lambda_2$.

{\bf{Case 2.2. $\lambda_2 \sim \lambda_3$, $\lambda_1 \ll \lambda_2$, and $\lambda_2 > 2^{2j}$.}}
As in the previous case we use
\begin{align*}
\sigma_j^{\lambda_2,\lambda_3} = \sum_{\lambda_1 \ll \lambda_2} \nu_j^\lambda,
\end{align*}
and note that in this case the bounds are
\begin{align*}
\Vert \widehat{\sigma_j^{\lambda_2,\lambda_3}} \Vert_{L^\infty} \lesssim \lambda_2^{-1},
\qquad \qquad
\Vert \sigma_j^{\lambda_2,\lambda_3} \Vert_{L^\infty} \lesssim 2^j \lambda_2.
\end{align*}
The operator norm bound is
\begin{align*}
\Vert T^{\lambda_2,\lambda_3}_j \Vert_{ L^{\mathfrak{p}_3}_{x_3} (L^{\mathfrak{p}_1}_{(x_1,x_2)}) \to L^{\mathfrak{p}_3'}_{x_3} (L^{\mathfrak{p}_1'}_{(x_1,x_2)}) }
   &\lesssim 2^{j\theta} \lambda_2^{\theta-1/2}.
\end{align*}
This is summable over $\lambda_2 > 2^{2j}$ for all $\theta \leq 1/3$.

{\bf{Case 3.1. $\lambda_1 \sim \lambda_2$, $\lambda_3 \ll \lambda_1$, and $\lambda_3^{1/2} \gtrsim 2^{-j} \lambda_1$.}}
In this case, by stationary phase in both variables, the estimates are
\begin{align}
\label{LowerThanTwo_CASE_3_1_basic_estiamtes}
\begin{split}
\Vert \widehat{\nu_j^\lambda} \Vert_{L^\infty} \lesssim \lambda_1^{-1/2} \lambda_3^{-1/2},
\qquad \qquad
\Vert \nu_j^\lambda \Vert_{L^\infty} \lesssim \lambda_1 \lambda_3^{1/2},
\end{split}
\end{align}
from which one can calculate that
\begin{align*}
\Vert T^\lambda_j \Vert_{ L^{\mathfrak{p}_3}_{x_3} (L^{\mathfrak{p}_1}_{(x_1,x_2)}) \to L^{\mathfrak{p}_3'}_{x_3} (L^{\mathfrak{p}_1'}_{(x_1,x_2)}) }
   &\lesssim \lambda_1^{(3\theta-1)/2}.
\end{align*}
The sum of the operator pieces in this case we denote by $T_{\delta,j}^{V}$. Then
\begin{align*}
\Vert T_{\delta,j}^{V} \Vert_{ L^{\mathfrak{p}_3}_{x_3} (L^{\mathfrak{p}_1}_{(x_1,x_2)}) \to L^{\mathfrak{p}_3'}_{x_3} (L^{\mathfrak{p}_1'}_{(x_1,x_2)}) }
   &\lesssim
      \sum_{\lambda_1=1}^{2^{j}} \sum_{\lambda_3=1}^{\lambda_1} \lambda_1^{(3\theta-1)/2}
      + \sum_{\lambda_1=2^j}^{2^{2j}} \sum_{\lambda_3=(2^{-j}\lambda_1)^2}^{\lambda_1} \lambda_1^{(3\theta-1)/2}\\
   &\lesssim
      \sum_{\lambda_1=1}^{2^{2j}} \sum_{\lambda_3=1}^{\lambda_1} \lambda_1^{(3\theta-1)/2}\\
   &\lesssim
      \sum_{\lambda_1=1}^{2^{2j}} \lambda_1^{(3\theta-1)/2} \log_2(\lambda_1).
\end{align*}
This is summable if and only if $\theta < 1/3$. For $\theta = 1/3$ we see
\begin{align*}
\Vert T^\lambda_j \Vert_{ L^{\mathfrak{p}_3}_{x_3} (L^{\mathfrak{p}_1}_{(x_1,x_2)}) \to L^{\mathfrak{p}_3'}_{x_3} (L^{\mathfrak{p}_1'}_{(x_1,x_2)}) }
   &\lesssim 1.
\end{align*}
Therefore, in this case we shall need the oscillatory sum lemma with two parameters (Lemma \ref{Oscillatory_auxiliary_two_parameter}) when applying
complex interpolation.

As usual we assume $\lambda_1 = \lambda_2$.
We consider the following function parametrised by the complex number $\zeta$:
\begin{align*}
\mu_\zeta = \gamma(\zeta) \, \sum_{\lambda_1, \lambda_3}
        \lambda_1^{\frac{1-3\zeta}{2}} \lambda_3^{\frac{1-3\zeta}{4}} \, \nu_j^\lambda,
\end{align*}
where $\gamma(\zeta)$ is to be defined later as appropriate.
The summation is over all $\lambda_1$ and $\lambda_3$ satisfying the conditions of this case (Case 3.1).
Notice that we necessarily have $\lambda_1 \gg 1$.

We denote by $T_\zeta$ the associated convolution operator against the Fourier transform of $\mu_\zeta$.
For $\zeta = 1/3$ we require that
\begin{align*}
\mu_\zeta = \sum_{\lambda_1, \lambda_3} \nu_j^\lambda,
\end{align*}
i.e., $\gamma(1/3) = 1$.
Then by interpolation it suffices to prove
\begin{align*}
\begin{split}
\Vert T_{it} \Vert_{ L^{2/(2-\tilde{\sigma})}_{x_3} (L^1_{(x_1,x_2)}) \to L^{2/\tilde{\sigma}}_{x_3} (L^\infty_{(x_1,x_2)}) }
   &\lesssim 1,\\
\Vert T_{1+it} \Vert_{L^2\to L^2}
   &\lesssim 1,
\end{split}
\end{align*}
with constants uniform in $t \in \R$.

In order to prove the first estimate, we need the decay bound \eqref{Auxiliary_mu_decay}, i.e.,
$$ |\widehat{\mu_{it}}(\xi)| \lesssim \frac{1}{(1+|\xi_3|)^{1/4}}.$$
But this follows automatically by \eqref{LowerThanTwo_CASE_3_1_basic_estiamtes}, the definition of $\mu_\zeta$, and
the fact that each $\widehat{\nu_j^\lambda}$ has its support located at $\lambda$.

It remains to prove the $L^2 \to L^2$ estimate by showing
\begin{align}
\label{LowerThanTwo_CASE_3_1_ComplexInterpolationFinal}
\Bigg\Vert \sum_{\lambda_1, \lambda_3} (\lambda_1)^{-1-\frac{3}{2}it} (\lambda_3)^{-\frac{1}{2}-\frac{3}{4}it} \, \nu_j^\lambda \Bigg\Vert_{L^\infty}
   \lesssim \frac{1}{|\gamma(1+it)|},
\end{align}
uniformly in $t$.

After substituting $\lambda_1 y_1 \mapsto y_1$ and $\lambda_3^{1/2} y_2 \mapsto y_2$ in the expression \eqref{LowerThanTwo_TheSecondDecomposition_SpaceForm_rewritten},
we get that the sum on the left hand side of \eqref{LowerThanTwo_CASE_3_1_ComplexInterpolationFinal} is
\begin{align*}
\sum_{\lambda_1, \lambda_3} (\lambda_1)^{-\frac{3}{2}it} (\lambda_3)^{-\frac{3}{4}it}  &\int
    \widecheck{\chi}_1 (\lambda_1 x_1-y_1) \,
    \widecheck{\chi}_1 (\lambda_1 x_2-2^{-j} \lambda_1 \lambda_3^{-1/2} y_2- \lambda_1 \psi_\omega(\lambda_1^{-1}y_1))\\
   &\times \widecheck{\chi}_1(\lambda_3 x_3 -  b^{\#}(\lambda_1^{-1}y_1,\lambda_3^{-1/2}y_2,\delta,j) y_2^2 - \lambda_3 \psi_\beta(\lambda_1^{-1}y_1)) \\
   &\times a(\lambda_1^{-1}y_1,\lambda_3^{-1/2}y_2, \delta, j) \, \chi_1(\lambda_1^{-1}y_1) \, \chi_0(\lambda_3^{-1/2}y_2) \, \my.
\end{align*}
Using the first two factors we can restrict ourselves to the case when $|(x_1,x_2)| \leq C$ for some large constant $C$.

Next, we use the substitution $\lambda_1 x_1 - y_1 \mapsto y_1$, conclude that
it is sufficient to consider integration over $|y_1| \leq \lambda_1^\varepsilon$ and that we have $x_1 \sim 1$.
Then, after using the Taylor approximation for $\psi_\omega$ and $\psi_\beta$ at $x_1$, one gets
\begin{align}
\label{LowerThanTwo_CASE_3_1_ComplexInterpolationFinal_main_integral}
\sum_{\lambda_1, \lambda_3} (\lambda_1)^{-\frac{3}{2}it} (\lambda_3)^{-\frac{3}{4}it} &\int
    \widecheck{\chi}_1 (y_1) \,
    \widecheck{\chi}_1 (\lambda_1 Q_\omega(x_1,x_2,\delta_1) - y_1 r_\omega(\lambda_1^{-1}y_1,x_1,\delta_1)- 2^{-j} \lambda_1 \lambda_3^{-1/2} y_2) \nonumber \\
   &\times \widecheck{\chi}_1(\lambda_3 Q_\beta(x_1,x_3,\delta_1) - \lambda_3 \lambda_1^{-1} y_1 r_\beta(\lambda_1^{-1}y_1,x_1,\delta_1) \\
   &\qquad \qquad - b^{\#}(x_1-\lambda_1^{-1}y_1,\lambda_3^{-1/2}y_2,\delta,j) y_2^2) \nonumber \\
   &\times a(x_1-\lambda_1^{-1}y_1,\lambda_3^{-1/2}y_2, \delta, j) \, \chi_1(x_1-\lambda_1^{-1}y_1) \, \chi_0(\lambda_3^{-1/2}y_2) \,
    \chi_0(\lambda_1^{-\varepsilon} y_1) \, \my, \nonumber 
\end{align}
where $|\partial_1^N r_\omega| \sim 1$ and $|\partial_1^N r_\beta| \sim 1$ for any $N \geq 0$. Recall that $2^{-j}\lambda_1\lambda_3^{-1/2} \lesssim 1$
and $\lambda_3 \lambda_1^{-1} \ll 1$.

If we define
$$
A \coloneqq \lambda_1 Q_\omega(x_1,x_2,\delta_1), \qquad \qquad B \coloneqq \lambda_3 Q_\beta(x_1,x_3,\delta_1),
$$
then we need to see what happens when either $|A| \gg 1$ or $|B| \gg 1$.
Let us assume that $C_B$ is a sufficiently large positive constant.

{\bf{Subcase $|B| > C_B$ and $|A| \lesssim 1$.}}
In this case we shall use the H\"older variant of the one parameter oscillatory sum lemma (Lemma \ref{Oscillatory_auxiliary_one_parameter_hoelder})
for each fixed $\lambda_3$.
We define
\begin{align}
\label{LowerThanTwo_CASE_3_1_ComplexInterpolationFinal_H_tilde}
\tilde{H}(z_1, z_2, z_3, z_4; &\lambda_3, x_1, x_3, \delta, 2^{-j}) \\
\coloneqq
&\int
    \widecheck{\chi}_1 (y_1) \,
    \widecheck{\chi}_1 (z_1 - y_1 r_\omega(z_3^{1/\varepsilon} y_1,x_1,\delta_1)- z_2 y_2) \nonumber \\
   &\times \widecheck{\chi}_1(\lambda_3 Q_\beta(x_1,x_3,\delta_1) - z_4 y_1 r_\beta(z_3^{1/\varepsilon} y_1,x_1,\delta_1)
    - b^{\#}(x_1-z_3^{1/\varepsilon} y_1,\lambda_3^{-1/2} y_2,\delta,j) y_2^2) \nonumber \\
   &\times a(x_1-z_3^{1/\varepsilon} y_1,\lambda_3^{-1/2} y_2, \delta, j) \, \chi_1(x_1-z_3^{1/\varepsilon} y_1) \, \chi_0(\lambda_3^{-1/2} y_2) \, \chi_0(z_3 y_1) \, \my,\nonumber
\end{align}
where we shall plug in
\begin{align*}
z_1 &= \lambda_1 Q_\omega(x_1,x_2,\delta_1), & z_2 &= 2^{-j} \lambda_1 \lambda_3^{-1/2}, \\
z_3 &= \lambda_1^{-\varepsilon}, & z_4 &= \lambda_3 \lambda_1^{-1}.
\end{align*}
Note that the parameters $\lambda_3$ and $x_3$ are not bounded.

Applying Lemma \ref{Oscillatory_auxiliary_one_parameter_hoelder} we get
\begin{align*}
\Bigg\Vert (\lambda_3)^{-\frac{1}{2}-\frac{3}{4}it} \sum_{\lambda_1} (\lambda_1)^{-1-\frac{3}{2}it} \, \nu_j^\lambda \Bigg\Vert_{L^\infty}
   \lesssim \frac{|\tilde{H}(0)| + \sum_{k=1}^4 C_k}{|2^{-\frac{3}{2}it}-1|}
   \lesssim \frac{\Vert\tilde{H}\Vert_{L^\infty} + \sum_{k=1}^4 C_k}{|\gamma(1+it)|},
\end{align*}
if we add an appropriate factor to $\gamma$
(i.e., our $\gamma$ needs to contain a factor equal to the expression \eqref{Oscillatory_auxiliary_one_parameter_gamma}).
It remains to prove that one can estimate $\Vert\tilde{H}\Vert_{L^\infty}$ and the constants $C_k$, $k =1,2,3,4$,
by $|B|^{-\varepsilon_B}$ since then we can sum in $\lambda_3$.

First let us consider the expression for $\tilde{H}(z)$.
The first three factors within the integral are behaving essentially like
$$
\widecheck{\chi}_1 (y_1) \widecheck{\chi}_1 (z_1 - y_1 - z_2 y_2) \widecheck{\chi}_1 (B - z_4 y_1 - y_2^2).
$$
Since we could otherwise use the first factor and estimate by $|B|^{-\varepsilon_B}$,
we may restrict our discussion to the part of the integration domain where $|y_1| \ll |B|^{\varepsilon_B}$.
Then we have $|B-z_4 y_1 r_\beta| \sim |B|$, and therefore
$$
\int |\widecheck{\chi}_1(B-z_4 y_1 r_\beta - y_2^2 b^{\#})| \my_2 \leq 2\int |\widecheck{\chi}_1(B-z_4 y_1 r_\beta -t)| \, |t|^{-1/2} \mathrm{d}t \leq C |B|^{-1/2},
$$
for a constant $C$.
Hence, we have the required bound for $\Vert\tilde{H}\Vert_{L^\infty}$.

Next, we see that taking derivatives in $z_1$ and $z_4$, doesn't change in an essential way the actual form of $\tilde{H}$
since we only obtain polynomial growth in $y_1$ which can be absorbed by $\widehat{\chi}_1(y_1)$,
and since derivatives of Schwartz functions are again Schwartz.
Therefore, we may estimate $C_k$, $k=1,4$, in the same way as we estimated the original integral.

Permuting the order of the variables $z_k$, $k=1,2,3,4$ appropriately, we see from the expressions for $C_k$ in Lemma \ref{Oscillatory_auxiliary_one_parameter_hoelder}
that we may now assume $z_1=z_4=0$. Taking the derivative in $z_3$ we obtain several terms.
We deal with the terms where a $y_1$ factor appears in the same way as we have dealt with in the previous cases.
It remains to deal with the term where $y_2^2$ factor appears, that is
\begin{align*}
&-z_3^{-1+1/\varepsilon} \, (\partial_1 b^{\#})(x_1-z_3^{1/\varepsilon} y_1,\lambda_3^{-1/2} y_2,\delta,j) \\
&\times \int
    \widecheck{\chi}_1 (y_1) \,
    \widecheck{\chi}_1 (z_1 - y_1 r_\omega(z_3^{1/\varepsilon} y_1,x_1,\delta_1)- z_2 y_2)  \\
   &\qquad \times (\widecheck{\chi}_1)'(\lambda_3 Q_\beta(x_1,x_3,\delta_1) - z_4 y_1 r_\beta(z_3^{1/\varepsilon} y_1,x_1,\delta_1)
    - b^{\#}(x_1-z_3^{1/\varepsilon} y_1,\lambda_3^{-1/2} y_2,\delta,j) y_2^2)  \\
   &\qquad \times a(x_1-z_3^{1/\varepsilon} y_1,\lambda_3^{-1/2} y_2, \delta, j) \, \chi_1(x_1-z_3^{1/\varepsilon} y_1) \, \chi_0(\lambda_3^{-1/2} y_2) \, \chi_0(z_3 y_1) \, y_1 y_2^2 \, \my.
\end{align*}
This integral can be estimated by
\begin{align}
\label{LowerThanTwo_CASE_3_1_ComplexInterpolationFinal_integral}
\int |\chi_0(\lambda_3^{-1/2} y_2) \, (\widecheck{\chi}_1)'(B - y_2^2 b^{\#})| \, |z_3|^{-1+1/\varepsilon} \, y_2^2 \my_2.
\end{align}
The key is now to notice that if we fix $\lambda_3$, then $\lambda_1$ goes over the set where $\lambda_1 \gg \lambda_3$.
In particular, since we shall plug in $z_3 = \lambda_1^{-\varepsilon}$, we have $|z_3|^{-1+1/\varepsilon} \lesssim \lambda_3^{-1+\varepsilon}$.
Therefore using the first factor in \eqref{LowerThanTwo_CASE_3_1_ComplexInterpolationFinal_integral}
we obtain the bound for \eqref{LowerThanTwo_CASE_3_1_ComplexInterpolationFinal_integral} to be
\begin{align*}
\int |(\widecheck{\chi}_1)'(B - y_2^2 b^{\#})| \, |y_2|^\varepsilon \my_2,
\end{align*}
for some different $\varepsilon$. Now one subsitutes $t = y_2^2 b^{\#}$ and easily obtains an admissible bound of the form $|B|^{-\varepsilon_B}$.

For the last constant $C_2$ we shall need to consider the H\"older norm.
Here we may assume $z_1=z_3=z_4=0$.
The derivative in $z_2$ can be estimated by the integral
\begin{align*}
\int \Big| \widecheck{\chi}_1 (y_1) \,
     (\widecheck{\chi}_1)' (- y_1 r_\omega(0,x_1,\delta_1)- z_2 y_2)
      \widecheck{\chi}_1(B - b^{\#}(x_1,\lambda_3^{-1/2} y_2,\delta,j) y_2^2) \, y_2 \Big| \mathrm{d}y.
\end{align*}
We shall now consider only the part where $y_2 \geq 0$ and $z_2 \geq 0$, as other cases can be treated in the same way.
Then substituting $t = y_2^2$ one gets that the estimate for $\partial_{z_2} \tilde{H}$ is
\begin{align*}
\iint \Big|  \widecheck{\chi}_1 (y_1) \,
     (\widecheck{\chi}_1)' (- y_1 r_\omega - z_2 t^{1/2})
     \widecheck{\chi}_1(B - t b^{\#}) \Big| \mathrm{d}y_1 \mathrm{d}t.
\end{align*}
From this form it is obvious that we may now restrict ourselves
to the part of the integration domain where $|y_1| \ll |B|^{\varepsilon_B}$
and $|t| \sim |B|$ by using the first and the third factor respectively.
If we denote this integration domain by $U_B$,
then the bound for the $C_2$ constant in Lemma \ref{Oscillatory_auxiliary_one_parameter_hoelder} reduces to estimating
\begin{align*}
&|z_2|^{1-\vartheta} \int_0^1 \iint_{U_B} \Big|  \widecheck{\chi}_1 (y_1) \, (\widecheck{\chi}_1)' (- y_1 r_\omega - s z_2 t^{1/2})
     \widecheck{\chi}_1(B - t b^{\#}) \Big| \mathrm{d}y_1 \mathrm{d}t \mathrm{d}s \\
     =
&|z_2|^{-\vartheta} \int_0^{z_2} \iint_{U_B} \Big|  \widecheck{\chi}_1 (y_1) \, (\widecheck{\chi}_1)' (- y_1 r_\omega - \tilde{s} t^{1/2})
     \widecheck{\chi}_1(B - t b^{\#}) \Big| \mathrm{d}y_1 \mathrm{d}t \mathrm{d}\tilde{s},
\end{align*}
where $\vartheta$ represents the H\"{o}lder exponent.
If $|z_2| \leq |B|^{-1/4}$, then we obviously have the required estimate.
Therefore, let us assume $|z_2| > |B|^{-1/4}$.
Then $|z_2|^{-\vartheta} < |B|^{\vartheta/4}$ and so integration on the domain $|\tilde{s}| \leq |B|^{-1/4}$ is not a problem.
On the other hand, if $|\tilde{s}| > |B|^{-1/4}$, then $|\tilde{s} t^{1/2}| \gtrsim |B|^{1/4}$ by our assumption on the size of $t$.
Thus we may use the Schwartz property of the second factor in the integral and obtain the required estimate.
This finishes the proof of the case where $|B| \gg 1$ and $|A| \lesssim 1$.

{\bf{Subcase $|B| > C_B$ and $|A| \gg 1$.}}
The preceding argumentation for the estimate of $\Vert \tilde{H} \Vert_{L^\infty}$ is also valid in this case since we have not used the second factor,
and so we see that we can always estimate the integral appearing in \eqref{LowerThanTwo_CASE_3_1_ComplexInterpolationFinal_main_integral}
by $|B|^{-1/2}$. It remains to gain a decay in $|A|$.

If we furthermore assume $|A| \leq |B|$, then $|B|^{-1/2} \leq |B|^{-1/4} |A|^{-1/4}$,
and so we can sum in both $\lambda_1$ and $\lambda_3$.
Therefore we may consider $|A| > |B|$ next, and
reduce our problem using the first factor in the integral in \eqref{LowerThanTwo_CASE_3_1_ComplexInterpolationFinal_main_integral}
to the part where $|y_1| \ll |A|^{\varepsilon_A}$.
Then $|z_1-y_1 r_\omega|=|A-y_1 r_\omega| \sim |A|$, and so we can gain an $|A|^{-\varepsilon_A}$ using the second factor in the integral,
unless $|z_2 y_2| \sim |A|$.
But since $|z_2| \lesssim 1$, we see that $|z_2 y_2| \sim |A|$ implies $|y_2| \gtrsim |A|$,
and so we can use finally the third factor where then the $y_2^2$ term is dominant.

{\bf{Subcase $|B| \leq C_B$ and $|A| > C_B^2$.}}
We can reduce ourselves to the integration over $|y_1| \ll |A|^{\varepsilon_A}$, and so $|A-y_1 r_\omega| \sim |A|$.
Therefore, if $|y_2| \leq |A|^{1/2}$, then using the second factor we get that the integral is bounded (up to a constant) by $|A|^{-1}$.
If $|y_2| > |A|^{1/2}$, then $|B - z_4 y_1 r_\beta - y_2^2 b^{\#}| \gtrsim |A|$,
and so we can use the third factor, and sum in both $\lambda_1$ and $\lambda_3$ (since $|B|<|A|$).

{\bf{Subcase $|B| \leq C_B$ and $|A| \leq C_B^2$.}}
Finally, if both $|A| \leq C_B^2$ and $|B| \leq C_B$ are bounded,
we use the two parameter oscillatory sum lemma. We define the function
\begin{align*}
H(z_1, z_2, z_3, z_4, z_5, z_6; x_1, \delta, 2^{-j}) \coloneqq\\
&\int
    \widecheck{\chi}_1 (y_1) \,
    \widecheck{\chi}_1 (z_1 - y_1 r_\omega(z_3^{1/\varepsilon} y_1,x_1,\delta_1)- z_5 y_2)\\
   &\times \widecheck{\chi}_1(z_2 - z_6 y_1 r_\beta(z_3^{1/\varepsilon} y_1,x_1,\delta_1) \\
   &\qquad \qquad - b^{\#}(x_1-z_3^{1/\varepsilon} y_1,z_4 y_2,\delta,j) y_2^2) \\
   &\times a(x_1-z_3^{1/\varepsilon} y_1,z_4 y_2, \delta, j) \, \chi_1(x_1-z_3^{1/\varepsilon} y_1) \, \chi_0(z_4 y_2) \, \chi_0(z_3 y_1) \, \my,
\end{align*}
where we shall plug in
\begin{align*}
z_1 &= \lambda_1 Q_\omega(x_1,x_2,\delta_1), & z_2 &= \lambda_3 Q_\beta(x_1,x_3,\delta_1), \\
z_3 &= \lambda_1^{-\varepsilon}, & z_4 &= \lambda_3^{-1/2}, \\
z_5 &= 2^{-j} \lambda_1 \lambda_3^{-1/2}, & z_6 &= \lambda_3 \lambda_1^{-1}. \\
\end{align*}
The associated exponents are $(\alpha_1,\alpha_2) = (-3/2,-3/4)$ and
\begin{align*}
(\beta_1^1,\beta_2^1) &= (1,0), & (\beta_1^2,\beta_2^2) &= (0,1), \\
(\beta_1^3,\beta_2^3) &= (-\varepsilon,0), & (\beta_1^4,\beta_2^4) &= (0,-1/2), \\
(\beta_1^5,\beta_2^5) &= (1,-1/2), & (\beta_1^6,\beta_2^6) &= (-1,1), \\
\end{align*}
and so for each $k$ the pairs $(\alpha_1,\alpha_2)$ and $(\beta_1^k,\beta_2^k)$ are linearly independent.
The $C^2$ norm of $H$ is uniformly bounded in the bounded paramteres $(x_1, \delta, 2^{-j})$ by arguing in the same manner as in Case 1.1.
Therefore we may apply Lemma \ref{Oscillatory_auxiliary_two_parameter}
if we take $\gamma$ to contain a factor equal to the expression \eqref{Oscillatory_auxiliary_two_parameter_gamma} (with $\theta = 1/3$).
Recall that $\gamma$ needs to contain also the factor from the case where we applied the one parameter lemma
(i.e., where we had $|B| > C_B$ and $|A| \lesssim 1$).

{\bf{Case 3.2. $\lambda_1 \sim \lambda_2$, $\lambda_3 \ll \lambda_1$, and $\lambda_3^{1/2} \ll 2^{-j} \lambda_1$.}}
Here we have the same bound for the Fourier transform as in the previous case.
Therefore
\begin{align*}
\Vert \widehat{\nu_j^\lambda} \Vert_{L^\infty} \lesssim \lambda_1^{-1/2} \lambda_3^{-1/2},
\qquad \qquad
\Vert \nu_j^\lambda \Vert_{L^\infty} \lesssim 2^j \lambda_3,
\end{align*}
from which one can get by interpolation
\begin{align*}
\Vert T^\lambda_j \Vert_{ L^{\mathfrak{p}_3}_{x_3} (L^{\mathfrak{p}_1}_{(x_1,x_2)}) \to L^{\mathfrak{p}_3'}_{x_3} (L^{\mathfrak{p}_1'}_{(x_1,x_2)}) }
   &\lesssim 2^{j\theta} \lambda_1^{(\theta-1)/2} \lambda_3^{\theta/2}.
\end{align*}
We first consider the case $\lambda_1 > 2^{2j}$ and denote its sum of the operator pieces by $T_{\delta,j}^{VI, 1}$. Then
\begin{align*}
\Vert T_{\delta,j}^{VI, 1} \Vert_{ L^{\mathfrak{p}_3}_{x_3} (L^{\mathfrak{p}_1}_{(x_1,x_2)}) \to L^{\mathfrak{p}_3'}_{x_3} (L^{\mathfrak{p}_1'}_{(x_1,x_2)}) }
   &\lesssim
      2^{j\theta} \sum_{\lambda_1=2^{2j}}^{\infty} \sum_{\lambda_3=1}^{\lambda_1} \lambda_1^{(\theta-1)/2} \lambda_3^{\theta/2}\\
   &\lesssim
      2^{j\theta} \sum_{\lambda_1=2^{2j}}^{\infty} \lambda_1^{\theta-1/2} \\
   &\lesssim
      2^{j(3\theta-1)} \lesssim 1.
\end{align*}
The other case is when $2^j \ll \lambda_1 \leq 2^{2j}$ and we denote the sum of these operator pieces by $T_{\delta,j}^{VI, 2}$. Then
\begin{align*}
\Vert T_{\delta,j}^{VI, 2} \Vert_{ L^{\mathfrak{p}_3}_{x_3} (L^{\mathfrak{p}_1}_{(x_1,x_2)}) \to L^{\mathfrak{p}_3'}_{x_3} (L^{\mathfrak{p}_1'}_{(x_1,x_2)}) }
   &\lesssim
      2^{j\theta} \sum_{\lambda_1=2^j}^{2^{2j}} \sum_{\lambda_3=1}^{(2^{-j}\lambda_1)^2} \lambda_1^{(\theta-1)/2} \lambda_3^{\theta/2}\\
   &\lesssim
      \sum_{\lambda_1=2^{j}}^{2^{2j}} \lambda_1^{(\theta-1)/2} \lambda_1^\theta \\
   &\lesssim
      \sum_{\lambda_1=2^{j}}^{2^{2j}} \lambda_1^{(3\theta-1)/2}.
\end{align*}
Again, this is summable if and only if $\theta < 1/3$. For $\theta = 1/3$, we have
\begin{align*}
\Vert T^\lambda_j \Vert_{ L^{\mathfrak{p}_3}_{x_3} (L^{\mathfrak{p}_1}_{(x_1,x_2)}) \to L^{\mathfrak{p}_3'}_{x_3} (L^{\mathfrak{p}_1'}_{(x_1,x_2)}) }
   &\lesssim 2^{j/3} (\lambda_1^2 \lambda_3^{-1})^{-1/6}.
\end{align*}
This operator norm estimate motivates us to define $k$ through $2^k \coloneqq \lambda_1^2 \lambda_3^{-1} = 2^{2k_1-k_3}$,
where $2^{k_1} = \lambda_1$ and $2^{k_3} = \lambda_3$.
Our goal is to prove for each $k$ that
\begin{align*}
\Bigg\Vert \sum_{\lambda_1^2 \lambda_3^{-1} = 2^k}
   T^\lambda_j \Bigg\Vert_{ L^{\mathfrak{p}_3}_{x_3} (L^{\mathfrak{p}_1}_{(x_1,x_2)}) \to L^{\mathfrak{p}_3'}_{x_3} (L^{\mathfrak{p}_1'}_{(x_1,x_2)}) }
   &\lesssim 2^{(2\varepsilon-1) (k-2j)/3}
\end{align*}
for some $0 \leq \varepsilon < 1/2$.
Since $k \geq 2j$, we then obtain the desired result by summation in $k$.

We shall slightly simplify the proof by assuming that $\lambda_1 = \lambda_2$.
Let us consider the following function parametrised by the complex number $\zeta$ and $k$:
\begin{align*}
\mu_\zeta^{k} = \gamma(\zeta) \, \sum_{\lambda_1^2 \lambda_3^{-1} = 2^k}
    (\lambda_3)^{\frac{1-3\zeta}{2}} \, \nu_j^\lambda,
\end{align*}
where
\begin{align*}
\gamma(\zeta) = 2^{-3(\zeta-1)/2} - 1.
\end{align*}
Let $T_\zeta^{k}$ denote the associated convolution operator. For $\zeta = 1/3$ we have
\begin{align*}
\mu_\zeta^{k} = \sum_{\lambda_1^2 \lambda_3^{-1} = 2^k} \nu_j^\lambda,
\end{align*}
and so, by interpolation, we need to prove
\begin{align}
\label{LowerThanTwo_CASE_3_2_ComplexInterpolation_endpoints}
\begin{split}
\Vert T_{it}^{k} \Vert_{ L^{2/(2-\tilde{\sigma})}_{x_3} (L^1_{(x_1,x_2)}) \to L^{2/\tilde{\sigma}}_{x_3} (L^\infty_{(x_1,x_2)}) }
   &\lesssim 2^{-k/4},\\
\Vert T_{1+it}^{k} \Vert_{L^2\to L^2}
   &\lesssim 2^j \, 2^{\varepsilon (k-2j)},
\end{split}
\end{align}
for some $0 \leq \varepsilon < 1/2$, and with constants uniform in $t \in \R$.
The first estimate follows right away since $\widehat{\nu_j^\lambda}$ have supports located at $\lambda$,
and therefore by the $L^\infty$ estimate for the Fourier transform of $\nu_j^\lambda$ we have
$$ | \widehat{\mu^{k}_{it}}(\xi)| \lesssim \frac{2^{-k/4}}{(1+|\xi_3|)^{1/4}}. $$

We prove the second estimate by using the oscillatory sum lemma. We need to prove
\begin{align}
\label{LowerThanTwo_CASE_3_2_ComplexInterpolationFinal}
\Bigg\Vert \sum_{\lambda_1^2 \lambda_3^{-1} = 2^k} \lambda_3^{-1-\frac{3}{2}it} \, \nu_j^\lambda \Bigg\Vert_{L^\infty}
   =
\Bigg\Vert \sum_{\lambda_1^2 \lambda_3^{-1} = 2^k}
2^{k} \lambda_1^{-2-3it} \, \nu_j^{( \lambda_1,  \lambda_1, 2^{-k} \lambda_1^2)} \Bigg\Vert_{L^\infty}
   \lesssim \frac{2^j \, 2^{\varepsilon (k-2j)}}{\Big| 2^{-\frac{3}{2}it} - 1 \Big|},
\end{align}
uniformly in $t$.

Let us discuss first the index ranges for $\lambda_1$, $\lambda_3$, and $2^k = \lambda_1^2 \lambda_3^{-1}$.
Recall that we are in the case where $2^j \ll \lambda_1 \leq 2^{2j}$ and $1 \leq \lambda_3 \ll \lambda_1^2 2^{-2j}$,
which implies $\lambda_3 \ll \lambda_1$ and $2^{2j} \ll 2^k \leq 2^{4j}$.
Let us now fix any $k$ satisfying $2^{2j} \ll 2^k \leq 2^{4j}$,
and let us consider all $(\lambda_1,\lambda_3)$ such that $2^k = \lambda_1^2 \lambda_3^{-1}$.
We shall use the oscillatory sum lemma by summing in $\lambda_1$ and consider $\lambda_3 = \lambda_1^2 2^{-k}$ as a function of $\lambda_1$ and $k$.
The conditions for $\lambda_1$ are then
\begin{align*}
2^j &\ll \lambda_1 \leq 2^{2j},\\
1 &\leq \lambda_1^2 2^{-k} \ll 2^{2j},
\end{align*}
which determine an interval of integers $I_{j,k}$ for $k_1$ (recall $\lambda_1 = 2^{k_1}$).

After substituting $\lambda_1 y_1 \mapsto y_1$ and $2^{-j} \lambda_1 y_2 \mapsto y_2$ in the expression \eqref{LowerThanTwo_TheSecondDecomposition_SpaceForm_rewritten},
we get that the sum on the left hand side of \eqref{LowerThanTwo_CASE_3_2_ComplexInterpolationFinal} is
\begin{align*}
2^j \sum_{k_1 \in I_{j,k}} \lambda_1^{-3it} &\int
    \widecheck{\chi}_1 (\lambda_1 x_1-y_1) \,
    \widecheck{\chi}_1 (\lambda_1 x_2- y_2- \lambda_1 \psi_\omega(\lambda_1^{-1}y_1))\\
   &\times \widecheck{\chi}_1(\lambda_3 x_3 -  2^{2j-k} b^{\#}(\lambda_1^{-1}y_1,2^j \lambda_1^{-1}y_2,\delta,j) y_2^2 - \lambda_3 \psi_\beta(\lambda_1^{-1}y_1)) \\
   &\times a(\lambda_1^{-1}y_1,2^j \lambda_1^{-1}y_2, \delta, j) \, \chi_1(\lambda_1^{-1}y_1) \, \chi_0(2^j \lambda_1^{-1}y_2) \, \my.
\end{align*}
Since using the first two factors we can get a decay in $\lambda_1$, we can restrict ourselves to the case $|(x_1, x_2)| \lesssim 1$.
When $|x_3| \gg 1$, then by using the third factor we can gain a factor
$\lambda_3^{-1} = (\lambda_1^{2} 2^{-k})^{-1}$, which sums up to a number of size $\sim 1$, by definition of $I_{j,k}$.
Therefore we may and shall assume $|x| \lesssim 1$.

Next, we use the substitution $\lambda_1 x_1 - y_1 \mapsto y_1$,
conclude that it is sufficient to consider the part of the integration domain where $|y_1| \leq \lambda_1^\varepsilon$,
and that we may assume $x_1 \sim 1$. If we use Taylor approximation for $\psi_\omega$ and $\psi_\beta$ at $x_1$, then one gets
\begin{align*}
2^j \sum_{k_1 \in I_{j,k}} \lambda_1^{-3it} &\int
    \widecheck{\chi}_1 (y_1) \,
    \widecheck{\chi}_1 (\lambda_1 Q_\omega(x_1,x_2,\delta_1) - y_1 r_\omega(\lambda_1^{-1}y_1,x_1,\delta_1)-y_2)\\
   &\times \widecheck{\chi}_1(\lambda_3 Q_\beta(x_1,x_3,\delta_1) - \lambda_3 \lambda_1^{-1} y_1 r_\beta(\lambda_1^{-1}y_1,x_1,\delta_1) \\
   & \qquad \qquad - 2^{2j-k} b^{\#}(x_1-\lambda_1^{-1}y_1,2^j\lambda_1^{-1}y_2,\delta,j) y_2^2) \\
   &\times a(x_1-\lambda_1^{-1}y_1,2^j\lambda_1^{-1}y_2, \delta, j) \, \chi_1(x_1-\lambda_1^{-1}y_1) \, \chi_0(2^j\lambda_1^{-1}y_2) \, \chi_0(\lambda_1^{-\varepsilon}y_1) \, \my,
\end{align*}
where $|\partial_1^N r_\omega| \sim 1$ and $|\partial_1^N r_\beta| \sim 1$ for any $N \geq 0$.
Note that $2^{2j-k} \ll 1$ and $\lambda_3 \lambda_1^{-1} \ll 1$,
and therefore it is sufficient to consider the cases when either $|A| \gg 1$ or $|B| \gg 1$, where
$$
A \coloneqq \lambda_1 Q_\omega(x_1,x_2,\delta_1), \qquad \qquad B \coloneqq \lambda_3 Q_\beta(x_1,x_3,\delta_1),
$$
since otherwise we may use the oscillatory sum lemma.
We remind that $\lambda_3 = \lambda_1^{2} 2^{-k}$ is considered to be a function of $\lambda_1$.

We concentrate on the first three factors within the integral:
$$
\widecheck{\chi}_1 (y_1) \widecheck{\chi}_1 (A - r_\omega y_1 - y_2)
\widecheck{\chi}_1 (B - \lambda_3 \lambda_1^{-1} r_\beta y_1 - 2^{2j-k} b^{\#} y_2^2),
$$
where $r_\omega$, $r_\beta$, and $b^{\#}$ are all converging in $C^\infty$ to constant functions of magnitude
$\sim 1$ when $\lambda_1 \to \infty$, $\delta \to 0$, and $j \to \infty$.

Let us denote by $M$ a large enough positive number.

{\bf Subcase $|B| > M^3$ and $|A| \leq M$.}
Then because of the first factor we may restrict our discussion to the integration domain where $|y_1| < |B|^{1/3}$.
There $|A-r_\omega y_1| \leq C|B|^{1/3}$ for some $C$.
We may then furthermore assume $|y_2| \leq 2C |B|^{1/3}$, since otherwise we could use the second factor.
Now, if we take $M$ sufficiently large, we have
$$
|\lambda_3 \lambda_1^{-1} r_\beta y_1 - 2^{2j-k} b^{\#} y_2^2| \ll B,
$$
and so we can now use the third factor's Schwartz property to obtain a factor $|B|^{-1}$, which gives summability in $\lambda_1$.

{\bf Subcase $|A| > M$.}
Here we shall need a slightly finer analysis. Note that using the first factor within the integral
we can actually reduce ourselves to the integration within the slightly narrower range $|y_1| < |A|^\varepsilon 2^{10\varepsilon(2j-k)}$ for some small $\varepsilon$
(see \eqref{LowerThanTwo_CASE_3_2_ComplexInterpolation_endpoints}), and therefore we can also assume using the second factor that
\begin{align*}
y_2 \in [A - C |A|^\varepsilon 2^{10\varepsilon(2j-k)}, A + C |A|^\varepsilon 2^{10\varepsilon(2j-k)}],
\end{align*}
for some $C$.

Now if $|A|^\varepsilon 2^{10\varepsilon(2j-k)} \leq 1$, we obtain that the bound on the integral is
$|A|^{2\varepsilon} 2^{20\varepsilon(2j-k)}$ (the area of the surface over which we integrate), and this is summable
in $\lambda_1$ over the set $|A|^\varepsilon 2^{10\varepsilon(2j-k)} \leq 1$.

Therefore, we assume $|A|^\varepsilon 2^{10\varepsilon(2j-k)} > 1$, that is $|A|^{1/10} > 2^{k-2j}$.
Now, if $M$ is sufficiently large, we then have by the restraint on $y_2$ that $|A|^2/2 < y_2^2 < 2 |A|^2$, and hence
$$
C_1 |A|^{2-1/10} < |2^{2j-k} b^{\#} y_2^2| < C_2 |A|^2.
$$
Therefore if either $|B| \ll C_1 |A|^{2-1/10}$ or $|B| \gg C_2 |A|^2$, we can simply use the Schwartz property of the third factor within
the integral. Let us now assume that $B$ is within the range $|B| \in [C_1 |A|^{2-1/10}, C_2 |A|^2]$.
We denote $\delta_A \coloneqq |A|^\varepsilon 2^{10\varepsilon(2j-k)}$ and recall $\delta_A > 1$ and $|y_1| < \delta_A \leq |A|^\varepsilon$.
Using the third factor within the integral we can reduce our problem to when
$$
|B - \lambda_3 \lambda_1^{-1} r_\beta y_1 - 2^{2j-k} b^{\#} y_2^2| \leq \delta_A.
$$
The implicit function theorem implies that
\begin{align*}
y_2^2 &\in [2^{k-2j}|B| - C' 2^{k-2j} \delta_A, 2^{k-2j}|B| + C' 2^{k-2j} \delta_A] \\
\Longleftrightarrow \frac{y_2^2}{2^{k-2j}|B|} &\in \Big[1 - \frac{C'\delta_A}{|B|}, 1 + \frac{C\delta_A}{|B|}\Big],
\end{align*}
for some $C'$. Since $\delta_A \leq |A|^{1/10}$ and $|B| > |A|^{3/2}$, we can conclude
$$
|y_2| \in [(2^{k-2j}|B|)^{1/2} - |A|^{-1/2}, (2^{k-2j}|B|)^{1/2} + |A|^{-1/2}],
$$
that is, $y_2$ goes over a set with length at most $C' |A|^{-1/2}$.
This implies that our integral is bounded by $C' |A|^{-1/2+\varepsilon}$, which is summable in $\lambda_1$.

{\bf{Case 4.1. $\lambda_1 \sim \lambda_2 \sim \lambda_3$ and $\lambda_1 > 2^{2j}$.}}
Here one first applies stationary phase in $x_2$.
Afterwards, as easily seen and explained in a bit more detail at the end of \cite[Chapter 4]{IM16} (and also in the next section of this article),
one gets a phase function in $x_1$ which has a singularity of Airy-type.
By using Lemma \ref{Oscillatory_auxiliary_van_der_corput}, with condition $(ii)$ and $M = 3$,
one gets that the Fourier transform estimate is
\begin{align*}
\Vert \widehat{\nu_j^\lambda} \Vert_{L^\infty} \lesssim \lambda_1^{-1/2} \, \lambda_1^{-1/3} = \lambda_1^{-5/6}.
\end{align*}
From \eqref{LowerThanTwo_TheSecondDecomposition_SpaceBound_last_case} the space-side estimate is
\begin{align*}
\Vert \nu_j^\lambda \Vert_{L^\infty} \lesssim 2^j \lambda_1,
\end{align*}
from which one gets by interpolation
\begin{align*}
\Vert T^\lambda_j \Vert_{ L^{\mathfrak{p}_3}_{x_3} (L^{\mathfrak{p}_1}_{(x_1,x_2)}) \to L^{\mathfrak{p}_3'}_{x_3} (L^{\mathfrak{p}_1'}_{(x_1,x_2)}) }
   &\lesssim 2^{j\theta} \lambda_1^{(5\theta-2)/6}.
\end{align*}
The bound on the operator norm is
\begin{align*}
\Vert \tilde{T}_{\delta,j}^{VII} \Vert_{ L^{\mathfrak{p}_3}_{x_3} (L^{\mathfrak{p}_1}_{(x_1,x_2)}) \to L^{\mathfrak{p}_3'}_{x_3} (L^{\mathfrak{p}_1'}_{(x_1,x_2)}) }
   &\lesssim
      2^{j\theta} \sum_{\lambda_1=2^{2j}}^{\infty} \lambda_1^{(5\theta-2)/6}\\
   &\lesssim
      2^{j(8\theta-2)/3},
\end{align*}
where $\tilde{T}_{\delta,j}^{VII}$ denotes the sum of the associated operator pieces.
This is uniformly bounded if and only if $\theta \leq 1/4$. For $\theta = 1/3$,
we can only sum in the range $\lambda_1 > 2^{6j}$ and so it remains to see what happens when
$2^{2j} < \lambda_1 \leq 2^{6j}$.
We denote the sum of the associated operator pieces for this remaining range by $T_{\delta,j}^{VII}$.
We shall deal with this case in the following section.

{\bf{Case 4.2. $\lambda_1 \sim \lambda_2 \sim \lambda_3$ and $\lambda_1 \leq 2^{2j}$.}}
Here only the space-side estimate changes and we have
\begin{align}
\label{LowerThanTwo_CASE_4_2_basic_estiamtes}
\begin{split}
\Vert \widehat{\nu_j^\lambda} \Vert_{L^\infty} \lesssim \lambda_1^{-5/6},
\qquad \qquad
\Vert \nu_j^\lambda \Vert_{L^\infty} \lesssim \lambda_1^{3/2}.
\end{split}
\end{align}
By interpolation one can obtain
\begin{align*}
\Vert T^\lambda_j \Vert_{ L^{\mathfrak{p}_3}_{x_3} (L^{\mathfrak{p}_1}_{(x_1,x_2)}) \to L^{\mathfrak{p}_3'}_{x_3} (L^{\mathfrak{p}_1'}_{(x_1,x_2)}) }
   &\lesssim \lambda_1^{(4\theta-1)/3}.
\end{align*}
We denote the sum of the associated operator pieces by $T_{\delta,j}^{VIII}$.
The above estimate is obviously summable if and only if $\theta < 1/4$.
For $\theta = 1/4$ we shall now use complex interpolation, and we deal with $\theta = 1/3$ in the next section.
We obviously may assume in this case $\lambda_i \gg 1$ for all $i=1,2,3$.

For simplicity, we assume that $\lambda_1 = \lambda_2 = \lambda_3$.
We consider the following function parametrised by the complex number $\zeta$:
\begin{align}
\label{LowerThanTwo_CASE_4_2_ComplexInterpolationMeasure}
\mu_\zeta = \gamma(\zeta) \, \sum_{1 \ll \lambda_1 \leq 2^{2j}}
    (\lambda_1)^{\frac{1-4\zeta}{2}} \, \nu_j^\lambda,
\end{align}
where
\begin{align*}
\gamma(\zeta) = \frac{2^{-2(\zeta-1)} - 1}{2^{\frac{3}{2}}-1}.
\end{align*}
The associated operator is denoted by $T_\zeta$.
For $\zeta = 1/4$ it holds
\begin{align*}
\mu_\zeta = \sum_{1 \ll \lambda_1 \leq 2^{2j}} \nu_j^\lambda,
\end{align*}
and so by Stein's interpolation theorem it suffices to prove
\begin{align*}
\begin{split}
\Vert T_{it} \Vert_{ L^{2/(2-\tilde{\sigma})}_{x_3} (L^1_{(x_1,x_2)}) \to L^{2/\tilde{\sigma}}_{x_3} (L^\infty_{(x_1,x_2)}) }
   &\lesssim 1,\\
\Vert T_{1+it} \Vert_{L^2\to L^2}
   &\lesssim 1,
\end{split}
\end{align*}
with constants uniform in $t \in \R$. Here $\tilde{\sigma} = 1/3$ since $\theta = 1/4$.

In order to prove the first estimate, we need the decay bound \eqref{Auxiliary_mu_decay}, i.e.,
$$ |\widehat{\mu_{it}}(\xi)| \lesssim \frac{1}{(1+|\xi_3|)^{1/3}}.$$
But this follows automatically by \eqref{LowerThanTwo_CASE_4_2_basic_estiamtes}, the definition of $\mu_\zeta$, and
the fact that each $\widehat{\nu_j^\lambda}$ has its support located around $\lambda$.

We prove the second $L^2 \to L^2$ estimate by using the oscillatory sum lemma \cite[Lemma 2.7]{IM16}. We need to prove
\begin{align}
\label{LowerThanTwo_CASE_4_2_ComplexInterpolationFinal}
\Bigg\Vert \sum_{1 \ll \lambda_1 \leq 2^{2j}} (\lambda_1)^{-\frac{3}{2}-2it} \, \nu_j^\lambda \Bigg\Vert_{L^\infty}
   \lesssim \frac{1}{\Big| 2^{-2it} - 1 \Big|},
\end{align}
uniformly in $t$.

After substituting $\lambda_1 y_1 \mapsto y_1$ and $\lambda_1^{1/2} y_2 \mapsto y_2$
in the expression \eqref{LowerThanTwo_TheSecondDecomposition_SpaceForm_rewritten},
we get that the sum on the left hand side of \eqref{LowerThanTwo_CASE_4_2_ComplexInterpolationFinal} is
\begin{align*}
\sum_{1 \ll \lambda_1 \leq 2^{2j}} \lambda_1^{-2it} &\int
    \widecheck{\chi}_1 (\lambda_1 x_1-y_1) \,
    \widecheck{\chi}_1 (\lambda_1 x_2-2^{-j}\lambda_1^{1/2} y_2- \lambda_1 \psi_\omega(\lambda_1^{-1}y_1))\\
   &\times \widecheck{\chi}_1(\lambda_1 x_3 -  b^{\#}(\lambda_1^{-1}y_1,\lambda_1^{-1/2}y_2,\delta,j) y_2^2 - \lambda_1 \psi_\beta(\lambda_1^{-1}y_1)) \\
   &\times a(\lambda_1^{-1}y_1,\lambda_1^{-1/2}y_2, \delta, j) \, \chi_1(\lambda_1^{-1}y_1) \, \chi_0(\lambda_1^{-1/2}y_2) \, \my.
\end{align*}
We may assume that $|x| \leq C$ for some large constant $C$,
since otherwise we could use the first three factors to gain a decay in $\lambda_1$.

Now as usual, we use the substitution $\lambda_1 x_1 - y_1 \mapsto y_1$,
conclude that it is sufficient to consider $|y_1| \leq \lambda_1^\varepsilon$
and $x_1 \sim 1$, and use Taylor approximation for $\psi_\omega$ and $\psi_\beta$ at $x_1$. Then one gets
\begin{align*}
\sum_{1 \ll \lambda_1 \leq 2^{2j}} \lambda_1^{-2it} &\int
    \widecheck{\chi}_1 (y_1) \,
    \widecheck{\chi}_1 (\lambda_1 Q_\omega(x_1,x_2,\delta_1) - y_1 r_\omega(\lambda_1^{-1}y_1,x_1,\delta_1)-2^{-j}\lambda_1^{1/2} y_2)\\
   &\times \widecheck{\chi}_1(\lambda_1 Q_\beta(x_1,x_3,\delta_1) - y_1 r_\beta(\lambda_1^{-1}y_1,x_1,\delta_1) -  b^{\#}(x_1-\lambda_1^{-1}y_1,\lambda_1^{-1/2}y_2,\delta,j) y_2^2) \\
   &\times a(x_1-\lambda_1^{-1}y_1,\lambda_1^{-1/2}y_2, \delta, j) \, \chi_1(x_1-\lambda_1^{-1}y_1) \, \chi_0(\lambda_1^{-1/2}y_2) \, \chi_0(\lambda_1^{-\varepsilon} y_1) \, \my,
\end{align*}
where $|\partial_1^N r_\omega| \sim 1$ and $|\partial_1^N r_\beta| \sim 1$ for any $N \geq 0$. Note that $2^{-j}\lambda_1^{1/2} \leq 1$,
and therefore it is sufficient to consider the cases when either $|A| \gg 1$ or $|B| \gg 1$, where
$$
A \coloneqq \lambda_1 Q_\omega(x_1,x_2,\delta_1), \qquad \qquad
B \coloneqq \lambda_1 Q_\beta(x_1,x_3,\delta_1),
$$
since otherwise we can use the oscillatory sum lemma.

The first three factors within the integral are behaving essentially like
$$
\widecheck{\chi}_1 (y_1) \widecheck{\chi}_1 (A - y_1 - 2^{-j}\lambda_1^{1/2} y_2) \widecheck{\chi}_1 (B - y_1 - y_2^2).
$$

If $|B| \gg 1$, then since we could otherwise use the first factor, we can assume $|y_1| \ll |B|^{\varepsilon_B}$.
Then $|B-y_1 r_\beta| \sim |B|$, and we can estimate the integral by
$$
\int |\widecheck{\chi}_1(B-y_1 r_\beta -y_2^2 b^{\#})| \my_2 \lesssim \int |\widecheck{\chi}_1(B-y_1 r_\beta -t)| \, |t|^{-1/2} \mathrm{d}t \lesssim |B|^{-1/2}.
$$
Now one can sum in $\lambda_1$.

Let us now assume $|B| \leq C_B$ for some large, but fixed constant $C_B$, and $|A| \gg C_B$.
Again, we can assume $|y_1| \ll |A|^{\varepsilon_A}$, and so $|A-y_1 r_\omega| \sim |A|$.
Therefore if $|y_2| \leq |A|^{1/2}$, then using the second factor we get that the integral is bounded (up to a constant) by $|A|^{-N}$.
If $|y_2| > |A|^{1/2}$, then $|B-y_1 r_\beta -y_2^2 b^{\#}| \gtrsim |A|$ and so we can use the third factor, and sum in $\lambda_1$.



\section{Airy-type analysis in the case $h_{\text{lin}}(\phi)$ < 2}
\label{Section_Airy}

In this section we begin with the proof of the estimates for $T^{VII}_{\delta, j}$ and $T^{VIII}_{\delta, j}$
when $\theta = 1/3$, i.e., when $\phi$ is of type $A_{n-1}$ with $m=2$ and finite $n\geq5$. In this case
$\tilde{\sigma} = 1/4$. We shall first recall some of the notation from \cite[Chapter 5]{IM16}.
From now on $\delta$ shall be a triple $(\delta_0, \delta_1, \delta_2)$ with $\delta_0 = 2^{-j}$,
we use $\lambda$ to denote the common value $\lambda_1 = \lambda_2 = \lambda_3$, and define
\begin{align*}
s_1 \coloneqq \frac{\xi_1}{\xi_3}, \quad s_2 \coloneqq \frac{\xi_2}{\xi_3}, \quad s_3 \coloneqq \frac{\xi_3}{\lambda}, \\
s \coloneqq (s_1,s_2,s_3), \quad s' \coloneqq (s_1,s_2).
\end{align*}
Then $|s_i| \sim 1$ for $i=1,2,3$, and we have
\begin{align*}
\xi &= \lambda s_3 (s_1,s_2,1),\\
\Phi(x,\delta,j,\xi) &= \lambda s_3 \tilde{\Phi}(x, \delta, \sigma, s_1, s_2),
\end{align*}
where $\Phi$ is the total phase from \eqref{LowerThanTwo_TheSecondDecomposition_CompletePhase_rewritten} and
\begin{align*}
\tilde{\Phi}(x,\delta,j,s_1,s_2) &= s_1 x_1 + s_2 x_1^2 \omega(\delta_1 x_1) + \sigma x_1^n \beta(\delta_1 x_1)\\
&\quad +\delta_0 s_2 x_2 + x_2^2 b_0(x,\delta).
\end{align*}
Recall that according to Case 4.2 from the last subsection of the previous section we have
\begin{align*}
\Vert T^\lambda_j \Vert_{ L^{\mathfrak{p}_3}_{x_3} (L^{\mathfrak{p}_1}_{(x_1,x_2)}) \to L^{\mathfrak{p}_3'}_{x_3} (L^{\mathfrak{p}_1'}_{(x_1,x_2)}) }
   &\lesssim \lambda^{1/9},
\end{align*}
and we can assume $\lambda \gg 1$.
Furthermore, recall that $\sigma \sim 1$, and that
\begin{align*}
b^{\#}(y,\delta_1,\delta_2,j) = b_0(y,\delta) \coloneqq b^a(\delta_1 y_1, \delta_0 \delta_2 y_2),
\end{align*}
where $b^a$ is the same function as in Subsection \ref{Knapp_h_lin_less_than_two}.
It is the function $b$ from Subsection \ref{Knapp_h_lin_less_than_two} expressed in adapted coordinates.
Recall that $\beta(0) \neq 0$, $\omega(0) \neq 0$, and $b_0(y,0) = b^a(0,0) = b(0,0) \neq 0$ for all $y$.

In terms of $s$ the expression for the Fourier transform of $\nu_\delta^\lambda \coloneqq \nu_j^\lambda$ becomes
\begin{align*}
\chi_1(s_1 s_3) \chi_1(s_2 s_3) \chi_1 (s_3) \int e^{-i\lambda s_3 \tilde{\Phi}(y,\delta,\sigma,s_1,s_2)} \tilde{a}(y,\delta) \mathrm{d}y,
\end{align*}
where the amplitude $\tilde{a}(y,\delta) \coloneqq a(y,\delta) \chi_1(y_1) \chi_0(y_2)$ is a smooth function supported
in the sets where $x_1 \sim 1$ and $|x_2| \lesssim 1$ and whose derivatives are uniformly bounded with respect to $\delta$.
If we denote
\begin{align*}
T_\delta^\lambda f \coloneqq f * \widehat{\nu_\delta^\lambda},
\end{align*}
then the estimate we need to prove is
\begin{align*}
\Big\Vert \sum_{1 \ll \lambda \leq \delta_0^{-6}} T_\delta^\lambda \Big\Vert_{ L^{\mathfrak{p}_3}_{x_3} (L^{\mathfrak{p}_1}_{(x_1,x_2)}) \to L^{\mathfrak{p}_3'}_{x_3} (L^{\mathfrak{p}_1'}_{(x_1,x_2)}) }
   \lesssim 1
\end{align*}
for
\begin{align*}
\Big( \frac{1}{\mathfrak{p}_1'}, \frac{1}{\mathfrak{p}_3'} \Big) &= \Big( \frac{1}{6}, \frac{1}{4} \Big).
\end{align*}
This estimate corresponds to the estimate of the sum $T_{\delta,j}^{VII} + T_{\delta,j}^{VIII}$
considered in the last subsection of the previous section (Case $4.1$ and Case $4.2$).

\subsection{First steps and estimates}

Our first step is to use the stationary phase in the $y_2$ variable,
ignoring the part away from the critical point where we can obtain absolutely summable estimates.
Then, as explained in \cite[Section 5.1]{IM16}, one obtains by using the implicit function theorem
that the critical point $x_2^c$ can be written as
\begin{align*}
x_2^c = \delta_0 s_2 Y_2(\delta_1 x_1, \delta_2, \delta_0 s_2),
\end{align*}
where $Y_2$ is smooth, $Y_2(0,0,0)=-1/(2b(0,0))$, and $|Y_2| \sim 1$.
Now one defines
\begin{align*}
\Psi(x_1,\delta,\sigma,s') \coloneqq \tilde{\Phi}(x_1, x_2^c, \sigma, s'),
\end{align*}
so we can write
\begin{align*}
\widehat{\nu_\delta^\lambda}(\xi) = \lambda^{-1/2} \chi_1(s_1 s_3) \chi_1(s_2 s_3) \chi_1 (s_3)
   \int e^{-i\lambda s_3 \Psi(y_1,\delta,\sigma,s_1,s_2)} a_0(y_1,s,\delta; \lambda) \mathrm{d}y_1,
\end{align*}
where $a_0$ is smooth and uniformly a classical symbol of order $0$ with respect to $\lambda$, and where
\begin{align}
\label{LowerThanTwo_Airy_Phase_first_form}
\Psi(y_1, \delta, \sigma, s_1, s_2) = s_1 y_1 + s_2 y_1^2 \omega(\delta_1 y_1) + \sigma y_1^n \beta(\delta_1 y_1) + (\delta_0 s_2)^2 Y_3(\delta_1 y_1, \delta_2, \delta_0 s_2)
\end{align}
for a smooth $Y_3$ with $Y_3(0,0,0) = -1/4b(0,0) \neq 0$.

Recall that as $a_0$ is a classical symbol we can express it as
$$a_0(y_1,s,\delta; \lambda) = a_0^0(y_1,s,\delta) + \lambda^{-1} a_0^1(y_1,s,\delta; \lambda),$$
where $a_0^0$ does not depend on $\lambda$ and $a_0^1$ has the same properties as $a_0$.
This induces the decomposition
$$
\nu_{\delta}^\lambda = \nu_{\delta,a_0^0}^\lambda + \nu_{\delta,a_0^1}^\lambda.
$$
The function $\nu_{\delta,a_0^1}^\lambda$ associated to the amplitude $a_0^1$ has Fourier transform bounded by $\lambda^{-3/2}$
and the $L^\infty$ norm on the space side is bounded by $\lambda^{3/2}$
(by the same reasoning as used to obtain \eqref{LowerThanTwo_CASE_4_2_basic_estiamtes}).
From these two bounds we can easily get the required estimate
for the operator associated to $\nu_{\delta,a_0^1}^\lambda$.
Therefore from now on, by an abuse of notation, we may and shall assume that $\nu_\delta^\lambda$
has an amplitude which does not depend on $\lambda$, i.e.,
\begin{align*}
\widehat{\nu_\delta^\lambda}(\xi) = \lambda^{-1/2} \chi_1(s_1 s_3) \chi_1(s_2 s_3) \chi_1 (s_3)
   \int e^{-i\lambda s_3 \Psi(y_1,\delta,\sigma,s_1,s_2)} a_0(y_1,s,\delta) \mathrm{d}y_1.
\end{align*}

The next step is to localise the integration in the above integral
to a small neighbourhood of the point where the second derivative vanishes.
For $\delta = 0$ this point is
$$
x_1^c = x_1^c(0,\sigma,s_2) \coloneqq \Bigg( - \frac{2 \omega(0)}{n(n-1) \sigma \beta(0)} s_2 \Bigg)^{1/(n-2)}.
$$
Away from this point the estimate for the integral is at worst $\lambda^{-1}$, by stationary phase or integration by parts.

We now briefly explain how to deal with the part away from $x_1^c$.
Recall from Case 4 in the last subsection of the previous section that the space bound on
$\nu_\delta^\lambda$ is $2^j \lambda = \delta_0^{-1} \lambda$ if $\lambda > 2^{2j} = \delta_0^{-2}$.
Now using the results from Subsection \ref{Mixed_norm_auxiliary} one can easily see that we can sum absolutely in $\lambda > \delta_0^{-2}$.
The case when $\lambda \leq \delta_0^{-2}$ has to be dealt with complex interpolation as in the Case 4.2. from the last subsection of the previous section.
In fact, the proof is completely the same, except that one needs to appropriately change $\gamma$ and the exponent over $\lambda_1$
in the expression for $\mu_\zeta$ in \eqref{LowerThanTwo_CASE_4_2_ComplexInterpolationMeasure} since $\theta = 1/3$ in this case,
and there it was $\theta =1/4$.
One also has a different amplitude $a$ localising near $x_2^c$ in $y_2$ integration and away from $x_1^c$ in $y_1$ integration.

Hence we may now consider only the part near the critical point $x_1^c$.
Abusing the notation again, we shall denote the part near the critical point $x_1^c$ by $\nu_\delta^\lambda$ too.
Following \cite{IM16} we shall furthermore assume without loss of generality
\begin{align}
\label{LowerThanTwo_Airy_constant_normalisation}
- \frac{2 \omega(0)}{n(n-1)\sigma \beta(0)} = 1, \quad \quad s_2 \sim 1,
\end{align}
and that in \eqref{LowerThanTwo_Airy_Phase_first_form} we are integrating over an arbitrarily small neighbourhood of $x_1^c$.
Therefore, we now have $x_1^c(0, \sigma, s_2) = s_2^{1/(n-2)}$, $|\Psi'''(x_1^c(\delta, \sigma, s_2),\delta,\sigma,s_1,s_2)| \sim 1$,
(by implicit function theorem) $x_1^c = x_1^c(\delta, \sigma, s_2)$ depends smoothly in all of its variables, and
\begin{align*}
\Psi''(x_1^c(\delta, \sigma, s_2),\delta,\sigma,s_1,s_2) = 0.
\end{align*}
We restate \cite[Lemma 5.2.]{IM16} how to locally develop $\Psi$ at the critical point of $\Psi'$, i.e., the point $x_1^c$.
Its proof is straightforward.

\begin{lemma}
\label{LowerThanTwo_Airy_lemma_phase_taylor}
The phase $\Psi$ given by \eqref{LowerThanTwo_Airy_Phase_first_form} can be developed locally around $x_1^c$ in the form
$$
\Psi(x_1^c(\delta,\sigma,s_2)+y_1, \delta, \sigma, s_1, s_2) = B_0(s',\delta,\sigma) - B_1(s', \delta, \sigma) y_1 + B_3(s_2,\delta,\sigma,y_1) y_1^3,
$$
where $B_0$, $B_1$, and $B_3$ are smooth functions, and where $|B_3(s_2, \delta, \sigma, y_1)| \sim 1$.
In fact, we can write (after taking \eqref{LowerThanTwo_Airy_constant_normalisation} into account)
\begin{align*}
\begin{split}
x_1^c(\delta,\sigma,s_2) &= s_2^{1/(n-2)} G_1(s_2, \delta, \sigma), \\
B_0(s',\delta,\sigma) &= s_1 s_2^{1/(n-2)} G_1(s_2, \delta, \sigma) - s_2^{n/(n-2)} G_2(s_2,\delta,\sigma), \\
B_1(s',\delta,\sigma) &= -s_1 + s_2^{(n-1)/(n-2)} G_3(s_2,\delta,\sigma), \\
B_3(s',\delta,\sigma,0) &= s_2^{(n-3)/(n-2)} G_4(s_2,\delta,\sigma),
\end{split}
\end{align*}
where $G_k$, $k=1,2,3,4$, are all smooth and of the following forms at $\delta = 0$:
\begin{align}
\label{LowerThanTwo_Airy_G_function_forms}
\begin{split}
G_1(s_2,0,\sigma) &= 1, \\
G_2(s_2,0,\sigma) &= \frac{n^2-n-2}{2} \sigma \beta(0), \\
G_3(s_2,0,\sigma) &= n(n-2) \sigma \beta(0), \\
G_4(s_2,0,\sigma) &= \frac{n(n-1)(n-2)}{6} \sigma \beta(0).
\end{split}
\end{align}
We shall also need $G_5 \coloneqq G_1 G_3 - G_2$. One can easily check that $G_k \neq 0$ for each $k = 1,2,3,4$,
since $n \geq 5$, and that
\begin{align*}
G_5(s_2,0,\sigma) &= \frac{(n-1)(n-2)}{2} \sigma \beta(0) \neq 0.
\end{align*}
\end{lemma}

By applying the lemma we may now write
\begin{align}
\label{LowerThanTwo_Airy_Fourier_measure_form}
\begin{split}
\widehat{\nu_\delta^\lambda}(\xi) = &\lambda^{-1/2} \chi_1(s_1 s_3) \chi_1(s_2 s_3) \chi_1 (s_3) e^{-i\lambda s_3 B_0(s',\delta,\sigma)}\\
   &\int e^{-i\lambda s_3 (B_3(s_2,\delta,\sigma,y_1) y_1^3 - B_1(s', \delta, \sigma) y_1)} a_0(y_1,s,\delta) \chi_0(y_1) \mathrm{d}y_1,
\end{split}
\end{align}
where $\chi_0$ is supported here in a sufficiently small neighbourhood of the origin and
$a_0$ denotes a slightly different function than before, but with the same relevant properties.
We now decompose $\widehat{\nu_\delta^\lambda}$ further, motivated by Lemma \ref{Oscillatory_auxiliary_airy}, into parts where
$\lambda^{2/3} |B_1(s',\delta,\sigma)| \lesssim 1$ near the Airy cone, and $(2^{-l} \lambda)^{2/3} |B_1(s, \delta, \sigma)| \sim 1$ away from
the Airy cone, for $M_0 \leq 2^l \leq \lambda/M_1$, where $M_0, M_1$ are sufficiently large. The Airy cone itself is given by the equation $B_1 = 0$.

In order to obtain such a decomposition
we take smooth cutoff functions $\chi_0$ and $\chi_1$ such that $\chi_0$ is supported in a sufficiently large neighbourhood of the origin
and $\chi_1(t)$ is supported in a neighbourhood of the points $-1$ and $1$ and away from the origin.
We furthermore assume that
$$
\sum_{l\in \Z} \chi_1(2^{-2l/3}t) = 1
$$
on $\R \setminus \{0\}$.
Then we can define
\begin{align}
\label{LowerThanTwo_Airy_decomposition}
\begin{split}
\widehat{\nu^\lambda_{\delta,Ai}}(\xi) &\coloneqq \chi_0(\lambda^{2/3} B_1(s',\delta,\sigma)) \widehat{\nu^\lambda_\delta}(\xi), \\
\widehat{\nu^\lambda_{\delta,l}}(\xi) &\coloneqq \chi_1((2^{-l}\lambda)^{2/3} B_1(s',\delta,\sigma)) \widehat{\nu^\lambda_\delta}(\xi),
\end{split}
\end{align}
where $M_0 \leq 2^l \leq \lambda/M_1$, so that
\begin{align*}
\nu^\lambda_{\delta}(\xi) = \nu^\lambda_{\delta,Ai} + \sum_{M_0 \leq 2^l \leq \lambda/M_1} \nu^\lambda_{\delta,l}.
\end{align*}
We denote the associated convolution operators,
convolving against the Fourier transform of $\nu^\lambda_{\delta,Ai}$ and $\nu^\lambda_{\delta,l}$,
by $T^\lambda_{\delta,Ai}$ and $T^\lambda_{\delta,l}$.
Note that the size of the number $M_0$ is related to how large of a neighbourhood of $0$ the cutoff function $\chi_0$ covers
in the first equation of \eqref{LowerThanTwo_Airy_decomposition},
and the size of the number $M_1$ is related to how small of a neighbourhood of $0$ we take
in \eqref{LowerThanTwo_Airy_Fourier_measure_form} for the $y_1$ variable.

\subsection{Estimates near the Airy cone}

From Lemma \ref{Oscillatory_auxiliary_airy}, (a), we get that the bound on the Fourier transform of $\nu^\lambda_{\delta, Ai}$ is $\lambda^{-5/6}$.
Unlike in \cite{IM16} we shall need to use complex interpolation to be able to estimate the part $T^\lambda_{\delta,Ai}$.
The proof here is actually similar to certain cases when $h_{\text{lin}}(\phi) \geq 2$ in \cite[Subsection 8.7.1]{IM16}.

We consider the following function parametrised by $\zeta \in \C$:
\begin{align*}
\mu_\zeta = \gamma(\zeta) \, \sum_{1 \ll \lambda \leq \delta_0^{-6}}
    \lambda^{\frac{7-21\zeta}{12}} \, \nu^\lambda_{\delta, Ai},
\end{align*}
where
\begin{align*}
\gamma(\zeta) = \frac{2^{-\frac{7(\zeta-1)}{4}} - 1}{2^{\frac{7}{6}}-1}.
\end{align*}
The associated operator acting by convolution against the Fourier transform of $\mu_\zeta$
is denoted by $T_\zeta$. For $\zeta = 1/3$ we see that
\begin{align*}
\mu_\zeta = \sum_{1 \ll \lambda  \leq \delta_0^{-6}} \nu^\lambda_{\delta, Ai},
\end{align*}
which means, by interpolation, that it is sufficient to prove
\begin{align*}
\begin{split}
\Vert T_{it} \Vert_{ L^{2/(2-\tilde{\sigma})}_{x_3} (L^1_{(x_1,x_2)}) \to L^{2/\tilde{\sigma}}_{x_3} (L^\infty_{(x_1,x_2)}) }
   &\lesssim 1,\\
\Vert T_{1+it} \Vert_{L^2\to L^2}
   &\lesssim 1,
\end{split}
\end{align*}
with constants uniform in $t \in \R$.

In order to prove the first estimate,
we need the decay bound \eqref{Auxiliary_mu_decay}, i.e.,
$$ |\widehat{\mu_{it}}(\xi)| \lesssim \frac{1}{(1+|\xi_3|)^{1/4}}.$$
This follows right away by using the estimate on the Fourier transform of $\nu^\lambda_{\delta, Ai}$, the definition of $\mu_\zeta$, and
the fact that each $\widehat{\nu^\lambda_{\delta, Ai}}$ has its support located at $(\lambda,\lambda,\lambda)$.

We prove the second $L^2 \to L^2$ estimate by using Lemma \ref{Oscillatory_auxiliary_one_parameter}.
We need to prove
\begin{align}
\label{LowerThanTwo_Airy_near_ComplexInterpolationFinal}
\Bigg\Vert \sum_{1 \ll \lambda \leq \delta_0^{-6}} \lambda^{-\frac{7}{6}- \frac{7}{4}it} \, \nu^\lambda_{\delta, Ai} \Bigg\Vert_{L^\infty}
   \lesssim \frac{1}{\Big| 2^{-\frac{7}{4}it} - 1 \Big|},
\end{align}
uniformly in $t$.

As in \cite[Subsection 5.1.1]{IM16} we now apply Fourier inversion using the formulas \eqref{LowerThanTwo_Airy_Fourier_measure_form},
\eqref{LowerThanTwo_Airy_decomposition}, and the form of the integral from Lemma \ref{Oscillatory_auxiliary_airy}, (a).
Then after changing coordinates in the integration from $(\xi_1,\xi_2,\xi_3)$ to $(s_1, s_2, s_3)$ one gets
\begin{align*}
\nu^{\lambda}_{\delta, Ai}(x) =
   &\lambda^{13/6} \int e^{-i\lambda s_3(B_0(s',\delta,\sigma)-s_1 x_1 -s_2 x_2 -x_3)} \chi_0(\lambda^{2/3} B_1(s',\delta,\sigma)) \\
   &\quad \quad \times g(\lambda^{2/3} B_1(s',\delta,\sigma),\lambda^{-1/3},\delta,\sigma,s) \tilde{\chi}_1(s) \mathrm{d}s_1 \mathrm{d}s_2 \mathrm{d}s_3,
\end{align*}
where $g$ is the smooth function from Lemma \ref{Oscillatory_auxiliary_airy}, (a), whose derivatives of any order are uniformly bounded,
and where
\begin{align*}
\tilde{\chi}_1(s) \coloneqq \chi_1(s_1 s_3) \chi_1(s_2 s_3) \chi_1(s_3) s_3^2.
\end{align*}
We may now also restrict ourselves to the situation where $|x| \lesssim 1$, since otherwise we can get a factor $\lambda^{-N}$ by integrating by parts.

Finally, we change coordinates from $s' = (s_1,s_2)$ to $(z,s_2)$, where $z \coloneqq \lambda^{2/3} B_1(s',\delta,\sigma)$,
and so by Lemma \ref{LowerThanTwo_Airy_lemma_phase_taylor} we have
\begin{align*}
z   &= \lambda^{2/3}(-s_1+s_2^{(n-1)/(n-2)}G_3(s_2,\delta,\sigma)),
\end{align*}
that is
\begin{align*}
s_1 &= s_2^{(n-1)/(n-2)} G_3(s_2,\delta,\sigma)-\lambda^{-2/3}z.
\end{align*}
Thus we obtain
\begin{align}
\label{LowerThanTwo_Airy_near_ComplexInterpolation_space_form_book}
\begin{split}
\nu^\lambda_{\delta, Ai} (x) = &\lambda^{3/2} \int e^{-i \lambda s_3 \Phi(z,s_2,x,\delta,\sigma)} \\
   &\times g \Big(z, \lambda^{-1/3}, \delta, \sigma, s_2^{(n-1)/(n-2)} G_3(s_2,\delta,\sigma)-\lambda^{-2/3}z,s_2,s_3\Big) \\
   &\times \tilde{\chi}_1 \Big(s_2^{(n-1)/(n-2)} G_3(s_2,\delta,\sigma)-\lambda^{-2/3}z,s_2 \Big) \chi_0(z) \mathrm{d}z \mathrm{d}s_2 \mathrm{d}s_3,
\end{split}
\end{align}
where by using the expressions for $B_0(s',\delta,\sigma)$ and $G_5(s_2,\delta,\sigma)$ from
Lemma \ref{LowerThanTwo_Airy_lemma_phase_taylor} one gets
\begin{align}
\label{LowerThanTwo_Airy_near_ComplexInterpolation_space_form_phase}
\begin{split}
\Phi(z,s_2,x,\delta,\sigma) \coloneqq &s_2^{n/(n-2)} G_5(s_2,\delta,\sigma) - s_2^{(n-1)/(n-2)} G_3(s_2,\delta,\sigma) x_1 - s_2 x_2 - x_3\\
    &+\lambda^{-2/3} z (x_1 - s_2^{1/(n-2)} G_1(s_2, \delta, \sigma)).
\end{split}
\end{align}
We may shorten the expression in \eqref{LowerThanTwo_Airy_near_ComplexInterpolation_space_form_book} to
\begin{align}
\label{LowerThanTwo_Airy_near_ComplexInterpolation_space_form}
\begin{split}
\nu^\lambda_{\delta, Ai} (x) = \lambda^{3/2} \int &e^{-i \lambda s_3 \Phi(z,s_2,x,\delta,\sigma)} \\
   &\times \tilde{g} (z, s_2^{1/(n-2)}, s_3, \lambda^{-1/3}, \delta, \sigma) \mathrm{d}z \mathrm{d}s_2 \mathrm{d}s_3,
\end{split}
\end{align}
where $\tilde{g}$ is smooth with uniformly bounded derivatives and localising the integration domain to $|z| \lesssim 1$, $s_2 \sim |s_3| \sim 1$.

Next, we notice that $\Phi(z,s_2^{n-2},x,0,\sigma)$ is a polynomial in $s_2$ by \eqref{LowerThanTwo_Airy_G_function_forms}.
We therefore substitute $s_0 = s_2^{1/(n-2)}$ and denote
$$
\tilde{\Phi}(z,s_0,x,\delta,\sigma) = \Phi(z,s_0^{1/(n-2)},x,\delta,\sigma).
$$
We are interested in localising the integration in \eqref{LowerThanTwo_Airy_near_ComplexInterpolation_space_form}
to the place where $\partial_{s_0}^2 \tilde{\Phi} = 0$ and $\partial_{s_0}^3 \tilde{\Phi} \neq 0$.
In order to carry out this reduction we need another simple lemma. It will be applied to the first three terms of
\begin{align*}
\tilde{\Phi}(z,s_0,x,0,\sigma) = &s_0^{n} G_5(s_0^{n-2},0,\sigma) - s_0^{n-1} G_3(s_0^{n-2},0,\sigma) x_1 - s_0^{n-2} x_2 - x_3\\
    &+ \lambda^{-2/3} z (x_1 - s_0 G_1(s_0^{n-2}, 0, \sigma)),
\end{align*}
which constitute a polynomial in $s_0$ whose derivatives have at most two zeros not located at the origin.
Note that the last term in the above expression is arbitrarily small.

\begin{lemma}
\label{LowerThanTwo_Airy_near_ComplexInterpolation_polynomial_lemma}
Assume $n \geq 5$ and consider a number $x_0 \sim 1$. Let us define a polynomial of the form
\begin{align*}
P(x) \coloneqq x^{n-2} (x^2 + bx + c) = x^n + b x^{n-1} + c x^{n-2}
\end{align*}
whose second derivative can be written as
\begin{align*}
P''(x) \coloneqq n(n-1) x^{n-4}(x-x_0+\varepsilon)(x-x_0-\varepsilon).
\end{align*}
If $|\varepsilon| \leq c_1$ for a sufficiently small constant $c_1$,
then $|P'(x)| \sim 1$ on a neighbourhood of $x_0$, which depends on $c_1$, but not on $\varepsilon$.
On the other hand, if $|\varepsilon| > c_2$ for some $c_2 > 0$ and $x_0-\varepsilon \sim 1$ (resp. $x_0+\varepsilon \sim 1$),
then $|P'''(x_0-\varepsilon)| \sim_{c_2} 1$ (resp. $|P'''(x_0+\varepsilon)| \sim_{c_2} 1$).
\end{lemma}

\begin{proof}
One needs to express $b$ and $c$ in terms of $x_0$ and $\varepsilon$,
after which it is easy to prove the lemma by a straightforward calculation.
\end{proof}

From the first conclusion of Lemma \ref{LowerThanTwo_Airy_near_ComplexInterpolation_polynomial_lemma} we see that
if the zeros of $\partial_{s_0}^2 \tilde{\Phi}$ which are away from the origin are too close to each other,
then we may use stationary phase or integration by parts to obtain a factor of $\lambda^{-1/2}$
(or better) and so the left hand side of \eqref{LowerThanTwo_Airy_near_ComplexInterpolationFinal} is absolutely summable.
Therefore we may assume that there is at least some distance between the zeros of $\partial_{s_0}^2 \tilde{\Phi}$.
From the second conclusion of Lemma \ref{LowerThanTwo_Airy_near_ComplexInterpolation_polynomial_lemma}
we obtain $|\partial_{s_0}^3 \tilde{\Phi}| \sim 1$ in a neighbourhood of those zeros within the integration domain
(i.e., for those located at $\sim 1$).

Therefore, we may now use the implicit function theorem and obtain a parametrisation
of a zero of the first three terms of $\partial_{s_0}^2 \tilde{\Phi}$:
\begin{align*}
\partial_{s_0}^2( s_0^{n} G_5(s_0^{n-2},\delta,\sigma) - s_0^{n-1} G_3(s_0^{n-2},\delta,\sigma) x_1 - s_0^{n-2} x_2),
\end{align*}
which we shall denote by $s_0^c(x,\delta,\sigma)$, and assume it is located away from the origin.
All such zeros can be treated the same way.

We may assume we integrate arbitrarily near the zero $s_0^c(x,\delta,\sigma)$ since again
we could otherwise use stationary phase or integration by parts. We may then use a Taylor approximation for
the first three terms in $\tilde{\Phi}$ at $s_0^c(x,\delta,\sigma)$ and obtain after translating $s_0 \mapsto s_0 + s_0^c$
that the phase has the form
\begin{align*}
\tilde{\Phi}_1(z,s_0,x,\delta,\sigma) =
    &\tilde{B}_0(x,\delta,\sigma) - \tilde{B}_1(x,\delta,\sigma) s_0 + \tilde{B}_3(s_0, x,\delta,\sigma) s_0^3 \\
    &+ \lambda^{-2/3} z \tilde{G}_1(s_0, x,\delta,\sigma) - \lambda^{-2/3} z \tilde{G}_2(s_0, x, \delta, \sigma) s_0
\end{align*}
with functions $\tilde{B}_{i}$, $i = 0,1,3$, being smooth and $|\tilde{B}_3| \sim 1$.
The functions $\tilde{G}_{i}$ are also smooth and have
the property that they do not depend on $s_0$ when $\delta = 0$.
Note also $\tilde{G}_2(s_0, x, 0, \sigma) = 1$.

Hence, we have obtained an Airy type integral with an error term of size at most $\lambda^{-2/3}$.
We denote this newly obtained function by $\tilde{\nu}^\lambda_{\delta, Ai}$:
\begin{align*}
\begin{split}
\tilde{\nu}^\lambda_{\delta, Ai} (x) = \lambda^{3/2} \int &e^{-i \lambda s_3 \tilde{\Phi}_1(z,s_0,x,\delta,\sigma)} \\
   &\times \tilde{g}_1 (z, s_0, s_3, \lambda^{-1/3}, \delta, \sigma) \mathrm{d}z \mathrm{d}s_0 \mathrm{d}s_3,
\end{split}
\end{align*}
where $\tilde{g}_1$ has the same properties as $\tilde{g}$, except that now the integration is over the domain
where $|z| \lesssim 1$, $|s_3| \sim 1$, and $|s_0| \ll 1$.

We now prove \eqref{LowerThanTwo_Airy_near_ComplexInterpolationFinal} for the remaining piece $\tilde{\nu}^\lambda_{\delta, Ai}$.
Let us begin with the case when
$$
A \coloneqq \lambda^{2/3} \tilde{B}_1(x,\delta,\sigma)
$$
satisfies $|A| \gg 1$.
We claim that in this case we can estimate the function $\tilde{\nu}^\lambda_{\delta, Ai}$ by $\lambda^{7/6} |A|^{-1/4}$, which
is absolutely summable in $\lambda$ in the expression \eqref{LowerThanTwo_Airy_near_ComplexInterpolationFinal} for $\mu_{1+it}$.
We need a modification of Lemma \ref{Oscillatory_auxiliary_airy}, (b).

\begin{lemma}
\label{LowerThanTwo_Airy_near_ComplexInterpolation_modification_Airy_b}
Consider the integral
\begin{align*}
   \int e^{ i\lambda(
       - b_1 s_0 + b_3(s_0) s_0^3 + \lambda^{-2/3} g(s_0)
   )} a_0(s_0,\lambda^{-1/3}) \chi_0(s_0) \mathrm{d}s_0,
\end{align*}
where all the appearing functions are smooth with uniformly bounded derivatives, and $|b_3(s_0)| \sim 1$.
This integral can be estimated up to a constant by $\lambda^{-1/3} |\lambda^{2/3} b_1|^{-1/4}$
if $|\lambda^{2/3} b_1| \gg 1$, $\lambda \gg 1$, and $\chi_0$ is supported in a sufficiently small neighbourhood of the origin.
\end{lemma}

\begin{proof}
Without loss of generality we may assume $b_3 > 0$.
We proceed similarily as in the proof of Lemma \ref{Oscillatory_auxiliary_airy}, (b).
The main point is that since we may assume $|b_1| \gg |\lambda^{-2/3} g^{(k)}|$ for finitely many $k \geq 0$,
the term $\lambda^{-2/3} g(s_0)$ will not have any significant influence.
The first derivative of the phase is
\begin{align*}
i \lambda(-b_1  + 3b_3(s_0) s_0^2 + b_3'(s_0) s_0^3 + \lambda^{-2/3} g'(s_0)),
\end{align*}
and hence if $b_1 < 0$ or $|b_1| \gtrsim 1$, then the phase has no critical points
since the first two terms are dominant, and its derivative is of size $\gtrsim |\lambda b_1|$.
Using integration by parts we get the estimate $|\lambda b_1|^{-1}$.

Therefore we may assume $0 < b_1 \ll 1$ and substitute $b_1^{1/2} s_0$ to obtain
\begin{align*}
   b_1^{1/2} \int &e^{ i \lambda b_1^{3/2}(-s_0 + b_3(b_1^{1/2} s_0) s_0^3 + b_1^{-3/2} \lambda^{-2/3} g(b_1^{1/2} s_0))} \\
       &\times a_0(b_1^{1/2} s_0,\lambda^{-1/3}) \chi_0(b_1^{1/2} s_0) \mathrm{d}s_0.
\end{align*}
One can now easily check that the function
$$
s_0 \mapsto -s_0 + b_3(b_1^{1/2} s_0) s_0^3 + b_1^{-3/2} \lambda^{-2/3} g(b_1^{1/2} s_0)
$$
has precisely two critical points near $\pm 1$.
Near these critical points the second derivative is of size $\sim 1$ and so by stationary phase one gets the bound
$b_1^{1/2} |\lambda b_1^{3/2}|^{-1/2} = \lambda^{-1/3} |\lambda^{2/3} b_1|^{-1/4}$.
Away from the critical points the size of the derivative of the phase is $\sim \lambda b_1^{3/2} \max\{s_0^2, 1\}$,
and so integrating by parts one gets the estimate $|\lambda b_1|^{-1}$.
\end{proof}

Therefore after one applies the above lemma, our problem is reduced to the case $|A| \lesssim 1$.
Our next step is to substitute $s_0 \mapsto \lambda^{-1/3} s_0$.
Then one gets
\begin{align*}
\begin{split}
\tilde{\nu}^\lambda_{\delta, Ai} (x) = \lambda^{7/6} \int &e^{-i \lambda s_3 \tilde{\Phi}_1(z,\lambda^{-1/3} s_0,x,\delta,\sigma)} \\
   &\times \tilde{g}_1 (z, \lambda^{-1/3} s_0, s_3, \lambda^{-1/3}, \delta, \sigma) \mathrm{d}z \mathrm{d}s_0 \mathrm{d}s_3,
\end{split}
\end{align*}
where
\begin{align*}
\lambda \tilde{\Phi}_1(z,\lambda^{-1/3} s_0,x,\delta,\sigma) =
    &\lambda \tilde{B}_0(x,\delta,\sigma) - A s_0 + \tilde{B}_3(\lambda^{-1/3} s_0, x,\delta,\sigma) s_0^3 \\
    &+ \lambda^{1/3} z \tilde{G}_1(\lambda^{-1/3} s_0, x,\delta,\sigma) - z \tilde{G}_2(\lambda^{-1/3} s_0, x, \delta, \sigma) s_0,
\end{align*}
and the new integration domain is $|z| \lesssim 1$, $|s_3| \sim 1$, and $|s_0| \ll \lambda^{1/3}$.

Using a Taylor approximation we can rewrite the $\tilde{G}_1$ term as
\begin{align*}
\tilde{G}_1(\lambda^{-1/3} s_0, x,\delta,\sigma) = \tilde{G}_1(0, x,\delta,\sigma) + \lambda^{-1/3} s_0 r(\lambda^{-1/3} s_0, x,\delta,\sigma).
\end{align*}
where $|\partial^N_t r(t, x,\delta,\sigma)| \ll_N 1$ for any $N \geq 0$ since $\tilde{G}_1$ is constant when $\delta = 0$.
Therefore, if we denote $\tilde{G}_3 = \tilde{G}_2 - r$, then $\tilde{G}_3$ has the same properties as $\tilde{G}_2$
(in particular $\tilde{G}_3 \sim 1$), and we can write
\begin{align*}
\lambda \tilde{\Phi}_1(z,\lambda^{-1/3} s_0,x,\delta,\sigma) =
    &\lambda \tilde{B}_0(x,\delta,\sigma) - A s_0 + \tilde{B}_3(\lambda^{-1/3} s_0, x,\delta,\sigma) s_0^3 \\
    &+ \lambda^{1/3} z \tilde{G}_1(0, x,\delta,\sigma) - z \tilde{G}_3(\lambda^{-1/3} s_0, x, \delta, \sigma) s_0.
\end{align*}

From this expression one sees that we can get an integrable factor of size $(1+|s_0|^2)^{-N/2}$
in the amplitude of $\tilde{\nu}^\lambda_{\delta, Ai}$ by using integration by parts in $s_0$, i.e., we can assume
\begin{align*}
\Big| \partial^{\alpha_1}_z \partial^{\alpha_2}_{s_0} \partial^{\alpha_3}_{s_3} \Big( \tilde{g}_1 (z, \lambda^{-1/3} s_0, s_3, \lambda^{-1/3}, \delta, \sigma) \Big) \Big|
\lesssim C_{\alpha_1, \alpha_2, \alpha_3} (1+|s_0|^2)^{-N/2},
\end{align*}
as the unbounded terms in the expression for the $s_0$ derivative of $\lambda \tilde{\Phi}_1(z,\lambda^{-1/3} s_0,x,\delta,\sigma)$ vanish.

Let us denote by
$$
E \coloneqq \lambda \tilde{B}_0(x,\delta,\sigma), \qquad \qquad F \coloneqq \lambda^{1/3} \tilde{G}_1(0, x,\delta,\sigma),
$$
the unbounded terms of the phase.
We need to reduce our problem to the case when $|E| \lesssim 1$ and $|F| \lesssim 1$
since then we can simply apply the oscillatory sum lemma.

We begin with the case $|F| \gg 1$.
Let us consider the $z$ integration. The factor tied with $z$ in the phase is
$$
F - \tilde{G}_3(\lambda^{-1/3} s_0, x, \delta, \sigma) s_0 = F - \tilde{G}_3 s_0,
$$
where $\tilde{G}_3(\lambda^{-1/3} s_0, x, \delta, \sigma) \sim 1$.
We may therefore assume we are integrating over the area in $s_0$ where
$$
|F - \tilde{G}_3 s_0| \lesssim |F|^{\varepsilon},
$$
since otherwise we can use integration by parts in $z$ and gain a factor $|F|^{-\varepsilon}$.
In particular, in this case we have $|s_0| \sim |F|$.
But then the integrable factor $(1+|s_0|^2)^{-N/2}$ is of size $|F|^{-N}$ and so we obtain the required bound.

It remains to consider the case $|F| \lesssim 1$ and $|E| \gg 1$.
The idea in this case is to use integration by parts in $s_3$, which enables us to localise the integration
to the set where $|\lambda \tilde{\Phi}_1| \lesssim |E|^\varepsilon$.
If we now take $|E|$ sufficiently large compared to both $|A|$ and $|F|$,
then we see that $|\lambda \tilde{\Phi}_1| \lesssim |E|^\varepsilon$ forces $|s_0| \sim |E|^{1/3}$.
But this implies that the integrable factor $(1+|s_0|^2)^{-N/2}$ is of size $|E|^{-N/3}$, which is what we wanted.
We are done with the part near the Airy cone.



\subsection{Estimates away from the Airy cone -- first considerations}

Recall from \eqref{LowerThanTwo_Airy_Fourier_measure_form} and \eqref{LowerThanTwo_Airy_decomposition} that we may write
\begin{align*}
\widehat{\nu^\lambda_{\delta,l}}(\xi) =
   &\lambda^{-1/2} \chi_1((2^{-l}\lambda)^{2/3} B_1(s',\delta,\sigma)) \\
   &\chi_1(s_1 s_3) \chi_1(s_2 s_3) \chi_1 (s_3) e^{-i\lambda s_3 B_0(s',\delta,\sigma)}\\
   &\int e^{-i\lambda s_3 (B_3(s_2,\delta,\sigma,y_1) y_1^3 - B_1(s', \delta, \sigma) y_1)} a_0(y_1,s,\delta) \chi_0(y_1) \mathrm{d}y_1,
\end{align*}
where $1 \ll 2^l \ll \lambda$. Applying Lemma \ref{Oscillatory_auxiliary_airy}, (b), we obtain
\begin{align*}
\begin{split}
\widehat{\nu^\lambda_{\delta,l}}(\xi) =
   &\lambda^{-1/2} \chi_1((2^{-l}\lambda)^{2/3} B_1(s',\delta,\sigma)) \\
   &\chi_1(s_1 s_3) \chi_1(s_2 s_3) \chi_1 (s_3) e^{-i\lambda s_3 B_0(s',\delta,\sigma)} \\
   &\Big( s_3^{-1/2} \lambda^{-1/2} |B_1(s', \delta, \sigma)|^{-1/4} \\
   &\times a(|B_1(s', \delta, \sigma)|^{1/2}, s; s_3\lambda |B_1(s', \delta, \sigma)|^{3/2}) \,
    e^{is_3 \lambda |B_1(s', \delta, \sigma)|^{3/2} q(|B_1(s', \delta, \sigma)|^{1/2},s_2)} \\
   &+ (s_3 \lambda |B_1(s', \delta, \sigma)|)^{-1} \, E(s_3\lambda |B_1(s', \delta, \sigma)|^{3/2}, |B_1(s', \delta, \sigma)|^{1/2}, s) \Big),
\end{split}
\end{align*}
where we have slightly simplified the situation by ignoring the sign of the function $q$ since
both $q_{+}$ and $q_{-}$ appearing in Lemma \ref{Oscillatory_auxiliary_airy}, (b), can be treated in the same way.
Note that $q$ depends in the second variable
only in $s_2$ and not $s$ since the same is true for $B_3$,
as can be readily seen from the proof of Lemma \ref{Oscillatory_auxiliary_airy}, (b).
Recall that $a$, $q$, and $E$ are smooth, and $|q| \sim 1$.
$E$ and all its derivatives have Schwartz decay in the first variable, and
$a$ is a classical symbol of order $0$ in the $s_3\lambda |B_1(s', \delta, \sigma)|^{3/2}$ variable.

We denote
$$
z = (2^{-l}\lambda)^{2/3} B_1(s',\delta,\sigma),
$$
and slightly change $a$ and $E$ in order to absorb the $s_3$ factors.
Then we can rewrite the previous expression for $\widehat{\nu^\lambda_{\delta,l}}$ as
\begin{align*}
\begin{split}
\widehat{\nu^\lambda_{\delta,l}}(\xi) =
   &\lambda^{-1/2} \chi_1(z) \chi_1(s_1 s_3) \chi_1(s_2 s_3) \chi_1 (s_3) e^{-i\lambda s_3 B_0(s',\delta,\sigma)} \\
   &\Big( \lambda^{-1/2} (2^{-l} \lambda)^{1/6} |z|^{-1/4} \\
   &\times a((2^l \lambda^{-1})^{1/3} |z|^{1/2}, s; 2^l |z|^{3/2}) \,
    e^{-is_3 2^l |z|^{3/2} q((2^l \lambda^{-1})^{1/3} |z|^{1/2}, s_2)} \\
   &+ \lambda^{-1} (2^{-l} \lambda)^{2/3} |z|^{-1} \, E(2^l |z|^{3/2}, (2^l \lambda^{-1})^{1/3} |z|^{1/2}, s) \Big).
\end{split}
\end{align*}
From this we easily see that
\begin{align*}
\Vert \widehat{\nu^\lambda_{\delta,l}} \Vert_{L^\infty} \lesssim \lambda^{-5/6} 2^{-l/6}.
\end{align*}

We plan to use complex interpolation and the two parameter oscillatory sum lemma (Lemma \ref{Oscillatory_auxiliary_two_parameter}).
We consider the following function parametrised by $\zeta \in \C$:
\begin{align*}
\mu_\zeta = \gamma(\zeta) \, \sum_{\substack{ 1 \ll \lambda \leq \delta_0^{-6} \\ M_0 \leq 2^l \leq \lambda/M_1}}
    \lambda^{\frac{7-21\zeta}{12}} \, 2^{\frac{1-3\zeta}{6}l} \, \nu^\lambda_{\delta,l},
\end{align*}
for an appropriate $\gamma(\zeta)$ to be chosen later as in \eqref{Oscillatory_auxiliary_two_parameter_gamma}.
We shall also use the one parameter oscillatory sum lemma for certain subcases,
and therefore we shall need to add appropriate factors to $\gamma$ of the form \ref{Oscillatory_auxiliary_one_parameter_gamma}.
The operator associated to $\mu_\zeta$ we denote by $T_\zeta$.

For $\zeta = 1/3$ we see that
\begin{align*}
\mu_\zeta = \sum_{\substack{ 1 \ll \lambda \leq \delta_0^{-6} \\ M_0 \leq 2^l \leq \lambda/M_1}} \nu^\lambda_{\delta,l},
\end{align*}
which means, by Stein's interpolation theorem, that it is sufficient to prove
\begin{align*}
\begin{split}
\Vert T_{it} \Vert_{ L^{2/(2-\tilde{\sigma})}_{x_3} (L^1_{(x_1,x_2)}) \to L^{2/\tilde{\sigma}}_{x_3} (L^\infty_{(x_1,x_2)}) }
   &\lesssim 1,\\
\Vert T_{1+it} \Vert_{L^2\to L^2}
   &\lesssim 1,
\end{split}
\end{align*}
with constants uniform in $t \in \R$.

In order to prove the first estimate we need the decay bound \eqref{Auxiliary_mu_decay}, i.e.,
$$ |\widehat{\mu_{it}}(\xi)| \lesssim \frac{1}{(1+|\xi_3|)^{1/4}}.$$
This bound follows easily by the $L^\infty$ bound on the Fourier transform of $\nu^\lambda_{\delta,l}$, the definition of $\mu_\zeta$, and
the fact that each $\widehat{\nu^\lambda_{\delta,l}}$ has its support located at $(\lambda,\lambda,\lambda)$.

It remains to prove the $L^2 \to L^2$ estimate
\begin{align}
\label{LowerThanTwo_Airy_away_ComplexInterpolationFinal}
\Bigg\Vert \sum_{\substack{ 1 \ll \lambda \leq \delta_0^{-6} \\ M_0 \leq 2^l \leq \lambda/M_1}}
   \lambda^{-\frac{7}{6}- \frac{7}{4}it} \, 2^{-\frac{1}{3}l-\frac{1}{2}ilt} \, \nu^\lambda_{\delta,l} \Bigg\Vert_{L^\infty}
   \lesssim \frac{1}{\Big| \gamma(1+it) \Big|},
\end{align}
uniformly in $t$.

We split the function $\nu^\lambda_{\delta,l}$ as
$$
\nu^\lambda_{\delta,l} = \nu^E_{\lambda, l} + \nu^a_{\lambda, l},
$$
where
\begin{align*}
\widehat{\nu^E_{\lambda, l}}(\xi) =
   &\lambda^{-5/6} 2^{-Nl} \tilde{\chi}_1(s,z) e^{-i\lambda s_3 B_0(s',\delta,\sigma)} \\
   &\times E(2^l |z|^{3/2}, (2^l \lambda^{-1})^{1/3} |z|^{1/2}, s)
\end{align*}
and
\begin{align*}
\widehat{\nu^a_{\lambda, l}}(\xi) =
   &\lambda^{-5/6} 2^{-l/6} \tilde{\chi}_1(s,z) e^{-i\lambda s_3 B_0(s',\delta,\sigma)} \\
   &\times a((2^l \lambda^{-1})^{1/3} |z|^{1/2}, s; 2^l |z|^{3/2}) \, e^{-is_3 2^l |z|^{3/2} q((2^l \lambda^{-1})^{1/3} |z|^{1/2}, s_2)},
\end{align*}
with appropriate (and in each of the above expressions possibly different) $\tilde{\chi}_1$ smooth cutoff functions
localising to the area where $|s_1| \sim s_2 \sim |s_3| \sim |z| \sim 1$.
In the expression for $\nu^E_{\lambda, l}$ we obtain the factor $2^{-Nl}$ by using the Schwartz property in the first variable of $E$,
and so the function $E$ is slightly different than before, but with the same properties.

\subsection{Estimates away from the Airy cone -- the estimate for $\nu^E_{\lambda, l}$}

The function $\nu^E_{\lambda, l}$ can be treated similarily as the function $\nu^\lambda_{\delta, Ai}$ in the case near the Airy cone.
We first apply the inverse of the Fourier transform to $\widehat{\nu^E_{\lambda, l}}$,
and then substitute $s=(s_1,s_2,s_3)$ for $\xi=(\xi_1, \xi_2, \xi_3)$.
Recall that $z = (2^{-l}\lambda)^{2/3} B_1(s',\delta,\sigma)$ and so by Lemma \ref{LowerThanTwo_Airy_lemma_phase_taylor} one has
\begin{align*}
s_1 = s_2^{(n-1)/(n-2)} G_3(s_2,\delta,\sigma)-(2^{l} \lambda^{-1})^{2/3}z.
\end{align*}
We plug in this expression for $s_1$ and also substitute $s_0$ for $s_2^{1/(n-2)}$.
In the end one gets
\begin{align*}
\nu^E_{\lambda, l} (x) = &\lambda^{3/2} 2^{-Nl} \int e^{-i \lambda s_3 \Phi_2(z,s_0,x,\delta,\sigma)} \\
   &\times g_2 \Big(2^l, (2^l \lambda^{-1})^{1/3}, z, s_0, s_3, \delta, \sigma\Big) \mathrm{d}z \mathrm{d}s_0 \mathrm{d}s_3,
\end{align*}
where $g_2$ is smooth and has all of its derivatives Schwartz in the first variable, and where
\begin{align*}
\Phi_2(z,s_0,x,\delta,\sigma) \coloneqq &s_0^{n} G_5(s_0^{n-2},\delta,\sigma) - s_0^{n-1} G_3(s_0^{n-2},\delta,\sigma) x_1 - s_0^{n-2} x_2 - x_3 \\
    &+(2^l \lambda^{-1})^{2/3} z (x_1 - s_0 G_1(s_0^{n-2}, \delta, \sigma)).
\end{align*}
The only difference compared to the phase in \eqref{LowerThanTwo_Airy_near_ComplexInterpolation_space_form_phase} is that
there $|z| \lesssim 1$, while here $|z| \sim 1$, and instead of the $\lambda^{-2/3}$ factor in front of $z$ in the phase
in \eqref{LowerThanTwo_Airy_near_ComplexInterpolation_space_form_phase},
here we have the much larger factor $(2^l \lambda^{-1})^{2/3}$.

We can now reduce to the situation where $|x| \lesssim 1$.
Namely, if $|x_1| \gg 1$ then we integrate by parts in $z$ to gain a factor $(\lambda (2^l \lambda^{-1})^{2/3})^{-N}$.
Otherwise if $|x_1| \lesssim 1$ and $|x_2| \gg 1$, then we integrate by parts in $s_0$ to obtain a factor $\lambda^{-N}$,
and if $|(x_1,x_2)| \lesssim 1$ and $|x_3| \gg 1$, we integrate by parts in $s_3$ to again gain a factor of $\lambda^{-N}$.

Next, recall that $(2^l \lambda^{-1})^{2/3} \ll 1$. Therefore, we may use again Lemma \ref{LowerThanTwo_Airy_near_ComplexInterpolation_polynomial_lemma}
and argue similarily as we did in the case near the Airy cone to reduce ourselves to a small neighbourhood of a point where the second derivative
in $s_0$ of the first three terms of $\Phi_2$ vanishes and $|\partial_{s_0}^3 \Phi_2| \sim 1$.
By the implicit function theorem we may parametrise this point as $s^c = s^c(x, \delta, \sigma)$:
$$
\left. \partial^2_{s_0} \right|_{s_0 = s^c}
(s_0^{n} G_5(s_0^{n-2},\delta,\sigma) - s_0^{n-1} G_3(s_0^{n-2},\delta,\sigma) x_1 - s_0^{n-2} x_2 - x_3) = 0.
$$
The point $s^c$ depends smoothly on $(x,\delta,\sigma)$.

Translating to the point $s^c$ and localising to a small neighbourhood we obtain a new function $\tilde{\nu}^E_{\lambda, l}$ of the form
\begin{align*}
\tilde{\nu}^E_{\lambda, l} (x) = &\lambda^{3/2} 2^{-Nl} \int e^{-i \lambda s_3 \Phi_2(z,s_0,x,\delta,\sigma)} \\
   &\times \tilde{g}_2 \Big(2^l, (2^l \lambda^{-1})^{1/3}, z, s_0, s_3, \delta, \sigma\Big) \mathrm{d}z \mathrm{d}s_0 \mathrm{d}s_3,
\end{align*}
where $\tilde{g}_2$ has the same properties as $g_2$, except that now $|s_0| \ll 1$.
The new phase is
\begin{align*}
\tilde{\Phi}_2(z,s_0,x,\delta,\sigma) =
    &\tilde{B}_0(x,\delta,\sigma) - \tilde{B}_1(x,\delta,\sigma) s_0 + \tilde{B}_3(s_0, x,\delta,\sigma) s_0^3 \\
    &+ (2^l \lambda^{-1})^{2/3} z H_0(s_0, x,\delta,\sigma) - (2^l \lambda^{-1})^{2/3} z H_1(s_0, x, \delta, \sigma) s_0,
\end{align*}
where $|\tilde{B}_3| \sim 1$ and $H_1 \sim 1$.
Additionally, one can see that $H_0$ and $H_1$ do not depend on $s_0$ when $\delta = 0$.

The next step is to develop the whole phase $\tilde{\Phi}_2$ at the point where $\partial_{s_0}^2 \tilde{\Phi}_2 = 0$.
The reason for this is that the factor $(2^l \lambda^{-1})^{2/3}$ is too large, and we cannot apply something similar to
Lemma \ref{LowerThanTwo_Airy_near_ComplexInterpolation_modification_Airy_b}.
Let us denote the critical point of $\partial_{s_0} \tilde{\Phi}_2$ by $s_0^c = s_0^c(x,\delta,\sigma, (2^l \lambda^{-1})^{2/3} z)$.
Note that $s_0^c$ is identically $0$ when either $\delta = 0$ or the variable refering to $(2^l \lambda^{-1})^{2/3} z$ is $0$.
Therefore, we can actually write
$$
s_0^c = (2^l \lambda^{-1})^{2/3} z \, \tilde{s}_0^c(x,\delta,\sigma, (2^l \lambda^{-1})^{2/3} z),
$$
where $\tilde{s}_0^c$ is smooth and identically $0$ when $\delta = 0$.

If we shorten $\rho = (2^l \lambda^{-1})^{2/3} z$,
then the expression for the first derivative of $\tilde{\Phi}_2$ at the point $s_0^c$ has the form
\begin{align*}
\partial_{s_0} \tilde{\Phi}_2 (z,s_0^c,x,\delta,\sigma) &=
    (s_0^c)^2 \, b(s_0^c, x,\delta,\sigma) - \rho h(s^c_0, x, \delta, \sigma) - \tilde{B}_1(x,\delta,\sigma) \\
    &= \rho^2 (\tilde{s}_0^c)^2 \, b(s_0^c, x,\delta,\sigma) - \rho h(s^c_0, x, \delta, \sigma) - \tilde{B}_1(x,\delta,\sigma),
\end{align*}
where $h(s^c_0, x, \delta, \sigma) \sim 1$ and $|b(s_0^c, x,\delta,\sigma)| \sim 1$ for some smooth functions $h$ and $b$.

One can easily check that $|\partial^3_{s_0} \tilde{\Phi}_2 (z,s_0,x,\delta,\sigma)| \sim 1$.
Therefore, developing the phase $\tilde{\Phi}_2$ at the point $s^c_0$, we may write
\begin{align}
\label{LowerThanTwo_Airy_away_Phase_tilde_3}
\tilde{\Phi}_3(z,s_0,x,\delta,\sigma) =
    &b_0(\rho) - \Big[ b_1 +  \rho \tilde{b}_1(\rho) \Big] s_0 + b_3(s_0, \rho) s_0^3,
\end{align}
where we suppressed the dependence of $b_0, b_1, \tilde{b}_1$, and $b_3$ on the bounded parameters $(x,\delta,\sigma)$.
Here we know that $\tilde{b}_1 \sim 1$ and $|b_3| \sim 1$.
We may again assume $|s_0| \ll 1$ as on the other part where $|s_0| \gtrsim 1$ we could use integration by parts or stationary phase
and obtain an expression which when plugged into \eqref{LowerThanTwo_Airy_away_ComplexInterpolationFinal}
would be absolutely summable in both $\lambda$ and $2^l$.

Finally, we develop the term $b_0$ at $0$ and substitute $s_0 \mapsto \lambda^{-1/3}s_0$.
Then
\begin{align*}
\lambda \tilde{\Phi}_3(z,s_0,x,\delta,\sigma)
    &= \lambda \Big(b_0^0 + \rho b_0^1 + \rho^2 \tilde{b}_0(\rho) - \lambda^{-1/3} \Big[ b_1 +  \rho \tilde{b}_1(\rho) \Big] s_0 + \lambda^{-1} b_3(\lambda^{-1/3} s_0, \rho) s_0^3 \Big) \\
    &= \lambda b_0^0 + \lambda^{1/3} 2^{2l/3} b_0^1 z + \lambda^{-1/3} 2^{4l/3} \tilde{b}_0(\rho) z^2 \\
    & \quad- \Big[ \lambda^{2/3} b_1 +  2^{2l/3} \tilde{b}_1(\rho) \, z\Big] s_0 + b_3(\lambda^{-1/3} s_0, \rho) s_0^3,
\end{align*}
and the remaining part of the function $\tilde{\nu}^E_{\lambda, l}$ is of the form
\begin{align}
\label{LowerThanTwo_Airy_away_nu_tilde_tilde}
\tilde{\tilde{\nu}}^E_{\lambda, l} (x) = &\lambda^{7/6} 2^{-Nl} \int e^{-i \lambda s_3 \tilde{\Phi}_3(z, \lambda^{-1/3} s_0,x,\delta,\sigma)} \\
   &\times g_3 \Big(2^l, (2^l \lambda^{-1})^{1/3}, z, \lambda^{-1/3} s_0, s_3, \delta, \sigma\Big) \mathrm{d}z \mathrm{d}s_0 \mathrm{d}s_3, \nonumber
\end{align}
where again $g_3$ has the same properties as $\tilde{g}_2$ and
in the area of integration we have $|s_0| \ll \lambda^{1/3}$.

Now, we first note that we can assume $\lambda^{-1/3} 2^{4l/3} \ll 1$ since otherwise
we can easily sum in both $\lambda$ and $l$ using the factor $2^{-Nl}$ for a sufficiently large $N$.
Next, we introduce
$$
A \coloneqq \lambda b_0^0, \qquad B \coloneqq \lambda^{1/3} 2^{2l/3} b_0^1, \qquad D \coloneqq \lambda^{2/3} b_1.
$$
We need to reduce our problem to the situation when $A, B$, and $D$ are bounded
since then we can simply apply the (one parameter) oscillatory sum lemma.
When this is the case, the size of the integration domain in \eqref{LowerThanTwo_Airy_away_nu_tilde_tilde}
is not a problem since, if we split the integration domain to the areas where $|s_0| \lesssim 2^{l/3}$ and $|s_0| \gg 2^{l/3}$,
the first part has domain size $2^{l/3}$, which is admissible,
and in the second part the amplitude is integrable in $s_0$ after using integration by parts.

{\bf{Case $|D| \gg 1$.}}
We consider two subcases.
The first subcase is when
\begin{align*}
|\lambda^{2/3} b_1 +  2^{2l/3} \tilde{b}_1(\rho) \, z| = |D +  2^{2l/3} \tilde{b}_1(\rho) \, z| > 1.
\end{align*}
Here we can actually use the Airy integral lemma (Lemma \ref{Oscillatory_auxiliary_airy}, (b))
applied to $s_0$ integration before substituting $s_0 \mapsto \lambda^{-1/3}s_0$, i.e., using the phase form \eqref{LowerThanTwo_Airy_away_Phase_tilde_3},
and obtain the bound
\begin{align*}
\Vert \tilde{\tilde{\nu}}^E_{\lambda, l} \Vert_{L^\infty} \lesssim \lambda^{7/6} 2^{-lN} \int \chi_1(z) |D +  2^{2l/3} \tilde{b}_1(\rho) \, z|^{-\varepsilon} \mathrm{d}z,
\end{align*}
for some constant $\varepsilon > 0$.
After plugging into \eqref{LowerThanTwo_Airy_away_ComplexInterpolationFinal} this is absolutely summable in $\lambda$.
Namely, in the cases $|D| \ll 2^{2l/3}$ and $|D| \gg 2^{2l/3}$ we get the estimate $|D|^{-\varepsilon}$, which is summable,
and the case $|D| \sim 2^{2l/3}$ happens for only $\mathcal{O}(1)$ $\lambda$'s, which depend on $l$.

The second subcase is when
\begin{align*}
|D +  2^{2l/3} \tilde{b}_1(\rho) \, z| \leq 1.
\end{align*}
Then necessarily again $|D| \sim |2^{2l/3}|$, and this can happen only for $\mathcal{O}(1)$ $\lambda$'s.
By \eqref{LowerThanTwo_Airy_away_nu_tilde_tilde} we have
\begin{align*}
\Vert \tilde{\tilde{\nu}}^E_{\lambda, l} \Vert_{L^\infty} \lesssim \lambda^{7/6} 2^{-lN},
\end{align*}
for maybe some different $N$.
The factor $\lambda^{7/6}$ is retained since in this case we can get an integrable factor in $s_0$ by using integration by parts.
After plugging into \eqref{LowerThanTwo_Airy_away_ComplexInterpolationFinal}
we may sum over the $\mathcal{O}(1)$ $\lambda$'s and then in $l$.

{\bf{Case $|D| \lesssim 1$, and $|A| \gg 1$ or $|B| \gg 1$.}}
The case $|A| \sim |B|$ can again happen only for $\mathcal{O}(1)$
number of $\lambda$'s and so we can assume that either $|A| \gg |B|$ or $|B| \gg |A|$. Both cases can be treated equally
and so we can assume without loss of generality that $|A| \gg |B|$. Then we can rewrite the phase in the form
\begin{align*}
\lambda \tilde{\Phi}_3(z,s_0,x,\delta,\sigma)
    &= B_0(\lambda, 2^l, z) - B_1(\lambda, 2^l, z) s_0 + b_3(\lambda^{-1/3} s_0, \rho) s_0^3,
\end{align*}
where we know that for $l$ sufficiently large $|B_0| \sim |A|$, $|B_1| \sim 2^{2l/3}$, and $|b_3| \sim 1$.

In order to simplify the situation a bit, we develop the amplitude function $g_3$ into a sum of tensor products,
separating the $s_3$ variable from the others.
It is sufficient to consider each of these tensor product terms separately, and so we can assume without loss of generality that
\begin{align*}
g_3 \Big(2^l, (2^l \lambda^{-1})^{1/3}, z, \lambda^{-1/3} s_0, s_3, \delta, \sigma\Big) =
  \tilde{g}_3 \Big(2^l, (2^l \lambda^{-1})^{1/3}, z, \lambda^{-1/3} s_0, \delta, \sigma\Big) \, \chi_1(s_3),
\end{align*}
where $\tilde{g}_3$ has the same properties as $g_3$, except it does not depend on $s_3$.

Then, after using the Fourier transform in $s_3$,
the integral in $s_0$ for the function $\tilde{\tilde{\nu}}^E_{\lambda, l}$ is of the form
\begin{align}
\label{LowerThanTwo_Airy_away_measure_E_part_final_integral}
    \int \widecheck{\chi}_1 \Big(B_0 - B_1 s_0 + b_3(\lambda^{-1/3} s_0, \rho) s_0^3\Big)
    \,\, \tilde{g}_3 \Big(2^l, (2^l \lambda^{-1})^{1/3}, z, \lambda^{-1/3} s_0, \delta, \sigma\Big) \mathrm{d}s_0,
\end{align}
where we have suppressed the variables of $B_0$ and $B_1$. One can easily check that this integral is bounded by $2^{l/3}$
by considering the situations where $|s_0| \lesssim 2^{l/3}$ and $|s_0| \gg 2^{l/3}$ separately. This is in fact true
if we use any $L^1 \cap L^\infty$ function instead of $\widecheck{\chi}_1$.

If now $|B_0 - B_1 s_0 + b_3(\lambda^{-1/3} s_0, \rho) s_0^3| \gtrsim |A|^\varepsilon$,
by using the Schwartz property we obtain the bound
\begin{align*}
\Vert \tilde{\tilde{\nu}}^E_{\lambda, l} \Vert_{L^\infty} \lesssim \lambda^{7/6} 2^{-lN} |A|^{-\varepsilon},
\end{align*}
with a different $N$, which after plugging into \eqref{LowerThanTwo_Airy_away_ComplexInterpolationFinal} is summable.

Next, if $|B_0 - B_1 s_0 + b_3(\lambda^{-1/3} s_0, \rho) s_0^3| \ll |A|^\varepsilon$, then
\begin{align*}
B_1 s_0 - b_3(\lambda^{-1/3} s_0, \rho) s_0^3 \in [B_0 - c |A|^\varepsilon, B_0 + c |A|^\varepsilon],
\end{align*}
for some small $c > 0$.
In particular, the fact $|B_0| \sim A$ gives us
\begin{align*}
|B_1 s_0 - b_3(\lambda^{-1/3} s_0, \rho) s_0^3| \sim |A|.
\end{align*}
First we consider integration over the domain $|s_0| \lesssim 2^{l/3}$.
In this case we get
$$
|B_1 s_0 - b_3(\lambda^{-1/3} s_0, \rho) s_0^3| \lesssim 2^l,
$$
which in turn implies that $|A| \lesssim 2^l$.
But this means we can trade a $2^{-l}$ factor for a $|A|^{-1}$ and so we are done.
The second part of the integral is where $|s_0| \gg 2^{l/3}$, which implies $|B_1 s_0 - b_3(\lambda^{-1/3} s_0, \rho) s_0^3| \sim |s_0|^3$,
i.e., $|s_0| \sim |A|^{1/3}$. But as the derivative of
$$
B_1 s_0 - b_3(\lambda^{-1/3} s_0, \rho) s_0^3
$$
is of size $|s_0|^2 \sim |A|^{2/3}$, then if we substitute $t = B_1 s_0 - b_3(\lambda^{-1/3} s_0, \rho) s_0^3$ in the integral
\eqref{LowerThanTwo_Airy_away_measure_E_part_final_integral}, the Jacobian is of size $|A|^{-2/3}$ and so the same $|A|^{-2/3}$ bound
holds for the integral. We are done with the estimate for the function $\nu^E_{\lambda, l}$.

\subsection{Estimates away from the Airy cone -- the estimate for $\nu^a_{\lambda, l}$}

Again substituting first $s$ for $\xi$, then $s_1$ for $z$, and then $s_0$ for $s_2^{1/(n-2)}$, we obtain the expression
\begin{align*}
\nu^a_{\lambda, l} (x) = &\lambda^{3/2} 2^{l/2} \int e^{-i \lambda s_3 \Phi_4(z,s_0,x,\delta,\sigma)} \\
   &\times g_4 \Big((2^l \lambda^{-1})^{1/3}, z, s_0, s_3, \delta, \sigma; 2^l\Big) \mathrm{d}z \mathrm{d}s_0 \mathrm{d}s_3,
\end{align*}
where $g_4$ is smooth in all of its variables and a classical symbol of order $0$ in the last $2^l$ variable, and where
\begin{align*}
\Phi_4(z,s_0,x,\delta,\sigma) \coloneqq &s_0^{n} G_5(s_0^{n-2},\delta,\sigma) - s_0^{n-1} G_3(s_0^{n-2},\delta,\sigma) x_1 - s_0^{n-2} x_2 - x_3\\
    &+(2^l \lambda^{-1})^{2/3} z (x_1 - s_0 G_1(s_0^{n-2}, \delta, \sigma))\\
    &+(2^l \lambda^{-1}) z^{3/2} q_0((2^l \lambda^{-1})^{1/3} z^{1/2}, s_0).
\end{align*}
We assume $z \sim 1$ since the case $z \sim -1$ can be treated in the same way.

We can restrict ourselves to the case $|x| \lesssim 1$ arguing in the same way as in the previous case.
In fact, we can restrict ourselves to the case $|x_1 - s_0 G_1(s_0^{n-2}, \delta, \sigma)| \ll 1$,
since otherwise we can use integration by parts in $z$.
From this it follows $|x_1| \sim 1$.
Since $G_1(s_0^{n-2},0,\sigma) = 1$, we can also localise the integration in $s_0$ to an arbitrarily small interval containing $x_1$.

\begin{lemma}
\label{LowerThanTwo_Airy_away_ComplexInterpolation_polynomial_x_1}
Define the polynomial
\begin{align*}
P(s_0; x_1, x_2, \sigma) &\coloneqq \frac{(n-1)(n-2)}{2} \sigma \beta(0) s_0^{n} - n(n-2) \sigma \beta(0) x_1 s_0^{n-1} - x_2 s_0^{n-2}.
\end{align*}
If $|x_1| \sim \sigma \sim |\beta(0)| \sim 1$, $n \geq 5$, and $|x_2| \lesssim 1$,
then
$$(n-3) P'(x_1; x_1, x_2, \sigma) = x_1 P''(x_1; x_1, x_2, \sigma),$$
and this expression is a polynomial in $(x_1,x_2)$.
\end{lemma}

\begin{proof}
Factoring out $(n-2) \sigma \beta(0)/2$ we can assume without loss of generality
\begin{align*}
P(s_0; x_1, x_2, \sigma) &\coloneqq (n-1) s_0^{n} - 2n x_1 s_0^{n-1} - \tilde{x}_2 s_0^{n-2},
\end{align*}
where $\tilde{x}_2 = (2 x_2)/[(n-2) \sigma \beta(0)]$. The first two derivatives of this polynomial are
\begin{align*}
P'(s_0; x_1, x_2, \sigma) &= n(n-1) s_0^{n-1} - 2n(n-1) x_1 s_0^{n-2} - (n-2) \tilde{x}_2 s_0^{n-3},\\
P''(s_0; x_1, x_2, \sigma) &= n(n-1)^2 s_0^{n-2} - 2n(n-1)(n-2) x_1 s_0^{n-3} - (n-2)(n-3) \tilde{x}_2 s_0^{n-4}.
\end{align*}
Plugging in $x_1$ we get
\begin{align*}
P'(x_1; x_1, x_2, \sigma) &= -n(n-1) x_1^{n-1} - (n-2) \tilde{x}_2 x_1^{n-3},\\
P''(x_1; x_1, x_2, \sigma) &= -n(n-1)(n-3) x_1^{n-2} - (n-2)(n-3) \tilde{x}_2 x_1^{n-4},
\end{align*}
and the claim follows.
\end{proof}

The coefficients of the polynomial in the above lemma come
from the first three terms of $\Phi_4(z,s_0,x,0,\sigma)$ and from Lemma \ref{LowerThanTwo_Airy_lemma_phase_taylor}.
Hence, the above lemma relates the first and the second $s_0$ derivative of $\Phi_4$ at $x_1$.

We develop the phase $\Phi_4$ in the variable $u \coloneqq x_1 - s_0 G_1(s_0^{n-2}, \delta, \sigma)$,
which is just a translation of $s_0$ to $x_1$ when $\delta = 0$.
Then we can write
\begin{align*}
\Phi_4(z,s_0,x,\delta,\sigma) =
    &b_0(x,\delta,\sigma) + b_1(x,\delta,\sigma) u + b_2(x,\delta,\sigma) u^2 + b_3(x,\delta,\sigma,u) u^3 \\
    &+(2^l \lambda^{-1})^{2/3} z u\\
    &+(2^l \lambda^{-1}) z^{3/2} q_1((2^l \lambda^{-1})^{1/3} z^{1/2}, u),
\end{align*}
where $|q_1| \sim 1$.
From Lemma \ref{LowerThanTwo_Airy_away_ComplexInterpolation_polynomial_x_1} one easily sees that
we can conclude that either $|b_1| \sim |b_2| \sim 1$ or $|b_1|, |b_2| \ll 1$.
Since $|u| \ll 1$, the case $|b_1| \sim |b_2| \sim 1$ would imply that
we can integrate by parts in $u$ and obtain a factor $\lambda^{-N}$.
Therefore, we may and shall assume that both $|b_1|$ and $|b_2|$ are very small,
and so we can apply Lemma \ref{LowerThanTwo_Airy_near_ComplexInterpolation_polynomial_lemma} to obtain $|b_3| \sim 1$
(this reduction one could have also gotten by checking the third derivative in Lemma \ref{LowerThanTwo_Airy_away_ComplexInterpolation_polynomial_x_1}).

Now note that if $|u|$ is not of size $(2^l \lambda^{-1})^{1/3}$,
then we can apply integration by parts in $z$ to gain a factor $2^{-lN}$.
In fact, after we substitute $u = (2^l \lambda^{-1})^{1/3} v$,
we can get a factor of size $2^{-lN} (1+|v|^2)^{-N/2}$ by integrating by parts in $z$.
Thus, we may restrict ourselves to the discussion of
\begin{align*}
\nu^a_{I} (x) = &\lambda^{7/6} 2^{-lN} \int e^{-i \lambda s_3 \Phi_5(z,v,x,\delta,\sigma)} \, (1+|v|^2)^{-N/2} \\
   &\times \tilde{g}_5 \Big((2^l \lambda^{-1})^{1/3}, z, (2^l \lambda^{-1})^{1/3} v, s_3, \delta, \sigma; 2^l\Big)
    (1-\chi_1(v)) \chi_0((2^l \lambda^{-1})^{1/3} v) \mathrm{d}z \mathrm{d}v \mathrm{d}s_3, \\
\nu^a_{II} (x) = &\lambda^{7/6} 2^{5l/6} \int e^{-i \lambda s_3 \Phi_5(z,v,x,\delta,\sigma)} \\
   &\times g_5 \Big((2^l \lambda^{-1})^{1/3}, z, (2^l \lambda^{-1})^{1/3} v, s_3, \delta, \sigma; 2^l\Big) \chi_1(v) \mathrm{d}z \mathrm{d}v \mathrm{d}s_3,
\end{align*}
where both $g_5$ and $\tilde{g}_5$ have the same properties as $g_4$.
In the expression for $\nu^a_{I}$ the $\chi_0((2^l \lambda^{-1})^{1/3} v)$ factor localises so that $|u| = |(2^l \lambda^{-1})^{1/3} v| \ll 1$.
Suppressing dependence on $(x,\delta,\sigma)$, the phase is of the form
\begin{align}
\label{LowerThanTwo_Airy_away_Phase_5}
\begin{split}
\lambda \Phi_5(z,v,x,\delta,\sigma) =
    &\lambda b_0 + \lambda^{2/3} 2^{l/3} b_1 v + \lambda^{1/3} 2^{2l/3} b_2 v^2 \\
    &+ 2^l \Big( b_3((2^l \lambda^{-1})^{1/3}v) v^3 + zv + z^{3/2} q_1((2^l \lambda^{-1})^{1/3} z^{1/2}, (2^l \lambda^{-1})^{1/3}v) \Big).
\end{split}
\end{align}

\medskip
\noindent
{\bf Estimates for $\nu^a_{I}$.}
In this case we plan to use the oscillatory sum lemma in $\lambda$ only and consider $2^l$ as a parameter.
Let us denote
$$
A \coloneqq \lambda b_0, \qquad  B \coloneqq  \lambda^{2/3} 2^{l/3} b_1, \qquad D \coloneqq \lambda^{1/3} 2^{2l/3} b_2.
$$
We need to reduce our problem to the case when $A$, $B$, and $D$ are bounded.
As here the integral itself is bounded by $\lesssim 1$,
we can assume that it is not the case that $|A| \sim |B|$, nor $|B| \sim |C|$, nor $|A| \sim |C|$,
since otherwise $\lambda$'s would go over a finite set, and we could sum in $l$.
Furthermore, as soon as $|A|$ (resp. $|B|$, or $|C|$) is greater than $1$, then we can automatically assume that
$|A| \gg 2^{4l}$ (resp. $|B| \gg 2^{4l}$, or $|C| \gg 2^{4l}$), since otherwise we could trade some factors $2^{-lN}$
to obtain a factor $|A|^{-\varepsilon}$ (resp. $|B|^{-\varepsilon}$, or $|D|^{-\varepsilon}$)
giving summability in $\lambda$ in the expression \eqref{LowerThanTwo_Airy_away_ComplexInterpolationFinal}.

If at least one of $|A|$, $|B|$, or $|C|$ are greater than $1$, we define
\begin{align*}
f(v,z,2^l \lambda^{-1}) \coloneqq b_3((2^l \lambda^{-1})^{1/3}v) v^3 + zv + z^{3/2} q_1((2^l \lambda^{-1})^{1/3} z^{1/2}, (2^l \lambda^{-1})^{1/3}v),
\end{align*}
and develop the function $\tilde{g}_5$ into a series of tensor products with variable $s_3$ separated, i.e.,
into a sum with terms of the form
\begin{align*}
h ((2^l \lambda^{-1})^{1/3}, z, (2^l \lambda^{-1})^{1/3} v, \delta, \sigma; 2^l ) \, \chi_1(s_3),
\end{align*}
where $h$ has the same properties as $\tilde{g}_5$, except it does not depend on $s_3$.
Then after taking the Fourier transform in $s_3$, we are reduced to estimating the integral
\begin{align}
\label{LowerThanTwo_Airy_away_nu_I_integral}
\lambda^{7/6} 2^{-lN} &\int (1+|v|^2)^{-N/2}  \, \widecheck{\chi}_1 (\lambda b_0 + \lambda^{2/3} 2^{l/3} b_1 v + \lambda^{1/3} 2^{2l/3} b_2 v^2 + 2^l f(v,z,2^l \lambda^{-1})) \nonumber \\
   &\quad \times \chi_1(z) (1-\chi_1(v)) \chi_0((2^l \lambda^{-1})^{1/3} v) \\
   &\quad \times h ((2^l \lambda^{-1})^{1/3}, z, (2^l \lambda^{-1})^{1/3} v, \delta, \sigma; 2^l ) \mathrm{d}z \mathrm{d}v. \nonumber
\end{align}

{\bf Case $|v| \ll 1$.}
The bound $|v| \ll 1$ gives
\begin{align*}
|f(v,z,2^l \lambda^{-1})| &\sim 1,\\
|\partial_v f(v,z,2^l \lambda^{-1})| &\sim 1,\\
|\partial^2_v f(v,z,2^l \lambda^{-1})| &\sim |v|.
\end{align*}
If $|A| \gg \max\{2^{4l}, |B|, |D|\}$, then we can easily gain a factor $|A|^{-1}$ using the Schwartz property of $\widecheck{\chi}_1$.
If $|B| \gg \max\{2^{4l}, |A|, |D|\}$, then the size of the derivative in $v$ of the function within $\widecheck{\chi}_1$ is $B$
and so we get the bound $|B|^{-1}$ by substitution.
Finally, if $|D| \gg \max\{2^{4l}, |A|, |B|\}$, we use the van der Corput lemma and obtain the bound $|D|^{-1/2}$.

{\bf Case $1 \ll |v| \ll (2^l \lambda^{-1})^{-1/3}$.}
In this case we can rewrite
\begin{align*}
f(v,z,2^l \lambda^{-1}) = v^3 \tilde{f}(v,z, (2^l \lambda^{-1})^{-1/3}),
\end{align*}
where $\tilde{f}$ is a smooth function with $|\tilde{f}| \sim 1$ and $|\partial^{k}_v \tilde{f}| \ll |v|^{-k}$ for all $k \geq 1$.
This means that $f$ is behaving essentially like $v^3$, and in particular
\begin{align*}
|f(v,z,2^l \lambda^{-1})| &\sim |v|^3,\\
|\partial_v f(v,z,2^l \lambda^{-1})| &\sim |v|^2.
\end{align*}

\noindent
{\bf Subcase $\max\{|B|,|D|\} \geq 1$.}
As mentioned before, this actually implies that we can assume $\max\{|B|,|D|\} \geq 2^{4l}$.
If now $|D| \gg |B|$, then since we could otherwise use the factor $(1+|v|^2)^{-N/2}$ in \eqref{LowerThanTwo_Airy_away_nu_I_integral},
we can restrain the integration to the domain $|v| \ll |D|^\varepsilon$.
Here the derivative in $v$ of the expression
\begin{align}
\label{LowerThanTwo_Airy_away_ComplexInterpolation_nondeg_derivative}
A+Bv+Dv^2+2^l \tilde{f}(v,z,2^l \lambda^{-1}) v^3
\end{align}
inside the Schwartz function $\widecheck{\chi}_1$ in \eqref{LowerThanTwo_Airy_away_nu_I_integral}
is of size $|B+c Dv|$ for some $|c| = |c(v)| \sim 1$.
But recall that $|v| \gg 1$ and so $|B+c Dv| \sim |Dv| \gg |D|$.
This means that substituting the above expression would give a Jacobian of size at most $|D|^{-1}$.

Next let us consider the case $|D| \lesssim |B|$.
If have the slightly stronger estimate $|D| \lesssim |B|^{1-\varepsilon}$,
and if we assume $|v| \ll |B|^\varepsilon$ (which we can because of the factor $(1+|v|^2)^{-N/2}$),
then in this case the derivative of \eqref{LowerThanTwo_Airy_away_ComplexInterpolation_nondeg_derivative} is of size $|B|$,
which means substituting this expression yields an admissible bound.

Therefore, we may now consider the case $|B|^{1-\varepsilon} \ll |D| \lesssim |B|$ and $|v| \ll |D|^\varepsilon$,
which implies, in case when $\varepsilon$ is sufficiently small, that $|D| \geq 2^{3l}$.
In particular, the derivative of \eqref{LowerThanTwo_Airy_away_ComplexInterpolation_nondeg_derivative}
can be again written as $|B+c Dv|$ with $|c| \sim 1$,
and we can reduce our problem to the part where $|B+c Dv| \ll |D|^\varepsilon$,
since otherwise substituting would give a Jacobian of size at most $|D|^{-\varepsilon}$.
But now $|B+c Dv| \ll |D|^\varepsilon$ implies $|v| \sim |B| |D|^{-1}$.
Hence, it suffices to estimate the integral
\begin{align*}
\int \Big| \widecheck{\chi}_1 (A+Bv+Dv^2+2^l \tilde{f}(v,z,2^l \lambda^{-1}) v^3) \chi_1(|B|^{-1} |D| v) \Big| \mathrm{d}v.
\end{align*}
We substitute $w = |B|^{-1}\,|D| v$ and write
\begin{align*}
|B| |D|^{-1} \int
\Big|
\widecheck{\chi}_1 (A+ (B |B| |D|^{-1}) w+ (D|B|^2 |D|^{-2}) w^2 + 2^l |B|^3 |D|^{-3} w^3 r(w)) \chi_1(w)
\Big| \mathrm{d}w.
\end{align*}
Applying the van der Corput lemma we obtain the estimate
$$
(|B| |D|^{-1}) \, (|B|^2 |D|^{-1})^{-1/2} = |D|^{-1/2},
$$
and so we are done with the case $\max\{|B|,|D|\} \geq 1$.

\noindent
{\bf Subcase $\max\{|B|,|D|\} \leq 1$ and $|A| \gg 1$.}
Again, we may actually assume $|A| \gg 2^{4l}$.
We may also then reduce ourselves to the discussion of the case $|v| \ll |A|^\varepsilon$,
since in the other part of the integration domain we can gain a factor $|A|^{-\varepsilon}$.
But then the expression \eqref{LowerThanTwo_Airy_away_ComplexInterpolation_nondeg_derivative}
is of size $\sim |A|$ and we can get a factor $|A|^{-1}$,
and hence we are also done with the function $\nu_I^a$.

\medskip
\noindent
{\bf Estimates for $\nu^a_{II}$.}
Here we have a non-degenerate critical point in $z$ which would give us a factor $2^{-l/2}$.
We shall not apply directly the stationary phase method here
since in this case some crucial information has been lost while we were deriving the form of the phase in this and the previous subsections.
It seems that one cannot prove the required bound for complex interpolation using the information from the form of the phase \eqref{LowerThanTwo_Airy_away_Phase_5}.
One needs to go back to the phase form in the original coordinates (the one before taking the inverse Fourier transform is \eqref{LowerThanTwo_Airy_Phase_first_form})
and find the critical point in the variables $(y_1,s_1)$.
This was carried out in \cite{IM16} (see the discussion before \cite[Lemma 5.6.]{IM16}).
Here we only sketch the steps.

The phase in \eqref{LowerThanTwo_Airy_Phase_first_form} is
\begin{align*}
\Psi(y_1, \delta, \sigma, s_1, s_2)
   = s_1 y_1 + s_2 y_1^2 \omega(\delta_1 y_1) + \sigma y_1^n \beta(\delta_1 y_1)
   + (\delta_0 s_2)^2 Y_3(\delta_1 y_1, \delta_2, \delta_0 s_2),
\end{align*}
and one integrates in the $y_1$ variable.
The phase function after one applies the Fourier transform is
\begin{align}
\label{LowerThanTwo_Airy_away_Phase_0}
\Phi_0(y_1, s_1, s_2, x, \delta, \sigma)
   = \Psi(y_1, \delta, \sigma, s_1, s_2) - s_1 x_1 - s_2 x_2 - x_3,
\end{align}
and one now integrates in the $s$ and $y_1$ variables, after substituting $s$ for $\xi$.
Recall that $s_0 = s_2^{1/(n-2)}$ and
\begin{align*}
s_1 &= s_0^{n-1} G_3(s_0^{n-2},\delta,\sigma)-\lambda^{-2/3}z, \\
v   &= (2^{l} \lambda^{-1})^{-1/3} (x_1 - s_0 G_1(s_0^{n-2},\delta,\sigma)).
\end{align*}
Therefore fixing $(s_2,s_3)$ is equivalent to fixing $(v,s_3)$, and in this case, finding the
critical point in $(y_1, s_1)$ is equivalent to finding the critical point in the $(y_1, z)$ coordinates.
Recall that the phase form in \eqref{LowerThanTwo_Airy_away_Phase_5} was derived by using the stationary phase method in $y_1$
(implicitly done as a part of Lemma \ref{Oscillatory_auxiliary_airy}) and changing variables from $s = (s_1, s_2, s_3)$ to $(z,v,s_3)$.

The key is now to notice that since the critical point is invariant with respect to coordinate changes, and so,
after applying the stationary phase in $z$ to the phase function \eqref{LowerThanTwo_Airy_away_Phase_5},
we get
\begin{align*}
\Phi_5(z^c,v,x,\delta,\sigma),
\end{align*}
which is the equal to the
phase function in \eqref{LowerThanTwo_Airy_away_Phase_0} after we apply the stationary phase in $(y_1,s_1)$:
\begin{align*}
\Phi_0(y_1^c, s_1^c, s_2, x, \delta, \sigma),
\end{align*}
and then change the coordinates from $s_2$ to $v$.
This was carried out in \cite{IM16} by explicitly calculating the critical point
in $(y_1,s_1)$ in \eqref{LowerThanTwo_Airy_away_Phase_0} (see \cite[Lemma 5.6]{IM16}).
One obtains that we can rewrite the function $\nu_{II}^a$ as
\begin{align*}
\nu^a_{II} (x) = \lambda^{7/6} 2^{l/3} \int e^{-i \lambda s_3 \Phi_6(\tilde{v},x,\delta,\sigma)}
   g_6 \Big((2^l \lambda^{-1})^{1/3}, \tilde{v}, s_3, \delta, \sigma; 2^l\Big) \chi_1(\tilde{v}) \mathrm{d}\tilde{v} \mathrm{d}s_3,
\end{align*}
where
\begin{align*}
\lambda \Phi_6(\tilde{v},x,\delta,\sigma) = &\lambda \tilde{b}_0(x,\delta,\sigma) + \lambda^{2/3} 2^{l/3} \tilde{b}_1 (x,\delta,\sigma) \tilde{v}\\
   &+ \delta_0^2 2^{2l/3} \lambda^{1/3} \tilde{b}_2(x,\delta_0 (2^l \lambda^{-1})^{1/3} \tilde{v},\delta,\sigma) \tilde{v}^2,
\end{align*}
with $\tilde{b}_0$, $\tilde{b}_1$, $\tilde{b}_2$ smooth, and $|\tilde{b}_2| \sim 1$.
The amplitude $g_6$ is a classical symbol of order $0$ in $2^l$, but we shall ignore this dependence since
the lower order terms can be treated similarily,
and even simpler since we can gain summability in $l$ and use the one parameter oscillatory sum lemma for $\lambda$.

We remark that the variable $\tilde{v}$ is only slightly different from the variable $v$ defined above
after the statement of Lemma \ref{LowerThanTwo_Airy_away_ComplexInterpolation_polynomial_x_1}.
Here $\tilde{v}$ corresponds to the $v$ variable of \cite[Subsection 5.2.3]{IM16}.
We explain briefly the relation between $v$ and $\tilde{v}$.
At the beginning of this subsection we obtained $\nu^a_{II}$ by localising to the part where
$$
|(2^l \lambda^{-1})^{1/3} v| = |x_1 - s_0 G_1(s_0^{n-2},\delta,\sigma)| = |x_1 - s_2^{n-2} G_1(s_2,\delta,\sigma)| \sim (2^l \lambda^{-1})^{1/3},
$$
i.e., $|v| \sim 1$.
Since $G_1(s_2,0,\sigma) = 1$, one can easily see by using the implicit function theorem that solving the equation
\begin{align*}
x_1 - s_2^{n-2} G_1(s_2,\delta,\sigma) = (2^l \lambda^{-1})^{1/3} v
\end{align*}
in $s_2$ one can write
\begin{align*}
s_2 = \tilde{G}_1(x_1, \delta, \sigma) + (2^l \lambda^{-1})^{1/3} v \, \tilde{G}((2^l \lambda^{-1})^{1/3} v, x_1, \delta, \sigma),
\end{align*}
where $|\tilde{G}| \sim \tilde{G}_1 \sim 1$.
Therefore if the $\tilde{v}$ variable is defined by
\begin{align*}
\tilde{v} = (2^l \lambda^{-1})^{-1/3} (s_2 - \tilde{G}_1(x_1,\delta,\sigma)),
\end{align*}
as is $v$ of \cite{IM16}, then
\begin{align*}
\tilde{v} = v \tilde{G}((2^l \lambda^{-1})^{1/3} v, x_1, \delta, \sigma).
\end{align*}
In particular, there is no significant difference between $v$ and $\tilde{v}$.

We define
$$
A \coloneqq \lambda \tilde{b}_0(x,\delta,\sigma), \qquad
B \coloneqq \lambda^{2/3} 2^{l/3} \tilde{b}_1 (x,\delta,\sigma), \qquad
D \coloneqq \delta_0^2 2^{2l/3} \lambda^{1/3},
$$
suppress the variables of $\tilde{b}_2$, and shorten $\rho = \delta_0 (2^l \lambda^{-1})^{1/3}$.
Then
\begin{align*}
\lambda \Phi_6(\tilde{v},x,\delta,\sigma) = A + B \tilde{v} + D \tilde{b}_2(\rho \tilde{v}) \tilde{v}^2,
\end{align*}
and in order to use the oscillatory sum lemma for two parameters we need to reduce the problem to the situation where
$|A|$, $|B|$, and $|D|$ are of size $\lesssim 1$. In the following we define $k$ through $\lambda = 2^{k}$.

First we treat the case when at least two of $|A|$, $|B|$, and $|D|$ are comparable. When this is the case, $\lambda$ can go over
only a finite set of indices (the index sets depending on $l$ and other constants), and it remains to sum only in $l$.
This is done in the following way.
If $|D| \gtrsim 1$, then we can use van der Corput lemma and obtain a factor $|D|^{-1/2}$, which is summable in $l$.
If $|D| \ll 1$, then the only case remaining is $|A| \sim |B|$, and here we can use integration by parts in $\tilde{v}$ and obtain
a factor $|B|^{-1}$ which we use to sum in $l$.

Next, we assume that we have a ``strict order'' between $|A|$, $|B|$, and $|D|$.
First we shall consider the cases when at least two of $|A|$, $|B|$, and $|D|$ are greater than $1$.
If $|A| \gg \max\{|B|, |D|\} \gtrsim 1$, we use integration by parts in $s_3$ and obtain
$$
|A|^{-1} \ll |A|^{-1/2} |\max\{|B|, |D|\}|^{-1/2},
$$
which is summable.
Similarly, if $|B| \gg \max\{|A|, |D|\} \gtrsim 1$, we can integrate by parts in $\tilde{v}$ and obtain the estimate
$$
|B|^{-1} \ll |B|^{-1/2} |\max\{|A|, |D|\}|^{-1/2},
$$
which is summable.
And if now $|D| \gg \max\{|A|, |B|\} \gtrsim 1$, we use the van der Corput lemma and obtain
$$
|D|^{-1/2} \ll |D|^{-1/4} |\max\{|A|, |B|\}|^{-1/4},
$$ which is again summable.
We are thus reduced to the case where one of $|A|$, $|B|$, or $|D|$ are greater than $1$, and the other two much smaller.

{\bf Case $|A| \geq 1$ and $\max\{|B|, |D|\} \ll 1$.}
In this case by using integration by parts in $s_3$ we can get a factor $|A|^{-1}$.
We use the one dimensional oscillatory sum lemma in $l$,
and afterwards, we can sum in $\lambda$ using the factor $|A|^{-1}$
which can be obtained as the bound on the $C^1$ norm of the function to which we applied the oscillatory sum lemma.

{\bf Case $|B| \geq 1$ and $\max\{|A|, |D|\} \ll 1$.}
Here we change the summation variables
\begin{align*}
2^{k_1} &\coloneqq \lambda^{2} 2^l,\\
2^{k_2} &\coloneqq \lambda,
\end{align*}
so that we now sum over $(k_1,k_2)$.
This change of variables corresponds to the system
\begin{align*}
k_1 &= 2k + l,\\
k_2 &= k,
\end{align*}
which has determinant equal to $1$,
and so the associated linear mapping is a bijection on $\Z^2$.

Since the summation bounds (without the constraints set by $A$, $B$, or $D$) are $1 \ll \lambda \leq \delta_0^{-6}$ and $1 \ll 2^l \ll \lambda$,
for each fixed $k_1$ the summation in $k_2$ is now within the range $2^{k_1/3} \ll 2^{k_2} \ll 2^{k_1/2}$,
and the summation in $k_1$ is for $1 \ll 2^{k_1} \ll \delta_0^{-18}$.

The quantities $B$ and $D$ can be rewritten as
\begin{align*}
B &= 2^{k_1/3} \tilde{b}_1 (x,\delta,\sigma),\\
D &= \delta_0^2 \, 2^{2k_1/3-k_2}.
\end{align*}
Now for a fixed $k_1$ we can apply the one-dimensional oscillatory sum lemma to sum in $2^{k_2}$
since all the terms coupled with $2^{k_2}$ are now within a bounded range.
In order to sum in $k_1$, one needs to estimate the $C^1$ norm of the function to which we have applied the oscillatory sum lemma.
One can easily see that integrating by parts in $s_0$ we obtain a factor $|B|^{-1}$
which in the new indices depends only on $2^{k_1}$. 

{\bf Case $|D| \geq 1$ and $\max\{|A|, |B|\} \ll 1$.}
In this case we also change the summation variables
\begin{align*}
2^{k_1} &\coloneqq \lambda 2^{2l},\\
2^{k_2} &\coloneqq \lambda,
\end{align*}
so that we now sum over $(k_1,k_2)$.
We have
\begin{align*}
\begin{split}
k_1 &= k+2l,\\
k_2 &= k,
\end{split}
\qquad\qquad\qquad
\Longleftrightarrow
\begin{split}
k &= k_2,\\
l &= (k_1-k_2)/2.
\end{split}
\end{align*}
Therefore when we fix $k_1$, the summation in $k_2$ goes over an interval of even or uneven integers, depending on the parity of $k_1$.
Since the summation bounds (without the constraints set by $A$, $B$, or $D$) are $1 \ll \lambda \leq \delta_0^{-6}$ and $1 \ll 2^l \ll \lambda$,
for each $k_1$ the summation in $k_2$ is now within the range $2^{k_1/3} \ll 2^{k_2} \ll 2^{k_1}$,
and the summation in $k_1$ is for $1 \ll 2^{k_1} \ll \delta_0^{-18}$.

The quantities $B$ and $D$ can be rewritten as
\begin{align*}
B &= 2^{k_1/2+3k_2/2} \tilde{b}_1 (x,\delta,\sigma),\\
D &= \delta_0^2 \, 2^{k_1/3}.
\end{align*}
For a fixed $k_1$ we want to apply the oscillatory sum lemma to the summation in $k_2$.
We remark that formally one should write $k_2$ as either $2r+1$ or $2r$ (depending on the parity of $k_1$),
and then apply the oscillatory sum lemma to the summation in $r$ instead of $k_2$.

Here we give a bit more details compared to the previous case
since the term $\rho$, which contains $(2^l \lambda^{-1})^{1/3}$, is coupled with $D$.
We need to estimate the $C^1$ norm of the function
\begin{align*}
H(z_1,z_2,z_3; x, \delta, \sigma) \coloneqq \int e^{-i s_3 (z_1 + z_2 \tilde{v} + D \tilde{b}_2(x,z_3 \delta_0 \tilde{v},\delta,\sigma) \tilde{v}^2)}
   g_6 (z_3, \tilde{v}, s_3, \delta, \sigma) \chi_1(\tilde{v}) \mathrm{d}\tilde{v} \mathrm{d}s_3.
\end{align*}
Formally, one should also add further dummy $z_i$'s for controlling the range of the summation indices.
Since we are in the case where $|D| \geq 1$, $|z_1| \ll 1$, and $|z_2| \ll 1$,
integrating by parts in $s_3$ we get that the $L^\infty$ estimate is $|D|^{-1}$.
Taking derivatives in $z_1$ and $z_2$ does not change the form of the integral in an essential way,
and so we can also estimate the $L^\infty$ norm of the these derivatives by $|D|^{-1}$.
Taking the derivative in $z_3$ a factor of size at most $|D|$ appears, but now we just apply integration by parts in $s_3$ two times
and get that we can estimate the $C^1$ norm of $H$ by $|D|^{-1}$.

{\bf Case $|A| \lesssim 1$, $|B| \lesssim 1$, and $|D| \lesssim 1$.}
Here we apply the two-parameter oscillatory sum lemma.
We only need to check the additional linear independence condition appearing in the assumptions of Lemma \ref{Oscillatory_auxiliary_two_parameter}.
The terms where $\lambda = 2^k$ and $2^l$ appear are
\begin{align*}
A &= 2^{\beta_1^1 k} \tilde{b}_0(x,\delta,\sigma), &
B &= 2^{\beta_1^2 k + \beta_2^2 l} 2^{l/3} \tilde{b}_1 (x,\delta,\sigma),\\
D &= \delta_0^2 \, 2^{\beta_1^3 k + \beta_2^3 l}, &
(2^l \lambda^{-1})^{1/3} &= 2^{\beta_1^4 k + \beta_2^4 l},
\end{align*}
where
\begin{align*}
(\beta_1^1,\beta_2^1) &= (1,0), &
(\beta_1^2,\beta_2^2) &= (2/3,1/3),\\
(\beta_1^3,\beta_2^3) &= (1/3,2/3), &
(\beta_1^4,\beta_2^4) &= (-1/3,1/3),
\end{align*}
and recall from \eqref{LowerThanTwo_Airy_away_ComplexInterpolationFinal} that
\begin{align*}
(\alpha_1,\alpha_2) &= (-7/4,-1/2).
\end{align*}
Formally, we also have to consider additionally
\begin{align*}
(\beta_1^5,\beta_2^5) &= (-1,0), &
(\beta_1^6,\beta_2^6) &= (0,-1),
\end{align*}
for implementing the lower summation bounds for $\lambda$ and $2^l$ as in \eqref{LowerThanTwo_Airy_away_ComplexInterpolationFinal}.
We see that the condition $\alpha_1 \beta_2^r \neq \alpha_2 \beta_1^r$ is satisfied for each $r = 1,\ldots,6$.
Therefore, we may now apply the lemma and obtain the inequality \eqref{LowerThanTwo_Airy_away_ComplexInterpolationFinal}.
This finishes the proof of Theorem \ref{Section_h_lin_less_than_2_main_theorem}.




\pagebreak


\begin{thebibliography}{1}

\bibitem{AKC79} G.I. Arhipov, A.A. Karacuba, V.N. \v{C}ubarikov, {\em Trigonometric integrals.}
Izv. Akad. Nauk. SSSR Ser. Mat., 43 (1979), 971--1003, 1197, in Russian.
English translation in Math. USSR-Izv., 15 (1980), 211--239.


\bibitem{BCP62} A. Benedek, A.-P. Calder\'on, R. Panzone, {\em Convolution operators on Banach space valued functions.}
Proc. Nat. Acad. Sci. U.S.A. 48 (1962) 356--365.

\bibitem{Bou91} J. Bourgain, {\em Besicovitch-type maximal operators and applications to Fourier analysis.}
Geom. Funct. Anal. 1, no. 2 (1991), 147--187.

\bibitem{BG11} J. Bourgain, L. Guth, {\em Bounds on oscillatory integral operators.}
C. R. Acad. Sci. Paris, Ser. I, 349 (2011) 137--141.

\bibitem{CS72} L. Carleson, P. Sj\"{o}lin, {\em Oscillatory integrals and a multiplier problem for the disc.}
Studia Math., 44 (1972), 287--299

\bibitem{vdC21} J.G. van der Corput, {\em Zahlentheoretische Absch\"atzungen.}
Math. Ann., 84 (1921), 53--79.

\bibitem{Do77} Y. Domar, {\em On the Banach algebra A(G) for smooth sets $\Gamma \subset \R^n$.}
Comment. Math. Helv., 52 (1977), no. 3, 357--371.


\bibitem{Dui74} J.J. Duistermaat, {\em Oscillatory integrals, Lagrange immersions and unfolding of singularities.}
Comm. Pure Appl. Math., 27 (1974), 207--281.

\bibitem{Fef70} C. Fefferman, {\em Inequalities for strongly singular convolution operators.}
Acta Math., (1974), 9--36.

\bibitem{Fer87} D.L. Fernandez, {\em Vector-valued singular integral operators on $L^p$-opaces with mixed norms and applications.}
Pacific J. Math., 129 (1987), no. 2, 257--275.

\bibitem{GV92} J. Ginibre, G. Velo, {\em Smoothing properties and retarded estimates for some dispersive evolution equations.}
Commun. Math. Phys., 144 (1992), 163--188.

\bibitem{Grb09} M. Greenblatt, {\em The asymptotic behavior of degenerate oscillatory integrals in two dimensions.}
J. Funct. Anal., 257 (6) (2009), 1759--1798.

\bibitem{Grl81} A. Greenleaf, {\em Principal curvature and harmonic analysis.}
Indiana Univ. Math. J., 30 (4) (1981), 519--537.

\bibitem{Gu16} L. Guth, {\em A restriction estimate using polynomial partitioning.}
J. Amer. Math. Soc. 29 (2016), 371--413.

\bibitem{IKM10} I.A. Ikromov, M. Kempe, D. M\"{u}ller, {\em Estimates for maximal functions associated to hypersurfaces in $\R^3$ and related problems of harmonic analysis.}
Acta Math. 204 (2010), 151--271.

\bibitem{IM11a} I.A. Ikromov, D. M\"{u}ller, {\em On adapted coordinate systems.}
Trans. Amer. Math. Soc., 363 (2011), no. 6, 2821--2848.

\bibitem{IM11b} I.A. Ikromov, D. M\"{u}ller, {\em Uniform estimates for the Fourier transform of surface carried measures in $\R^3$ and an application to Fourier restriction.}
J. Fourier Anal. Appl., 17 (2011), no. 6, 1292--1332.

\bibitem{IM16} I.A. Ikromov, D. M\"{u}ller, {\em Fourier Restriction for Hypersurfaces in Three Dimensions and Newton Polyhedra.}
Princeton University Press 2016.

\bibitem{IS97} A. Iosevich, E. Sawyer, {\em Maximal averages over surfaces.}
Adv. Math., 132 (1997), 46--119.

\bibitem{Kar84} V.N. Karpushkin, {\em A theorem on uniform estimates for oscillatory integrals with a phase depending on two variables.}
Trudy Semin. Petrovsk., 10 (1984), 150--169, 238, in Russian.
English translation in J. Sov. Math. 35, 2809--2826 (1986)

\bibitem{KM93} S. Klainerman, M. Machedon, {\em Space-time estimates for null forms and the local existence theorem.}
Comm. Pure Appl. Math., Vol. 46, no. 9 (1993), 1221--1268.

\bibitem{KT98} M. Keel, T. Tao, {\em Endpoint Strichartz Estimates.}
American Journal of Mathematics, Vol. 120 (1998), no. 5, 955--980.

\bibitem{LV10} S. Lee, V. Vargas, {\em Restriction estimates for some surfaces with vanishing curvatures.}
J. Funct. Anal., 258. (2010), no. 9, 2884--2909.

\bibitem{Liz70} P.I. Lizorkin, {\em Multipliers of Fourier integrals and bounds of convolutions in spaces with mixed norm. Applications.}
Izv. Akad. Nauk SSSR Ser. Mat. 34 (1970), 218--247, in Russian.
English translation in Mathematics of the USSR-Izvestiya, Vol. 4, no. 1 (1970), 225--255.

\bibitem{Mo98} S.J. Montgomery-Smith, {\em Time decay for the bounded mean oscillation of solutions of the Schr\"{o}dinger and wave equations.}
Duke Math J., Vol. 91, no. 2 (1998), 393--408.

\bibitem{MVV96} A. Moyua, A. Vargas, L. Vega, {\em Schr\"{o}dinger maximal function and restriction properties of the Fourier transform.}
Internat. Math. Res. Notices 16 (1996), 793--815.

\bibitem{PS97} D.H. Phong, E.M. Stein, {\em The Newton polyhedron and oscillatory integral operators.}
Acta Math., 179 (1997), no. 1, 105--152.

\bibitem{PSS97} D.H. Phong, E.M. Stein, J.A. Sturm, {\em On the growth and stability of real-analytic functions.}
Amer. J. Math., 121 (1999), no. 3, 519--554.

\bibitem{Si74} D. Siersma, {\em Classification and deformation of singularities.}
Doctoral dissertation, University of Amsterdam, (1974), 1--115.


\bibitem{Ste93} E.M. Stein, {\em Harmonic analysis: Real-variable methods, orthogonality, and oscillatory integrals.}
Princeton Mathematical Series 43. Princeton University Press, Princeton, NJ, 1971.

\bibitem{Str77} R. S. Strichartz, {\em Restrictions of Fourier transforms to quadratic surfaces and decay of solutions of wave equations.}
Duke Math. J., 44 (1977), 705--714.

\bibitem{T03} T. Tao, {\em A sharp bilinear restriction estimate for paraboloids.}
Geom. Funct. Anal. 13 (2003) no. 6, 1359--1384.

\bibitem{To75} P.A. Tomas, {\em A restriction theorem for the Fourier transform.}
Bull. Amer. Math. Soc. 81 (1975), 477--478.

\bibitem{Var76} A.N. Varchenko, {\em Newton polyhedra and estimates of oscillating integrals.}
Funkcional. Anal. i Priložen, 10 (1976), 13--38, in Russian.
English translation in Functional Anal. Appl., 18 (1976), 175--196.

\bibitem{W95} T. Wolff, {\em An improved bound for Kakeya type maximal functions.}
Revista Mat. Iberoamericana, 11 (1995), 651--674.

\bibitem{W01} T. Wolff, {\em A sharp bilinear cone restriction estimate.}
Ann. of Math. (2), 153 (2001) no. 5, 661--698.

\bibitem{Zyg74} A. Zygmund, {\em On Fourier coefficients and transforms of functions of two variables.}
Studia Math., 50 (1974), 189--201.

\end{thebibliography}
\end{document}